%% file: main.tex
\documentclass[preprint,fleqn,10pt]{cls/elsarticle}
\input{setup/format}
\input{setup/package}
\input{setup/command}

\begin{document}
	\begin{frontmatter}
		\title{
  Error Analysis and Numerical Algorithm for PDE Approximation  with Hidden-Layer Concatenated Physics Informed Neural Networks 
  } 
        \author[lab_1]{Yanxia Qian}\ead{yxqian@xaut.edu.cn}
        \author[lab_2]{Yongchao Zhang}\ead{yoczhang@nwu.edu.cn}
        \author[lab_3]{Suchuan Dong\corref{label4}} \ead{sdong@purdue.edu}
		
		\address[lab_1]{Department of Mathematics, School of Science, Xi'an University of Technology, Xi'an, Shaanxi, 710048, China}
		\address[lab_2]{School of Mathematics, Northwest University, Xi'an, Shaanxi 710069, China}
		\address[lab_3]{Center for Computational and Applied Mathematics, Department of Mathematics, Purdue University, West Lafayette, IN47907, USA}
        \cortext[label4]{Author of Correspondence.}
		
		
		\input{content/abstract}
	\end{frontmatter}


	\input{content/introduction}
	\input{content/pinn}

	\input{content/heat}

	\input{content/burger}
    \input{content/wave}
    \input{content/sinegordon}
    \input{content/numerical_examples}
    \input{content/summary}

\section*{Acknowledgments}
The work was partially supported by the NSF of China (No.12101495), China Postdoctoral Science Foundation (No.2021M702747), Natural Science Foundation of Hunan Province (No.2022JJ40422), General Special Project of Education Department of Shaanxi Provincial Government (No.21JK0943), and the US National Science Foundation (DMS-2012415).
    
	\input{content/auxiliary}

	\setlength{\bibsep}{0.5ex}  
	\bibliography{main.bbl}
	
\end{document}

%% file: setup/format.tex
\graphicspath{./figures}

%% file: setup/package.tex
\usepackage{amsmath,amssymb,amsthm}
\allowdisplaybreaks[4]  
\usepackage[hidelinks,colorlinks,linkcolor=blue]{hyperref}
\usepackage{cleveref}
\usepackage{graphicx}
\usepackage{graphics}
\usepackage {graphicx,fancyhdr}
\usepackage{graphics}
\usepackage{color}
\usepackage{flafter}
\usepackage{multirow}
\usepackage{mathrsfs} 
\usepackage{subfig}
\usepackage{tabularx}
\usepackage{lineno}
\usepackage{xfrac}
\usepackage{rotating} 
\usepackage{stmaryrd} 
\usepackage{bm} 
\usepackage{cases} 
\usepackage{enumerate}
\usepackage{ulem}  
\usepackage{url}

\usepackage{enumitem}
\setlist{nolistsep}

\modulolinenumbers[5]  

%% file: setup/command.tex
\newcommand{\dx}{\ensuremath{{\,\hspace{-0.06em}\rm{d}}\bm{x}}}
\newcommand{\dxx}{\ensuremath{{\,\hspace{-0.06em}\rm{d}}x}}
\newcommand{\ds}{\ensuremath{{\,\hspace{-0.06em}\rm{d}}s(\bm{x})}}
\newcommand{\dt}{\ensuremath{{\,\hspace{-0.06em}\rm{d}}t}}
\newcommand{\diff}[1]{\ensuremath{{\,\hspace{-0.06em}\rm{d}}{#1}}}

\newcommand{\comment}[1]{\iffalse #1 \fi }

\theoremstyle{plain}

\newtheorem{Theorem}{Theorem}[section]
\newtheorem{Lemma}[Theorem]{Lemma}
\newtheorem{Def}[Theorem]{Definition}
\newtheorem{Remark}[Theorem]{Remark}
{\normalfont}


%% file: content/abstract.tex
\begin{abstract}

We present the hidden-layer concatenated physics informed neural network (HLConcPINN) method, which combines hidden-layer concatenated feed-forward neural networks, a modified  block time marching strategy, and a physics informed approach for approximating partial differential equations (PDEs). We analyze the convergence properties and establish the error bounds of this method for two types of PDEs: parabolic (exemplified by the heat and Burgers' equations) and hyperbolic (exemplified by the wave and nonlinear Klein-Gordon equations). We  show that its approximation error of the solution can be effectively controlled by the training loss for dynamic simulations with long time horizons. The HLConcPINN method in principle allows an arbitrary number of hidden layers not smaller than two and any of the commonly-used smooth activation functions for the hidden layers beyond the first two, with theoretical guarantees. This generalizes several recent neural-network techniques, which have theoretical guarantees but are confined to two hidden layers in the network architecture and the $\tanh$ activation function. Our theoretical analyses subsequently inform the formulation of appropriate training loss functions for these PDEs, leading to physics informed neural network (PINN) type computational algorithms that differ from the standard PINN formulation. Ample numerical experiments are presented based on the proposed algorithm to validate the effectiveness of this method and confirm aspects of the theoretical analyses. 

\end{abstract}

\begin{keyword}
physics informed neural network; hidden-layer concatenation; block time-marching; deep neural networks; deep learning; scientific machine learning
\end{keyword}

%% file: content/introduction.tex
\section{Introduction}\label{sec:intro}

The rapid growth in data availability and computing resources has ushered in a transformative era in machine learning and data analytics, fueling remarkable advancements across diverse scientific and engineering disciplines \cite{LeCun2015DP}. These breakthroughs have a significant impact on fields such as natural language processing, robotics, computer vision, speech and image recognition, and genomics. Of particular promise is the use of neural network (NN) based approaches to tackle challenges such as high-dimensional problems, including  high-dimensional partial differential equation (PDE). This is due to the intractable computational workload caused by the curse of dimensionality associated with conventional numerical techniques, rendering such techniques practically infeasible. Deep learning algorithms, on the other hand, can offer invaluable support.  Pertaining to PDE problems specifically, neural network methods provide implicit regularization, exhibiting a great potential to alleviate or overcome challenges related to high dimensionality \cite{Beck2019Machine, Berner2020Analysis}.

This surge of progress has driven extensive research efforts in recent years, fostering the integration of deep learning techniques into scientific computing \cite{Karniadakisetal2021, SirignanoS2018, Raissi2019pinn, EY2018, Lu2021DeepXDE}. 
Notably, the physics informed neural networks (PINNs) approach, introduced in \cite{Raissi2019pinn}, has demonstrated remarkable success in addressing various forward and inverse PDE problems, establishing itself as a widely adopted methodology in scientific machine learning \cite{Raissi2019pinn, HeX2019, CyrGPPT2020, JagtapKK2020, WangL2020, JagtapK2020, CaiCLL2020, Tartakovskyetal2020, DongN2021, CalabroFS2021, WanW2022, FabianiCRS2021, KrishnapriyanGZKM2021, DongY2022, DongY2022rm, WangYP2022, DongW2022, Siegeletal2022,Hu2022XPINN,HuLWX2022,Penwardenetal2023}. 
Comprehensive reviews of PINNs, including its benefits and limitations, can be found in \cite{Karniadakisetal2021, Cuomo2022Scientific}.

Theoretical understanding of the physics informed neural network approach has attracted extensive research, and contributions to the theoretical analysis of PDEs using PINNs have grown steadily and substantially in recent years. 
Shin et al.~\cite{Shin2020On,Shin2010.08019} conducted an extensive investigation into PINNs, demonstrating their consistency when applied to linear elliptic and parabolic PDEs. Mishra and Molinaro proposed an abstract framework to estimate the generalization error of PINNs in forward PDE problems \cite{Mishra2022Estimates} and extended it to inverse PDE problems in \cite{Mishra2022inverse}. Bai and Koley \cite{Bai2021PINN} focused on evaluating the approximation performance of PINNs in nonlinear dispersive PDEs. Biswa et al. \cite{Biswas2022Error} supplied error estimates and stability analysis for the incompressible Navier-Stokes equations. Zerbinati~\cite{Zerbinati2022pinns} treated PINNs as a point matching localization method and provided error estimates for elliptic problems. De Ryck et al.~\cite{DeRyck2021On} presented crucial theoretical findings on PINNs with tanh activation functions and analyzed their approximation errors. These results underlie the  theoretical studies on PINNs for the Navier-Stokes equations~\cite{2023_IMA_Mishra_NS}, high-dimensional radiative transfer problem~\cite{Mishra2021pinn}, and dynamic PDEs of second order in time~\cite{Qian2303.12245}. Hu et al.~\cite{Ruimeng2209.11929} provided valuable insights into the accuracy and convergence properties of PINNs for approximating the primitive equations. Berrone et al.~\cite{Berrone2022posteriori} conducted a posterior error analysis of variational PINNs for solving elliptic boundary-value problems. In~\cite{Hu2022XPINN} the generalized Barron space has been considered for the neural network and a prior and posterior generalization bound on the PDE residuals are provided. Panos et al.~\cite{Panos2023IFENN} concentrated on two critical aspects: error convergence and engineering-guided hyperparameter search, aiming to optimize the performance of the integrated finite element neural network. Gao and Zakharian~\cite{Gao2305.11915} shed insight into the error estimation for solving nonlinear equations using PINNs in the context of $\mathbb{R}$-smooth Banach spaces. 

%

The analyses of PINNs in the aforementioned contributions all involve feed-forward neural networks (FNNs). The network output represents the PDE solution, and the neural network is trained by a ``physics informed" approach, i.e.~by minimizing a loss function related to  residuals of the PDE and the boundary/initial conditions.
The PINN methods of~\cite{2023_IMA_Mishra_NS,Mishra2021pinn,Qian2303.12245} (among others) theoretically guarantee for a class of PDEs that (i) the approximation error of the solution field will be bounded by the training loss, and (ii) there indeed exist FNNs that can make the training loss arbitrarily small. 
While these methods with theoretical guarantees are attractive and important, they suffer from two limitations: (i) the neural network must have two hidden layers, and (ii) the activation function is restricted to $\tanh$ (hyperbolic tangent) only. 

We are interested in the following question: 
\begin{itemize}
\item Is it possible to develop a PINN technique that retains the theoretical guarantees and additionally allows (i) an arbitrary number of hidden layers larger than two, and (ii) activation functions other than $\tanh$?
\end{itemize}
In this work we develop a PINN approach to address the above question. Our method provides an  answer in the affirmative, and it alleviates and largely overcomes the two aforementioned limitations.


A key strategy in our approach is the adoption of a type of modified FNNs, known as hidden-layer concatenated feed-forward neural networks (HLConcFNNs). HLConcFNN was proposed by~\cite{NiDong_hLC_2023} originally for extreme learning machines (ELMs), in order to overcome the issue that achieving high accuracy often necessitates a wide last hidden layer in conventional ELMs~\cite{2021_CMAME_LEMDD,DongY2022rm,DongW2022,WangD2024}. Building upon FNNs, HLConcFNNs establish direct connections between all hidden layer nodes and the output layer through a logical concatenation of the hidden layers. HLConcFNNs have the interesting property that, by appending hidden layers or adding extra nodes to existing hidden layers of a network structure, its representation capacity is guaranteed to be not decreasing (in practice strictly increasing with nonlinear activation functions)~\cite{NiDong_hLC_2023}. By contrast, conventional FNNs lack such a property. The properties of HLConcFNNs prove crucial to generalizing the theoretical analysis of PINN to network architectures with more than two hidden layers and with other activation functions than $\tanh$.


Another strategy in our approach and theoretical analysis is the block time marching (BTM) scheme~\cite{2021_CMAME_LEMDD} for dynamic simulations of time-dependent PDEs with long (or longer) time horizons. For long-time dynamic simulations, training the neural network on the entire spatial-temporal domain with a large dimension in time proves to be especially difficult. In this case, dividing the domain into ``time blocks" with moderate sizes and training the neural network on the space-time domain of each time block individually and successively, with the initial conditions informed by computation results from the preceding time block, can significantly improve the accuracy and ease the training~\cite{2021_CMAME_LEMDD}. This is the essence of block time marching. We refer to e.g.~\cite{KrishnapriyanGZKM2021,Penwardenetal2023} (among others) for analogous strategies.  
%
Block time marching, as formulated in its existing form~\cite{2021_CMAME_LEMDD}, is not amenable to theoretical analysis. The problem lies in the data regularity for the initial conditions, when multiple time blocks are present. To overcome this issue, we present in this work a modified BTM scheme, denoted by ``ExBTM" (standing for Extended BTM), which enables the analysis of the block time marching strategy.


In this paper we present the hidden-layer concatenated PINN (or HLConcPINN) method, by combining hidden-layer concatenated FNNs, the modified block time marching strategy, and the physics informed neural network approach, for approximating parabolic and hyperbolic type PDEs. We analyze the convergence properties and error bounds  of this method for parabolic equations, exemplified by the heat and viscous Burgers' equations, and hyperbolic equations, exemplified by the wave and nonlinear Klein-Gordon equations. Our analyses show that the approximation error of the HLConcPINN solution can be effectively controlled by the training error for long-term dynamic simulations. The  network architecture for HLConcPINNs can in principle contain any number of hidden layers larger than two, and the activation function for all hidden layers beyond the first two can essentially be any of the commonly-used activation functions with sufficient regularity, as long as the first two hidden layers adopt the tanh activation function. 
We note that the analyses of the HLConcPINN method presented herein, excluding the BTM component, can be extended to elliptic type equations (not included here).

These theoretical analyses subsequently inform the formulation of appropriate training loss functions, giving rise to PINN-type computational algorithms, which differ from the standard PINN and BTM formulations for these PDEs. We present ample numerical experiments based on the proposed algorithm. The numerical results demonstrate the effectiveness of this method in accurately capturing the solution field and affirm the relationship between the approximation error and the training loss from the theoretical analyses. Extensive numerical comparisons between the current algorithm and that employing the original BTM scheme are also presented.


The main contributions of this paper lie in two aspects: (i) the hidden-layer concatenated PINN methodology, and (ii) the analyses of the convergence properties and error estimates for this technique. The HLConcPINN method has the salient property that it allows an arbitrary number of hidden layers not smaller than two in the network structure, and allows essentially all of the commonly-used smooth activation functions, with theoretical guarantees.

The remainder of this paper is structured as follows. Section \ref{Preliminaries} provides an overview of HLConcPINN and the BTM strategy. In Sections~\ref{Heat}--\ref{Klein-Gordon}, 
we analyze the convergence and errors of the HLConcPINN algorithm for approximating the heat equation, Burgers' equation, wave equation and the nonlinear Klein-Gordon equation.
Section~\ref{numerical_examples} provides a set of
numerical experiments with these PDEs to show the effectiveness of the HLConcPINN method and to supplement  our theoretical analyses. Section~\ref{sec_summary} concludes the presentation with some further remarks.
Finally, the appendix (Section~\ref{Appendix}) summarizes several auxiliary results and provides  proofs for the theorems discussed in Section \ref{Heat}.

%% file: content/pinn.tex
\section{Hidden-Layer Concatenated Physics Informed Neural Networks and Block Time Marching }\label{Preliminaries}

\subsection{Generic PDE}\label{Generic_PDE}

Consider a compact domain $D\subset \mathbb{R}^d$ ($d>0$ being an integer) and the following initial/boundary value problem on this domain,
\begin{subequations}\label{general}
	\begin{align}\label{general_eq1}
		&\frac{\partial u}{\partial t}(\bm{x},t)+
		\mathcal{L}[u](\bm{x},t) =0 \qquad\qquad\quad (\bm{x},t)\in D\times[0,T],\\
		\label{general_eq2}
		&\mathcal{B}u(\bm{x},t)=u_{d}(\bm{x},t)\qquad\qquad\qquad\quad \ \ (\bm{x},t) \in \partial D\times[0,T],\\
		\label{general_eq3}
		&u(\bm{x},0)=u_{in}(\bm{x}) \qquad\qquad\qquad\qquad\ \ \, \bm{x}\in D,
	\end{align}
\end{subequations}
Here, $u: D\times [0,T]\subset\mathbb R^{d+1}\rightarrow \mathbb R^m$ ($m\geq 1$ being an integer) is the unknown field solution. $u_{d}$ is the boundary data, and $u_{in}$ is the initial distribution for $u$. $\mathcal{L}$ and $\mathcal{B}$ denote the differential and boundary operators. $T$ is the  dimension in time.

\subsection{Physics Informed Neural Networks}\label{PINN1}

{Physics informed neural network (PINN) refers} to the general approach for approximating a PDE problem using feedforward neural networks (FNNs) by minimizing the residuals involved in the problem.
We first define feedforward neural networks, and then discuss the related machinery for the analysis of PINN.

\begin{figure}[tb]
	\vspace*{-10pt}
	\centering
		\subfloat[Conventional FNN]{
  \includegraphics[width=0.5\linewidth]{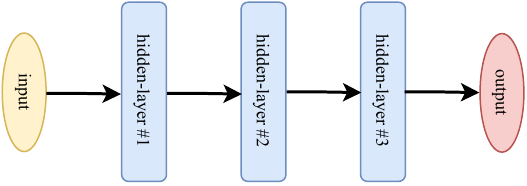}
  } \hspace{1em}
		\subfloat[Hidden-layer concatenated FNN]{
  \includegraphics[width=0.66\linewidth]{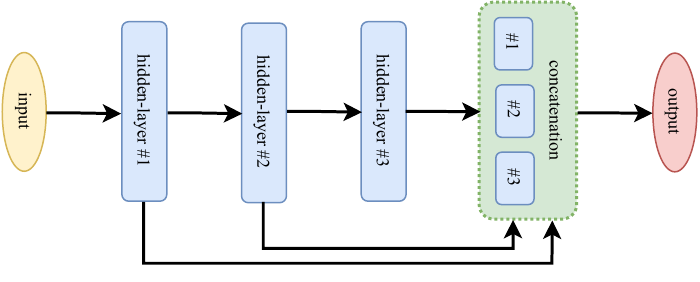}}
	\caption{Illustration of network structures (with 3 hidden layers) for conventional and hidden-layer concatenated neural networks. In hidden-layer concatenated FNN, all the hidden nodes are exposed to the output nodes, while in conventional FNN only the last hidden-layer nodes are exposed to the output nodes.}
	\label{pre_nn}
 \vspace*{-5pt}
\end{figure}

Let $\sigma: \mathbb{R} \rightarrow \mathbb{R}$ denote an activation function. 
For any $n\in \mathbb{N}$ and $z=(z_1,z_2,\cdots,z_n)\in \mathbb{R}^n$, we define $\sigma(z): = (\sigma(z_1),\cdots,\sigma(z_n))$.
A feedforward neural network  (with three hidden layers) is illustrated in Figure~\ref{pre_nn}(a) using a cartoon.
It is formally defined as follows.

\begin{Def}\label{pre_lem1}
Let $W$ and $L$ be integers, and $1\leq l_i\leq W$ ($0\leqslant i\leqslant L$) denote $(L+1)$ positive integers.
Let $R\in \mathbb{R}$ represent a bounded positive real number, $z_0\in \mathbb{R}^{l_0}$ denote the input variable, $\sigma: \mathbb{R} \rightarrow \mathbb{R}$ be a twice differentiable activation function, and $\mathcal{A}_k^{\vartheta}$ ($1\leq k\leq L$) be an affine mapping $\mathcal{A}_k^{\vartheta}: \mathbb{R}^{l_{k-1}}\rightarrow \mathbb{R}^{l_k}$ given by 
\[
z_{k-1}\mapsto W_kz_{k-1}+b_k \qquad \text{for} \ 1\leq k\leq L,
\]
where $W_k \in [-R,R]^{l_k\times l_{k-1}}\subset \mathbb{R}^{l_k\times l_{k-1}}$ and $b_k\in [-R,R]^{l_{k}} \subset\mathbb{R}^{l_{k}}$ are referred to as the weight/bias coefficients for $1\leq k\leq L$. Let  $\Theta$ denote the collection of all weights/biases and $\vartheta\in\Theta$. 

A feedforward neural network is defined as the mapping 
$u_{\vartheta}: \mathbb{R}^{l_0}\rightarrow \mathbb{R}^{l_L}$ given by
\begin{equation}\label{pre_eq2}
u_{\vartheta}(z_0)=\mathcal{A}_{L}^{\vartheta}\circ\sigma \circ\mathcal{A}_{L-1}^{\vartheta}\circ\cdots\circ\sigma\circ\mathcal{A}_{1}^{\vartheta}(z_0),\qquad z_0\in \mathbb{R}^{l_0},
\end{equation} 
where $\circ$ denotes a function composition. 
\end{Def}

The feedforward neural network in the definition contains $(L+1)$ layers ($L\geq 2$) with widths $(l_0,l_1, \cdots,l_L)$, respectively. The input layer and the output layer have $l_0$ and $l_L$ nodes, respectively. The $(L-1)$ layers between the input/output layers are the hidden layers, with widths $l_k$ ($1\leq k\leq L-1$). 
From layer to layer, the network logic represents an affine transform, followed by a function composition with the activation function $\sigma$. No activation function is applied to the output layer. 
Hereafter we refer to the vector of positive integers, $\bm l=(l_0,l_1,\dots,l_L)$, as an architectural vector, which characterizes the architecture of an FNN.

The neural network $u_{\vartheta}$, defined by \eqref{pre_eq2}, is a parameterized function of the input $z_0=$ $(\bm{x},t)$, with the parameter $\vartheta$ of weights and biases. We represent the solution field $u$ to problem~\eqref{general} by the neural network $u_{\vartheta}$, and 
wish to find the parameters $\vartheta$ such that $u_{\vartheta}$ approximates $u$ well. 

It is necessary to approximate  function integrals during the analysis of physics informed neural networks. 
Given a subset $\Lambda \subset \mathbb{R}^d$ and a function $f\in L^1(\Lambda)$, a quadrature rule provides an approximation of the integral by
$ 
\int_{\Lambda}f(z)\diff{z}\approx
\frac{1}{M}\sum_{n=1}^M \omega_n f(z_n),
$ 
where $z_n\in \Lambda$ ($1\leq n\leq M$) represents the quadrature points and $\omega_n$ ($1\leq n\leq M$) denotes the appropriate quadrature weights. The approximation accuracy is influenced by the regularity of $f$, the type of quadrature rule and the number of quadrature points ($M$). 
In the partial differential equations considered in this work, we assume that the problem dimension is low, thus allowing 
the use of standard deterministic values for the integrating points. 
Following \cite{2023_IMA_Mishra_NS,Qian2303.12245}, we employ the midpoint rule for numerical integrals. 
We partition $\Lambda$ into $M\sim N^d$ cubes 
with an edge length $\frac{1}{N}$.
The approximation accuracy is determined by
\begin{equation}\label{int1}
	\left|\int_{\Lambda}f(z)\diff{z}-
	\mathcal{Q}_M^{\Lambda}[f]
	\right|\leq C_fM^{-2/d},
\end{equation}
where $\mathcal{Q}_M^{\Lambda}[f]:=\frac{1}{M}\sum_{n=1}^M\omega_nf(z_n)$, $C_f\lesssim \|f\|_{C^2(\Lambda)}$ ($a\lesssim b$ denotes $a\leq Cb$ for some constant $C$) 
and $\{z_n\}_{n=1}^M$ denote the midpoints of these cubes~\cite{DavisR2007}.


\subsection{Hidden-Layer Concatenated Physics Informed Neural Networks (HLConcPINNs)}\label{HLC}

We consider a type of modified FNN, termed hidden-layer concatenated feed-forward neural network (HLConcFNN) proposed by~\cite{NiDong_hLC_2023}, for PDE approximation in this work.
%
HLConcFNNs differ from traditional FNNs by a modification that establishes direct connections between all hidden nodes and the output layer. 
The modified network, as illustrated in Figure \ref{pre_nn}(b) with three hidden layers, incorporates a logical concatenation layer between the last hidden layer and the output layer. This layer aggregates the output fields of all hidden nodes across the network, spanning from the first to the last hidden layers. From the logical concatenation layer to the output layer, a typical affine transformation is applied, possibly followed by an activation function, to obtain the output of the overall neural network. Notably, the logical concatenation layer does not introduce any trainable parameters. 
Hereafter we refer to the modified network as the hidden-layer concatenated FNN (HLConcFNN), as opposed to the original FNN, which serves as the base neural network. The incorporation of logical concatenation ensures that all the hidden nodes in the base network architecture have direct connections to the output nodes in  HLConcFNN. This direct connectivity facilitates the flow of information from the hidden layers to the output layer, enhancing the network's capacity to capture intricate relations. We refer to the approach,   combining physics informed neural networks with hidden-layer concatenated FNNs, as hidden-layer concatenated physics informed neural networks (HLConcPINNs).

Given architectural vector $\bm{l}= (l_0, l_1, \cdots, l_L)$, 
the logical concatenation layer contains a total of $N_c(\bm{l})=\sum_{i=1}^{L-1}l_i$ virtual nodes, with the total number of hidden-layer coefficients in the neural network given by $N_h(\bm{l})=\sum_{i=1}^{L-1}(l_{i-1}+1)l_i$. 
The total number of network parameters in  HLConcFNN is $N_a(\bm l)=N_h(\bm l)+ [N_c(\bm l)+1]l_L$.  
The HLConcFNN is formally defined as the mapping $u_{\theta}:\mathbb R^{l_0}\rightarrow\mathbb R^{l_L}$ given by
\begin{equation}\label{pre_eq3}
	u_{\theta}(z)=\sum_{i=1}^{L-1} M_iu_i^{\vartheta}(z)+b_{L},\qquad z\in \mathbb{R}^{l_0},
\end{equation} 
where $\theta\in\mathbb{R}^{N_a}$ denotes all the network parameters in HLConcFNN, and $\vartheta\in\mathbb{R}^{N_h}$ denotes the hidden-layer parameters. $u_i^{\vartheta}(z)=\sigma \circ\mathcal{A}_{i}^{\vartheta}\circ\sigma \circ\mathcal{A}_{i-1}^{\vartheta}\circ\cdots\circ\sigma\circ\mathcal{A}_{1}^{\vartheta}(z)$ $(1\leq i \leq L-1)$, with 
$M_i\in \mathbb{R}^{l_L\times l_i}$ $(1\leq i \leq L-1)$ denoting the connection coefficients between the output layer and the $i$-th hidden layer. $b_L\in\mathbb{R}^{l_L}$ is the bias of the output layer.

Given an architectural vector $\bm{l}$ and an activation function $\sigma$, let HLConcFNN$(\bm{l},\sigma)$ denote the  hidden-layer concatenated neural network associated with this architecture. For a given domain $D\subset \mathbb{R}^d$, 
we define
\begin{align}\label{pre_eq4}
	U(D,\bm{l},\sigma)=\{\ u_{\theta}(z)\ |\ u_{\theta}(z) \ {\rm is} \  {\rm the} \ {\rm output} \ {\rm of} \ {\rm HLConcFNN}(\bm{l},\sigma),\ z\in D,\ \theta\in\mathbb{R}^{N_a(\bm l)}\ \}
\end{align} 
as the collection of all possible output fields of this HLConcFNN$(\bm{l},\sigma)$. $U(D,\bm{l},\sigma)$ denotes the set of  functions that can be exactly represented by this HLConcFNN$(\bm{l},\sigma)$ on $D$. Following~\cite{NiDong_hLC_2023}, we refer to $U(D,\bm{l},\sigma)$ as the representation capacity of the HLConcFNN$(\bm{l},\sigma)$ for the domain $D$.


HLConcFNNs exhibit a hierarchical structure in terms of their representation capacity, which is crucial to the current analyses. Specifically,  given a base network architecture $\bm l$, if a new hidden layer is appended to this network or if extra nodes are added to an existing hidden layer, the representation capacity of the HLConcFNN associated with the new architecture is guaranteed to be not smaller than that of the original architecture. These points are made precise by Lemmas~\ref{Ar_5} and~\ref{Ar_6} in Section~\ref{Auxiliary lemmas}. As a result, starting with an initial network architecture, one can attain a sequence of network architectures, by either appending one or more hidden layers to or adding extra nodes to  existing hidden layers of the preceding architecture. Then we can conclude that the HLConcFNNs associated with this sequence of network architectures have non-decreasing representation capacities. By contrast, conventional FNNs do not have such a property. It is also noted that the HLConcFNN associated with a network architecture has a representation capacity at least as large as that of the conventional FNN associated with this architecture.

\begin{figure}[!htb]
	\vspace*{-5pt}
	\centering
	\subfloat{
 \includegraphics[width=1\linewidth]{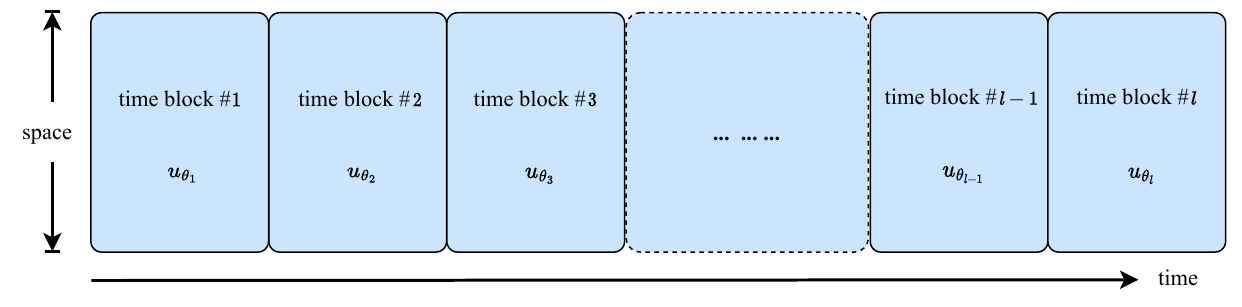}}
	\vspace*{-10pt}
	\caption{Illustration of the block time marching (BTM) strategy. 
 The large time domain is partitioned into multiple blocks, with each block 
 computed individually and successively.
 Solution in one block informs the initial condition for the subsequent time block.
 }
	\label{pre_btm}
 \vspace*{-5pt}
\end{figure}

\subsection{Block Time Marching (BTM)}\label{BTM}

Long-time dynamic simulation of time-dependent PDEs is a challenging issue to neural network-based methods. NN-based techniques oftentimes adopt a space-time approach for solving dynamic PDEs, in which the space and time variables are treated on the same footing. When the temporal dimension of the space-time domain becomes large, as necessitated by the interest in long-time dynamics, network training can become immensely difficult, leading to grossly inaccurate NN predictions, especially toward later time instants.

Block time marching (BTM)~\cite{2021_CMAME_LEMDD} is an effective strategy that can alleviate the challenge posed by long time horizons and facilitate accurate long-term simulations with neural networks. 
BTM addresses the challenge 
by partitioning the large time domain into multiple windows (time blocks) and successively advancing the solution through these blocks.
Figure \ref{pre_btm} provides a visual representation of the block time marching strategy. The space-time domain, with a long time horizon, is divided into time blocks along the time axis.  
Each time block should have a moderate  size in time, facilitating  effective capture of the dynamics. The initial-boundary value problem is solved individually and sequentially on the space-time domain of each time block using a suitable method, in particular with HLConcPINN in this work. The solution obtained on one time block, evaluated at the last time instant, informs the initial conditions for the computation of the subsequent time block. Starting with the first time block, we march in time block by block, until the last time block is traversed. 

The basic BTM formulation as described above, unfortunately, is not amenable to theoretical analysis. Our analysis requires a modification to the basic formulation, which will be discussed in subsequent sections.

\subsection{Residuals and Training Sets}
\label{sec_setting}

We combine block time marching and hidden-layer concatenated PINNs for solving the system~\eqref{general}. We provide a   theoretical analysis of the resultant method, and investigate  the numerical algorithms as suggested by the theory.

For simplicity we consider uniform time blocks in BTM. We divide the temporal dimension $T$ into $l\geq 1$ uniform time blocks, where $l$ is chosen such that the block size $\Delta T=T/l$ ($\Delta t$ or $\Delta T$, in Sect 3-6, time block $\Delta t = t_{i+1}-t_i$) is of a moderate value. Let $[t_{i-1},t_i]$ ($1\leq i\leq l$) denote the $i$-th block in time, where $t_i=i\Delta T$ (or $t_i=i\Delta t$) and $t_0$ denotes the initial time.
We march in time block by block, and within each time block solve the system~\eqref{general} using hidden-layer concatenated PINN.

To solve~\eqref{general},
it is necessary to specify the residuals  and the set of training collocation points. 
%
Let $\mathcal{S}_i\subset \overline{D}\times [0,t_i]$ denote the set of collocation points for training the HLConcPINN on the $i$-th time block ($1\leq i\leq l$).
We define  $\mathcal{S}_i = \mathcal{S}_{int_{i}} \cup \mathcal{S}_{sb_{i}} \cup \mathcal{S}_{tb_{i}}$ with,
\begin{itemize}
	\item Interior training points $\mathcal{S}_{int_{i}}=\{y_n, 1\leq n \leq N_{int_{i}}\}$, where $y_n= (\bm{x}_n,t_n) \in D\times (t_{i-1},t_i)$. 
	\item Spatial boundary training points $\mathcal{S}_{sb_{i}}=\{y_n, 1\leq n \leq N_{sb_{i}}\}$, where $y_n= (\bm{x}_n,t_n) \in \partial D\times (t_{i-1},t_i)$. 
	\item Temporal boundary training points $\mathcal{S}_{tb_{i}}=\{y_n, 1\leq n \leq N_{tb_{i}}\}$, where $y_n=(\bm{x}_n,t_n) \in D\times \{t_0,t_1,\dots,t_{i-1} \}$.
\end{itemize}
Here $(N_{{int}_i},N_{{sb}_i},N_{{tb}_i})$ denote the number of interior points, spatial boundary points, and temporal boundary points for the $i$-th block, respectively.

Define space-time domains $\Omega_i = D\times [t_{i-1},t_i]$ and $\Omega_{*i}= \partial D\times [t_{i-1},t_i]$ for time block $i$. We employ a HLConcFNN $u_{\theta}: \Omega_i\rightarrow \mathbb{R}^m$ to approximate the solution $u$  on $\Omega_i$, and
 define the residual functions by ($1\leq i\leq l$),
\begin{subequations}\label{general_pinn}
	\begin{align}
		\label{general_pinn_eq1}
		&\mathcal{R}_{int_i}[u_{\theta_i}](\bm{x},t)=
		\frac{\partial u_{\theta_i}}{\partial t}(\bm{x},t)+
		\mathcal{L}[u_{\theta_i}](\bm{x},t)\qquad\qquad\quad (\bm{x},t)\in \Omega_i , \\
		\label{general_pinn_eq2}
		&\mathcal{R}_{sb_i}[u_{\theta_i}](\bm{x},t)=\mathcal{B}u_{\theta_i}(\bm{x},t)-u_{d}(\bm{x},t)\qquad\qquad\qquad\ \ \ (\bm{x},t) \in \Omega_{*i},\\
		\label{general_pinn_eq3}
		&\mathcal{R}_{tb_i}[u_{\theta_i}](\bm{x},t_{i-1})=u_{\theta_i}(\bm{x},t_{i-1})-u_{\theta_{i-1}}(\bm{x},t_{i-1})\qquad\ \ \, \bm{x}\in D,
	\end{align}
\end{subequations}
where $u_{\theta_{0}}(\bm{x},t_0) = u_{in}(\bm{x})$. These residuals characterize the extent to which a given function $u$ satisfies the initial/boundary value problem~\eqref{general} for time block $i$. If $u$ is the exact solution, then $\mathcal{R}_{int_i}[u]=\mathcal{R}_{sb_i}[u]=\mathcal{R}_{tb_1}[u]\equiv 0$.
The above settings on the time block partitions and training sets will be employed throughout the subsequent sections.

\comment{
The question \eqref{general} belongs to parabolic-type partial differential equations. For hyperbolic-type PDE $\frac{\partial^2 u}{\partial t^2}(\bm{x},t)$ $+$ $\mathcal{L}[u](\bm{x},t)=0$, by introducing $v=\frac{\partial u}{\partial t}$, we divide the interior residual into the following two parts:
\begin{subequations}
\begin{align}
	\label{general1_pinn_eq1.1}
	&\mathcal{R}_{int1_i}[u_{\theta_i},v_{\theta_i}](\bm{x},t)=
	\frac{\partial u_{\theta_i}}{\partial t}(\bm{x},t)-v_{\theta_i}(\bm{x},t),\\
	\label{general1_pinn_eq1.2}
	&\mathcal{R}_{int2_i}[u_{\theta_i},v_{\theta_i}](\bm{x},t)=
	\frac{\partial v_{\theta_i}}{\partial t}(\bm{x},t)+
	\mathcal{L}[u_{\theta_i}](\bm{x},t).
\end{align}
\end{subequations}

In the absence of the term $\frac{\partial u}{\partial t}$ and time space, \eqref{general} can be transformed into an elliptic-type PDE. And then,  we choose the training set $\mathcal{S}^{E}$ as a subset of the domain $D$ and partition the full training set $\mathcal{S}^{E}$ into two distinct parts $\mathcal{S}^{E} =\mathcal{S}_{int}^{E}\cup \mathcal{S}_{sb}^{E}$: 
\begin{itemize}
	\item Interior training points $\mathcal{S}_{int}^{E}=\{\bm{x}_n\}$ for $1\leq n \leq N_{int}$, with each $\bm{x}_n \in D$. 
	\item Spatial boundary training points $\mathcal{S}_{sb}^{E}=\{\bm{x}_n\}$ for $1\leq n \leq N_{sb}$, with each $\bm{x}_n \in \partial D$. 
\end{itemize}
These analytical techniques and theoretical results can be extended to elliptic-type equations. 
} 

\vspace{0.1in} 
In the forthcoming sections, we  focus on four time-dependent partial differential equations: the heat equation, viscous Burgers’ equation, wave equation, and nonlinear Klein-Gordon equation,
 representative of parabolic and hyperbolic type PDEs. 
We provide an analysis of the HLConcPINN method for approximating the solutions to these equations for long-time dynamic simulations, and investigate the PINN type  computational algorithm stemming from these analyses. 
We implement these algorithms and numerically demonstrate the effectiveness of the HLConcPINN method using extensive experiments.

%% file: content/heat.tex
\section{HLConcPINN  for Approximating the Heat Equation}\label{Heat}

\subsection{Heat Equation}

Let $D \subset \mathbb{R}^d$ denote an open connected bounded domain with a $C^k$ boundary $\partial D$. We consider the
 heat equation:
\begin{subequations}\label{heat}
	\begin{align}\label{heat_eq1}
		&\frac{\partial u}{\partial t}-\Delta u = f     \qquad\qquad\quad \text{in}\ D\times[0,T],\\
		\label{heat_eq2}
		&u(\bm{x},0)=u_{in}(\bm{x})\qquad\quad\ \ \, \text{in}\ D,\\
		\label{heat_eq3}
		&u|_{\partial D}=u_{d}\qquad\qquad\qquad\ \, \text{in}\ \partial D\times[0,T].
	\end{align}
\end{subequations}
Here $u(\bm x,t)$ is the field solution, $f$ is a source term, and $u_{in}$ and $u_d$ denote the initial distribution and the boundary data, respectively.


\subsection{Hidden-Layer Concatenated Physics Informed Neural Networks}


We divide the temporal domain into $l$ blocks, and seek $l$ deep neural networks $u_{\theta_{i}} : D\times [0,t_i] \rightarrow \mathbb{R}$, parameterized by $\theta_{i}$, 
to approximate the solution $u$ of \eqref{heat} for $1\leq i\leq l$. 
For any 
$u_{\theta_{i}} : D\times [0,t_i] \rightarrow \mathbb{R}$ ($1\leq i\leq l$), we define the residuals:
\begin{subequations}\label{heat_pinn}
	\begin{align}\label{heat_pinn_eq1}
		&R_{int_{i}}[u_{\theta_{i}}](\bm{x},t) =\frac{\partial u_{\theta_i}}{\partial t}-\Delta u_{\theta_{i}} - f,
  \qquad\qquad\qquad\quad \ \ (\bm{x},t)\in \Omega_i,
  \\
		\label{heat_pinn_eq2}
		&R_{tb_{i}}[u_{\theta_{i}}](\bm{x},t_{j-1}) =u_{\theta_i}(\bm{x},t_{j-1})-u_{\theta_{j-1}}(\bm{x},t_{j-1}),
  \qquad \bm{x}\in D, \quad \text{for}\ 1\leq j\leq i,
  \\
		\label{heat_pinn_eq3}
		&R_{sb_{i}}[u_{\theta_{i}}](\bm{x},t) =u_{\theta_i}(\bm{x},t)-u_{d}(\bm{x},t),
  \qquad\qquad\qquad\quad\ \, (\bm{x},t)\in \Omega_{*i}.
	\end{align}
\end{subequations}
Here $u_{\theta_{0}}(\bm{x},t_{0})= u_{in}(\bm{x})$, and $\Omega_i$ and $\Omega_{*i}$ are defined in Section~\ref{sec_setting}. Note that for the exact solution $R_{int_{i}}[u]=R_{tb_{i}}[u]=R_{sb_{i}}[u]=0$.

With HLConcPINN, we try to find a sequence of neural networks $u_{\theta_{i}}$ ($1\leq i\leq l$), for which all the residuals are  minimized. Specifically, we minimize  the quantity,
\begin{align}
	\label{heat_G}
	&\mathcal{E}_{G_i}(\theta_i)^2=\widetilde{\mathcal{E}}_{G_i}(\theta_i)^2+\mathcal{E}_{G_{i-1}}(\theta_{i-1})^2,
\end{align}
sequentially for $1\leq i\leq l$, where
\begin{align}
	\label{heat_Gi}
	&\widetilde{\mathcal{E}}_{G_i}(\theta_i)^2=\int_{\Omega_i}|R_{int_{i}}[u_{\theta_{i}}](\bm{x},t)|^2\dx\dt
	+\sum_{j=1}^{i}\int_{D}|R_{tb_{i}}[u_{\theta_{i}}](\bm{x},t_{j-1})|^2\dx
	\nonumber\\
	&\qquad+\left(\int_{\Omega_{*i}}|R_{sb_{i}}[u_{\theta_{i}}](\bm{x},t)|^2\ds\dt\right)^{\frac{1}{2}}.
\end{align}
In equation~\eqref{heat_G} we set $\mathcal{E}_{G_{i-1}}(\theta_{i-1}) =0$ for $i=1$. The quantity $\mathcal{E}_{G_i}(\theta)$ is commonly known as the population risk or generalization error of the neural networks $u_{\theta_i}$. 

\begin{Remark}\label{heat_Remark1} 
	In the original block time marching scheme from~\cite{2021_CMAME_LEMDD}, when computing a particular time block, the initial condition is taken to be the solution data from the preceding time block evaluated at the last time instant. Theorem \ref{Ar_7} (in the appendix) suggests that this initial value may have a different regularity from the true initial  data for the problem. This difference can affect the regularity of the computed solution in the current time block.
	
	To address this issue, we make the following crucial modification to  block time marching. We employ the true initial data for the problem  as the initial value for all time blocks within the interval $[0, t_i]$ for $1 \leq i \leq l$, as specified by~\eqref{heat_pinn_eq2}. This ensures that the regularity of the initial value is maintained throughout the time blocks. Essentially, we enforce the PDE and the boundary conditions only on the interval $[t_{i-1}, t_i]$ in time. 
 For the time periods $[0, t_{i-1}]$, however, we enforce the residuals solely at the discrete points $t_{j-1}$ ($1\leq j \leq i$).
	By using the true initial data consistently and training the neural network within individual time blocks successively, we can maintain the regularity of the solution across all time blocks. The initial condition \eqref{heat_pinn_eq2} and the setting for training data points in subsequent discussions employ this modified BTM formulation. 	
\end{Remark}

The integrals in \eqref{heat_G} can be approximated numerically, leading to a training loss function. Following the discussions of Section~\ref{sec_setting}, the full training set consists of $\mathcal{S} = \bigcup_{i=1}^l\mathcal{S}_{i}$ with $\mathcal{S}_i = \mathcal{S}_{int_{i}} \cup \mathcal{S}_{sb_{i}} \cup \mathcal{S}_{tb_{i}}$ and we employ the midpoint rule for the numerical quadrature. This leads to the following approximation:
\begin{align}
	\label{heat_T}
	&\mathcal{E}_{T_i}(\theta_i,\mathcal{S}_i)^2=\widetilde{\mathcal{E}}_{T_i}(\theta_i,\mathcal{S}_i)^2+\mathcal{E}_{T_{i-1}}(\theta_{i-1},\mathcal{S}_{i-1})^2,\\
	\label{heat_Ti}
	&\widetilde{\mathcal{E}}_{T_i}(\theta_i,\mathcal{S}_i)^2=\mathcal{E}_T^{int_{i}}(\theta_i,\mathcal{S}_{int_{i}})^2+\mathcal{E}_T^{tb_{i}}(\theta_i,\mathcal{S}_{tb_{i}})^2+\mathcal{E}_T^{sb_{i}}(\theta_i,\mathcal{S}_{sb_{i}}),
\end{align}
where
\begin{subequations}\label{heat_TT}
	\begin{align}
		\label{heat_T1}
		&\mathcal{E}_T^{int_{i}}(\theta_i,\mathcal{S}_{int_{i}})^2 = \sum_{n=1}^{N_{int_{i}}}\omega_{int_{i}}^n|R_{int_{i}}[u_{\theta_{i}}](\bm{x}_{int_{i}}^n,t_{int_{i}}^n))|^2,\\
		\label{heat_T2}
		&\mathcal{E}_T^{sb_{i}}(\theta_i,\mathcal{S}_{sb_{i}})^2 = \sum_{n=1}^{N_{sb_{i}}}\omega_{sb_{i}}^n|R_{sb_{i}}[u_{\theta_{i}}](\bm{x}_{sb_{i}}^n,t_{sb_{i}}^n))|^2,\\\label{heat_T3}
		&\mathcal{E}_T^{tb_{i}}(\theta_i,\mathcal{S}_{tb_{i}})^2 = \sum_{j=1}^{i}\sum_{n=1}^{N_{tb_{i}}}\omega_{tb_{i}}^n|R_{tb_{i}}[u_{\theta_{i}}](\bm{x}_{tb_i}^n,t_{j-1})|^2,
	\end{align}
\end{subequations}
with the term $\mathcal{E}_{T_{i-1}}(\theta_{i-1}$,$\mathcal{S}_{i-1}) =0$ for $i=1$. Here, the quadrature points in space-time constitute the data sets $\mathcal{S}_{int_i} = \{(\bm{x}_{int_{i}}^n,t_{int_{i}}^n)\}_{n=1}^{N_{int_i}}$, $\mathcal{S}_{tb_i} = \{\bm{x}_{tb_{i}}^n\}_{n=1}^{N_{tb_i}}$ and $\mathcal{S}_{sb_i} = \{(\bm{x}_{sb_{i}}^n,t_{sb_{i}}^n)\}_{n=1}^{N_{sb_i}}$, and $\omega_{\star_i}^n$ are the quadrature weights with $\star$ denoting $int$, $tb$ or $sb$. 

\subsection{Error Analysis}  

Let $\hat{u}_{i}= u_{\theta_{i}} - u$ denote the error of the HLConcPINN approximation ($u_{\theta_i}$) against the true solution ($u$). 
By using equation \eqref{heat} and  definitions of the  residuals~\eqref{heat_pinn}, we obtain
\begin{subequations}\label{heat_error}
	\begin{align}\label{heat_error_eq1}
		&R_{int_{i}}=\frac{\partial \hat{u}_{i}}{\partial t}-\Delta\hat{u}_{i},\\
		\label{heat_error_eq2}
		&R_{tb_{i}}|_{t=t_{j-1}}=\hat{u}_{i}|_{t=t_{j-1}} - \hat{u}_{j-1}|_{t=t_{j-1}},
  \qquad j = 1,2,\cdots, i,\\
		\label{heat_error_eq3}
		&R_{sb_{i}}=\hat{u}_{i}|_{\partial D},
	\end{align}
\end{subequations}
where $\hat{u}_{0}|_{t=t_{0}} = 0$. We define the total error of the HLConcPINN approximation by
\begin{equation}\label{heat_total}
	\mathcal{E}(\theta_i)^2=\int_{t_{i-1}}^{t_{i}}\int_{D}|\hat{u}_i(\bm{x},t)|^2\dx\dt,
 \quad 1\leq i\leq l.
\end{equation}

The bounds on the HLConcPINN residuals and its approximation errors are provided by the following three theorems.
The proofs for these theorems are given in the appendix (Section \ref{Proof}).
\begin{Theorem}\label{sec4_Theorem1} 
	Let $\widetilde\Omega_i = D\times [0,t_i]$ and $\widetilde\Omega_{*i}= \partial D\times [0,t_i]$. Suppose $n$, $d$, $k \in \mathbb{N}$ with $n\geq2$ and $k\geq 3$, and $u\in H^k(\widetilde\Omega_i)$. For every integer $N>5$, there exists a HLConcPINN $u_{\theta_i}$ such that
	\begin{equation}\label{lem4.1}
		\|R_{int_{i}}\|_{L^2(\widetilde\Omega_{i})}\lesssim N^{-k+2}{\rm ln}^2N;
		\qquad
		\|R_{tb_{i}}(\bm{x},t_{j-1})\|_{L^2(D)},\  \|R_{sb_{i}}\|_{L^2(\widetilde\Omega_{*i})}\lesssim N^{-k+1}{\rm ln}N,\quad 1\leq j\leq i.
	\end{equation}
\end{Theorem}

 Theorem \ref{sec4_Theorem1} implies that one can make the HLConcPINN residuals~\eqref{heat_pinn} arbitrarily small by choosing $N$ to be sufficiently large. It follows that the generalization error $\mathcal{E}_{G_i}(\theta_i)^2$ in~\eqref{heat_G} can be made arbitrarily small.

The next two theorems indicate that the approximation error $\mathcal{E}(\theta_i)^2$ is also small when the generalization error $\mathcal{E}_{G_i}(\theta_i)^2$ is small with the HLConcPINN approximation $u_{\theta_i}$. Moreover,
the approximation error $\mathcal{E}(\theta_i)^2$ can be arbitrarily small, provided that the training error $\mathcal{E}_{T_i}(\theta_i,\mathcal{S}_i)^2$ is sufficiently small and the sample set is sufficiently large. 


\begin{Theorem}\label{sec4_Theorem2} Let $d\in \mathbb{N}$, and $u\in C^1(\widetilde\Omega_i)$ be the classical solution to \eqref{heat}. Let $u_{\theta_{i}}$ be a HLConcPINN with parameter $\theta_i$, $t_{i-1}\leq\tau \leq t_i$, and $\Delta t=t_i-t_{i-1}$  (time block size). Then the following relation holds,
	\begin{equation}\label{lem4.3}
		\int_{D} |\hat{u}_{i}(\bm{x},\tau)|^2\dx
		\leq C_{G_{i}}\exp(\Delta t),
		\qquad
		\int_{t_{i-1}}^{t_{i}}\int_{D} |\hat{u}_{i}(\bm{x},t)|^2\dx\dt\leq C_{G_{i}}\Delta t\exp(\Delta t),
	\end{equation}
	where
	\begin{align*}
		&C_{G_{i}}= \widetilde{C}_{G_{i}}  + 2C_{G_{i-1}}\exp(\Delta t), \qquad C_{G_{0}}=0, \\
		&\widetilde{C}_{G_{i}} = 2\sum_{j=1}^{i}\int_{D}|R_{tb_{i}}(\bm{x},t_{j-1})|^2\dx+ \int_{t_{i-1}}^{t_{i}}\int_{D}|R_{int_{i}}|^2\dx\dt+2C_{\partial D_i}|\Delta t|^{\frac{1}{2}}\big(\int_{t_{i-1}}^{t_{i}}\int_{\partial D}|R_{sb_{i}}|^2\ds\dt\big)^{\frac{1}{2}},
	\end{align*} 
	and $C_{\partial D_i}=|\partial D|^{\frac{1}{2}}\big(\|u\|_{C^1(\widetilde\Omega_{*i})}+\|u_{\theta_{i}}\|_{C^1(\widetilde\Omega_{*i})}\big)$.
\end{Theorem}

\begin{Theorem}\label{sec4_Theorem3} Let $d\in \mathbb{N}$ and $T>0$. Let $u\in C^4(\widetilde\Omega_i)$ be the classical solution to \eqref{heat},  and  $u_{\theta_{i}}$ $(1\leq i\leq l)$ be a HLConcPINN with parameter $\theta_i$. Then the total error satisfies
	\begin{align}\label{lem4.4}
		&\int_{t_{i-1}}^{t_{i}}\int_{D} |\hat{u}_{i}(\bm{x},t)|^2\dx\dt\leq C_{T_{i}}\Delta t\exp(\Delta t)
		\nonumber\\
		&\qquad=\mathcal{O}\left(\mathcal{E}_{T_i}(\theta_i,\mathcal{S}_i)^2 + M_{int_i}^{-\frac{2}{d+1}} +M_{tb_i}^{-\frac{2}{d}}+M_{sb_i}^{-\frac{1}{d}}\right),
	\end{align}
	where the constant $C_{T_i}$ is defined by 
	\begin{align}\label{lem4.5}
		&C_{T_{i}}= \widetilde{C}_{T_{i}}  + 2C_{T_{i-1}}\exp(\Delta t), \qquad C_{T_{0}}=0, \\
		&\widetilde{C}_{T_{i}} = 2\sum_{j=1}^{i}\big(C_{({R_{tb_{i}}^2(\bm{x},t_{j-1})})}M_{tb_{i}}^{-\frac{2}{d}}+\mathcal{Q}_{M_{tb_{i}}}^{D}(R_{tb_{i}}^2(\bm{x},t_{j-1}))\big)
		\nonumber\\
		&\qquad+C_{({R_{int_{i}}^2})}M_{int_{i}}^{-\frac{2}{d+1}}+\mathcal{Q}_{M_{int_{i}}}^{\Omega_i}(R_{int_{i}}^2)+2C_{\partial D_i}|\Delta t|^{\frac{1}{2}}\big(C_{({R_{sb_{i}}^2})}M_{sb_{i}}^{-\frac{2}{d}}+\mathcal{Q}_{M_{sb_{i}}}^{\Omega_{*i}}(R_{sb_{i}}^2)\big)^{\frac{1}{2}}.\nonumber
	\end{align} 
\end{Theorem}

%% file: content/burger.tex
\section{HLConcPINN for Approximating the Burgers' Equation}\label{Burger}

\subsection{Viscous Burgers' Equation}
We consider the 1D viscous Burgers' equation on the domain $D=[a,b] \subset \mathbb{R}$:
\begin{subequations}\label{burger}
	\begin{align}\label{burger_eq1}
		&\frac{\partial u}{\partial t}-\nu\frac{\partial^2 u}{\partial x^2} +u\frac{\partial u}{\partial x} = f(x,t)     \qquad\qquad\  (x,t)\in D\times[0,T],\\
		\label{burger_eq2}
		&u(x,0)=u_{in}(x)\qquad\qquad\qquad\qquad\quad\ \ \, x\in D,\\
		\label{burger_eq3}
		&u(a,t)=g_1(t), \qquad u(b,t)=g_2(t),
	\end{align}
\end{subequations}
where the constant $\nu$ denotes the viscosity, $f$ is a prescribed source term, $g_1(t)$ and $g_2(t)$ denote the boundary data, and $u_{in}(x)$ is the initial distribution.


\subsection{Hidden-Layer Concatenated Physics Informed Neural Networks}

We follow the settings from Section~\ref{sec_setting}, and
seek deep neural networks $u_{\theta_{i}} : D\times [0,t_i] \rightarrow \mathbb{R}$ for $1\leq i\leq l$  ($l$ denoting the number of time blocks) to  approximate the solution $u$ of \eqref{burger}. 
Define the following residual functions, for $1\leq i\leq l$,
\begin{subequations}\label{burger_pinn}
\begin{align}\label{burger_pinn_eq1}
		&R_{int_{i}}[u_{\theta_{i}}](x,t) =\frac{\partial u_{\theta_i}}{\partial t}-\nu\frac{\partial^2 u_{\theta_i}}{\partial x^2} +u_{\theta_i}\frac{\partial u_{\theta_i}}{\partial x} - f,\\
		\label{burger_pinn_eq2}
		&R_{tb_{i}}[u_{\theta_{i}}](x,t_{j-1}) =u_{\theta_{i}}|_{t=t_{j-1}}-u_{\theta_{j-1}}|_{t=t_{j-1}} \qquad1\leq j\leq i,\\
		\label{burger_pinn_eq3}
		&R_{sb1_{i}}[u_{\theta_{i}}](a,t) =u_{\theta_{i}}(a,t)-g_1(t),\qquad R_{sb2_{i}}[u_{\theta_{i}}](b,t) =u_{\theta_{i}}(b,t)-g_2(t).
	\end{align}
\end{subequations}
In these equations $\left.u_{\theta_{0}}\right|_{t=t_0}= u_{in}(x)$. Note that $R_{int_{i}}[u]=R_{tb_{i}}[u]=R_{sb_{i}}[u]=0$ for the exact solution $u$. 
With HLConcPINN we seek $\theta_i$ ($1\leq i\leq l$) to minimize the following quantity,
\begin{align}
	\label{burger_G}
	&\mathcal{E}_{G_i}(\theta_i)^2=\widetilde{\mathcal{E}}_{G_i}(\theta_i)^2+\mathcal{E}_{G_{i-1}}(\theta_{i-1})^2,
 \quad 1\leq i\leq l,\\
	\label{burger_Gi}
	&\widetilde{\mathcal{E}}_{G_i}(\theta_i)^2=\int_{t_{i-1}}^{t_{i}}\int_{D}|R_{int_{i}}[u_{\theta_{i}}](x,t)|^2\dxx\dt
	+\int_{t_{i-1}}^{t_{i}}(|R_{sb1_{i}}[u_{\theta_{i}}](a,t)|^2+|R_{sb2_{i}}[u_{\theta_{i}}](b,t)|^2)\dt
	\nonumber\\
	&\qquad+\left(\int_{t_{i-1}}^{t_{i}}|R_{sb1_{i}}[u_{\theta_{i}}](a,t)|^2\dt\right)^{\frac{1}{2}}+\left(\int_{t_{i-1}}^{t_{i}}|R_{sb2_{i}}[u_{\theta_{i}}](b,t)|^2\dt\right)^{\frac{1}{2}} \nonumber\\
&\qquad+\sum_{j=1}^{i}\int_{D}|R_{tb_{i}}[u_{\theta_{i}}](x,t_{j-1})|^2\dxx,
\end{align}
where $\mathcal{E}_{G_{i-1}}(\theta_{i-1}) =0$ for $i=1$.

The training data set consists of $\mathcal{S} = \bigcup_{i=1}^l\mathcal{S}_{i}$, with $\mathcal{S}_i = \mathcal{S}_{int_{i}} \cup \mathcal{S}_{sb_{i}} \cup \mathcal{S}_{tb_{i}}$. 
The spatial boundary training points are $\mathcal{S}_{sb_{i}}=\{y_n\}$ for $1\leq n \leq N_{sb_{i}}$, with  $y_n= (\bm{x},t)_n \in  \{a, b\} \times (t_{i-1},t_i)$. We approximate the integrals in \eqref{burger_G} by the mid-point rule, leading to the training loss functions,
\begin{align}
	\label{burger_T}
	&\mathcal{E}_{T_i}(\theta_i,\mathcal{S}_i)^2=\widetilde{\mathcal{E}}_{T_i}(\theta_i,\mathcal{S}_i)^2+\mathcal{E}_{T_{i-1}}(\theta_{i-1},\mathcal{S}_{i-1})^2, \quad 1\leq i\leq l,\\
	\label{burger_Ti}
	&\widetilde{\mathcal{E}}_{T_i}(\theta_i,\mathcal{S}_i)^2=\mathcal{E}_T^{int_{i}}(\theta_i,\mathcal{S}_{int_{i}})^2+\mathcal{E}_T^{sb1_{i}}(\theta_i,\mathcal{S}_{sb_{i}})^2+\mathcal{E}_T^{sb2_{i}}(\theta_i,\mathcal{S}_{sb_{i}})^2
	\nonumber\\
	&\qquad+\mathcal{E}_T^{sb1_{i}}(\theta_i,\mathcal{S}_{sb_{i}})+\mathcal{E}_T^{sb2_{i}}(\theta_i,\mathcal{S}_{sb_{i}})+\mathcal{E}_T^{tb_{i}}(\theta_i,\mathcal{S}_{tb_{i}})^2,
\end{align}
where $\mathcal{E}_T^{sb1_{i}}(\theta_i,\mathcal{S}_{sb_{i}})^2 = \sum_{n=1}^{N_{sb_{i}}}\omega_{sb_{i}}^n|R_{sb_{i}}[u_{\theta_{i}}](a,t_{sb_{i}}^n)|^2$, $\mathcal{E}_T^{sb2_{i}}(\theta_i,\mathcal{S}_{sb_{i}})^2 = \sum_{n=1}^{N_{sb_{i}}}\omega_{sb_{i}}^n|R_{sb_{i}}[u_{\theta_{i}}](b,t_{sb_{i}}^n)|^2$, and the remaining terms are defined according to equation \eqref{heat_TT}. 
Note that $\mathcal{E}_{T_{i-1}}(\theta_{i-1},\mathcal{S}_{i-1}) =0$ for $i=1$.

\subsection{Error Analysis}  

Let $\hat{u}_i= u_{\theta_{i}} - u$ denote the error of the HLConcPINN approximation ($u$ denoting the exact solution). 
Applying the Burgers' equation \eqref{burger} and the definitions of the different residuals, we obtain for $1\leq i\leq l$,
\begin{subequations}\label{burger_error}
	\begin{align}\label{burger_error_eq1}
		&R_{int_{i}}=\frac{\partial \hat{u}_i}{\partial t}
		-\nu\frac{\partial^2 \hat{u}_i}{\partial x^2} +u_{\theta_i}\frac{\partial u_{\theta_i}}{\partial x} - u\frac{\partial u}{\partial x},\\
		\label{burger_error_eq2}
		&R_{tb_{i}}|_{t=t_{j-1}}=\hat{u}_{i}|_{t=t_{j-1}} - \hat{u}_{j-1}|_{t=t_{j-1}}\qquad j = 1,2,\cdots, i,\\
		\label{burger_error_eq3}
		&R_{sb1_{i}}(a,t)=\hat{u}_i(a,t),\qquad R_{sb2_{i}}(b,t)=\hat{u}_i(b,t),
	\end{align}
\end{subequations}
where $\hat{u}_{0}|_{t=t_{0}} = 0$. Then, we define the total error of the HLConcPINN approximation as
\begin{equation}\label{burger_total}
	\mathcal{E}(\theta_i)^2=\int_{t_{i-1}}^{t_{i}}\int_{D}|\hat{u}_i(\bm{x},t)|^2\dxx\dt.
\end{equation}

\begin{Theorem}\label{sec4b_Theorem1} 
	Let $\widetilde\Omega_i = D\times [0,t_i]$.  Suppose $n$, $d$, $k \in \mathbb{N}$ with $n\geq2$ and $k\geq 3$, and $u\in H^k(\widetilde\Omega_i)$. For every integer $N>5$, there exists a HLConcPINN $u_{\theta_i}$ such that
	\begin{subequations}
		\begin{align}
			\label{lem4b.1}
			&\|R_{int_{i}}\|_{L^2(\widetilde\Omega_{i})}\lesssim N^{-k+2}{\rm ln}^2N,\\
			\label{lem4b.2}
			&\|R_{tb_{i}}(x,t_{j-1})\|_{L^2(D)}, \|R_{sb1_{i}}\|_{L^2(\{a\}\times [0,t_{i}])}, \|R_{sb2_{i}}\|_{L^2(\{b\}\times [0,t_{i}])}
			\lesssim N^{-k+1}{\rm ln}N,\quad 1\leq j\leq i.
		\end{align}
	\end{subequations}
\end{Theorem}
\begin{proof} By applying $u\in H^k(\widetilde\Omega_i)$, Lemmas \ref{Ar_2} and \ref{Ar_7}, we can conclude the proof.
\end{proof}

\begin{Theorem}\label{sec4b_Theorem2} Let $u\in C^1(\widetilde\Omega_i)$ be the classical solution to \eqref{burger}. Let $u_{\theta_{i}}$ $(1\leq i\leq l)$ be a HLConcPINN with parameter $\theta_i$. Then the following relation holds,
	\begin{align}\label{lem4b.3}
		&\int_{D} |\hat{u}_i(x,\tau)|^2\dxx
		\leq C_{G_{i}}\exp((1+C_{D_i})\Delta t), \quad \tau\in[t_{i-1},t_i],
  \\
		\label{lem4b.4}
		&\int_{t_{i-1}}^{t_{i}}\int_{D} |\hat{u}_i(x,t)|^2\dxx\dt\leq C_{G_{i}}\Delta t\exp((1+C_{D_i})\Delta t),
	\end{align}
	where
	\begin{align*}
		&C_{G_{i}}= 2C_{G_{i-1}}\exp((1+C_{D_i})\Delta t) + \widetilde{C}_{G_{i}}, \qquad C_{G_{0}}=0, \\
		&\widetilde{C}_{G_{i}} = 2\sum_{j=1}^{i}\int_{D}|R_{tb_{i}}(x,t_{j-1})|^2\dxx+ \int_{t_{i-1}}^{t_{i}}\int_{D}|R_{int_{i}}|^2\dxx\dt
		+C_{\partial D_{1i}}(\int_{t_{i-1}}^{t_{i}}|R_{sb1_{i}}|^2\dt)^{\frac{1}{2}}
		\nonumber\\
		&\qquad+C_{\partial D_{1i}}(\int_{t_{i-1}}^{t_{i}}|R_{sb2_{i}}|^2\dt)^{\frac{1}{2}}+C_{\partial D_{2i}}\int_{t_{i-1}}^{t_{i}}\big(|R_{sb1_{i}}|^2+|R_{sb2_{i}}|^2\big)\dt,
	\end{align*} 
  $C_{\partial D_{1i}}=2\nu{\Delta t}^{\frac{1}{2}}(\|u\|_{C^1(\widetilde\Omega_{*i})}+\|u_{\theta_{i}}\|_{C^1(\widetilde\Omega_{*i})})$, $C_{D_i}=2\Delta t^{\frac{1}{2}}\big(\|u_{\theta_{i}}\|_{C^1(\widetilde\Omega_{i})}+\frac{1}{2}\|u\|_{C^1(\widetilde\Omega_{i})}\big)$ and $C_{\partial D_{2i}}=\Delta t^{\frac{1}{2}}\|u\|_{C^0(\widetilde\Omega_{*i})}$.
\end{Theorem}
\begin{proof} 
Equation \eqref{burger_error_eq1} can be re-written as
	\begin{equation}\label{sec4b_eq0}
		R_{int_{i}}=\frac{\partial \hat{u}_i}{\partial t}-\nu\frac{\partial^2 \hat{u}_i}{\partial x^2} +\hat{u}_i\frac{\partial \hat{u}_i}{\partial x} 
		+\hat{u}_i\frac{\partial u}{\partial x}
		+ u\frac{\partial \hat{u}_i}{\partial x}.
	\end{equation}
	Note the following relation,
	\[
	\int_{D}u\frac{\partial \hat{u}_i}{\partial x}\hat{u}_{i}\dxx=\frac{1}{2}\int_{D}u\frac{\partial \hat{u}_i^2}{\partial x}\dxx
	= \frac{1}{2}u\hat{u}_i^2\big|_{a}^{b} 
	- \frac{1}{2}\int_{D}|\hat{u}_i^2|\frac{\partial u}{\partial x}\dxx.
	\]
 The rest of the proof follows the same approach in the proof of Theorem \ref{sec4_Theorem2}.
\end{proof} 

\begin{Theorem}\label{sec4b_Theorem3} Let $u\in C^4(\widetilde\Omega_i)$ be the classical solution of the Burgers' equation \eqref{burger},  and let $u_{\theta_{i}}$ $(1\leq i\leq l)$ be a HLConcPINN with parameter $\theta_i$. Then the total apporximation error satisfies
	\begin{align}\label{lem4b.5}
		&\int_{t_{i-1}}^{t_{i}}\int_{D} |\hat{u}_i(x,t)|^2\dxx\dt\leq C_{T_{i}}\Delta t\exp((1+C_{D_i})\Delta t)
		\nonumber\\
		&\qquad=\mathcal{O}(\mathcal{E}_{T_i}(\theta_i,\mathcal{S}_i)^2 + M_{int_i}^{-\frac{2}{d+1}} +M_{tb_i}^{-\frac{2}{d}}+M_{sb_i}^{-\frac{1}{d}}),
	\end{align}
	where
	\begin{align}\label{sec4b_eq10}
		&C_{T_{i}}= \widetilde{C}_{T_{i}}  + 2C_{T_{i-1}}\exp((1+C_{D_i})\Delta t), \qquad C_{T_{0}}=0, \\
		&\widetilde{C}_{T_{i}} = 2\sum_{j=1}^{i}\big(C_{({R_{tb_{i}}^2(x,t_{j-1})})}M_{tb_{i}}^{-\frac{2}{d}}+\mathcal{Q}_{M_{tb_{i}}}^{D}(R_{tb_{i}}^2(x,t_{j-1}))\big)+C_{({R_{int_{i}}^2})}M_{int_{i}}^{-\frac{2}{d+1}}+\mathcal{Q}_{M_{int_{i}}}^{\Omega_i}(R_{int_{i}}^2)
		\nonumber\\
		&\qquad+C_{\partial D_{1i}}\big(C_{({R_{sb1_{i}}^2})}M_{sb_{i}}^{-\frac{2}{d}}+\mathcal{Q}_{M_{sb_{i}}}^{\Omega_{*i}}(R_{sb1_{i}}^2)\big)^{\frac{1}{2}}+C_{\partial D_{1i}}\big(C_{({R_{sb2_{i}}^2})}M_{sb_{i}}^{-\frac{2}{d}}+\mathcal{Q}_{M_{sb_{i}}}^{\Omega_{*i}}(R_{sb2_{i}}^2)\big)^{\frac{1}{2}}
		\nonumber\\
		&\qquad+C_{\partial D_{2i}}\big(C_{({R_{sb1_{i}}^2})}M_{sb_{i}}^{-\frac{2}{d}}+\mathcal{Q}_{M_{sb_{i}}}^{\Omega_{*i}}(R_{sb1_{i}}^2)+C_{({R_{sb2_{i}}^2})}M_{sb_{i}}^{-\frac{2}{d}}+\mathcal{Q}_{M_{sb_{i}}}^{\Omega_{*i}}(R_{sb2_{i}}^2)\big).
	\end{align} 
\end{Theorem}
\begin{proof} The proof follows from  Lemma \ref{Ar_3}, Theorem \ref{sec4b_Theorem2}, and the quadrature error formula \eqref{int1}.
\end{proof} 

%% file: content/wave.tex
\section{HLConcPINN for Approximating the Wave Equation}\label{Wave}

\subsection{Wave Equation}

Consider the wave equation on the torus $D=[0,1)^d \subset \mathbb{R}^d$ with periodic boundary conditions:
\begin{subequations}\label{wave}
	\begin{align}
		\label{wave_eq0}
		&\frac{\partial u}{\partial t}- v = 0 \quad\qquad\qquad\qquad \text{in}\ D\times[0,T],\\
		\label{wave_eq1}
		&\frac{\partial v}{\partial t}-\Delta u = f  \qquad\qquad\qquad \text{in}\ D\times[0,T],\\
		\label{wave_eq2}
		&u(\bm{x},0)=\psi_{1}(\bm{x})  \quad\qquad\qquad\, \text{in}\ D,\\
		\label{wave_eq3}
		&v(\bm{x},0)=\psi_{2}(\bm{x})\quad\qquad\qquad\,\text{in}\ D,\\
		\label{wave_eq4}
		&u(\bm{x},t)=u(\bm{x}+1,t), \qquad\ \ \, \text{in}\ \partial D\times[0,T],\\
		\label{wave_eq5}
		&\nabla u(\bm{x},t)=\nabla u(\bm{x}+1,t),  \quad \, \text{in}\ \partial D\times[0,T],
	\end{align}
\end{subequations}
where $(u,v)$ are the field functions to solve, $f$ is a source term, and $(\psi_1,\psi_2)$ denote the initial data for $(u,v)$.


\subsection{Hidden-Layer Concatenated Physics Informed Neural Networks}\label{wave_HLConcPINN}


Following the settings from Section~\ref{sec_setting},
we seek neural networks $u_{\theta_{i}}: D\times [0,t_i] \rightarrow \mathbb{R}$ and $v_{\theta_{i}}: D\times [0,t_i] \rightarrow \mathbb{R}$ with $1\leq i\leq l$, parameterized by $\theta_{i}$, 
that approximate the solutions $u$ and $v$ of \eqref{wave}. We define the residuals (for $1\leq i\leq l$ and $1\leq j \leq i$),
\begin{subequations}\label{wave_pinn}
	\begin{align}
		\label{wave_pinn_eq1}
		&R_{int1_{i}}[u_{\theta_{i}},v_{\theta_{i}}](\bm{x},t) =\frac{\partial u_{\theta_i}}{\partial t}-v_{\theta_i},\quad
		R_{int2_{i}}[u_{\theta_{i}},v_{\theta_{i}}](\bm{x},t) =\frac{\partial v_{\theta_i}}{\partial t}-\Delta u_{\theta_i} - f,\\
		\label{wave_pinn_eq3}
		&R_{tb1_{i}}[u_{\theta_{i}}](\bm{x}, t_{j-1}) =u_{\theta_i}(\bm{x},t_{j-1})-u_{\theta_{j-1}}(\bm{x},t_{j-1}),\quad 
		R_{tb2_{i}}[v_{\theta_{i}}](\bm{x}, t_{j-1}) =v_{\theta_i}(\bm{x},t_{j-1})-v_{\theta_{j-1}}(\bm{x},t_{j-1}),\\
		&R_{sb1_{i}}[v_{\theta_{i}}](\bm{x},t) =v_{\theta_i}(\bm{x},t)-v_{\theta_i}(\bm{x}+1,t),\quad
		\label{wave_pinn_eq6}
		R_{sb2_{i}}[u_{\theta_{i}}](\bm{x},t) =\nabla u_{\theta_i}(\bm{x},t)-\nabla u_{\theta_i}(\bm{x}+1,t),
	\end{align}
\end{subequations}
where $u_{\theta_{0}}(\bm{x},t_{0})= \psi_{1}(\bm{x})$ and $v_{\theta_{0}}(\bm{x},t_{0})= \psi_{2}(\bm{x})$. For the exact solution $(u,v)$, we have  $R_{int1_{i}}[u,v]=R_{int2_{i}}[u,v]=R_{tb1_{i}}[u]=R_{tb2_{i}}[v]=R_{sb1_{i}}[v]=R_{sb2_{i}}[u]=0$. 

With the HLConcPINN algorithm, we minimize the quantity (for $1\leq i\leq l$),
\begin{align}\label{wave_G}
	&\mathcal{E}_{G_i}(\theta_i)^2=\widetilde{\mathcal{E}}_{G_i}(\theta_i)^2+\mathcal{E}_{G_{i-1}}(\theta_{i-1})^2,\\
	\label{wave_Gi}
	&\widetilde{\mathcal{E}}_{G_i}(\theta_i)^2=\int_{\Omega_i}(|R_{int1_{i}}[u_{\theta_{i}},v_{\theta_{i}}](\bm{x},t)|^2+|R_{int2_{i}}[u_{\theta_{i}},v_{\theta_{i}}](\bm{x},t)|^2+|\nabla R_{int1_{i}}[u_{\theta_{i}},v_{\theta_{i}}](\bm{x},t)|^2)\dx\dt
	\nonumber\\
	&\qquad+\sum_{j=1}^i\int_{D}(|R_{tb1_{i}}[u_{\theta_{i}}](\bm{x}, t_{j-1})|^2+|R_{tb2_{i}}[v_{\theta_{i}}](\bm{x}, t_{j-1})|^2+|\nabla R_{tb1_{i}}[u_{\theta_{i}}](\bm{x}, t_{j-1})|^2)\dx
	\nonumber\\
	&\qquad
	+\left(\int_{\Omega_{*i}}|R_{sb1_{i}}[v_{\theta_{i}}](\bm{x},t)|^2\ds\dt\right)^{\frac{1}{2}}+\left(\int_{\Omega_{*i}}|R_{sb2_{i}}[u_{\theta_{i}}](\bm{x},t)|^2\ds\dt\right)^{\frac{1}{2}}.
\end{align}
Here, $\mathcal{E}_{G_{0}}(\theta_{0}) =0$. 

	
	

Adopting the full training set $\mathcal{S} = \bigcup_{i=1}^l\mathcal{S}_{i}$ with $\mathcal{S}_i = \mathcal{S}_{int_{i}} \cup \mathcal{S}_{sb_{i}} \cup \mathcal{S}_{tb_{i}}$ as given in Section~\ref{sec_setting}, 
we approximate the integrals in \eqref{wave_G} by the midpoint rule, 
resulting in the training loss,
\begin{align}
	\label{wave_T}
	&\mathcal{E}_{T_i}(\theta_i,\mathcal{S}_i)^2=\widetilde{\mathcal{E}}_{T_i}(\theta_i,\mathcal{S}_i)^2+\mathcal{E}_{T_{i-1}}(\theta_{i-1},\mathcal{S}_{i-1})^2,\\
	\label{wave_Ti}
	&\widetilde{\mathcal{E}}_{T_i}(\theta_i,\mathcal{S}_i)^2
	=\mathcal{E}_T^{int1_{i}}(\theta_{i},\mathcal{S}_{int_{i}}) ^2
	+\mathcal{E}_T^{int2_{i}}(\theta_{i},\mathcal{S}_{int_{i}}) ^2
	+\mathcal{E}_T^{int3_{i}}(\theta_{i},\mathcal{S}_{int_{i}}) ^2
	+\mathcal{E}_T^{tb1_{i}}(\theta_{i},\mathcal{S}_{tb_{i}}) ^2\nonumber\\
	&\qquad
	+\mathcal{E}_T^{tb2_{i}}(\theta_{i},\mathcal{S}_{tb_{i}}) ^2+\mathcal{E}_T^{tb3_{i}}(\theta_{i},\mathcal{S}_{tb_{i}}) ^2 
	+\mathcal{E}_T^{sb1_{i}}(\theta_{i},\mathcal{S}_{sb_{i}})
	+\mathcal{E}_T^{sb2_{i}}(\theta_{i},\mathcal{S}_{sb_{i}}),
\end{align}
where
\begin{subequations}\label{wave_TT}
	\begin{align}
		\label{wave_T1}
		&\mathcal{E}_T^{int1_{i}}(\theta_{i},\mathcal{S}_{int_{i}}) ^2 = \sum_{n=1}^{N_{int_i}}\omega_{int_i}^n|R_{int1_{i}}[u_{\theta_{i}},v_{\theta_{i}}](\bm{x}_{int_{i}}^n,t_{int_{i}}^n)|^2,\\
		\label{wave_T01}
		&\mathcal{E}_T^{int2_{i}}(\theta_{i},\mathcal{S}_{int_{i}}) ^2 = \sum_{n=1}^{N_{int_i}}\omega_{int_i}^n|R_{int2_{i}}[u_{\theta_{i}},v_{\theta_{i}}]](\bm{x}_{int_{i}}^n,t_{int_{i}}^n)|^2,\\
		\label{wave_T001}
		&\mathcal{E}_T^{int3_{i}}(\theta_{i},\mathcal{S}_{int_{i}}) ^2 = \sum_{n=1}^{N_{int_i}}\omega_{int_i}^n|\nabla R_{int1_{i}}[u_{\theta_{i}},v_{\theta_{i}}](\bm{x}_{int_{i}}^n,t_{int_{i}}^n)|^2,\\
		\label{wave_T2}
		&\mathcal{E}_T^{tb1_{i}}(\theta_{i},\mathcal{S}_{tb_{i}}) ^2 = \sum_{j=1}^{i}\sum_{n=1}^{N_{tb_i}}\omega_{tb_i}^n|R_{tb1_{i}}[u_{\theta_{i}}](\bm{x}_{tb_{i}}^n,t_{j-1})|^2,\\
		\label{wave_T02}
		&\mathcal{E}_T^{tb2_{i}}(\theta_{i},\mathcal{S}_{tb_{i}}) ^2 = \sum_{j=1}^{i}\sum_{n=1}^{N_{tb_i}}\omega_{tb_i}^n|R_{tb2_{i}}[v_{\theta_{i}}](\bm{x}_{tb_{i}}^n,t_{j-1})|^2,\\
		\label{wave_T002}
		&\mathcal{E}_T^{tb3_{i}}(\theta_{i},\mathcal{S}_{tb_{i}}) ^2 = \sum_{j=1}^{i}\sum_{n=1}^{N_{tb_i}}\omega_{tb_i}^n|\nabla R_{tb1_{i}}[u_{\theta_{i}}](\bm{x}_{tb_{i}}^n,t_{j-1})|^2,\\
		\label{wave_T3}
		&\mathcal{E}_T^{sb1_{i}}(\theta_{i},\mathcal{S}_{sb_{i}}) ^2 = \sum_{n=1}^{N_{sb_i}}\omega_{sb_i}^n|R_{sb1_{i}}[v_{\theta_{i}}](\bm{x}_{sb_{i}}^n,t_{sb_{i}}^n)|^2,\\
		\label{wave_T03}
		&\mathcal{E}_T^{sb2_{i}}(\theta_{i},\mathcal{S}_{sb_{i}}) ^2 = \sum_{n=1}^{N_{sb_i}}\omega_{sb_i}^n|R_{sb2_{i}}[u_{\theta_{i}}](\bm{x}_{sb_{i}}^n,t_{sb_{i}}^n)|^2.
	\end{align}
\end{subequations}
Here, the quadrature points in space-time constitute the data sets $\mathcal{S}_{int_i} = \{(\bm{x}_{int_{i}}^n,t_{int_{i}}^n)\}_{n=1}^{N_{int_i}}$, $\mathcal{S}_{tb_i} = \{\bm{x}_{tb_{i}}^n\}_{n=1}^{N_{tb_i}}$ and $\mathcal{S}_{sb_i} = \{(\bm{x}_{sb_{i}}^n,t_{sb_{i}}^n)\}_{n=1}^{N_{sb_i}}$, and $\omega_{\star_i}^n$ are suitable quadrature weights with $\star$ denoting $int$, $tb$ or $sb$. Notice that $\mathcal{E}_{T_{i-1}}(\theta_{i-1},\mathcal{S}_{i-1}) =0$ for $i=1$.

\subsection{Error Analysis}  

Let 
$  
\hat{u}_i = u_{\theta_i}-u, \ \hat{v}_i = v_{\theta_i}-v	
$  
denote the error of the  HLConcPINN approximation of the solution $(u,v)$.
We define the total approximation error  by
\begin{equation}\label{wave_total}
	\mathcal{E}(\theta_i)^2=\int_{t_{i-1}}^{t_{i}}\int_{D}(|\hat{u}_i(\bm{x},t)|^2+|\nabla\hat{u}_i(\bm{x},t)|^2+|\hat{v}_i(\bm{x},t)|^2)\dx\dt.
\end{equation}
In light of the wave equations \eqref{wave} and the definitions of   residuals \eqref{wave_pinn}, we have
\begin{subequations}\label{wave_error}
	\begin{align}
		\label{wave_error_eq1}
		&R_{int1_{i}}=\frac{\partial \hat{u}_i}{\partial t}-\hat{v}_i,\\
		\label{wave_error_eq2}
		&R_{int2_{i}}=\frac{\partial \hat{v}_i}{\partial t}-\Delta \hat{u}_i,\\
		\label{wave_error_eq3}
		&R_{tb1_{i}}|_{t=t_{j-1}}=\hat{u}_{i}|_{t=t_{j-1}} - \hat{u}_{j-1}|_{t=t_{j-1}}\qquad j = 1,2,\cdots, i,\\
		\label{wave_error_eq4}
		&R_{tb2_{i}}|_{t=t_{j-1}}=\hat{v}_{i}|_{t=t_{j-1}} - \hat{v}_{j-1}|_{t=t_{j-1}}\qquad j = 1,2,\cdots, i,\\
		\label{wave_error_eq5}
		&R_{sb1_{i}}=\hat{v}_i(\bm{x},t)-\hat{v}_i(\bm{x}+1,t),\quad
		R_{sb2_{i}}=\nabla\hat{u}_i(\bm{x},t)-\nabla\hat{u}_i(\bm{x}+1,t).
	\end{align}
\end{subequations}

\begin{Theorem}\label{sec5_Theorem1} 
	Let $\widetilde\Omega_i = D\times [0,t_i]$ and $\widetilde\Omega_{*i}= \partial D\times [0,t_i]$ ($1\leq i\leq l$). Let $n$, $d$, $k \in \mathbb{N}$ with $n\geq2$ and $k\geq 3$, $u\in H^k(\widetilde\Omega_i)$ and $v\in H^{k-1}(\widetilde\Omega_i)$. For every integer $N>5$ and $1\leq j\leq i\leq l$, there exist HLConcPINNs $u_{\theta_i}$ and $v_{\theta_i}$, such that
	\begin{subequations}
		\begin{align}
			\label{lem5.1}
			&\|R_{int1_{i}}\|_{L^2(\widetilde\Omega_i)},\|R_{tb1_{i}}(\bm{x},t_{j-1})\|_{L^2(D)}\lesssim N^{-k+1}{\rm ln}N,\\
			\label{lem5.2}
			&\|R_{int2_{i}}\|_{L^2(\widetilde\Omega_i)},\|\nabla R_{int1_{i}}\|_{L^2(\widetilde\Omega_i)}, \|\nabla R_{tb1_{i}}(\bm{x},t_{j-1})\|_{L^2(D)}, \|R_{sb2_{i}}\|_{L^2(\widetilde\Omega_{*i})}\lesssim N^{-k+2}{\rm ln}^2N,\\
			\label{lem5.3}
			&\|R_{tb2_{i}}(\bm{x},t_{j-1})\|_{L^2(D)},\|R_{sb1_{i}}\|_{L^2(\widetilde\Omega_{*i})}\lesssim N^{-k+2}{\rm ln}N.
		\end{align}
	\end{subequations}
\end{Theorem}
\begin{proof} Similar to  Theorem \ref{sec4_Theorem1}, the proof follows by noting $u\in H^k(\widetilde\Omega_i)$, $v\in H^{k-1}(\widetilde\Omega_i)$, Lemmas \ref{Ar_3} and \ref{Ar_7}.
\end{proof}

Theorem \ref{sec5_Theorem1} implies that one can make the HLConcPINN residuals~\eqref{wave_pinn} arbitrarily small by choosing $N$ to be sufficiently large. It follows that the generalization error $\mathcal{E}_{G_i}(\theta_i)^2$ in~\eqref{wave_G} can be made arbitrarily small.  
The next two theorems show that: (i) the total approximation error $\mathcal{E}(\theta_i)^2$ is small when the generalization error $\mathcal{E}_{G_i}(\theta_i)^2$ is small with the HLConcPINN approximation $(u_{\theta_i},v_{\theta_i})$, and (ii) the total approximation error $\mathcal{E}(\theta_i)^2$ can be arbitrarily small if the training error $\mathcal{E}_{T_i}(\theta_i,\mathcal{S}_i)^2$ is sufficiently small and the sample set is sufficiently large. 

\begin{Theorem}\label{sec5_Theorem2} Let $d\in \mathbb{N}$, $u\in C^1(\widetilde\Omega_i)$ and $v\in C^0(\widetilde\Omega_i)$ be the classical solution to  \eqref{wave}. Let $u_{\theta_i}$ and $v_{\theta_i}$ denote the HLConcPINN approximation with parameter $\theta_i$. For all  $1\leq i\leq l$,  the following relation holds,
	\begin{align}\label{lem5.4}
		&\int_{D}(|\hat{u}_i(\bm{x},\tau)|^2+|\nabla\hat{u}_i(\bm{x},\tau)|^2+|\hat{v}_i(\bm{x},\tau)|^2)\dx\leq C_{G_{i}}\exp(2\Delta t), \quad \tau \in [t_{i-1},
	t_i],
  \\
		\label{lem5.5}
		&\int_{t_{i-1}}^{t_{i}}\int_{D}(|\hat{u}_i(\bm{x},t)|^2+|\nabla\hat{u}_i(\bm{x},t)|^2+|\hat{v}_i(\bm{x},t)|^2)\dx\dt\leq C_{G_{i}}\Delta t\exp(2\Delta t),
	\end{align}
	where $\Delta t=t_i-t_{i-1}$ and for $1\leq i\leq l$,
	\begin{align*}
		&C_{G_{i}}= \widetilde{C}_{G_{i}}  + 2C_{G_{i-1}}\exp(2\Delta t), \qquad C_{G_{0}}=0, \nonumber\\
		&\widetilde{C}_{G_{i}} =\int_{\Omega_i}(|R_{int1_{i}}|^2+|R_{int2_{i}}|^2+|\nabla R_{int1_{i}}|^2)\dx\dt
        \nonumber\\
		&\qquad +2\sum_{j=1}^i\int_{D}(|R_{tb1_{i}}(\bm{x},t_{j-1})|^2+|R_{tb2_{i}}(\bm{x},t_{j-1})|^2+|\nabla R_{tb1_{i}}(\bm{x},t_{j-1})|^2)\dx
		\nonumber\\
		&\qquad 
        +2|\Delta t|^{\frac{1}{2}}C_{\partial D_{1i}}\left(\int_{\Omega_{*i}}|R_{sb1_{i}}|^2\ds\dt\right)^{\frac{1}{2}}+2|\Delta t|^{\frac{1}{2}}C_{\partial D_{2i}}\left(\int_{\Omega_{*i}}|R_{sb2_{i}}|^2\ds\dt\right)^{\frac{1}{2}},
	\end{align*} 
 $C_{\partial D_{1i}}=|\partial D|^{\frac{1}{2}}(\|u\|_{C^1(\widetilde\Omega_{*i})}+\|u_{\theta_i}\|_{C^1(\widetilde\Omega_{*i})})$ and
 $C_{\partial D_{2i}}=|\partial D|^{\frac{1}{2}}(\|v\|_{C(\widetilde\Omega_{*i})}+\|v_{\theta_i}\|_{C(\widetilde\Omega_{*i})})$. 
\end{Theorem}
\begin{proof} The proof follows the same techniques as in the proof of Theorem \ref{sec4_Theorem2} and in the proof of Theorem 3.4 of~\cite{Qian2303.12245} .
\end{proof} 

\begin{Theorem}\label{sec5_Theorem3} Let $d\in \mathbb{N}$ and $T>0$. Let $u\in C^4(\widetilde\Omega_i)$ and $v\in C^3(\widetilde\Omega_i)$ be the classical solution to \eqref{wave},  and let $(u_{\theta_i},v_{\theta_i})$ denote the HLConcPINN approximation with parameter $\theta_i$. Then the total approximation error satisfies
	\begin{align}\label{lem5.6}
		&\int_{t_{i-1}}^{t_{i}}\int_{D}(|\hat{u}_i(\bm{x},t)|^2+|\nabla \hat{u}_i(\bm{x},t)|^2+|\hat{v}_i(\bm{x},t)|^2)\dx\dt\leq C_{T_{i}}\Delta t\exp(2\Delta t)
		\nonumber\\
		&\quad=\mathcal{O}(\mathcal{E}_{T_i}(\theta_i,\mathcal{S}_i)^2 + M_{int_i}^{-\frac{2}{d+1}} +M_{tb_i}^{-\frac{2}{d}}+M_{sb_i}^{-\frac{1}{d}}),
	\end{align}
	where 
	\begin{align}\label{sec5_eq5}
		&C_{T_{i}}= \widetilde{C}_{T_{i}}  + 2C_{T_{i-1}}\exp(2\Delta t), \qquad C_{T_{0}}=0, \nonumber\\
		&\widetilde{C}_{T_{i}}=
		2\sum_{j=1}^{i}\big(C_{({R_{tb1_{i}}^2(\bm{x},t_{j-1})})}M_{tb_{i}}^{-\frac{2}{d}}+\mathcal{Q}_{M_{tb_{i}}}^{D}(R_{tb1_{i}}^2(\bm{x},t_{j-1}))+C_{({R_{tb2_{i}}^2(\bm{x},t_{j-1})})}M_{tb_{i}}^{-\frac{2}{d}}+\mathcal{Q}_{M_{tb_{i}}}^{D}(R_{tb2_{i}}^2(\bm{x},t_{j-1}))\big)
		\nonumber\\
		&\qquad+2\sum_{j=1}^{i}\big(C_{({|\nabla R_{tb1_{i}}(\bm{x},t_{j-1})|^2})}M_{tb_{i}}^{-\frac{2}{d}}+\mathcal{Q}_{M_{tb_{i}}}^{D}(|\nabla R_{tb1_{i}}(\bm{x},t_{j-1})|^2)\big)
		+C_{({R_{int1_{i}}^2})}M_{int_{i}}^{-\frac{2}{d+1}}+\mathcal{Q}_{M_{int_{i}}}^{\Omega_i}(R_{int1_{i}}^2)
		\nonumber\\
		&\qquad+C_{(R_{int2_{i}}^2)}M_{int_{i}}^{-\frac{2}{d+1}}+\mathcal{Q}_{M_{int_{i}}}^{\Omega_i}(R_{int2_{i}}^2)+C_{(|\nabla R_{int1_{i}}|^2)}M_{int_{i}}^{-\frac{2}{d+1}}+\mathcal{Q}_{M_{int_{i}}}^{\Omega_i}(|\nabla R_{int1_{i}}|^2)
		\nonumber\\
        &\qquad+2|\Delta t|^{\frac{1}{2}}\big(C_{\partial D_{1i}}(C_{({R_{sb1_{i}}^2})}M_{sb_{i}}^{-\frac{2}{d}}+\mathcal{Q}_{M_{sb_{i}}}^{\Omega_{*i}}(R_{sb1_{i}}^2))^{\frac{1}{2}}+C_{\partial D_{2i}}(C_{({R_{sb2_{i}}^2})}M_{sb_{i}}^{-\frac{2}{d}}+\mathcal{Q}_{M_{sb_{i}}}^{\Omega_{*i}}(R_{sb2_{i}}^2))^{\frac{1}{2}}\big).
	\end{align}
\end{Theorem}
\begin{proof} Using Lemma \ref{Ar_3}, Theorem \ref{sec5_Theorem2} and the quadrature error formula \eqref{int1} leads to this result.
\end{proof} 

%% file: content/sinegordon.tex
\section{HLConcPINN for Approximating the Nonlinear Klein-Gordon Equation}\label{Klein-Gordon}
\subsection{Nonlinear Klein-Gordon Equation}

Let $D\subset \mathbb{R}^d$ be  an open connected bounded set with  boundary $\partial D$. We consider the  nonlinear Klein-Gordon equation:
\begin{subequations}\label{SG}
	\begin{align}
		\label{SG_eq0}
		&\frac{\partial u}{\partial t} - v = 0 
		\ \qquad\qquad\qquad\qquad\qquad\qquad\qquad \text{in}\ D\times[0,T],\\
		\label{SG_eq1}
		&\varepsilon^2\frac{\partial v}{\partial t} = a^2\Delta u - \varepsilon_1^2u-g(u)+f   \  \ \quad\qquad\qquad \text{in}\ D\times[0,T],\\
		\label{SG_eq2}
		&u(\bm{x},0)=\psi_{1}(\bm{x})\qquad\qquad\qquad\qquad\qquad\qquad \ \, \text{in}\ D,\\
		\label{SG_eq3}
		&v(\bm{x},0)=\psi_{2}(\bm{x}) \qquad\qquad\qquad\qquad\qquad\qquad \ \, \text{in}\ D,\\
		\label{SG_eq4}
		&u(\bm{x},t)|_{\partial D}=u_{d}(\bm{x},t)   \qquad\qquad\qquad\qquad\qquad\ \text{in}\ \partial D\times[0,T],
	\end{align}
\end{subequations}
where $u$ and $v$ are the field functions to be solved for, $f$ is a source term, and $u_d$, $\psi_1$ and $\psi_2$ denote the boundary/initial conditions. $\varepsilon>0$, $a>0$ and $\varepsilon_1\geq 0$ are  constants.
$g(u)$ is a nonlinear term. We assume that $g$ is globally Lipschitz,
i.e.~there exists a constant $L$ (independent of $v$ and $w$) such that
\begin{equation}\label{non2}
	|g(v) - g(w)|\leq L|v-w|, \qquad \forall v, \, w \in \mathbb{R}.
\end{equation} 


\subsection{Hidden-Layer Concatenated Physics Informed Neural Networks}

Following the settings in Section~\ref{sec_setting},
we define the following residuals for the HLConcPINN approximation $u_{\theta_{i}}: D\times [0,t_i] \rightarrow \mathbb{R}$ and $v_{\theta_{i}}:$ $D\times$ $[0,t_i]$ $\rightarrow \mathbb{R}$ (for $1\leq j \leq i\leq l$) of the equations in \eqref{SG}:
\begin{subequations}\label{SG_pinn}
	\begin{align}
		\label{SG_pinn_eq1}
		&R_{int1_i}[u_{\theta_i},v_{\theta_i}](\bm{x},t) =\frac{\partial u_{\theta_i}}{\partial t}-v_{\theta_i},\quad
		R_{int2_i}[u_{\theta_i},v_{\theta_i}](\bm{x},t) =\varepsilon^2\frac{\partial v_{\theta_i}}{\partial t}-a^2\Delta u_{\theta_i} +\varepsilon_1^2u_{\theta_i}+g(u_{\theta_i})-f,\\
		\label{SG_pinn_eq3}
		&R_{tb1_i}[u_{\theta_i}](\bm{x},t_{j-1}) =u_{\theta_i}(\bm{x},t_{j-1})-u_{\theta_{j-1}}(\bm{x},t_{j-1}),\quad
		R_{tb2_i}[v_{\theta_i}](\bm{x},t_{j-1}) =v_{\theta_i}(\bm{x},t_{j-1})-v_{\theta_{j-1}}(\bm{x},t_{j-1}),\\
		\label{SG_pinn_eq5}
		&R_{sb_i}[v_{\theta_i}](\bm{x},t) =v_{\theta_i}(\bm{x},t)|_{\partial D}-u_{dt}(\bm{x},t),
	\end{align}
\end{subequations}
where $u_{dt}=\frac{\partial u_d}{\partial t}$, $u_{\theta_0}(\bm{x},t_0)=\psi_{1}(\bm{x})$ and $v_{\theta_0}(\bm{x},t_0)$ $=\psi_{2}(\bm{x})$. Notice that $R_{int1_i}[u,v]=R_{int2_i}[u,v]=R_{tb1_i}[u]=R_{tb2_i}[v]=R_{sb_i}[v]=0$ for the exact solution $(u,v)$. We minimize the following generalization error (for $1\leq i\leq l$),
\begin{align}\label{SG_G}
	&\mathcal{E}_{G_i}(\theta_i)^2=\widetilde{\mathcal{E}}_{G_i}(\theta_i)^2+\mathcal{E}_{G_{i-1}}(\theta_{i-1})^2,\\
	\label{SG_Gi}
	&\widetilde{\mathcal{E}}_{G_i}(\theta_i)^2=\int_{\Omega_{i}}\left(|R_{int1_i}[u_{\theta_i},v_{\theta_i}](\bm{x},t)|^2+|R_{int2_i}[u_{\theta_i},v_{\theta_i}](\bm{x},t)|^2+|\nabla R_{int1_i}[u_{\theta_i},v_{\theta_i}](\bm{x},t)|^2\right)\dx\dt
	\nonumber\\
	&\qquad\qquad+\sum_{j=1}^i\int_{D}\left(|R_{tb1_i}[u_{\theta_i}](\bm{x},t_{j-1})|^2+|R_{tb2_i}[v_{\theta_i}](\bm{x},t_{j-1})|^2+|\nabla R_{tb1_i}[u_{\theta_i}](\bm{x},t_{j-1})|^2\right)\dx
	\nonumber\\
	&\qquad\qquad
	+\left(\int_{\Omega_{*i}}|R_{sb_i}[v_{\theta_i}](\bm{x},t)|^2\ds\dt\right)^{\frac{1}{2}},
\end{align}
where $\mathcal{E}_{G_0}(\theta_0)=0$. 

Employing the  training set $\mathcal{S} = \bigcup_{i=1}^l\mathcal{S}_{i}$ with $\mathcal{S}_i = \mathcal{S}_{int_{i}} \cup \mathcal{S}_{sb_{i}} \cup \mathcal{S}_{tb_{i}}$ as given in Section~\ref{sec_setting} and the midpoint rule for approximating the residuals, we arrive at the training loss 
as follows (for $1\leq i\leq l$),
\begin{align}\label{SG_T}
	&\mathcal{E}_{T_i}(\theta_i,\mathcal{S}_i)^2=\widetilde{\mathcal{E}}_{T_i}(\theta_i,\mathcal{S}_i)^2+\mathcal{E}_{T_{i-1}}(\theta_{i-1},\mathcal{S}_{i-1})^2,\\
	\label{SG_Ti}
	&\widetilde{\mathcal{E}}_{T_i}(\theta_i,\mathcal{S}_i)^2
	=\mathcal{E}_T^{int1_i}(\theta_i,\mathcal{S}_{int_i})^2
	+\mathcal{E}_T^{int2_i}(\theta_i,\mathcal{S}_{int_i})^2
	+\mathcal{E}_T^{int3_i}(\theta_i,\mathcal{S}_{int_i})^2
	+\mathcal{E}_T^{tb1_i}(\theta_i,\mathcal{S}_{tb_i})^2
	\nonumber\\
	&\qquad\qquad+\mathcal{E}_T^{tb2_i}(\theta_i,\mathcal{S}_{tb_i})^2 +\mathcal{E}_T^{tb3_i}(\theta_i,\mathcal{S}_{tb_i})^2 
	+\mathcal{E}_T^{sb_i}(\theta_i,\mathcal{S}_{sb_i}),
\end{align}
where $\mathcal{E}_T^{sb_i}(\theta_i,\mathcal{S}_{sb_i})^2= \sum_{n=1}^{N_{sb_i}}\omega_{sb_i}^n|R_{sb_i}[v_{\theta_i}](\bm{x}_{sb_i}^n,t_{sb_i}^n)|^2$ and the other terms are defined in \eqref{wave_TT}. Notice that $\mathcal{E}_{T_0}(\theta_0,\mathcal{S}_0)=0$. 

\subsection{Error Analysis}  

Let $\hat{u}_i = u_{\theta_i}-u$ and $\hat{v}_i = v_{\theta_i}-v$ denote the errors of HLConcPINN approximation,
where $(u,v)$ are the exact solutions. We define the total approximation error of  HLConcPINN as, 
\begin{equation}\label{SG_total}
	\mathcal{E}(\theta_i)^2=\int_{t_{i-1}}^{t_{i}}\int_{D}(|\hat{u}_i(\bm{x},t)|^2+a^2|\nabla \hat{u}_i(\bm{x},t)|^2+\varepsilon^2|\hat{v}_i(\bm{x},t)|^2)\dx\dt.
\end{equation}
 Subtracting the equations \eqref{SG} from the residual equations~\eqref{SG_pinn} leads to,
\begin{subequations}\label{SG_error}
	\begin{align}
		\label{SG_error_eq1}
		&R_{int1_i}=\frac{\partial \hat{u}_i}{\partial t}-\hat{v}_i,\\
		\label{SG_error_eq2}
		&R_{int2_i}=\varepsilon^2\frac{\partial \hat{v}_i}{\partial t}-a^2\Delta \hat{u}_i +\varepsilon_1^2\hat{u}_i+g(u_{\theta_i})-g(u),\\
		\label{SG_error_eq3}
		&R_{tb1_{i}}|_{t=t_{j-1}}=\hat{u}_{i}|_{t=t_{j-1}} - \hat{u}_{j-1}|_{t=t_{j-1}}\qquad j = 1,2,\cdots, i,\\
		\label{SG_error_eq4}
		&R_{tb2_{i}}|_{t=t_{j-1}}=\hat{v}_{i}|_{t=t_{j-1}} - \hat{v}_{j-1}|_{t=t_{j-1}}\qquad j = 1,2,\cdots, i,\\
		\label{SG_error_eq5}
		&R_{sb_i}=\hat{v}_i|_{\partial D}.
	\end{align}
\end{subequations}

The following theorems summarize the results of the HLConcPINN approximation for the nonlinear Klein-Gordon equation.


\begin{Theorem}\label{sec6_Theorem1} 
	Let $n\geq2$, $d$, $k \in \mathbb{N}$ with $k\geq 3$. Suppose that $g(u)$ is Lipschitz continuous, $u \in C^k(D\times[0,t_i])$ and $v \in C^{k-1}(D\times[0,t_i])$ ($1\leq i\leq l$). Then for every integer $N>5$, there exist HLConcPINNs $u_{\theta_i}$ and $v_{\theta_i}$, such that
	\begin{subequations}
		\begin{align}
			\label{lem6.1}
			&\|R_{int1_i}\|_{L^2(D\times[0,t_i])},\|R_{tb1_i}\|_{L^2(D)}\lesssim N^{-k+1}{\rm ln}N,\\
			\label{lem6.2}
			&\|R_{int2_i}\|_{L^2(D\times[0,t_i])},\|\nabla R_{int1_i}\|_{L^2(D\times[0,t_i])}, \|\nabla R_{tb1_i}\|_{L^2(D)}\lesssim N^{-k+2}{\rm ln}^2N,\\
			\label{lem6.3}
			&\|R_{tb2_i}\|_{L^2(D)},\|R_{sb_i}\|_{L^2(\partial D\times [0,t_i])}\lesssim N^{-k+2}{\rm ln}N.
		\end{align}
	\end{subequations}
\end{Theorem}
\begin{proof} Similar to that of Theorem \ref{sec4_Theorem1}, the proof follows by noting  $u \in C^k(D\times[0,t_{i}])$, $v \in C^{k-1}(D\times[0,t_{i}])$, Lemmas \ref{Ar_3}, \ref{Ar_4} and \ref{Ar_7},  and the globally Lipschitz condition \eqref{non2}.
\end{proof} 

This theorem implies that the HLConcPINN residuals in \eqref{SG_pinn} can be made arbitrarily small by choosing a sufficiently large $N$. Therefore, the generalization error $\mathcal{E}_{G_i}(\theta_i)^2$ can be made arbitrarily small. We next show that the HLConcPINN total approximation error $\mathcal{E}(\theta_i)^2$ can be controlled by the generalization error $\mathcal{E}_{G_i}(\theta_i)^2$ (Theorem \ref{sec6_Theorem2} below), and by the training error~$\mathcal{E}_{T_i}(\theta_i,\mathcal{S}_i)^2$ (Theorem \ref{sec6_Theorem3} below). 

\begin{Theorem}\label{sec6_Theorem2} Let $d\in \mathbb{N}$, $u\in C^1(D\times[0,t_i])$ and $v\in C^0(D\times[0,t_i])$ be the classical solution of \eqref{SG}. Let $(u_{\theta_i},v_{\theta_i})$ denote the  HLConcPINN approximation with parameter $\theta_i$. For 
 $1\leq i\leq l$,  the following relation holds,	
	\begin{align}\label{lem6.4}
		&\int_{D}(|\hat{u}_i(\bm{x},t)|^2+a^2|\nabla \hat{u}_i(\bm{x},t)|^2+\varepsilon^2|\hat{v}_i(\bm{x},t)|^2)\dx\leq C_{G_i}\exp\left((2+\varepsilon_1^2+L+a^2)\Delta t\right),\\
        \label{lem6.5}
		&\int_{t_{i-1}}^{t_{i}}\int_{D}(|\hat{u}_i(\bm{x},t)|^2+a^2|\nabla \hat{u}_i(\bm{x},t)|^2+\varepsilon^2|\hat{v}_i(\bm{x},t)|^2)\dx\dt
		\leq C_{G_i}\Delta t\exp\left((2+\varepsilon_1^2+L+a^2)\Delta t\right),
	\end{align}
	where
	\begin{align*}
		&C_{G_{i}}= \widetilde{C}_{G_{i}}  + 2C_{G_{i-1}}\exp((2+\varepsilon_1^2+L+a^2)\Delta t), \qquad C_{G_{0}}=0, \nonumber\\
		&\widetilde{C}_{G_{i}} =\int_{\Omega_i}(|R_{int1_{i}}|^2+|R_{int2_{i}}|^2+a^2|\nabla R_{int1_{i}}|^2)\dx\dt+2C_{\partial D_i}|\Delta t|^{\frac{1}{2}}\left(\int_{\Omega_{*i}}|R_{sb_i}|^2\ds\dt\right)^{\frac{1}{2}}
		\nonumber\\
		&\qquad +2\sum_{j=1}^i\int_{D}(|R_{tb1_{i}}(\bm{x},t_{j-1})|^2+\varepsilon^2|R_{tb2_{i}}(\bm{x},t_{j-1})|^2+a^2|\nabla R_{tb1_{i}}(\bm{x},t_{j-1})|^2)\dx,
	\end{align*} 
	and $C_{\partial D_i}=a^2|\partial D|^{\frac{1}{2}}(\|u\|_{C^1(\partial D\times[0,t_i])}+||u_{\theta_i}||_{C^1(\partial D\times[0,t_i])})$.
\end{Theorem}
\begin{proof} The proof is similar to that of Theorem \ref{sec4_Theorem2},  by noting \eqref{non2}.
\end{proof} 

\begin{Theorem}\label{sec6_Theorem3} Let $d\in \mathbb{N}$ and $T>0$, and let $u\in C^4(D\times[0,t_i])$ and $v\in C^3(D\times[0,t_i])$ be the classical solution to  \eqref{SG}. Let $(u_{\theta_i},v_{\theta_i})$ denote the HLConcPINN approximation with parameter $\theta_i$. Then the following relation holds ($1\leq i\leq l$), 
	\begin{align}\label{lem6.6}
		&\int_{\Omega_{i}}(|\hat{u}_i(\bm{x},t)|^2+a^2|\nabla \hat{u}_i(\bm{x},t)|^2+\varepsilon^2|\hat{v}_i(\bm{x},t)|^2)\dx\dt\leq C_{T_i}\Delta t\exp\left((2+\varepsilon_1^2+L+a^2)\Delta t\right)
		\nonumber\\
		&\quad=\mathcal{O}(\mathcal{E}_{T_i}(\theta_i,\mathcal{S}_i)^2 + M_{int_i}^{-\frac{2}{d+1}} +M_{tb_i}^{-\frac{2}{d}}+M_{sb_i}^{-\frac{1}{d}}),                       
	\end{align}
	where 
	\begin{align*}
		&C_{T_{i}}= \widetilde{C}_{T_{i}}  + 2C_{T_{i-1}}\exp\left((2+\varepsilon_1^2+L+a^2)\Delta t\right), \qquad C_{T_{0}}=0, \\
		&\widetilde{C}_{T_{i}}=2\sum_{j=1}^{i}\big(C_{({R_{tb1_{i}}^2(\bm{x},t_{j-1})})}M_{tb_{i}}^{-\frac{2}{d}}+\mathcal{Q}_{M_{tb_{i}}}^{D}(R_{tb1_{i}}^2(\bm{x},t_{j-1}))+\varepsilon^2C_{({R_{tb2_{i}}^2(\bm{x},t_{j-1})})}M_{tb_{i}}^{-\frac{2}{d}}+\varepsilon^2\mathcal{Q}_{M_{tb_{i}}}^{D}(R_{tb2_{i}}^2(\bm{x},t_{j-1}))\big)
		\nonumber\\
		&\qquad+2a^2\sum_{j=1}^{i}\big(C_{({|\nabla R_{tb1_{i}}(\bm{x},t_{j-1})|^2})}M_{tb_{i}}^{-\frac{2}{d}}+\mathcal{Q}_{M_{tb_{i}}}^{D}(|\nabla R_{tb1_{i}}(\bm{x},t_{j-1})|^2)\big)
		+C_{({R_{int1_{i}}^2})}M_{int_{i}}^{-\frac{2}{d+1}}+\mathcal{Q}_{M_{int_{i}}}^{\Omega_i}(R_{int1_{i}}^2)
		\nonumber\\
		&\qquad+C_{({R_{int2_i}^2})}M_{int}^{-\frac{2}{d+1}}+\mathcal{Q}_{M_{int}}^{\Omega_{i}}(R_{int2_i}^2)
		+a^2\left(C_{(|\nabla R_{int1_i}|^2)}M_{int}^{-\frac{2}{d+1}}+\mathcal{Q}_{M_{int}}^{\Omega_{i}}(|\nabla R_{int1_i}|^2)\right),\\
		&\qquad+2C_{\partial D_i}|\Delta t|^{\frac{1}{2}}\left(C_{({R_{sb_i}^2})}M_{sb}^{-\frac{2}{d}}+\mathcal{Q}_{M_{sb}}^{\Omega_{*i}}(R_{sb_i}^2)\right)^{\frac{1}{2}}.
	\end{align*}
\end{Theorem}
\begin{proof} The proof follows from  Lemma \ref{Ar_3}, Theorem \ref{sec6_Theorem2} and the quadrature error formula \eqref{int1}.
\end{proof} 

It follows from Theorem \ref{sec6_Theorem3} that the HLConcPINN approximation error $\mathcal{E}(\theta_i)^2$ can be arbitrarily small, provided that the training error $\mathcal{E}_{T_i}(\theta_i,\mathcal{S}_i)^2$ is sufficiently small  and the sample set is sufficiently large.

%% file: content/numerical_examples.tex
\begin{table}[tb]\small
\centering\small
\begin{tabular}{c|c|c|c}
\hline
$N_c$                    & Hidden layers                                    & Activation functions                 & Figures/Tables    \\ \hline
\multirow{2}{*}{varied}  & {[}90, 90{]}                               & {[}tanh, tanh{]}                     & Table \ref{num_heat_tab2}                \\ \cline{2-4} 
                      & {[}90, 90, 10{]}                            & {[}tanh, tanh, sine{]}               & Tables \ref{tab_Burger_err_1}, \ref{tab_wave_err_1}, \ref{tab_KG_err_1}   \\ \hline
\multirow{6}{*}{2000} & varied                                      & all tanh                             & Table \ref{num_heat_tab1}                \\ \cline{2-4} 
                      & {[}90, 90{]}                              & {[}tanh, tanh{]}                     & Figure \ref{PINN_num_heat_fig4_a}  \\ \cline{2-4} 
                      & {[}90, 90, 10{]}                          & {[}tanh, tanh, tanh{]}               & Figures \ref{PINN_num_heat_fig1_1}$-$\ref{PINN_num_heat_fig3_1}, \ref{PINN_num_heat_fig4_b}  \\ \cline{2-4} 
                      & \multirow{3}{*}{{[}90, 90, 10, 10{]}} & {[}tanh, tanh, tanh, tanh{]}         & Table \ref{num_heat_tab3}                \\ \cline{3-4} 
                      &                                           & {[}tanh, tanh, Gaussian, Gaussian{]}       & Figure \ref{PINN_num_heat_fig4_c}\\ \cline{3-4} 
                      &                                           & {[}tanh, tanh, softplus, softplus{]} & Figure \ref{PINN_num_heat_fig4_d}\\ \hline
\multirow{2}{*}{2500} & {[}90, 90, 10{]}                          & varied                                  & Tables \ref{tab_Burger_err_2}, \ref{tab_wave_err_3}, \ref{tab_KG_err_3} \\ \cline{2-4} 
                      & {[}90, 90, 10{]}                          & {[}tanh, tanh, sine{]}               & Figures \ref{PINN_num_Burger_fig1}$-$\ref{PINN_num_KG_fig5_1}        \\ \hline
\end{tabular}
\caption{Summary of neural network settings (network architecture, activation functions, training data points) for the test problems in Section~\ref{numerical_examples}. Shown in the second column are the nodes in the hidden layers only.
}
\label{num_table0}
\end{table}


\section{Computational Examples}\label{numerical_examples}

We next present a set of numerical examples  to test the performance of the HLConcPINN method developed herein. This method  has several distinctive features,  distinguishing it from the standard PINN and  recent neural networks with theoretical guarantees. Specifically, these include:
\begin{itemize}
\item The method is based on hidden-layer concatenated FNNs (HLConcFNN), in which the output nodes and all the hidden nodes are logically connected. This architecture is critical to the theoretical analyses, and it endows the method with the subsequent properties. 

\item The current error analyses hold for network architectures with  two or more hidden layers, and with essentially any activation function having a sufficient regularity for all hidden layers beyond the first two.  This generalized capability contrasts starkly with the recent PINN methods that have a theoretical guarantee for solving PDEs but are confined to  network architectures having two hidden layers and the $\tanh$ activation function (see e.g.~\cite{2023_IMA_Mishra_NS,Qian2303.12245}).

\item The method espouses a modified block time marching (ExBTM) strategy for long-time dynamic simulations. In the modified scheme, the ``initial condition" for a particular time block is informed by the  approximations from all previous time blocks evaluated at a set of discrete time instants. The modified BTM scheme is crucial for the error analyses. In contrast, the original BTM formulation as given in e.g.~\cite{2021_CMAME_LEMDD} is not amenable to theoretical analysis.

\end{itemize}

We consider the heat, Burgers', wave and the nonlinear Klein-Gordon equations in one spacial dimension plus time, with the following common settings in the numerical tests.
We partition the temporal dimension into five uniform time blocks. Within each time block, we utilize $N_c$ collocation points sampled from a uniform random distribution within the spatial-temporal domain. Additionally, $N_c$ uniform random points are selected along each spatial boundary and the initial boundary. Simulations were performed by systematically varying $N_c$ between $1500$ and $3000$. After the neural networks are trained, we compare the HLConcPINN solution with the exact solution on a set of $N_{ev} = 1000 \times 1000$ uniform grid points (test/evaluation points) within each time block that encompasses the problem domain and its boundaries.

The HLConcPINN errors reported below have been calculated as follows. Suppose $ z_n = (\bm x, t)_n \in D\times[0,T]$ ($n=1,\cdots, N_{ev} $) denote the set of test points. The  errors are then defined by 
\begin{align}\label{BTM_num_eq0}
	  &l^2\text{-error} =
          \frac{\sqrt{\sum_{n=1}^{N_{ev}} |u(z_n) - u_{\theta}(z_n)|^2}}{\sqrt{\sum_{n=1}^{N_{ev}} u(z_n)^2}},
          \quad
		l^\infty\text{-error}= \frac{\max\{|u(z_n) - u_{\theta}(z_n)|\}_{n=1}^{N_{ev}} }{\sqrt{\left(\sum_{n=1}^{N_{ev}} u(z_n)^2\right)/N_{ev}}},
\end{align}
where $ u_\theta $ denotes the HLConcPINN solution and $u$ denotes the exact solution. 

Following the  analyses in previous sections, we employ network architectures with two or more hidden layers in the numerical tests, with the $\tanh$ activation function for the first two hidden layers. For the subsequent hidden layers, we have tested a range of activation functions. Table~\ref{num_table0} provides an overview of the neural network settings for the test results reported in the subsequent subsections.
%

As discussed in Remark~\ref{heat_Remark1}, we adopt a modified block time marching scheme in this work.
This is different from the original block time marching scheme of~\cite{2021_CMAME_LEMDD}, which uses the solution data of the preceding time block at the last time instant as the initial condition. Both the original and the modified block time marching schemes have been tested in the simulations, with their results marked by ``HLConcPINN-BTM" and ``HLConcPINN-ExBTM" in the following discussions, respectively.



In the following simulations, the neural network has been trained by a combination of the Adam optimizer and the L-BFGS optimizer.
Within each time block, the network is trained first by Adam for 100 epochs. The training then continues with the L-BFGS optimizer for another 30,000 iterations. Our application code is implemented in Python with the PyTorch library.

\input{content/numerical_examples_heat}

\input{content/numerical_examples_Burgers}

\input{content/numerical_examples_wave}
\input{content/numerical_examples_KG}

%% file: content/numerical_examples_heat.tex
\begin{figure}[tb]
	\centering
	\subfloat[$u$]{\includegraphics[width=0.15\linewidth]{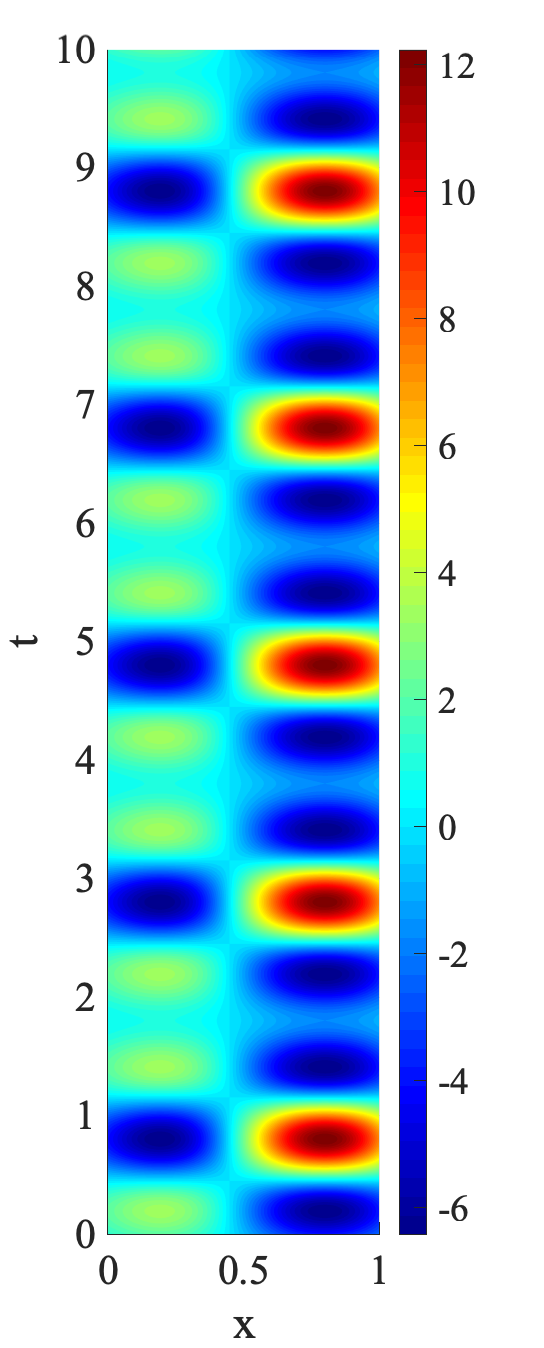}}
    \subfloat[$u_\theta$]{\includegraphics[width=0.15\linewidth]{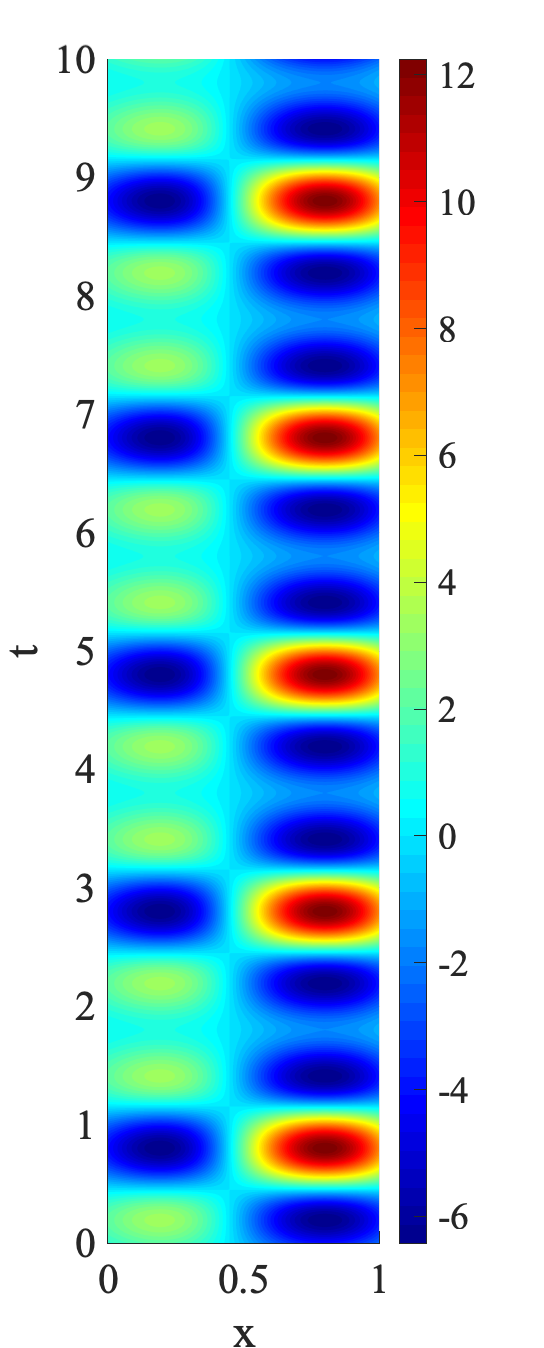}}
    \subfloat[$|u-u_\theta|$]{\includegraphics[width=0.15\linewidth]{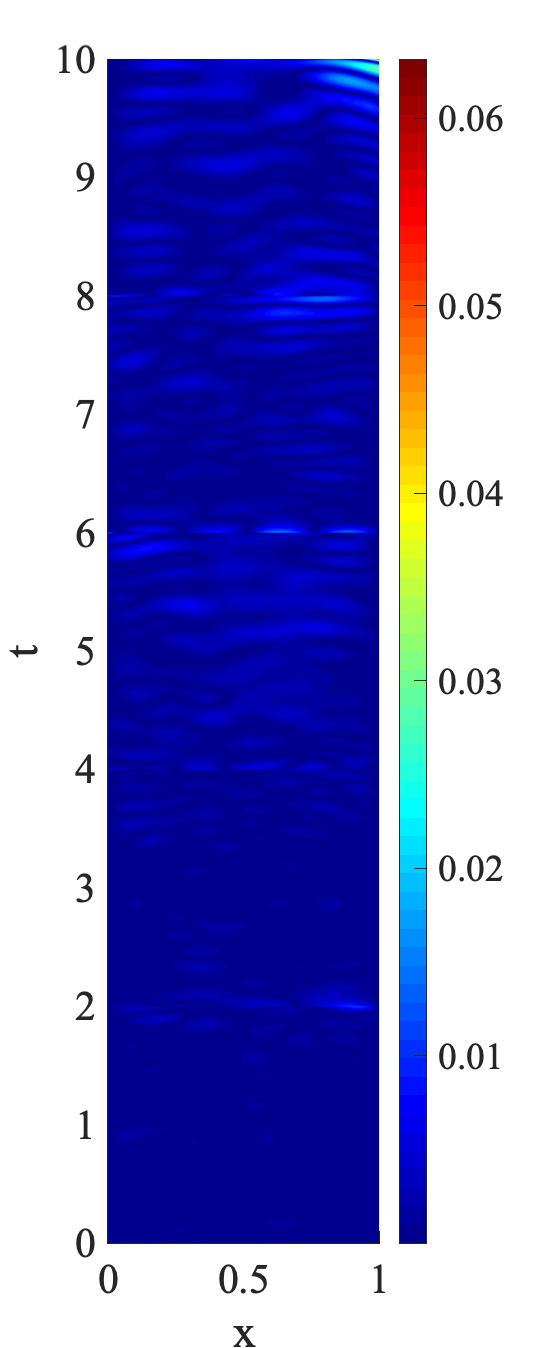}}
    \subfloat[$u^*_\theta$]{\includegraphics[width=0.15\linewidth]{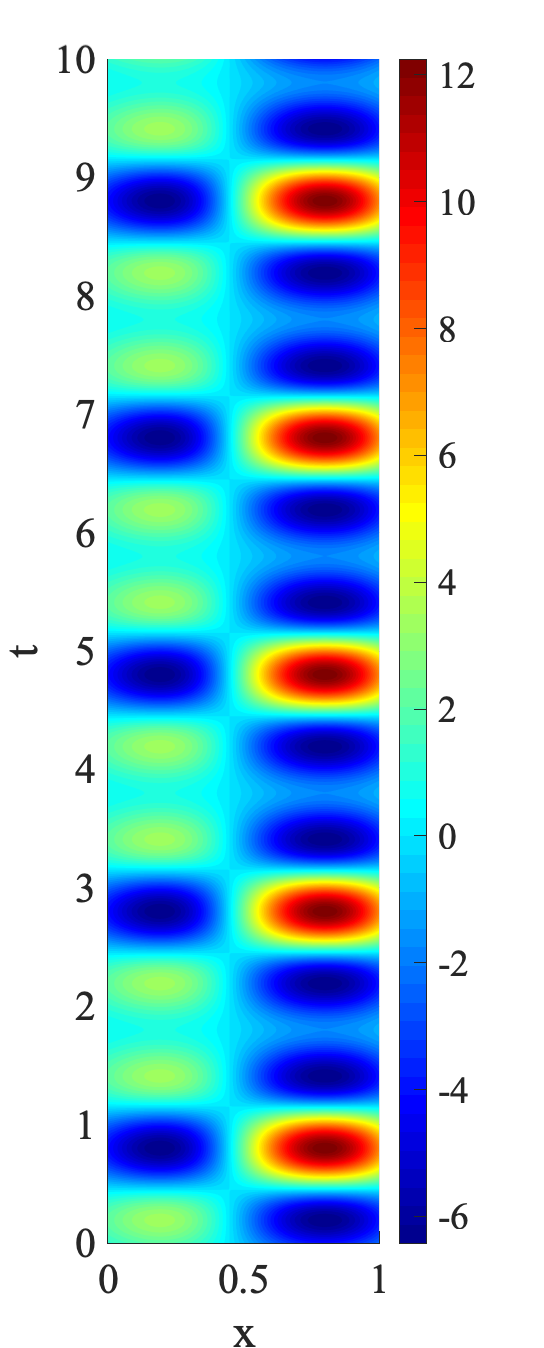}}
    \subfloat[$|u - u^*_\theta|$]{\includegraphics[width=0.15\linewidth]{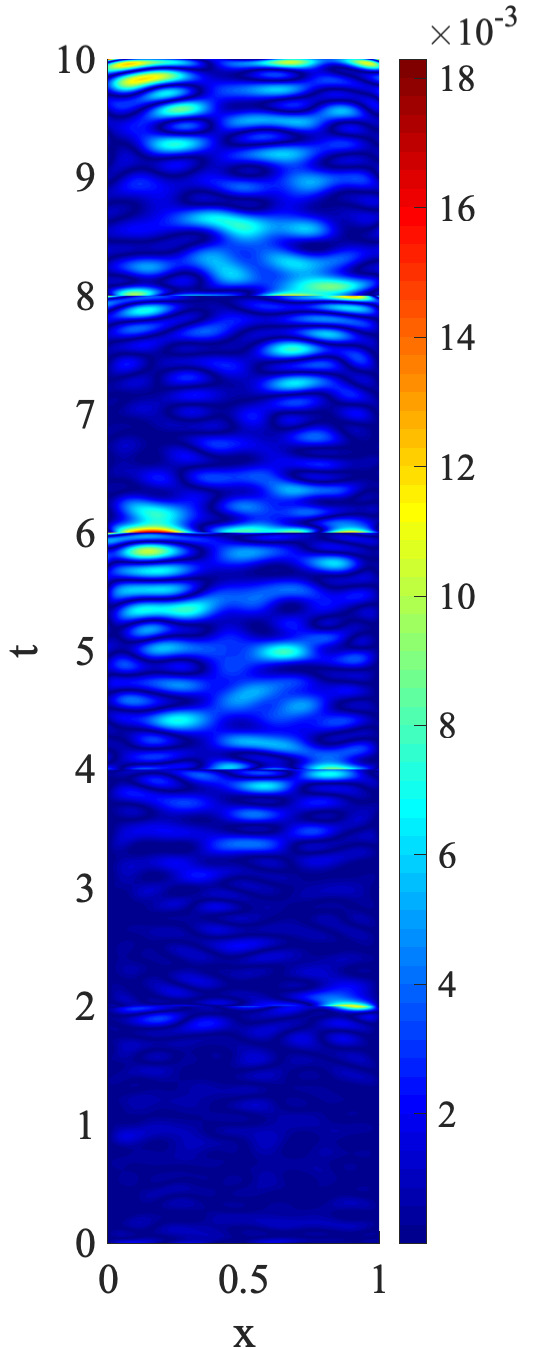}}
	\caption{Heat equation: Distributions of the true solution (a), the HLConcPINN-ExBTM solution (b) and its point-wise absolute error (c), the HLConcPINN-BTM solution (d) (denoted by $u^*_\theta$) and its point-wise absolute error (e), in the spacial-temporal domain.
 NN architecture: [2, 90, 90, 10, 1], with the $\tanh$ activation function; $N_c=2000$ for the collocation points.
  }
	\label{PINN_num_heat_fig1_1}
\end{figure}

\begin{figure}[tb]
	\centering
	\subfloat[$ t=2.5 $]{
		\begin{minipage}[b]{0.25\textwidth}
			\includegraphics[scale=0.25]{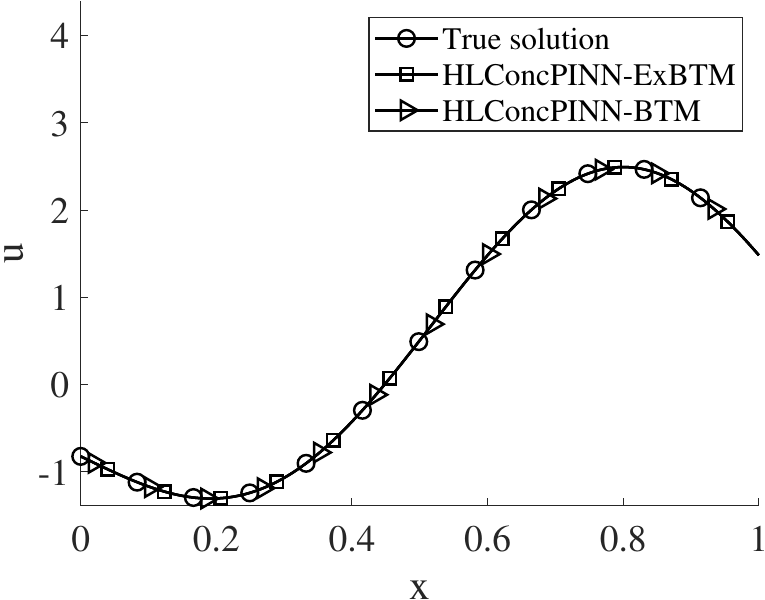}\\
			\includegraphics[scale=0.25]{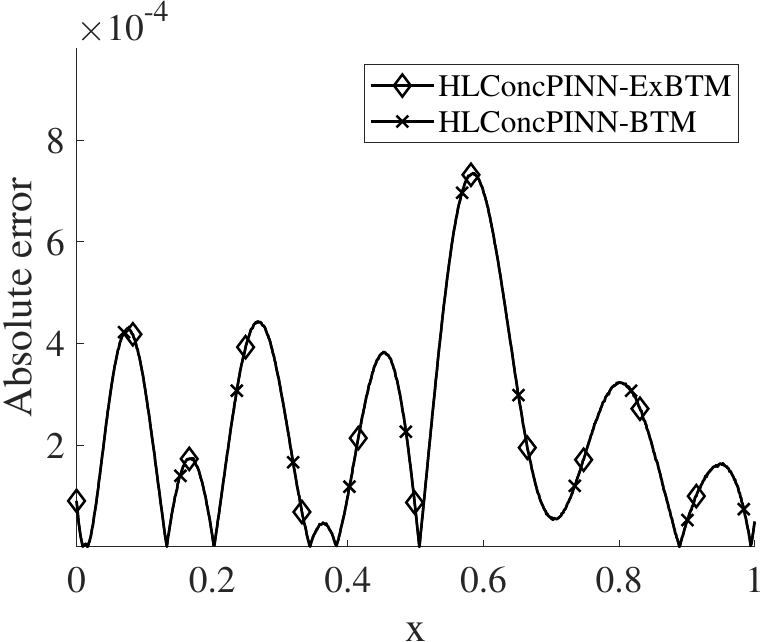}
		\end{minipage}
	}
	\subfloat[$ t=5 $]{
		\begin{minipage}[b]{0.25\textwidth}
			\includegraphics[scale=0.25]{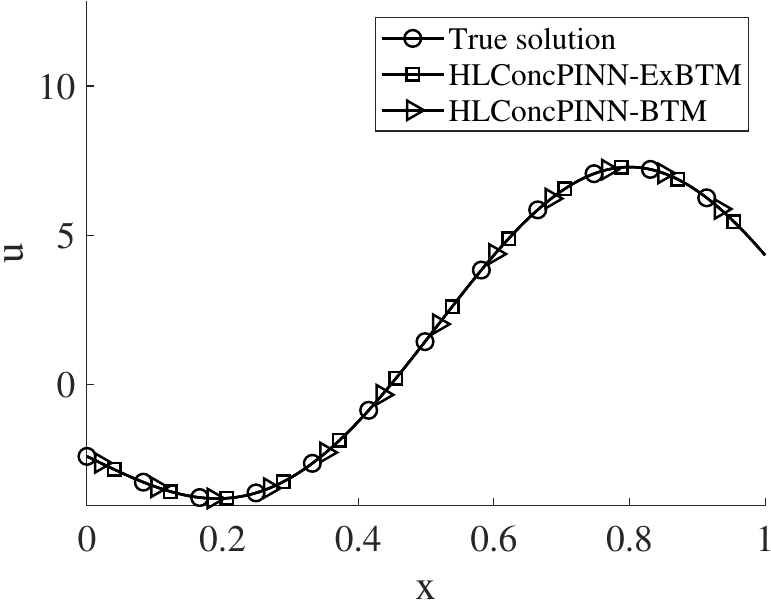}\\
			\includegraphics[scale=0.25]{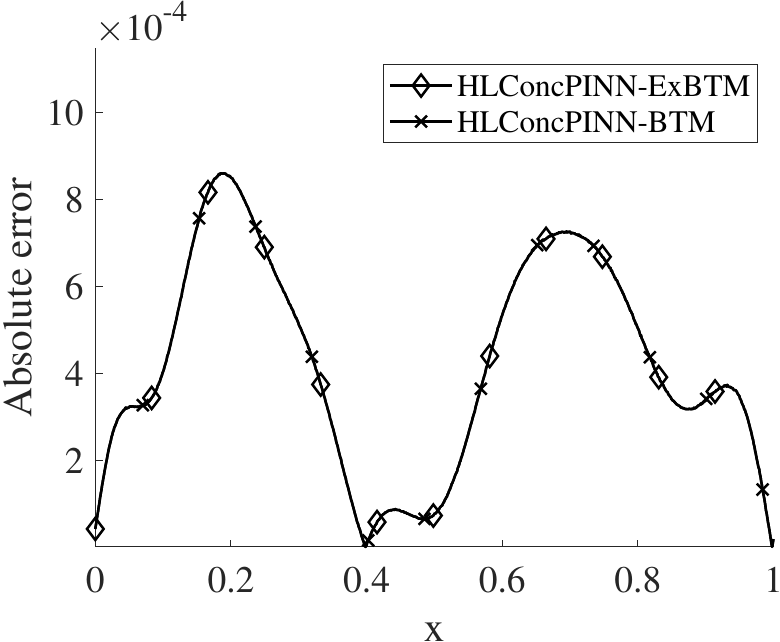}
		\end{minipage}
	}
	\subfloat[$ t=9.5 $]{
		\begin{minipage}[b]{0.25\textwidth}
			\includegraphics[scale=0.25]{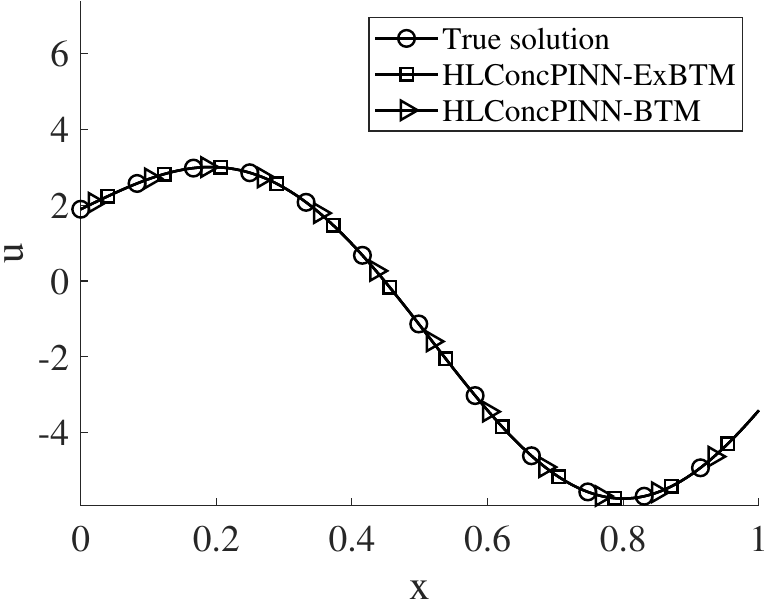}\\
			\includegraphics[scale=0.25]{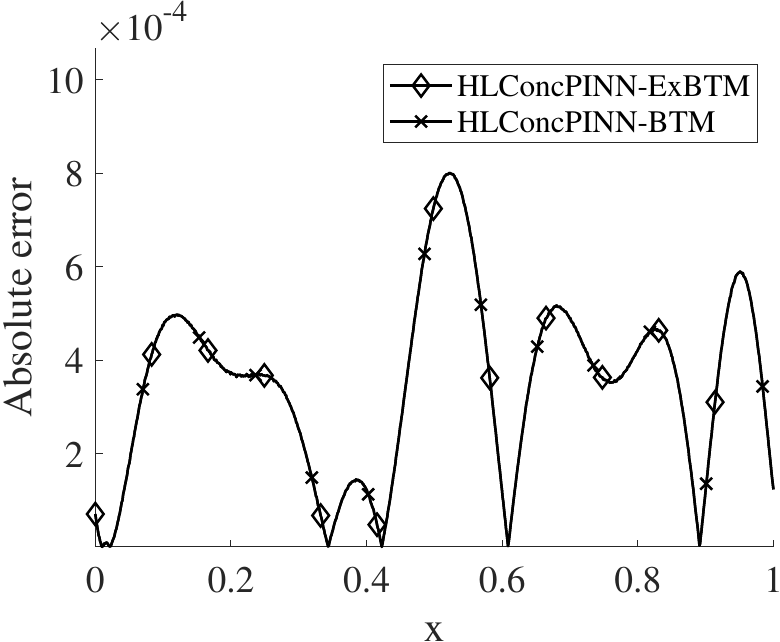}
		\end{minipage}
	}
	\caption{Heat equation: Top row, comparison of profiles of the true solution,  HLConcPINN-ExBTM solution, and HLConcPINN-BTM solution at several time instants. Bottom row, profiles of the absolute error of the HLConcPINN-ExBTM and HLConcPINN-BTM solutions. 
 NN architecture: [2, 90, 90, 10, 1] with $\tanh$ activation function;
$N_c=2000$ for the training collocation points. 
 }
	\label{PINN_num_heat_fig2_1}
\end{figure}

\begin{figure}[tb]
	\centering
	\subfloat[HLConcPINN-ExBTM]{\includegraphics[width=0.33\linewidth]{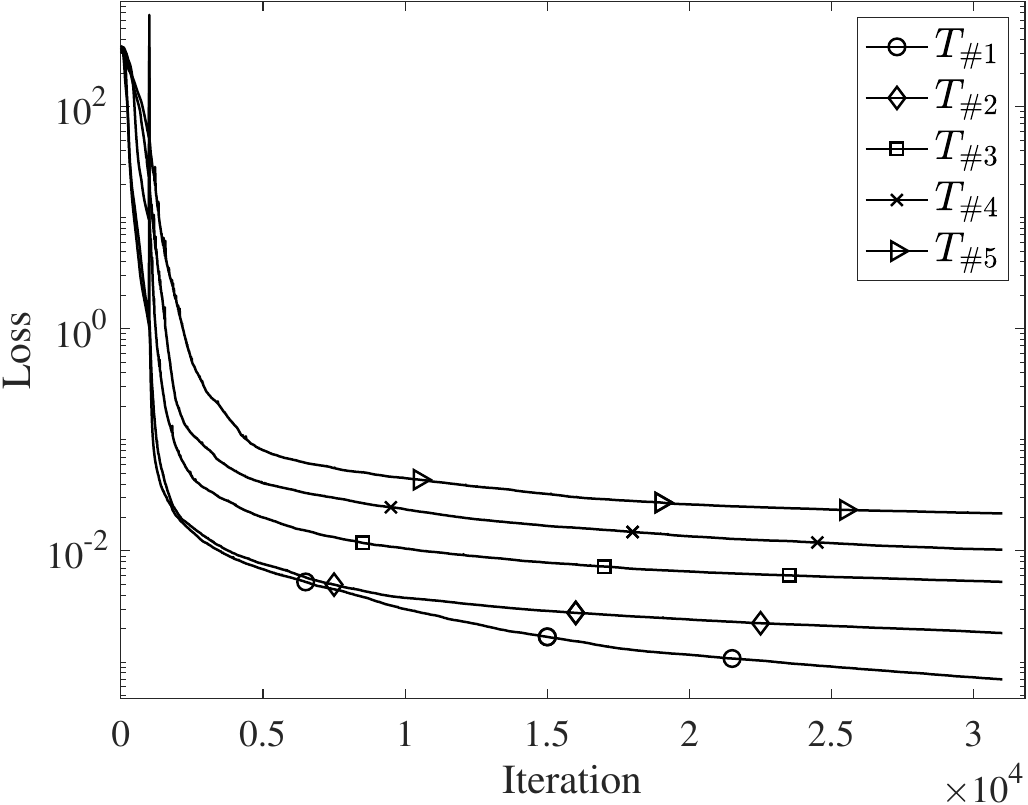}}\quad
	\subfloat[HLConcPINN-BTM]{\includegraphics[width=0.33\linewidth]{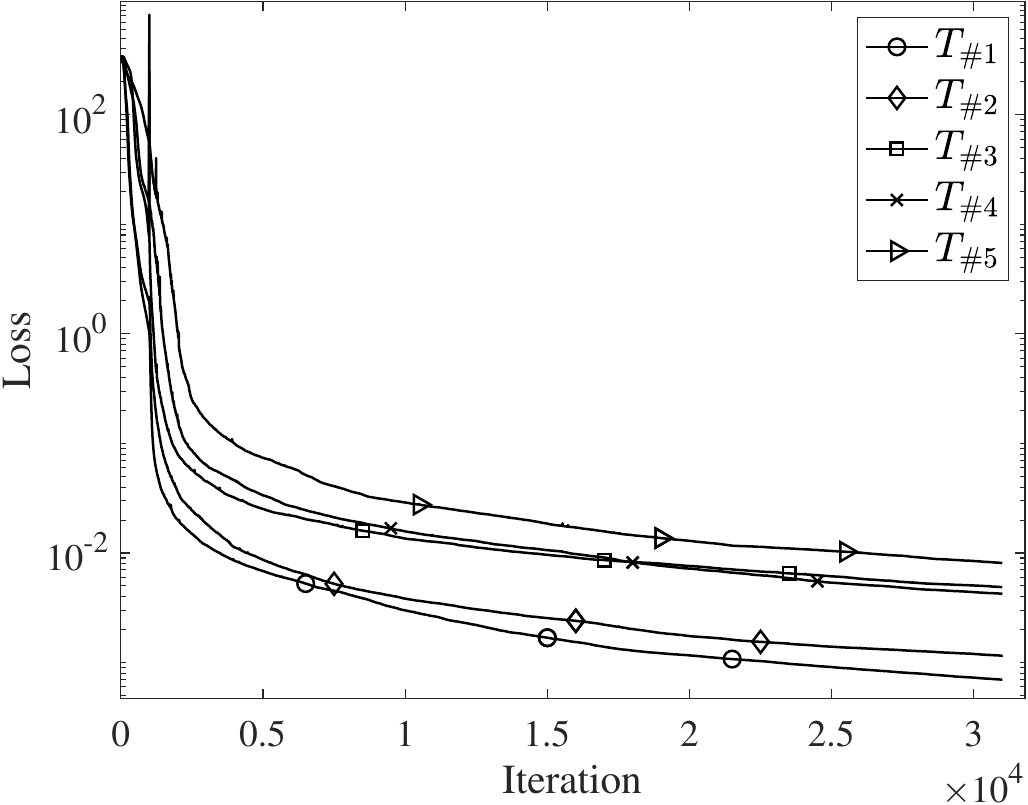}}\hspace{0.1em}
	\caption{Heat equation: Training loss versus the training iterations for different time blocks with the (a) HLConcPINN-ExBTM and (b) HLConcPINN-BTM methods.
 NN architecture: [2, 90, 90, 10, 1], $\tanh$ activation function; $N_c=2000$ for the training collocation points. The legend shows the time block index, with e.g.~$T_{\#2}$ denoting the second time block.
 }
	\label{PINN_num_heat_fig3_1}
\end{figure}

\subsection{Heat Equation}

We test the HLConcPINN scheme for solving the heat equation in one spatial dimension (plus time).
Consider the spatial-temporal domain $\Omega=\{ (x,t) | x\in [0, 1], t\in[0, 10] \}$, and the following initial/boundary-value problem,
\begin{subequations}\label{num_heat_eq1}
    \begin{align}
    &\frac{\partial u}{\partial t} - \nu \frac{\partial^2 u}{\partial x^2} = f(x,t),\\
    &u(0, t)=g_1(t), \quad u(1,t)=g_2(t),\\
    &u(x,0) = h(x),
\end{align}
\end{subequations}
where $u(x,t)$ is the field function to be solved for, $f(x,t)$ is a source term, and $\nu=0.1$ is the diffusion coefficient. $g_1(x)$ and $g_2(x)$ are the boundary conditions, and $h(x)$ is the initial field distribution. 
In this test, we choose the source term $f$ such that the following field function satisfies \eqref{num_heat_eq1},
\begin{align}\label{num_heat_eq2}
    u(x,t) = \Big( 2\cos(\pi x + \frac{\pi}{5}) + \frac{3}{2}\cos(2\pi x-\frac{3\pi}{5})\Big) \Big( 2\cos(\pi t + \frac{\pi}{5}) + \frac{3}{2}\cos(2\pi t-\frac{3\pi}{5})\Big),
\end{align}
and we choose the initial/boundary conditions by restricting \eqref{num_heat_eq2} to the corresponding boundaries.


We consider two forms for the HLConcPINN loss function, corresponding to the original block time marching (BTM) scheme from~\cite{2021_CMAME_LEMDD} and the modified BTM scheme (denoted by ExBTM) developed in this work. The loss function for the current ExBTM scheme is given by, for time block $i$ ($1\leq i\leq l$),
\begin{align}\label{num_heat_eq_loss1}
    Loss_i^{I} = &\frac{W_1}{N_c}\sum_{n=1}^{N_c}\left[\frac{\partial u_{\theta_i}}{\partial t}(x_{int}^n,t_{int}^n) - \Delta u_{\theta_i}(x_{int}^n,t_{int}^n) - f(x_{int}^n,t_{int}^n)\right]^2 \notag \\
    & + \frac{W_2}{N_c} \sum_{j=1}^i \sum_{n=1}^{N_c}\left[ u_{\theta_i}(x_{tb}^n, t_{j-1}) - u_{\theta_{j-1}}(x_{tb}^n, t_{j-1}) \right]^2 \notag\\
    & + W_3\Big(\frac{1}{N_c}\sum_{n=1}^{N_c}\left[ (u_{\theta_i}(0, t_{sb}^n) - g_1(t_{sb}^n))^2 + (u_{\theta_i}(1, t_{sb}^n) - g_2(t_{sb}^n))^2 \right]  \Big)^{1/2} \notag\\
    & + Loss_{i-1}^{I},
\end{align}
where $Loss_{0}^{I}=0$, and we have added a set of penalty coefficients $W_k>0$ ($k=1,2,3$) for different loss terms. Note also that in the simulations we have approximated the integral by averaging over the collocation points in the domain,  while in the analysis the mid-point rule has been adopted.
The loss form corresponding to the original BTM scheme is, for time block $i$ ($1\leq i\leq l$),
\begin{align}\label{num_heat_eq_loss2}
    Loss_i^{II}=&\frac{W_1}{N_c}\sum_{n=1}^{N_c}\left[\frac{\partial u_{\theta_i}}{\partial t}(x_{int}^n,t_{int}^n) - \Delta u_{\theta_i}(x_{int}^n,t_{int}^n) - f(x_{int}^n,t_{int}^n)\right]^2 \notag \\
    & + \frac{W_2}{N_c}\sum_{n=1}^{N_c}\left[ u_{\theta_i}(x_{tb}^n, t_{i-1}) - u_{\theta_{i-1}}(x_{tb}^n, t_{i-1}) \right]^2 \notag\\
    & + W_3\Big(\frac{1}{N_c}\sum_{n=1}^{N_c}\left[ (u_{\theta_i}(0, t_{sb}^n) - g_1(t_{sb}^n))^2 + (u_{\theta_i}(1, t_{sb}^n) - g_2(t_{sb}^n))^2 \right]  \Big)^{1/2},
\end{align}
where $u_{\theta_{0}}(x,t_0)= h(x)$.
In subsequent simulations the penalty coefficients are fixed to $(W_1, W_2, W_3)=(0.8, 0.9, 0.9)$ in both $Loss_i^{I}$ and $Loss_i^{II}$, and $5$ uniform time blocks are employed in block time marching. The HLConcPINN schemes employing these two distinctive loss functions will be designated as HLConcPINN-ExBTM ($Loss_i^{I}$) and HLConcPINN-BTM ($Loss_i^{II}$), respectively.


An overview of the solution field and the training histories is provided by Figures~\ref{PINN_num_heat_fig1_1},~\ref{PINN_num_heat_fig2_1}, and~\ref{PINN_num_heat_fig3_1} for the HLConcPINN-ExBTM and the HLConcPINN-BTM methods. 
Figure~\ref{PINN_num_heat_fig1_1} shows distributions in the space-time domain of the true solution, the HLConcPINN-ExBTM solution, and the HLConcPINN-BTM solution, as well as the point-wise absolute errors of the HLConcPINN-ExBTM and HLConcPINN-BTM solutions. Figure~\ref{PINN_num_heat_fig2_1} compares profiles of the true solution, the HLConcPINN-ExBTM and HLConcPINN-BTM solutions at three time instants ($t=2.5$, $5$ and $9.5$), and also shows the error profiles of the HLConcPINN-ExBTM and HLConcPINN-BTM methods. Figure~\ref{PINN_num_heat_fig3_1} depicts the training loss histories for each of the $5$ time blocks with the HLConcPINN-ExBTM and HLConcPINN-BTM methods. In this set of simulations, three hidden layers and the $\tanh$ activation function are employed in the neural network. The specific parameter values are provided in the captions of these figures; see also Table \ref{num_table0}. The HLConcPINN-ExBTM and the HLConcPINN-BTM methods are able to capture the solution quite accurately, with the HLConcPINN-BTM solution appearing slightly better.


\begin{table}[tb]\small
	\centering\small
\begin{tabular}{c|@{}c@{}|cc|cc|cc|cc}
\hline
\multirow{2}{*}{Error}    & \multirow{2}{*}{Time block} & \multicolumn{2}{c|}{$N_c=1500$} & \multicolumn{2}{c|}{$N_c=2000$} & \multicolumn{2}{c|}{$N_c=2500$} & \multicolumn{2}{c}{$N_c=3000$} \\ \cline{3-10}
                      &  & ExBTM  & BTM  & ExBTM  & BTM   & ExBTM  & BTM  & ExBTM  & BTM   \\ \hline
\multirow{5}{*}{$l^2$} & $T_{\#1}$  & 3.63e-04  & 3.63e-04  & 4.16e-04  & 4.16e-04  & 2.58e-04  & 2.58e-04  & 3.88e-04  & 3.88e-04  \\ \cline{3-10} 
                      & $T_{\#2}$  & 1.00e-03  & 5.93e-04  & 5.73e-04  & 5.93e-04  & 9.14e-04  & 7.85e-04  & 4.42e-04  & 5.51e-04  \\ \cline{3-10}
                      & $T_{\#3}$  & 7.29e-04  & 6.37e-04  & 7.75e-04  & 2.93e-04  & 9.89e-04  & 1.00e-03  & 7.03e-04  & 4.86e-04  \\ \cline{3-10} 
                      & $T_{\#4}$  & 5.21e-04  & 7.84e-04  & 6.93e-04  & 5.52e-04   & 8.38e-04  & 6.03e-04  & 7.49e-04  & 9.14e-04  \\ \cline{3-10} 
                      & $T_{\#5}$  & 9.14e-04  & 1.03e-03  & 1.77e-03  & 1.40e-03  & 6.04e-04  & 6.01e-04  & 1.21e-03  & 1.21e-03  \\ \hline
\multirow{5}{*}{$l^\infty$} & $T_{\#1}$  & 1.64e-03  & 1.64e-03  & 2.49e-03  & 2.49e-03  & 1.31e-03  & 1.31e-03  & 2.34e-03  & 2.34e-03  \\ \cline{3-10} 
                      & $T_{\#2}$  & 5.74e-03  & 5.22e-03  & 3.02e-03  & 3.43e-03  & 6.72e-03  & 5.04e-03  & 2.09e-03  & 4.82e-03  \\ \cline{3-10}
                      & $T_{\#3}$  & 4.77e-03  & 3.67e-03  & 4.16e-03  & 2.27e-03  & 4.39e-03  & 4.75e-03  & 8.61e-03  & 2.02e-03  \\ \cline{3-10} 
                      & $T_{\#4}$  & 3.09e-03  & 1.45e-02  & 3.81e-03  & 5.86e-03  & 9.06e-03  & 3.84e-03  & 4.23e-03  & 5.23e-03  \\ \cline{3-10} 
                      & $T_{\#5}$  & 4.70e-03  & 1.08e-02  & 3.22e-02  & 2.29e-02  & 5.23e-03  & 1.90e-03  & 6.52e-03  & 1.86e-02  \\ \hline
\end{tabular}
\caption{Heat equation:  $l^2$ and $l^\infty$ errors in different time blocks corresponding to a range of training collocation points $N_c$ for the HLConcPINN-ExBTM and HLConcPINN-BTM methods. NN architecture: $[2, 90, 90, 1]$, with $\tanh$ activation function. 
}
	\label{num_heat_tab2}
\end{table}

Table~\ref{num_heat_tab2} shows a study of the effect of training collocation points on the results of the HLConcPINN-ExBTM and HLConcPINN-BTM methods. The $l^2$ and $l^{\infty}$ errors of these methods in different time blocks obtained with collocation points ranging from $N_c=1500$ to $N_c=3000$ are listed in the table. Here the neural network has an architecture $[2, 90, 90, 1]$, with the $\tanh$ activation function for all hidden layers. The data indicate that the errors of these methods are not sensitive to the number of training collocation points. In most of subsequent tests we employ a fixed $N_c=2000$ for the training collocation points.

\begin{table}[tb]\small
	\centering\small
\begin{tabular}{c|@{}c@{}|cc|cc|cc|cc}
\hline
\multirow{2}{*}{Error}    & \multirow{2}{*}{Time block} & \multicolumn{2}{c|}{{[}2,90,90,1{]}} & \multicolumn{2}{c|}{{[}2,90,90,10,1{]}} & \multicolumn{2}{c|}{{[}2,90,90,10,10,1{]}} & \multicolumn{2}{c}{{[}2,90,90,10,10,10,1{]}} \\ \cline{3-10}
                      &  & ExBTM  & BTM  & ExBTM  & BTM   & ExBTM  & BTM  & ExBTM  & BTM   \\ \hline
\multirow{5}{*}{$l^2$} & $T_{\#1}$  & 4.16e-04  & 4.16e-04 & 1.79e-04  & 1.79e-04  & 1.99e-04  & 1.99e-04  & 2.44e-04  & 2.44e-04  \\ \cline{3-10}
                      & $T_{\#2}$  & 5.73e-04  & 5.93e-04 & 2.40e-04  & 3.46e-04  & 3.13e-04  & 2.03e-04  & 3.95e-04  & 3.02e-04  \\ \cline{3-10}
                      & $T_{\#3}$  & 7.75e-04  & 2.93e-04 & 5.33e-04  & 7.24e-04  & 8.17e-04  & 7.28e-04  & 8.91e-04  & 9.68e-04  \\ \cline{3-10}
                      & $T_{\#4}$  & 6.93e-04  & 5.52e-04 & 5.92e-04  & 6.30e-04  & 1.70e-03  & 7.86e-04  & 1.01e-03  & 1.44e-03  \\ \cline{3-10}
                      & $T_{\#5}$  & 1.77e-03  & 1.40e-03 & 9.06e-04  & 8.46e-04  & 1.49e-03  & 9.15e-04  & 1.59e-03  & 2.29e-03  \\ \hline
\multirow{5}{*}{$l^\infty$} & $T_{\#1}$  & 2.49e-03  & 2.49e-03 & 2.97e-03  & 2.97e-03  & 9.11e-04  & 9.11e-04  & 1.47e-03  & 1.47e-03  \\ \cline{3-10}
                      & $T_{\#2}$  & 3.02e-03  & 3.43e-03 & 2.62e-03  & 3.05e-03  & 1.43e-03  & 1.82e-03  & 1.89e-03  & 1.92e-03  \\ \cline{3-10}
                      & $T_{\#3}$  & 4.16e-03  & 2.27e-03 & 3.90e-03  & 3.94e-03  & 4.05e-03  & 4.08e-03  & 4.79e-03  & 4.92e-03  \\ \cline{3-10}
                      & $T_{\#4}$  & 3.81e-03  & 5.86e-03 & 4.24e-03  & 4.45e-03  & 1.29e-02  & 8.33e-03  & 1.04e-02  & 2.49e-02  \\ \cline{3-10}
                      & $T_{\#5}$  & 3.22e-02  & 2.29e-02 & 1.62e-02  & 4.70e-03  & 8.92e-03  & 7.20e-03  & 1.29e-02  & 1.62e-02  \\ \hline
\end{tabular}
\caption{Heat equation: $l^2$ and $l^\infty$ errors in different time blocks corresponding to a series of  network architectures with varying number of hidden layers for the HLConcPINN-ExBTM and HLConcPINN-BTM methods. 
$\tanh$ activation function; $N_c=2000$ for the collocation points. NN architectural vectors are specified in row one of the table.
}
	\label{num_heat_tab1}
\end{table}

A salient feature of the current method lies in that the theoretical analyses are applicable to neural network architectures with more than two hidden layers.
Table~\ref{num_heat_tab1} shows a test of the network depth (number of hidden layers) on the HLConcPINN-ExBTM and HLConcPINN-BTM results for the heat equation. It lists the $l^2$ and $l^{\infty}$ errors in different time blocks obtained by these methods using network architectures with $2$ to $5$ hidden layers. The network architectural vectors are given in the table. We employ the $\tanh$ activation function for all hidden layers, and a fixed $N_c=2000$ for the training collocation points in these tests. We can make several observations. First, the errors grow over time with both methods. For example, the $l^2$ errors increase from around $10^{-4}$ in time block \#1 to around $10^{-3}$ in  time block \#5. Second, increasing the number of hidden layers only slightly influences the accuracy of results. The errors in general appear to decrease from two to three hidden layers. As the number of hidden layers further increases to three to five, the errors tend to increase slightly compared with those of two hidden layers. Third, the errors obtained with HLConcPINN-ExBTM and HLConcPINN-BTM are generally comparable, with one slightly better than the other in different cases.

\begin{table}[tb]\small
	\centering\small
\begin{tabular}{c|@{}c@{}|cc|cc|cc|cc}
\hline
\multirow{2}{*}{Error}    & \multirow{2}{*}{Time block} & \multicolumn{2}{c|}{sine} & \multicolumn{2}{c|}{Gaussian} & \multicolumn{2}{c|}{swish} & \multicolumn{2}{c}{softplus} \\ \cline{3-10}
                      &  & ExBTM  & BTM  & ExBTM  & BTM   & ExBTM  & BTM  & ExBTM  & BTM   \\ \hline
\multirow{5}{*}{$l^2$} & $T_{\#1}$  & 1.34e-04  & 1.34e-04  & 1.03e-04  & 1.03e-04  & 2.61e-04  & 2.61e-04  & 2.96e-04  & 2.96e-04  \\ \cline{3-10} 
                      & $T_{\#2}$  & 1.53e-04  & 1.91e-04  & 1.86e-04  & 2.28e-04  & 2.56e-04  & 2.38e-04  & 3.68e-04  & 3.48e-04  \\ \cline{3-10}
                      & $T_{\#3}$  & 3.06e-04  & 2.96e-04  & 4.41e-04  & 4.11e-04  & 5.27e-04  & 4.09e-04  & 3.28e-04  & 3.90e-04  \\ \cline{3-10} 
                      & $T_{\#4}$  & 6.03e-04  & 4.98e-04  & 5.80e-04  & 8.05e-04  & 7.75e-04  & 6.49e-04  & 8.40e-04  & 1.02e-03  \\ \cline{3-10} 
                      & $T_{\#5}$  & 6.94e-04  & 6.30e-04  & 7.35e-04  & 8.98e-04  & 7.45e-04  & 3.22e-03  & 1.50e-03  & 1.13e-03  \\ \hline
\multirow{5}{*}{$l^\infty$} & $T_{\#1}$  & 8.18e-04  & 8.18e-04  & 1.02e-03  & 1.02e-03  & 1.44e-03  & 1.44e-03  & 1.73e-03  & 1.73e-03  \\ \cline{3-10} 
                      & $T_{\#2}$  & 8.95e-04  & 1.29e-03  & 1.53e-03  & 1.52e-03  & 1.73e-03  & 1.17e-03  & 2.08e-03  & 1.33e-03  \\ \cline{3-10}
                      & $T_{\#3}$  & 8.82e-04  & 2.39e-03  & 4.22e-03  & 2.65e-03  & 3.41e-03  & 1.97e-03  & 2.86e-03  & 3.96e-03  \\ \cline{3-10} 
                      & $T_{\#4}$  & 3.97e-03  & 3.06e-03  & 4.68e-03  & 3.86e-03  & 4.47e-03  & 3.35e-03  & 3.87e-03  & 4.19e-03  \\ \cline{3-10} 
                      & $T_{\#5}$  & 4.03e-03  & 4.53e-03  & 2.73e-03  & 5.57e-03  & 7.40e-03  & 1.90e-02  & 6.22e-03  & 5.66e-03  \\ \hline
\end{tabular}
\caption{Heat equation:  $l^2$ and $l^\infty$ errors in different time blocks of the HLConcPINN-ExBTM and HLConcPINN-BTM methods obtained with several activation functions. 
NN architecture: $[2, 90, 90, 10, 10, 1]$; $N_c=2000$ for training collocation points.
The activation function is fixed to $\tanh$ in the first two hidden layers, and is varied among sine, Gaussian, swish, and softplus in the subsequent hidden layers.
}
	\label{num_heat_tab3}
\end{table}

The current HLConcPINN-ExBTM method admits theoretical analyses in cases where more general activation functions are employed.
Table~\ref{num_heat_tab3} provides a study of the effect of the activation functions on the simulation results of the HLConcPINN-ExBTM and HLConcPINN-BTM methods. We employ a neural network architecture $[2, 90, 90, 10, 10, 1]$, and $N_c=2000$ for the training collocation points. The activation function in the first two hidden layers is fixed to $\tanh$. For the subsequent hidden layers we vary the activation function among the sine, Gaussian, swish, or softplus functions.  The $l^2$ and $l^{\infty}$ errors of HLConcPINN-ExBTM and HLConcPINN-BTM in different time blocks corresponding to these activation functions are provided in the table. These results can be compared with those in Table~\ref{num_heat_tab1} for the same network architecture, where the $\tanh$ activation function has been used for all hidden layers. Overall, the sine activation function appears to produce the best results for HLConcPINN-ExBTM and HLConcPINN-BTM. The results obtained with the Gaussian, $\tanh$,  swish and softplus functions seem comparable to one another in terms of the accuracy. 

\begin{figure}[tb]
	\centering
	\subfloat[NN: {[2,90,90,1]}. Activation: $\tanh$ in all hidden layers.]{\label{PINN_num_heat_fig4_a}
    \includegraphics[width=0.4\linewidth]{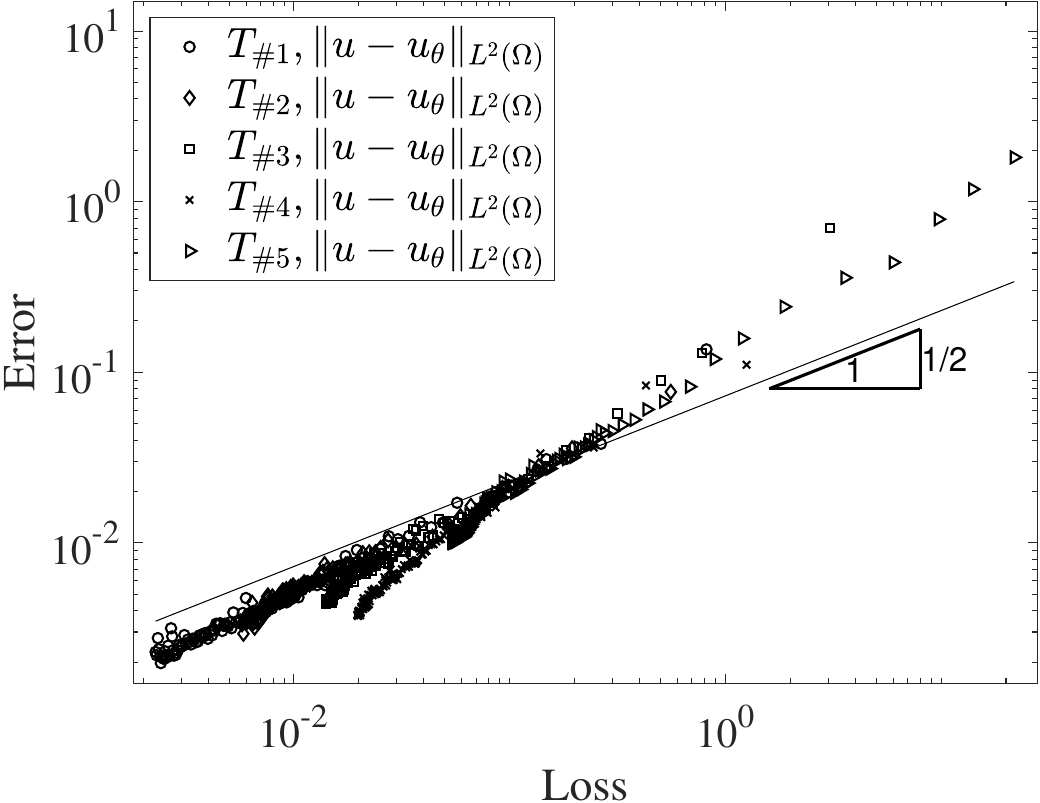}} \qquad
    \subfloat[NN: {[2,90,90,10,1]}. Activation: $\tanh$ in all hidden layers.]{\label{PINN_num_heat_fig4_b}
    \includegraphics[width=0.4\linewidth]{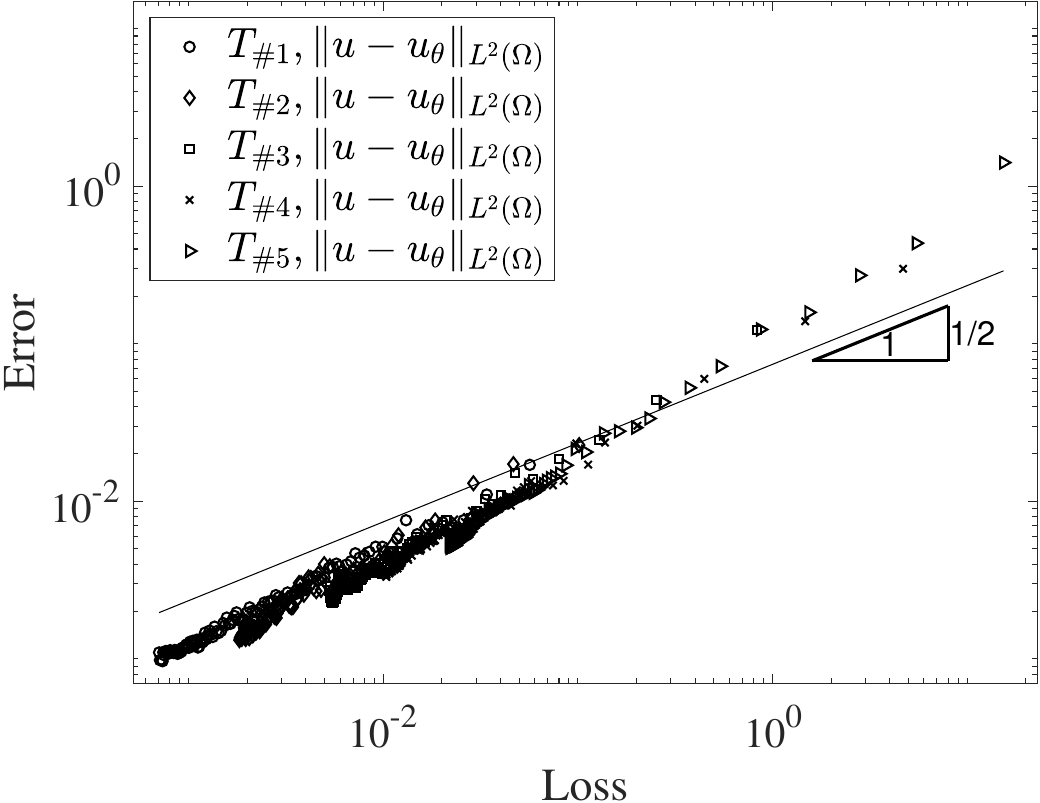}}\\
    \subfloat[NN: {[2,90,90,10,10,1]}. Activation: $\tanh$ in first two and Gaussian in subsequent hidden layers.]{\label{PINN_num_heat_fig4_c}
    \includegraphics[width=0.4\linewidth]{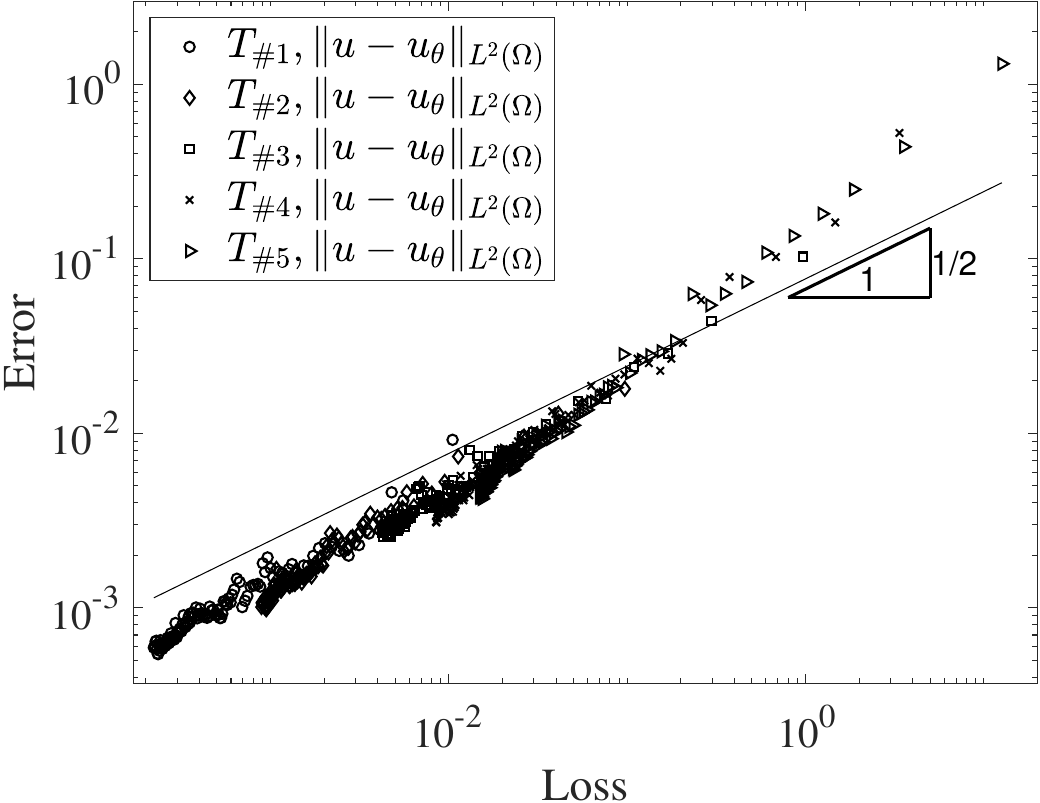}}\qquad
    \subfloat[NN: {[2,90,90,10,10,1]}. Activation: $\tanh$ in first two and softplus in subsequent hidden layers.]{\label{PINN_num_heat_fig4_d}
    \includegraphics[width=0.4\linewidth]{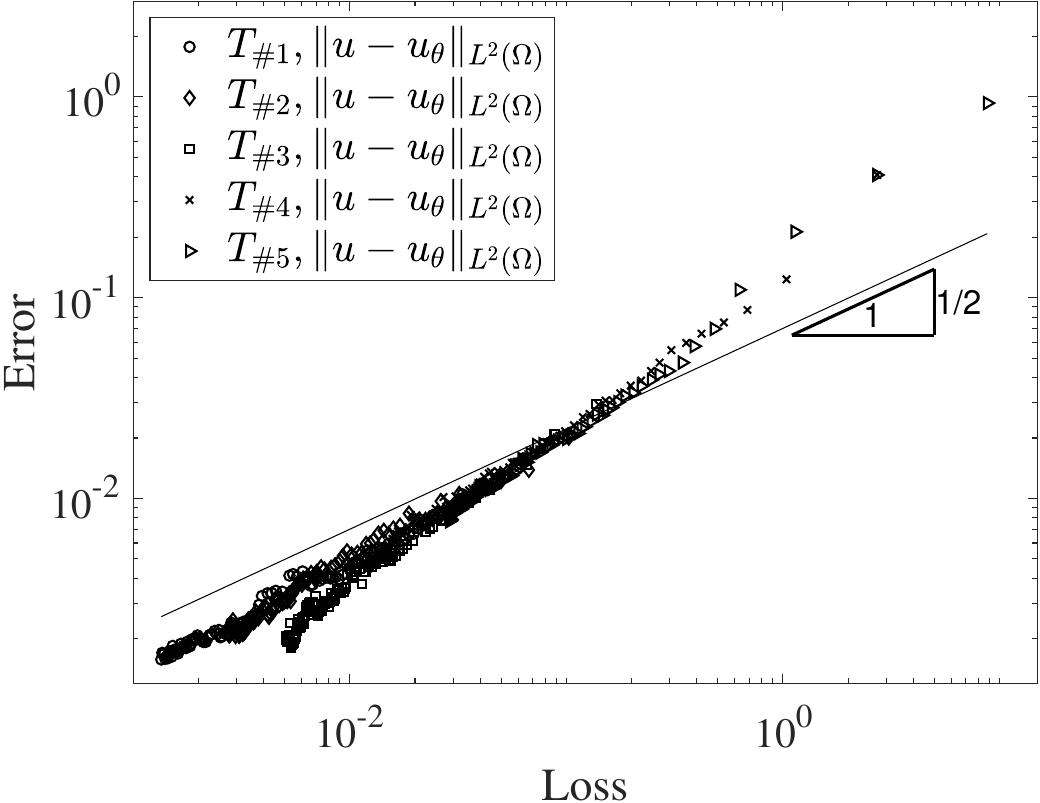}}
	\caption{Heat equation:  $l^2$ errors of HLConcPINN-ExBTM as a function of the training loss values, obtained with different network architectures and activation functions. $N_c=2000$ for the training collocation points.
 }
	\label{PINN_num_heat_fig4}
\end{figure}

Theorem \ref{sec4_Theorem3} indicates that the approximation error of the solution to the heat equation obtained with the HLConcPINN-ExBTM method scales as the square root of the training loss  for all time blocks. Figure \ref{PINN_num_heat_fig4} provides numerical evidence corroborating this statement. Here we  plot the $l^2$ errors of the solution as a function of the training loss value (in logarithmic scale) for HLConcPINN-ExBTM from our simulations. The number of hidden layers varies from two to four in these tests, with  $\tanh$ as the activation functions for the first two hidden layers and Gaussian or softplus for the subsequent hidden layers. We have used $N_c=2000$ for the collocation points. The data generally signify a square root scaling consistent with the theoretical analysis, with some deviation (faster than square root) toward larger training loss values.

\comment{
\begin{figure}[tb]
	\centering
	\subfloat[$u$]{\includegraphics[width=0.19\linewidth]{./figures/Heat_BTM_2000_90,90_u}}
    \subfloat[$u_\theta$]{\includegraphics[width=0.19\linewidth]{./figures/Heat_BTM_2000_90,90_unn_Ex}}
    \subfloat[$|u-u_\theta|$]{\includegraphics[width=0.19\linewidth]{./figures/Heat_BTM_2000_90,90_unn_Ex_err}}
    \subfloat[$u^*_\theta$]{\includegraphics[width=0.19\linewidth]{./figures/Heat_BTM_2000_90,90_unn}}
    \subfloat[$|u - u^*_\theta|$]{\includegraphics[width=0.19\linewidth]{./figures/Heat_BTM_2000_90,90_unn_err}}
	\caption{Heat equation: Solutions ($u_\theta$ is the solution of PINN-ExBTM, $u_\theta^*$ is the solution of PINN-BTM). The two hidden layers with $[90, 90]$ neurons. [All $\tanh$, $N=2000$ for training, $N=1000$ for prediction]
 }
	\label{PINN_num_heat_fig1}
\end{figure}

\begin{figure}[tb]
	\centering
	\subfloat[$ t=2.5 $]{
		\begin{minipage}[b]{0.22\textwidth}
			\includegraphics[scale=0.25]{./figures/Heat_BTM_2000_90,90_TBlock2_T2.5_u}\\
			\includegraphics[scale=0.25]{./figures/Heat_BTM_2000_90,90_TBlock2_T2.5_err}
		\end{minipage}
	}\qquad
	\subfloat[$ t=5 $]{
		\begin{minipage}[b]{0.22\textwidth}
			\includegraphics[scale=0.25]{./figures/Heat_BTM_2000_90,90_TBlock3_T5_u}\\
			\includegraphics[scale=0.25]{./figures/Heat_BTM_2000_90,90_TBlock3_T5_err}
		\end{minipage}
	}\qquad
	\subfloat[$ t=9.5 $]{
		\begin{minipage}[b]{0.22\textwidth}
			\includegraphics[scale=0.25]{./figures/Heat_BTM_2000_90,90_TBlock5_T9.5_u}\\
			\includegraphics[scale=0.25]{./figures/Heat_BTM_2000_90,90_TBlock5_T9.5_err}
		\end{minipage}
	}
	\caption{Heat equation: Top row, comparison of profiles between the exact solution and PINN-ExBTM/PINN-BTM solutions for $u$ at several time instants. Bottom row, profiles of the absolute error of the PINN-ExBTM and PINN-BTM solutions for $u$. $N=2000$ training collocation points. The two hidden layers with $[90, 90]$ neurons. [All $\tanh$, $N=2000$ for training, $N=1000$ for prediction]}
	\label{PINN_num_heat_fig2}
\end{figure}

\begin{figure}[tb]
	\centering
	\subfloat[PINN-ExBTM]{\includegraphics[width=0.5\linewidth]{./figures/Heat_BTM_2000_90,90_lossHistory_Ex}}
	\subfloat[PINN-BTM]{\includegraphics[width=0.5\linewidth]{./figures/Heat_BTM_2000_90,90_lossHistory}}\hspace{0.1em}
	\caption{Heat equation: Loss histories of (a) PINN-ExBTM and (b) PINN-BTM corresponding to various numbers of training collocation points. The two hidden layers with $[90, 90]$ neurons. [All $\tanh$, $N=2000$ for training, $N=1000$ for prediction]}
	\label{PINN_num_heat_fig3}
\end{figure}

} 

\comment{
\begin{figure}[tb]
	\centering
	\subfloat[PINN-ExBTM]
 { 
    \includegraphics[width=0.5\linewidth]{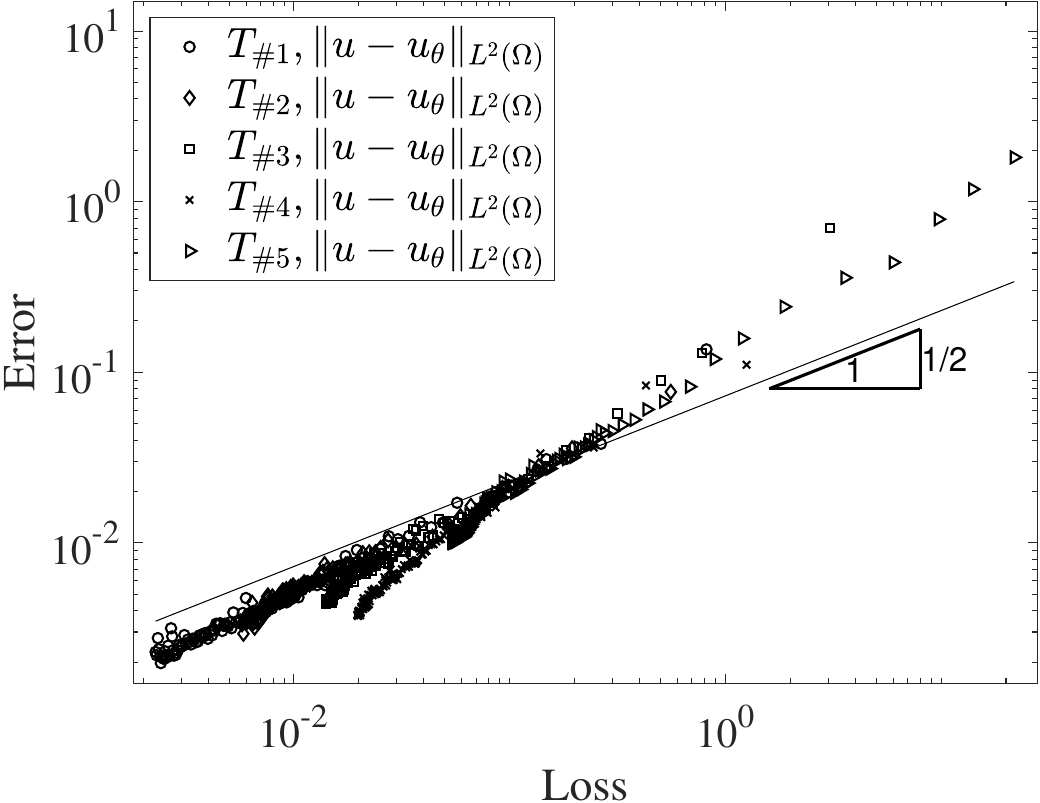}}
	\subfloat[PINN-BTM]
 { 
    \includegraphics[width=0.5\linewidth]{./figures/Heat_BTM_2000_90,90_TBlock5_errorRatio}}\\
    \subfloat[PINN-ExBTM]{ 
    \includegraphics[width=0.5\linewidth]{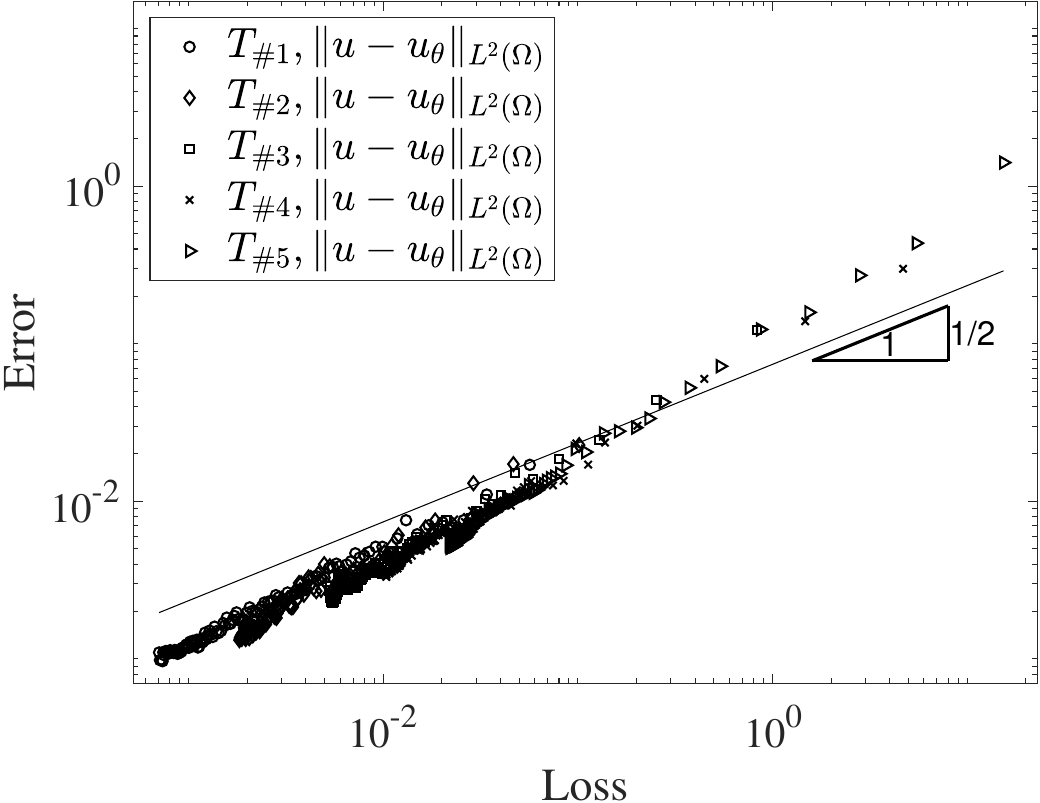}}
	\subfloat[PINN-BTM]{ 
    \includegraphics[width=0.5\linewidth]{./figures/Heat_BTM_2000_90,90,10_TBlock5_errorRatio}}\hspace{0.5em}
    \hspace{0.5em}
	\caption{Heat equation: The $l^2$ errors of $u$ as a function of the training loss value. [(a)$-$(b) The two hidden layers with $[90, 90]$ neurons. All $\tanh$, $N_c=2000$ for training, $N=1000$ for prediction] [(c)$-$(d) The three hidden layers with $[90, 90, 10]$ neurons. All $\tanh$, $N_c=2000$ for training, $N_{ev}=1000$ for prediction]}
	\label{PINN_num_heat_fig4}
\end{figure}

\begin{figure}[tb]
	\centering
	\subfloat[PINN-ExBTM]{\label{PINN_num_heat_fig5_a}
    \includegraphics[width=0.5\linewidth]{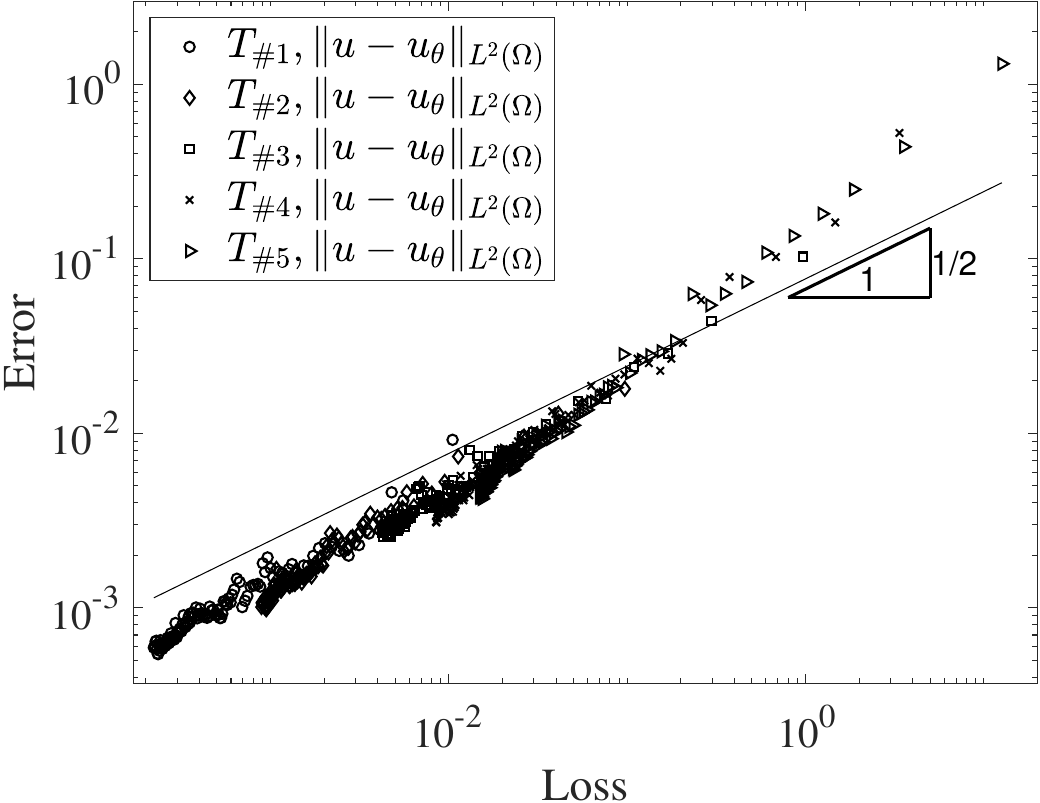}}
	\subfloat[PINN-BTM]{\label{PINN_num_heat_fig5_b}
    \includegraphics[width=0.5\linewidth]{./figures/Heat_BTM_2000_90,90,10,10_gauss_TBlock5_errorRatio}}\\
    \subfloat[PINN-ExBTM]{\label{PINN_num_heat_fig5_c}
    \includegraphics[width=0.5\linewidth]{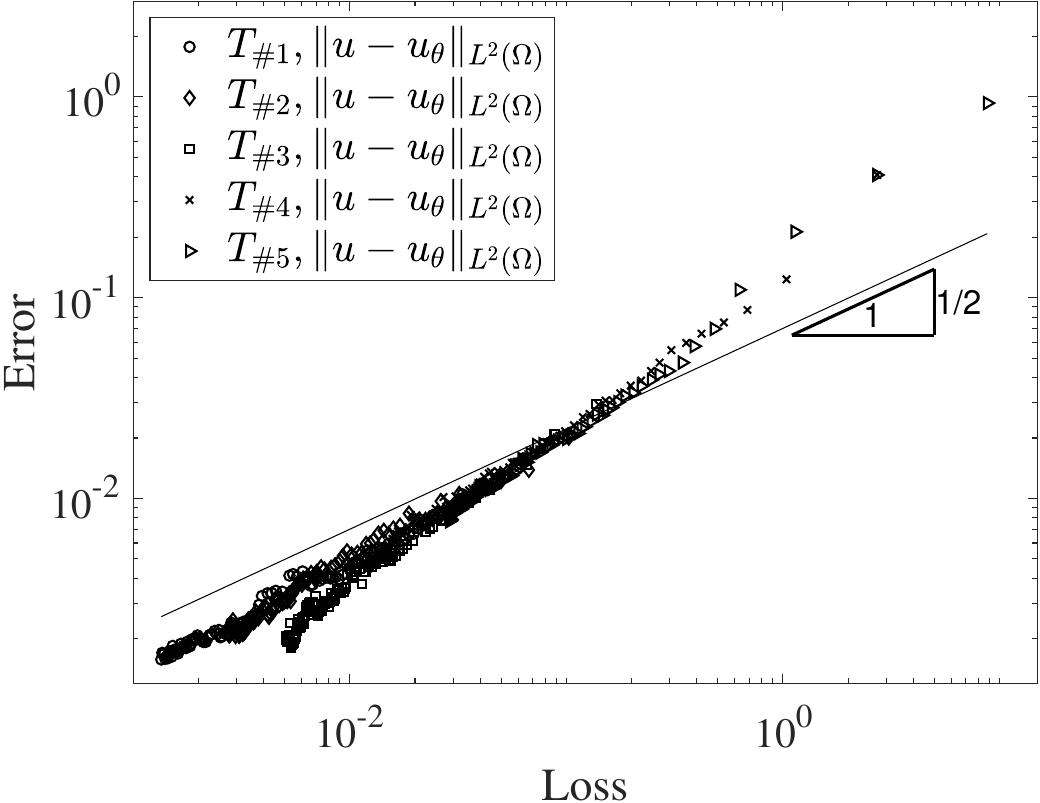}}
	\subfloat[PINN-BTM]{\label{PINN_num_heat_fig5_d}
    \includegraphics[width=0.5\linewidth]{./figures/Heat_BTM_2000_90,90,10,10_softplus_TBlock5_errorRatio}}\hspace{0.5em}
	\caption{Heat equation: The $l^2$ errors of $u$ as a function of the training loss value. The layers with $[90, 90, 10, 10]$ neurons. [(a)$-$(b) Last two layers: Gauss, $N_c=2000$ for training, $N_{ev}=1000$ for prediction] [(c)$-$(d) Last two layers: Softplus, $N_c=2000$ for training, $N_{ev}=1000$ for prediction]}
	\label{PINN_num_heat_fig5}
\end{figure}

} 

%% file: content/numerical_examples_Burgers.tex
\begin{figure}[!tb]
	\centering
	\subfloat[$u$]{\includegraphics[width=0.2\linewidth]{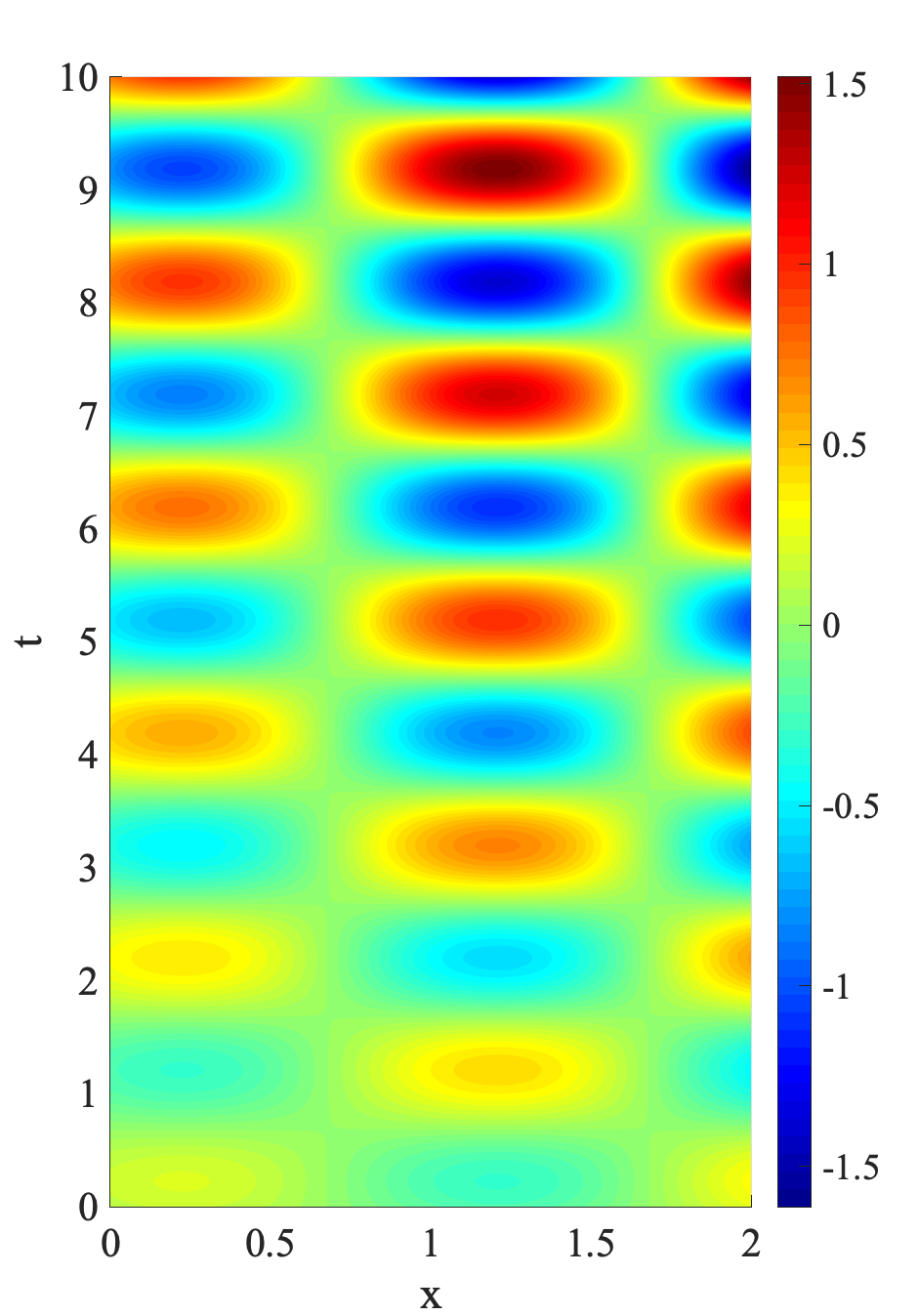}}
    \subfloat[$u_\theta$]{\includegraphics[width=0.2\linewidth]{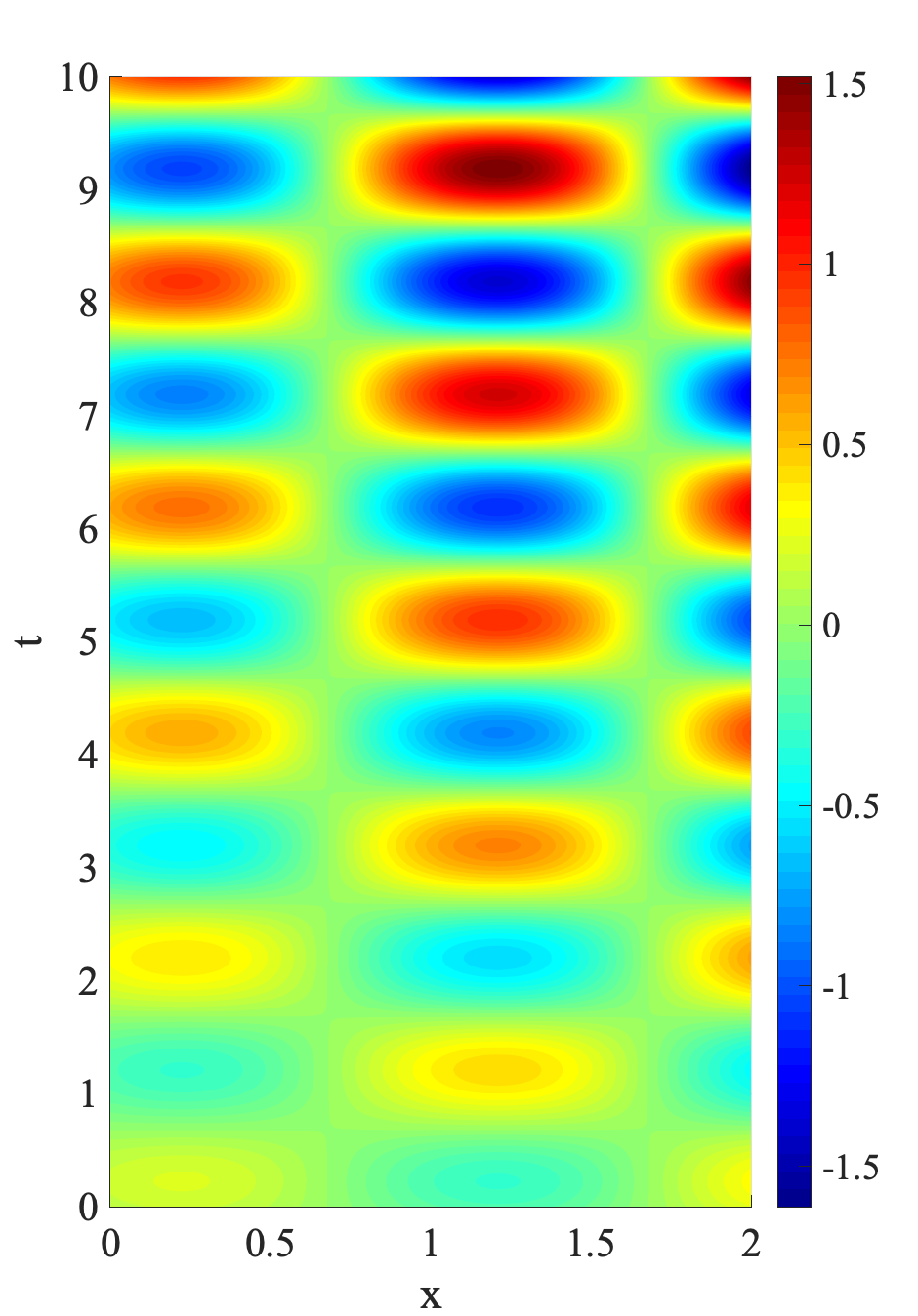}}
    \subfloat[$|u-u_\theta|$]{\includegraphics[width=0.2\linewidth]{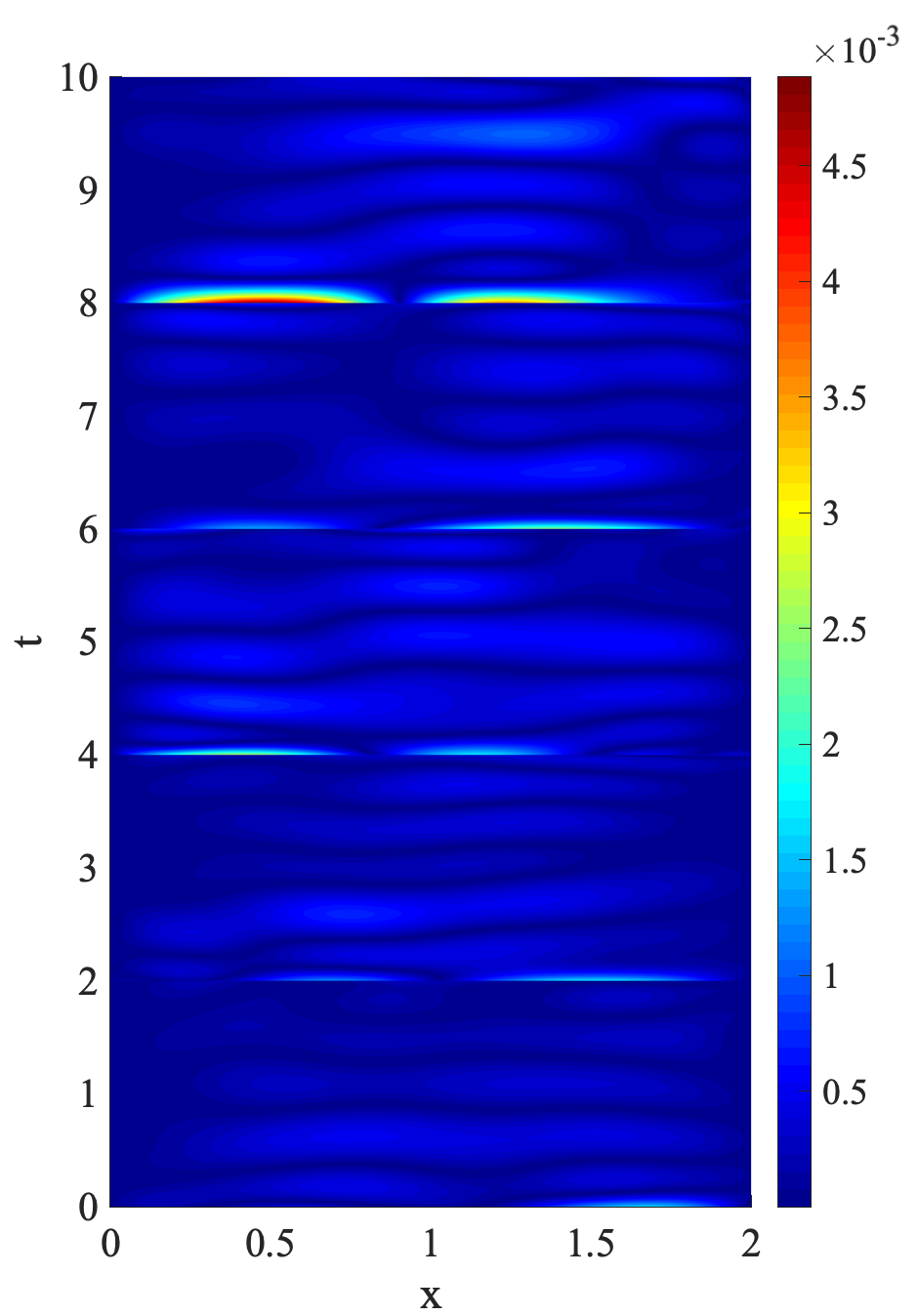}}
    \subfloat[$u^*_\theta$]{\includegraphics[width=0.2\linewidth]{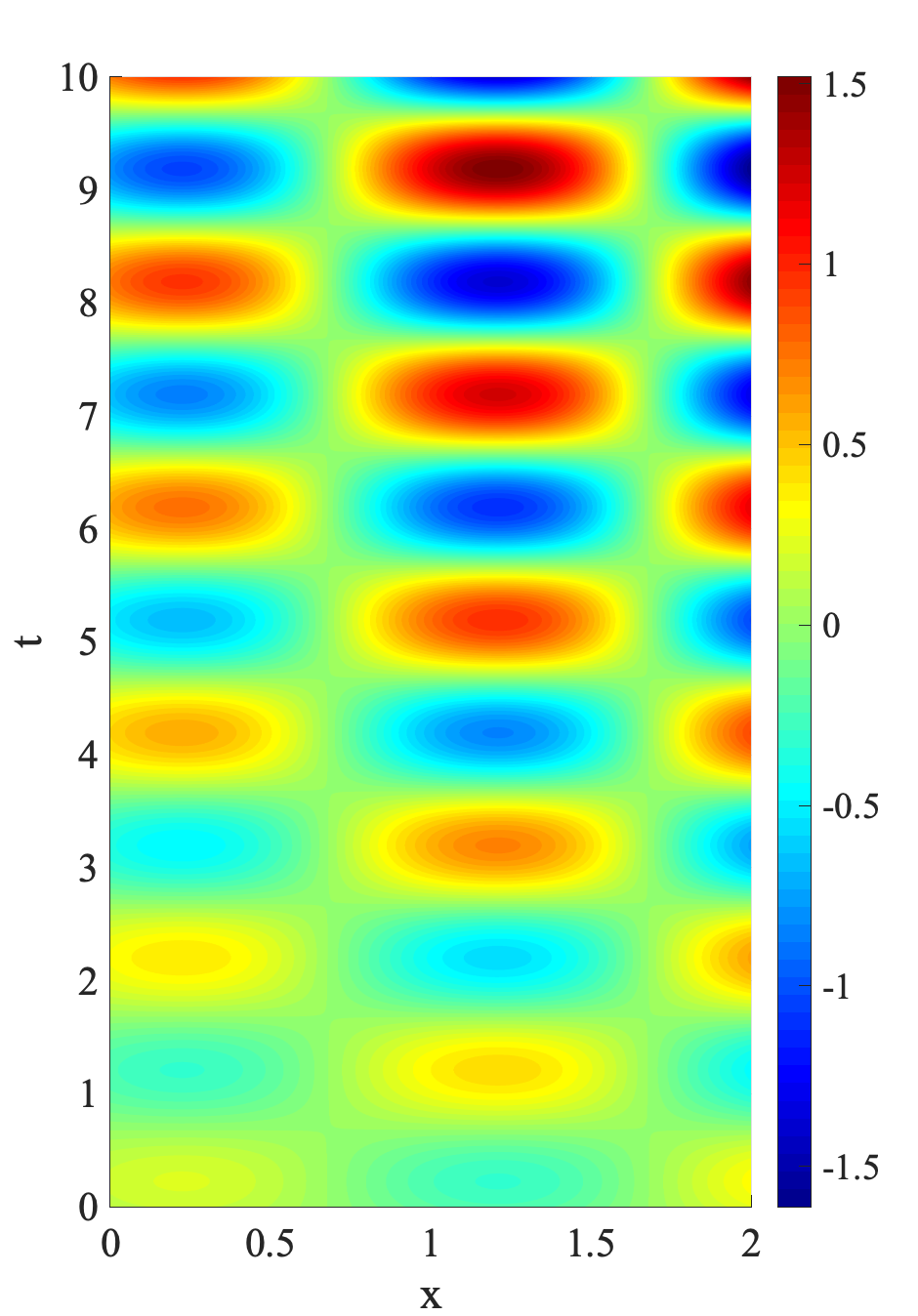}}
    \subfloat[$|u - u^*_\theta|$]{\includegraphics[width=0.2\linewidth]{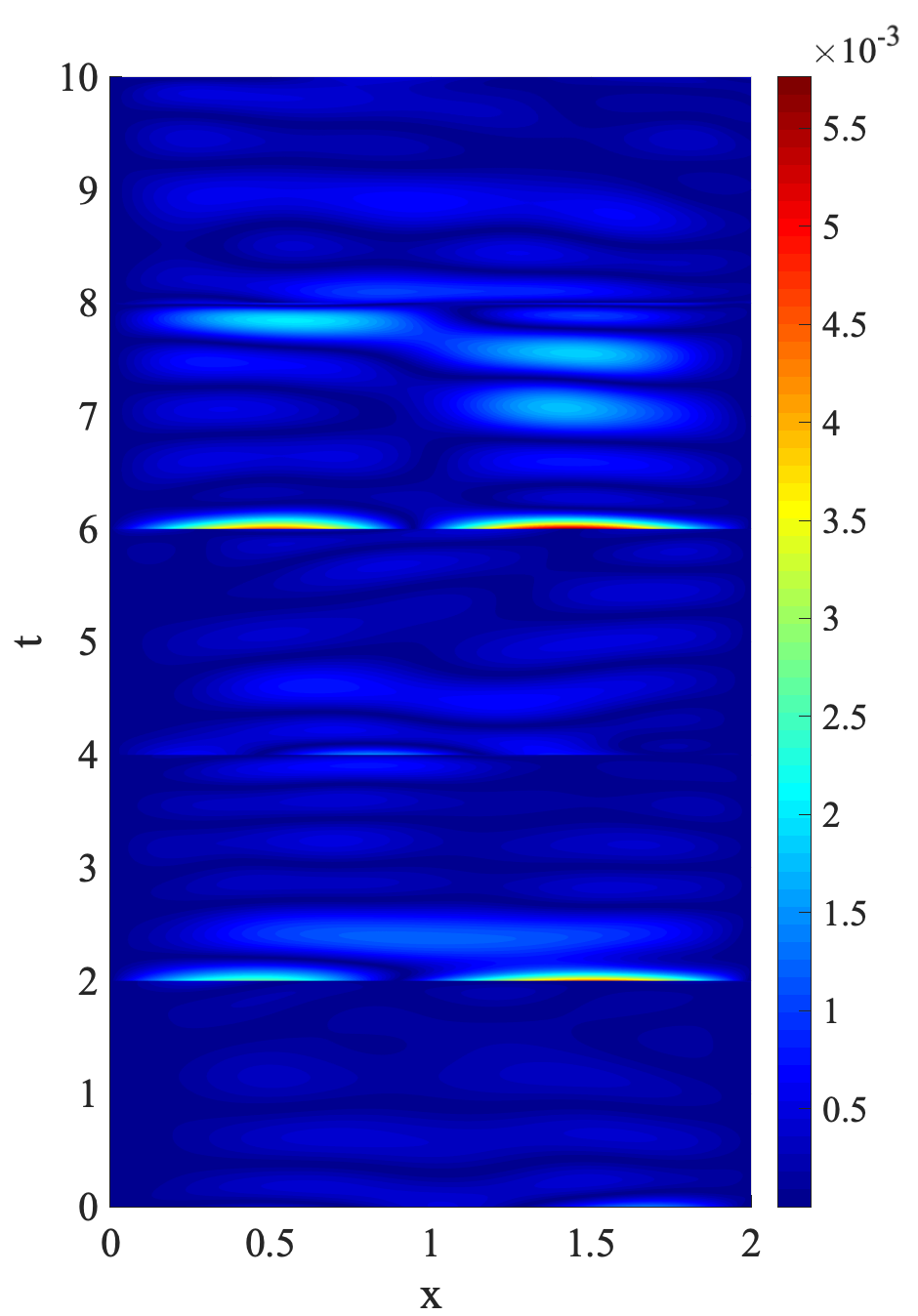}}
	\caption{Burgers' equation: Distributions of the exact solution (a), the HLConcPINN-ExBTM solution and its point-wise error (b,c), and the HLConcPINN-BTM solution and its point-wise error (d,e). NN: [2,90,90,10,1], with $\tanh$, $\tanh$ and sine activation functions for the three hidden layers; $N_c=2500$ for the training collocation points.
 }
	\label{PINN_num_Burger_fig1}
\end{figure}

\begin{figure}[tb]
	\centering
	\subfloat[HLConcPINN-ExBTM]
 {\includegraphics[width=0.35\linewidth]{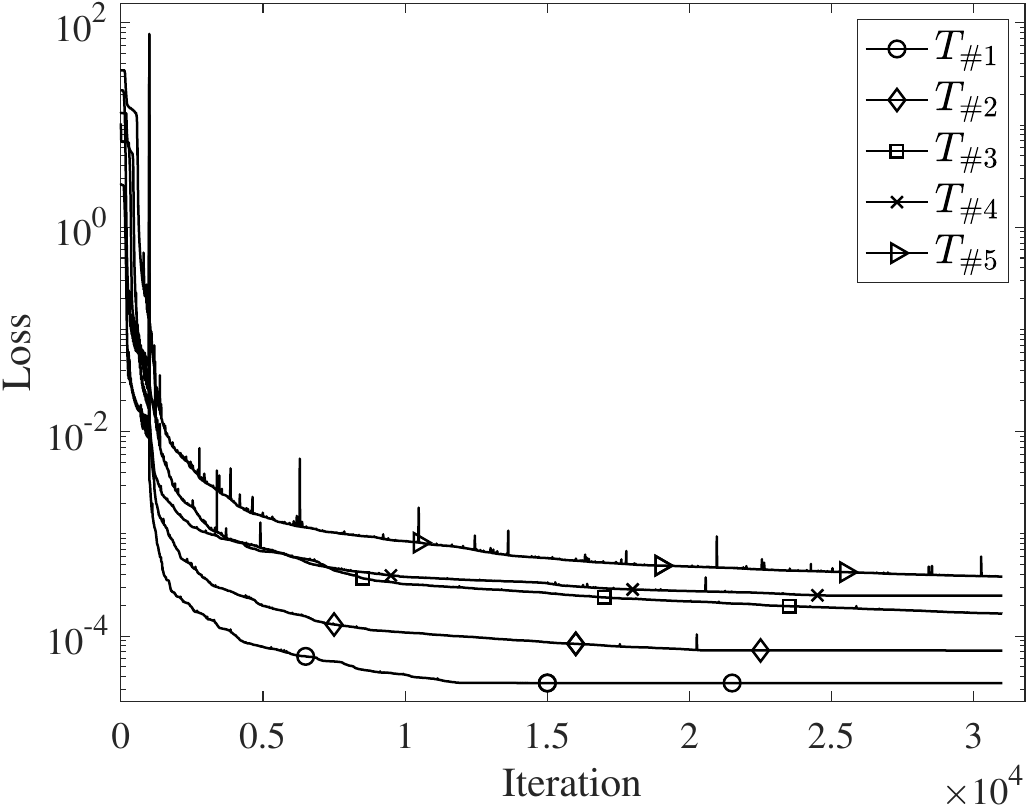}}\qquad
	\subfloat[HLConcPINN-BTM]{\includegraphics[width=0.35\linewidth]{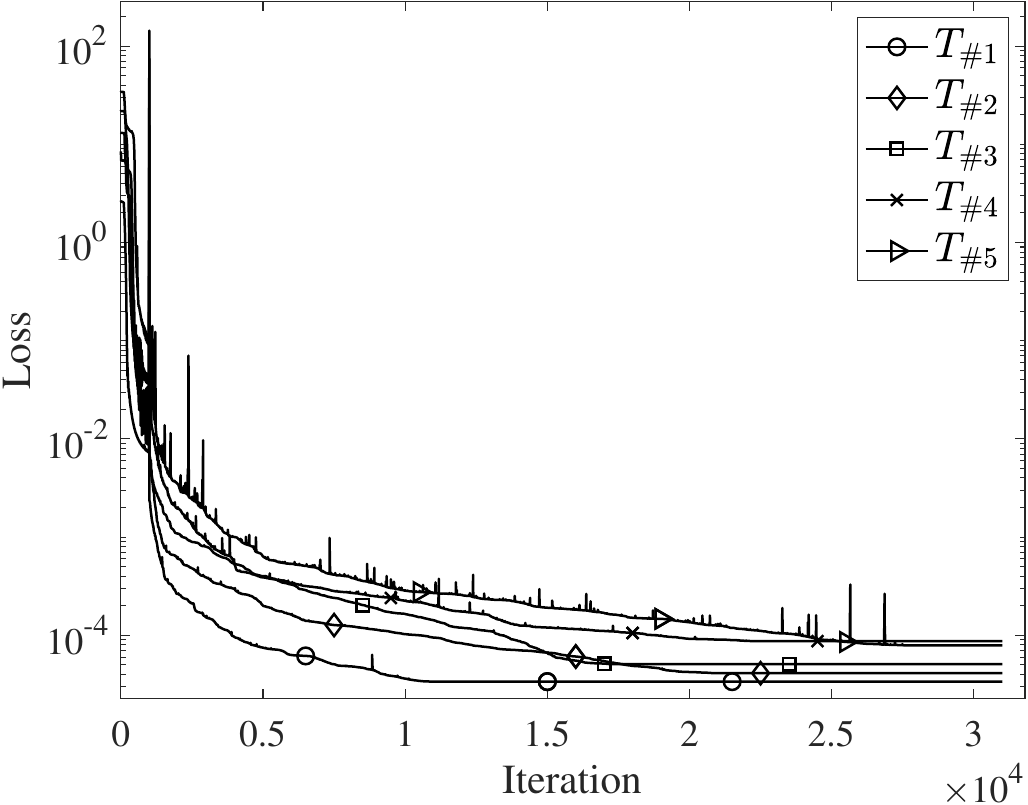}}\hspace{0.1em}
	\caption{Burgers' equation: Loss histories of HLConcPINN-ExBTM and HLConcPINN-BTM in different time blocks. NN settings and simulation parameters follow those of Figure~\ref{PINN_num_Burger_fig1}.
 }
	\label{PINN_num_Burger_fig4_1}
\end{figure}

\begin{figure}[tb]
	\centering
 \includegraphics[width=0.35\linewidth]{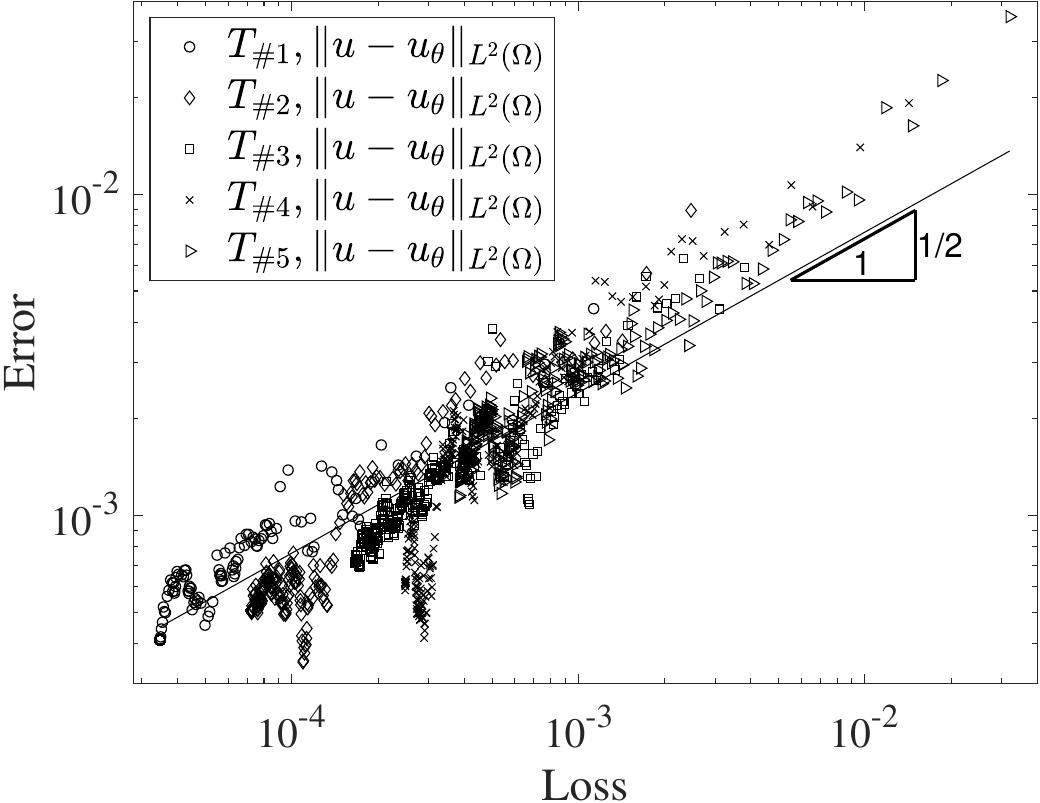}
    \hspace{0.5em}
	\caption{Burgers' equation: The $l^2$ errors of $u$ as a function of the training loss for the HLConcPINN-ExBTM method. 
 The NN settings and simulation parameters follow those of Figure~\ref{PINN_num_Burger_fig1}.
 }
	\label{PINN_num_Burger_fig5_1}
\end{figure}

\subsection{Burgers' Equation}

We next consider the viscous Burgers' equation on the spatial-temporal domain
$(x,t)\in\Omega= D\times [0, T] = [0, 2] \times [0, 10]$,
\begin{subequations}\label{num_Burgers_eq0}
\begin{align}
	& \frac{\partial u}{\partial t} -\nu\frac{\partial^2 u}{\partial x^2} + u\frac{\partial u}{\partial x} = f(x,t), \\
 &u({0},t)=\phi_1(t),  \quad u({2},t)=\phi_2(t), \quad
	u({x},0)=\psi({x}).
\end{align}
\end{subequations}
Here $u(x,t)$ is the field to be solved for,   $\nu$ denotes the viscosity, $f(x,t)$ is a source term, $\phi_1(t)$ and $\phi_2(t)$ denote the boundary data, and $\psi({x})$ is the initial distribution. 
We take $\nu = 1$ and choose the source term and the boundary/initial condition such that the function 
\begin{align}
    u(x,t) = \left(\frac{1}{5} + \frac{x}{10}\right)\left(\frac{1}{5} + \frac{t}{10}\right) \left[ 2\sin\left(\pi x + \frac{2\pi}{5}\right) + \frac{1}{2}\cos\left(\pi x - \frac{3\pi}{5}\right)\right] \left[ 2\sin\left(\pi t + \frac{2\pi}{5}\right) + \frac{1}{2}\cos\left(\pi t - \frac{3\pi}{5}\right) \right], \nonumber
\end{align}
solves the problem \eqref{num_Burgers_eq0}.

\begin{table}[tb]\small
	\centering\small
\begin{tabular}{c|@{}c@{}|cc|cc|cc|cc}
\hline
\multirow{2}{*}{Error}    & \multirow{2}{*}{Time block} & \multicolumn{2}{c|}{$N_c=1500$} & \multicolumn{2}{c|}{$N_c=2000$} & \multicolumn{2}{c|}{$N_c=2500$} & \multicolumn{2}{c}{$N_c=3000$} \\ \cline{3-10}
                      &  & ExBTM  & BTM  & ExBTM  & BTM   & ExBTM  & BTM  & ExBTM  & BTM   \\ \hline
\multirow{5}{*}{$l^2$} & $T_{\#1}$  & 2.26e-03 & 2.41e-03 & 1.19e-03 & 1.29e-03 & 1.14e-03 & 1.08e-03 & 1.75e-03 & 1.62e-03 \\ \cline{3-10} 
                      & $T_{\#2}$  & 6.31e-04 & 1.72e-03 & 6.92e-04 & 6.40e-04 & 8.35e-04 & 1.80e-03 & 6.50e-04 & 4.54e-04 \\ \cline{3-10}
                      & $T_{\#3}$  & 7.05e-04 & 7.63e-04 & 7.72e-04 & 7.59e-04 & 8.32e-04 & 6.74e-04 & 8.69e-04 & 9.56e-04 \\ \cline{3-10} 
                      & $T_{\#4}$  & 1.60e-03 & 8.04e-04 & 6.76e-04 & 6.78e-04 & 5.61e-04 & 1.48e-03 & 7.45e-04 & 8.76e-04 \\ \cline{3-10} 
                      & $T_{\#5}$  & 1.74e-03 & 1.81e-03 & 4.77e-04 & 1.32e-03 & 8.50e-04 & 4.81e-04 & 8.23e-04 & 1.17e-03 \\ \hline
\multirow{5}{*}{$l^\infty$} & $T_{\#1}$  & 2.01e-02 & 1.97e-02 & 1.03e-02 & 1.17e-02 & 8.77e-03 & 7.49e-03 & 1.23e-02 & 4.28e-03 \\ \cline{3-10} 
                      & $T_{\#2}$  & 4.43e-03 & 9.01e-03 & 3.68e-03 & 2.84e-03 & 5.71e-03 & 1.46e-02 & 4.61e-03 & 3.63e-03 \\ \cline{3-10}
                      & $T_{\#3}$  & 4.58e-03 & 6.60e-03 & 7.41e-03 & 4.13e-03 & 6.32e-03 & 3.58e-03 & 5.27e-03 & 7.47e-03 \\ \cline{3-10} 
                      & $T_{\#4}$  & 5.06e-03 & 3.36e-03 & 6.37e-03 & 4.61e-03 & 4.52e-03 & 1.05e-02 & 4.13e-03 & 6.44e-03 \\ \cline{3-10} 
                      & $T_{\#5}$  & 1.28e-02 & 1.09e-02 & 2.36e-03 & 3.40e-03 & 7.24e-03 & 1.52e-03 & 3.64e-03 & 3.31e-03 \\ \hline
\end{tabular}
\caption{ Burgers' equation: the $l^2$ and $l^\infty$ errors corresponding to different training collocation points $N_c$. NN: [2,90,90,10,1], with tanh, tanh and sine activation functions for the three hidden layers.
}
	\label{tab_Burger_err_1}
\end{table}

\begin{table}[tb]\small
	\centering\small
\begin{tabular}{c|@{}c@{}|cc|cc|cc|cc}
\hline
\multirow{2}{*}{Error}    & \multirow{2}{*}{Time block} & \multicolumn{2}{c|}{tanh} & \multicolumn{2}{c|}{Gaussian} & \multicolumn{2}{c|}{swish} & \multicolumn{2}{c}{softplus} \\ \cline{3-10}
                      &  & ExBTM  & BTM  & ExBTM  & BTM   & ExBTM  & BTM  & ExBTM  & BTM   \\ \hline
\multirow{5}{*}{$l^2$} & $T_{\#1}$ & 1.44e-03 & 1.04e-03 & 1.59e-03 & 2.23e-03 & 1.60e-03 & 2.23e-03 & 2.10e-03 & 1.80e-03 \\ \cline{3-10} 
                      & $T_{\#2}$  & 9.66e-04 & 1.19e-03 & 1.01e-03 & 2.15e-03 & 2.35e-03 & 3.96e-03 & 8.69e-04 & 7.42e-04 \\ \cline{3-10}
                      & $T_{\#3}$  & 1.69e-03 & 8.65e-04 & 1.39e-03 & 5.70e-04 & 1.87e-03 & 1.12e-03 & 1.77e-03 & 1.37e-03 \\ \cline{3-10} 
                      & $T_{\#4}$  & 1.05e-03 & 1.26e-03 & 1.05e-03 & 1.07e-03 & 1.42e-03 & 1.22e-03 & 2.31e-03 & 1.04e-03 \\ \cline{3-10} 
                      & $T_{\#5}$  & 1.66e-03 & 2.13e-03 & 1.32e-03 & 2.89e-03 & 2.55e-03 & 1.53e-03 & 2.41e-03 & 1.39e-03 \\ \hline
\multirow{5}{*}{$l^\infty$} & $T_{\#1}$  & 9.19e-03 & 6.56e-03 & 1.69e-02 & 2.26e-02 & 1.02e-02 & 2.00e-02 & 1.87e-02 & 1.31e-02 \\ \cline{3-10} 
                      & $T_{\#2}$  & 8.05e-03 & 1.19e-02 & 4.97e-03 & 1.83e-02 & 3.22e-02 & 4.87e-02 & 4.43e-03 & 2.78e-03 \\ \cline{3-10}
                      & $T_{\#3}$  & 2.23e-02 & 7.64e-03 & 1.68e-02 & 5.14e-03 & 2.20e-02 & 1.11e-02 & 9.10e-03 & 1.34e-02 \\ \cline{3-10} 
                      & $T_{\#4}$  & 5.90e-03 & 1.11e-02 & 9.88e-03 & 6.82e-03 & 1.58e-02 & 6.79e-03 & 2.00e-02 & 6.14e-03 \\ \cline{3-10} 
                      & $T_{\#5}$  & 1.33e-02 & 1.82e-02 & 1.14e-02 & 9.61e-03 & 1.47e-02 & 5.31e-03 & 9.31e-03 & 4.58e-03 \\ \hline
\end{tabular}
\caption{Burgers' equation: the $l^2$ and $l^\infty$ errors obtained with several different activation functions. NN: [2,90,90,10,1], with tanh activation function for the first two hidden layers, while the activation function for the last hidden layer is varied as given in the table. $N_c=2500$ training collocation points.
}
	\label{tab_Burger_err_2}
\end{table}

The loss function for the HLConcPINN-ExBTM method is given by,
\begin{align}\label{num_Burger_eq_loss0}
    Loss_i^{I} = &\frac{W_1}{N_c}\sum_{n=1}^{N_c}\left[\frac{\partial u_{\theta_i}}{\partial t}(x_{int}^n,t_{int}^n) - \nu \frac{\partial^2 u_{\theta_i}}{\partial x^2}(x_{int}^n,t_{int}^n) + u_{\theta_i}(x_{int}^n,t_{int}^n)\frac{ \partial u_{\theta_i}}{\partial x}(x_{int}^n,t_{int}^n) - f(x_{int}^n,t_{int}^n)\right]^2 \notag \\
    & + \frac{W_2}{N_c} \sum_{j=1}^i \sum_{n=1}^{N_c}\left[ u_{\theta_i}(x_{tb}^n, t_{j-1}) - u_{\theta_{j-1}}(x_{tb}^n, t_{j-1}) \right]^2 \notag\\
    & + \frac{W_3}{N_c}\sum_{n=1}^{N_c}\left[ (u_{\theta_i}(0, t_{sb}^n) - g_1(t_{sb}^n))^2 + (u_{\theta_i}(2, t_{sb}^n) - g_2(t_{sb}^n))^2 \right] \notag\\
    & + W_4\Big(\frac{1}{N_c}\sum_{n=1}^{N_c} (u_{\theta_i}(0, t_{sb}^n) - g_1(t_{sb}^n))^2 \Big)^{1/2} + W_5\Big(\frac{1}{N_c}\sum_{n=1}^{N_c} (u_{\theta_i}(2, t_{sb}^n) - g_2(t_{sb}^n))^2 \Big)^{1/2} \notag\\
    & + Loss_{i-1}^{I},
\end{align}
where $Loss_{0}^{I}=0$, and we have added a set of penalty coefficients $W_k>0$ ($k=1,\dots,5$) for different loss terms. 
%
The loss function for  HLConcPINN-BTM is,
\begin{align}\label{num_Burger_eq_loss1}
    Loss_i^{II} = &\frac{W_1}{N_c}\sum_{n=1}^{N_c}\left[\frac{\partial u_{\theta_i}}{\partial t}(x_{int}^n,t_{int}^n) - \nu \frac{\partial^2 u_{\theta_i}}{\partial x^2}(x_{int}^n,t_{int}^n) + u_{\theta_i}(x_{int}^n,t_{int}^n)\frac{\partial u_{\theta_i}}{\partial x}(x_{int}^n,t_{int}^n) - f(x_{int}^n,t_{int}^n)\right]^2 \notag \\
    & + \frac{W_2}{N_c} \sum_{n=1}^{N_c}\left[ u_{\theta_i}(x_{tb}^n, t_{i-1}) - u_{\theta_{j-1}}(x_{tb}^n, t_{i-1}) \right]^2 \notag\\
    & + \frac{W_3}{N_c}\sum_{n=1}^{N_c}\left[ (u_{\theta_i}(0, t_{sb}^n) - g_1(t_{sb}^n))^2 + (u_{\theta_i}(2, t_{sb}^n) - g_2(t_{sb}^n))^2 \right] \notag\\
    & + W_4\Big(\frac{1}{N_c}\sum_{n=1}^{N_c} (u_{\theta_i}(0, t_{sb}^n) - g_1(t_{sb}^n))^2 \Big)^{1/2} + W_5\Big(\frac{1}{N_c}\sum_{n=1}^{N_c} (u_{\theta_i}(2, t_{sb}^n) - g_2(t_{sb}^n))^2 \Big)^{1/2}.
\end{align}
For both methods, we employ $(W_1, ..., W_5) = (0.6, 0.4, 0.4, 0.4, 0.4)$ in the following simulations. Five uniform time blocks are employed in block time marching.

Figure \ref{PINN_num_Burger_fig1} shows  distributions of the true solution, the HLConcPINN-ExBTM and HLConcPINN-BTM solutions and their absolute errors. The neural network structure and other parameters are provided in the figure caption. 
%
%
The histories of the training loss functions for HLConcPINN-ExBTM and HLConcPINN-BTM are shown in Figure \ref{PINN_num_Burger_fig4_1}. Both methods have captured the solution well. 

Tables \ref{tab_Burger_err_1} and \ref{tab_Burger_err_2} illustrate the effects of the training collocation points and the activation function on the simulation points. In these simulations the neural network structure is characterized by $[2,90,90,10,1]$, with the activation function $\tanh$ for the first two hidden layers. In Table~\ref{tab_Burger_err_1}, the activation function for the last hidden layer is set to sine, and the number of training collocation points is varied systematically. In Table~\ref{tab_Burger_err_2} the activation function for the last hidden layer is varied (tanh, Gaussian, swish, or softplus), with fixed training collocation points $N_c=2500$. The simulation results appear not sensitive to the training collocation points, similar to observations with the previous test problem. Among the activation functions tested, the sine function appears to produce the best result.


Figure \ref{PINN_num_Burger_fig5_1} illustrates the relation between the $l^2$ error of $u$  and the training loss value for different time blocks obtained with the HLConcPINN-ExBTM method in our simulations. The scaling manifested in the data is consistent with Theorem \ref{sec4b_Theorem3} from our analyses.


\comment{
\begin{figure}[tb]
	\centering
	\subfloat[$ t=2.5 $]{
		\begin{minipage}[b]{0.22\textwidth}
			\includegraphics[scale=0.25]{./figures/Burger_BTM_2500_90,90,10_sin_TBlock2_T2.5_u}\\
			\includegraphics[scale=0.25]{./figures/Burger_BTM_2500_90,90,10_sin_TBlock2_T2.5_uerr}
		\end{minipage}
	}
	\subfloat[$ t=5 $]{
		\begin{minipage}[b]{0.22\textwidth}
			\includegraphics[scale=0.25]{./figures/Burger_BTM_2500_90,90,10_sin_TBlock3_T5_u}\\
			\includegraphics[scale=0.25]{./figures/Burger_BTM_2500_90,90,10_sin_TBlock3_T5_uerr}
		\end{minipage}
	}
	\subfloat[$ t=9.5 $]{
		\begin{minipage}[b]{0.22\textwidth}
			\includegraphics[scale=0.25]{./figures/Burger_BTM_2500_90,90,10_sin_TBlock5_T9.5_u}\\
			\includegraphics[scale=0.25]{./figures/Burger_BTM_2500_90,90,10_sin_TBlock5_T9.5_uerr}
		\end{minipage}
	}
	\caption{Burger's equation: Top row, comparison of profiles between the exact solution and PINN-ExBTM/PINN-BTM solutions for $u$ at several time instants. Bottom row, profiles of the absolute error of the PINN-ExBTM and PINN-BTM solutions for $u$. $N=2000$ training collocation points. The three hidden layers with $[90, 90, 10]$ neurons. [$\tanh-\tanh-\sin$, $N_c=2500$ for training, $N_{ev}=1000$ for prediction]}
	\label{PINN_num_Burger_fig3_1}
\end{figure}
} 



%% file: content/numerical_examples_wave.tex
\begin{figure}[tb]
	\centering
	\subfloat[$u$]{\includegraphics[width=0.2\linewidth]{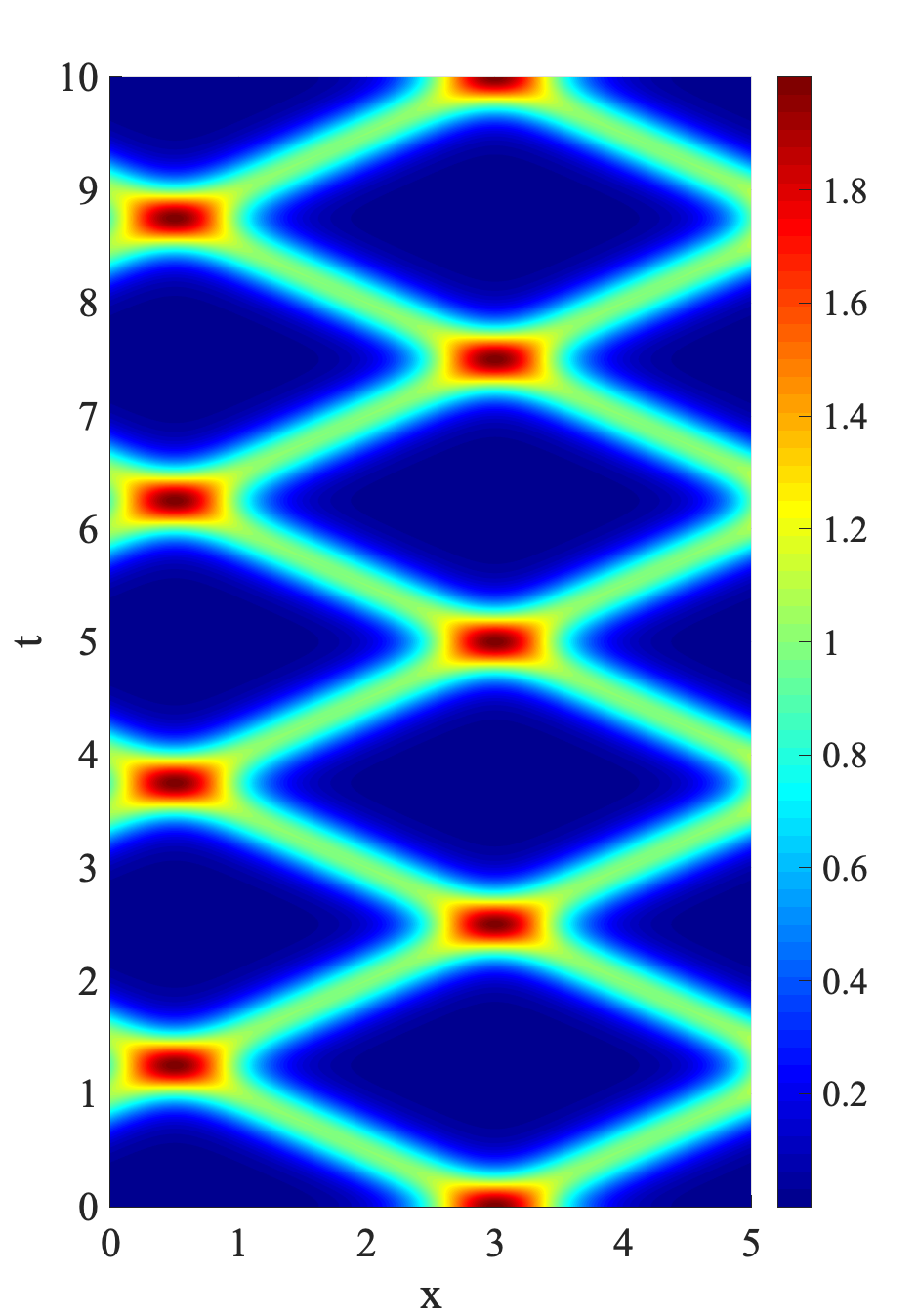}}
    \subfloat[$u_\theta$]{\includegraphics[width=0.2\linewidth]{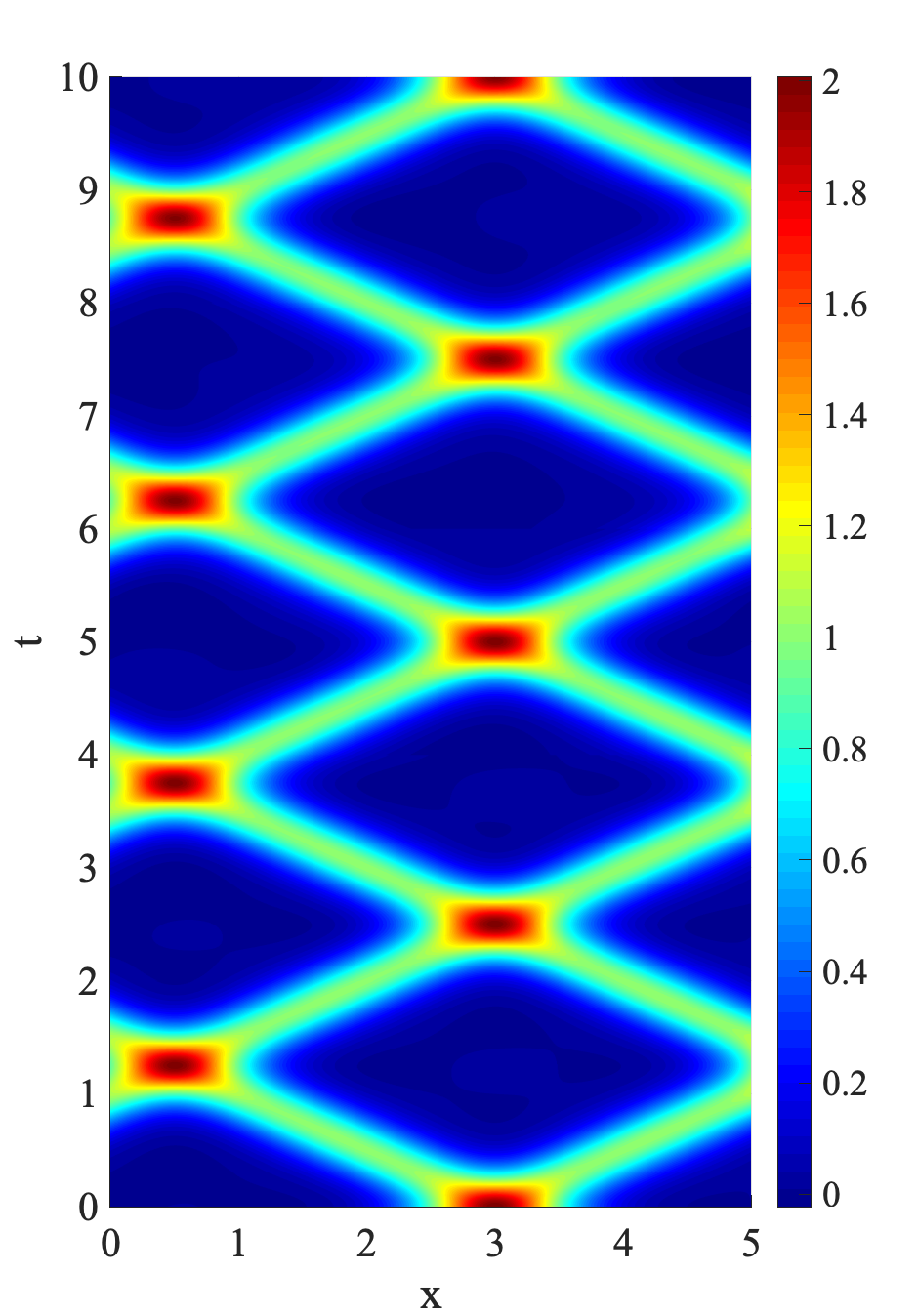}}
    \subfloat[$|u-u_\theta|$]{\includegraphics[width=0.2\linewidth]{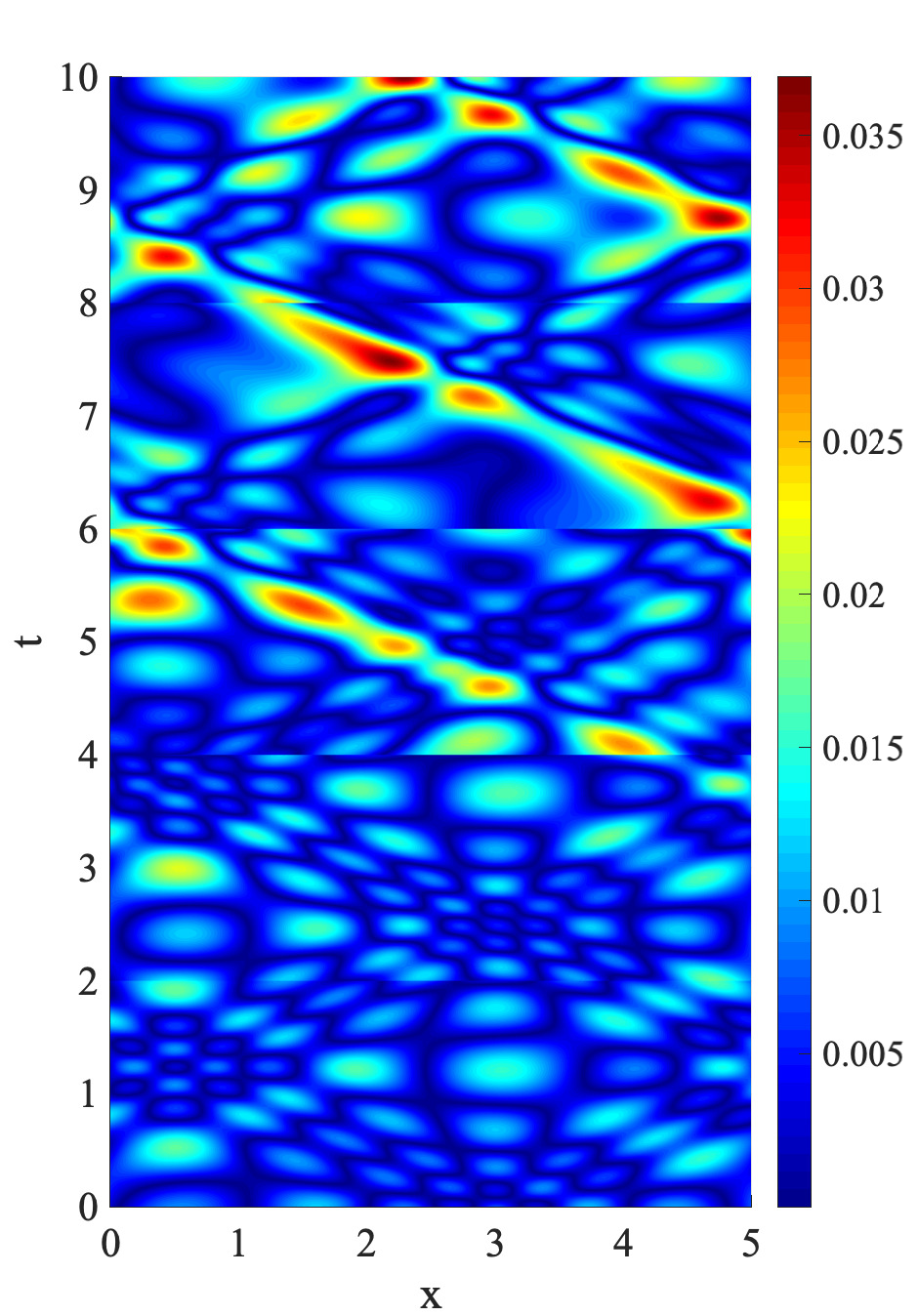}}
    \subfloat[$u^*_\theta$]{\includegraphics[width=0.2\linewidth]{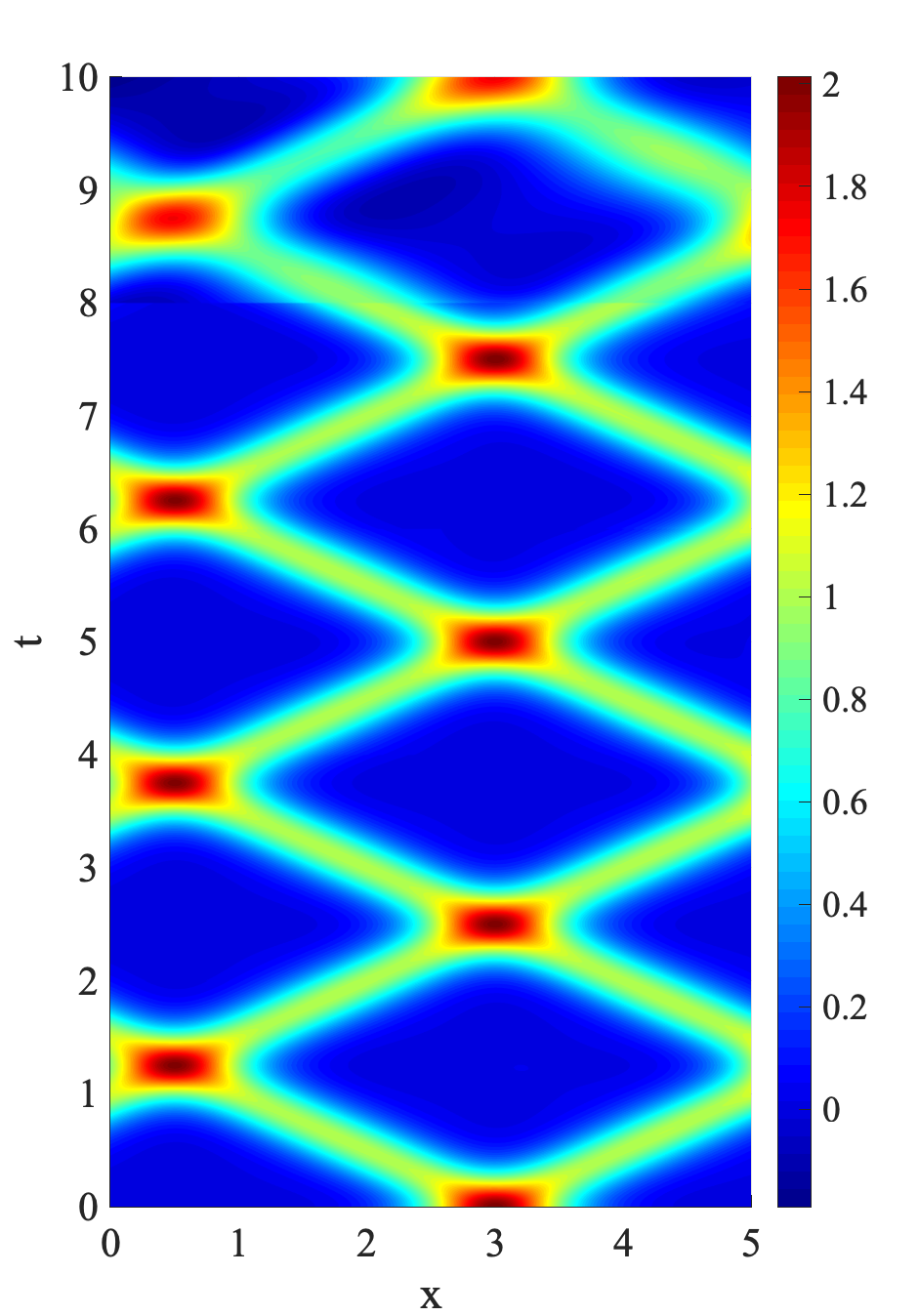}}
    \subfloat[$|u - u^*_\theta|$]{\includegraphics[width=0.2\linewidth]{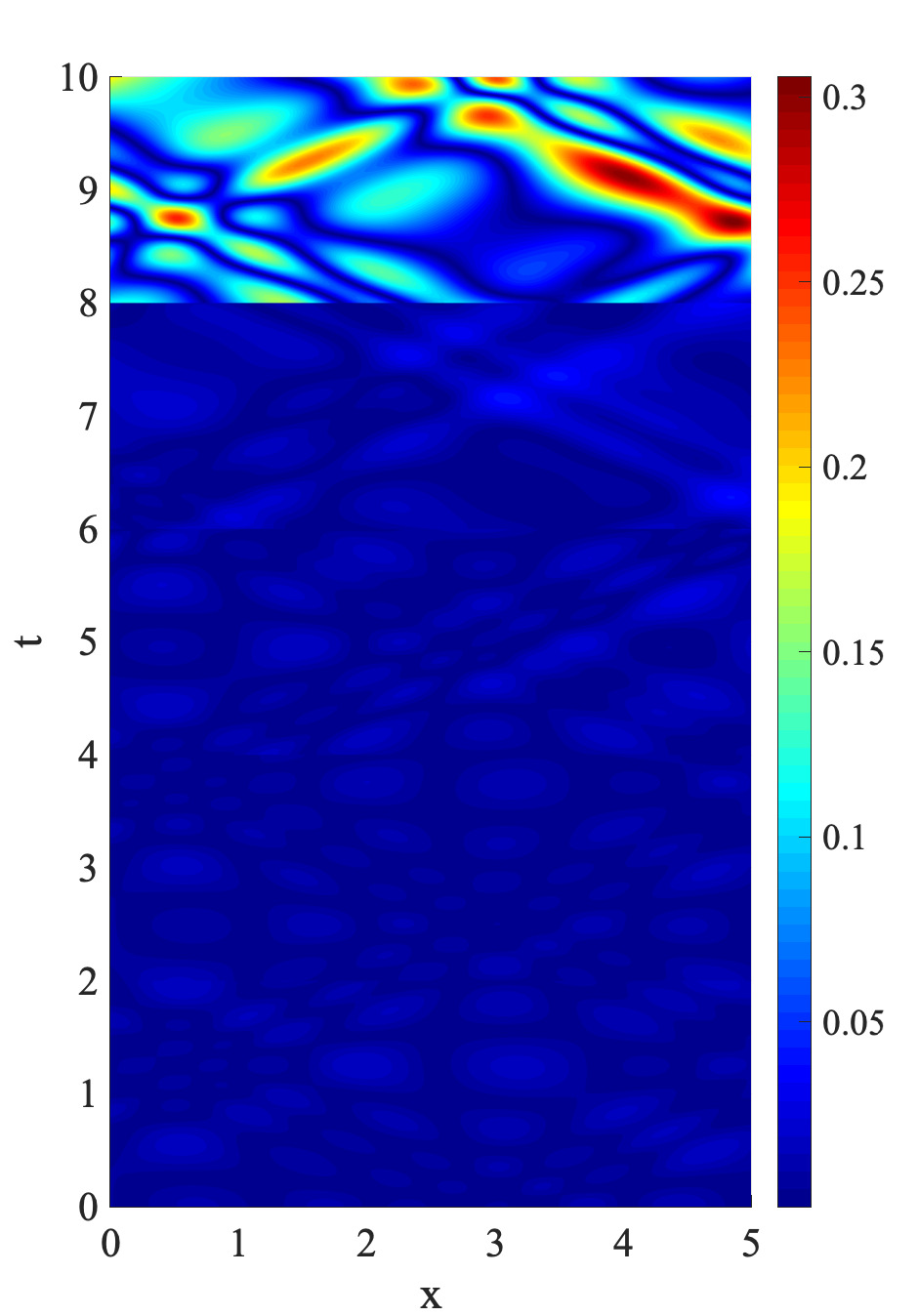}}
	\caption{Wave equation: Solution distributions ($u$: true solution; $u_\theta$: HLConcPINN-ExBTM solution; $u_\theta^*$ HLConcPINN-BTM solution). 
 NN: $[2,90,90,10,2]$; activation function: $\tanh$ for the first two hidden layers, sine for the last hidden layer; $N_c=2500$ for  training collocation points.
 }
	\label{PINN_num_wave_fig1}
\end{figure}

\begin{figure}[tb]
	\centering
	\subfloat[$v$]{\includegraphics[width=0.2\linewidth]{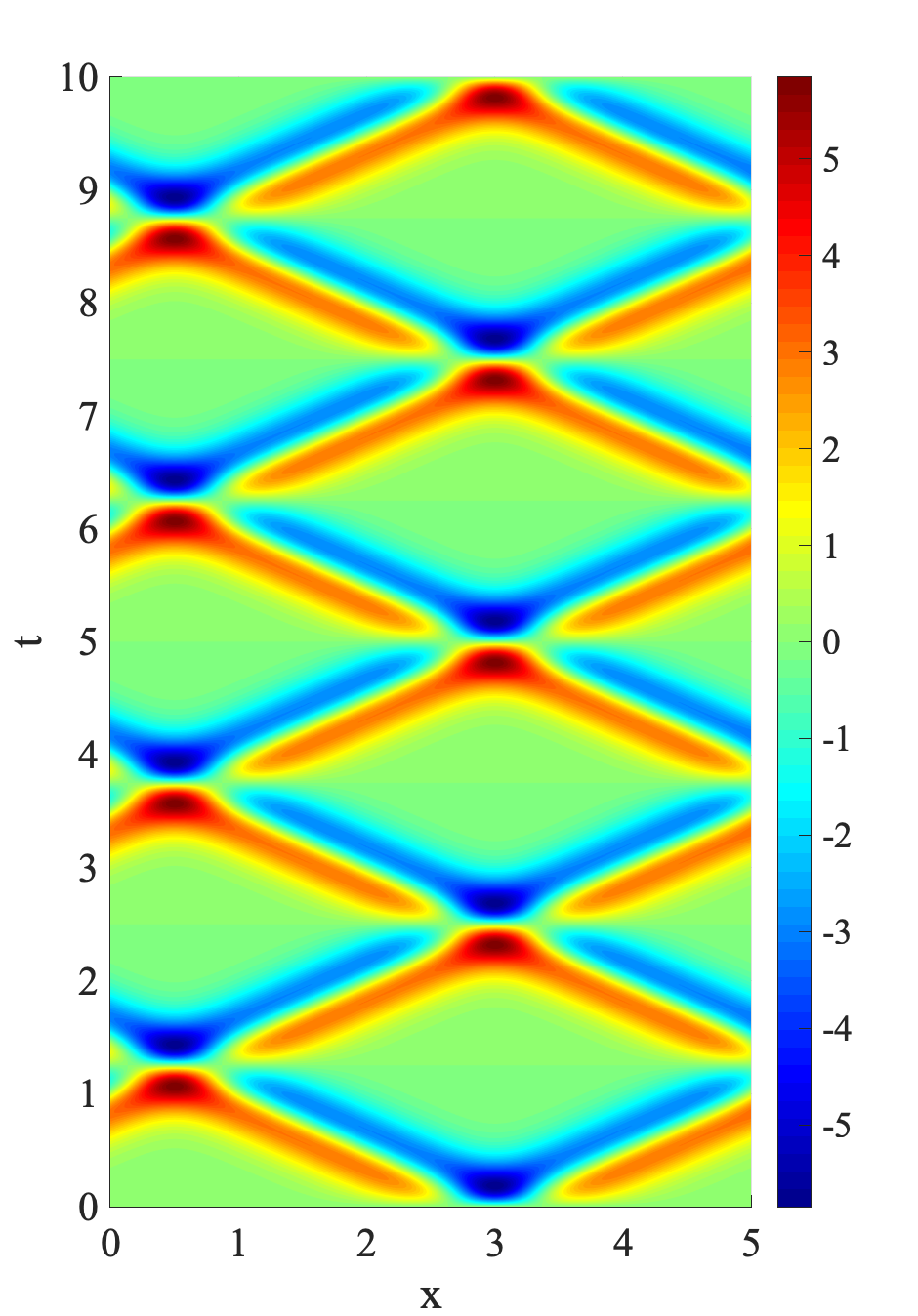}}
    \subfloat[$v_\theta$]{\includegraphics[width=0.2\linewidth]{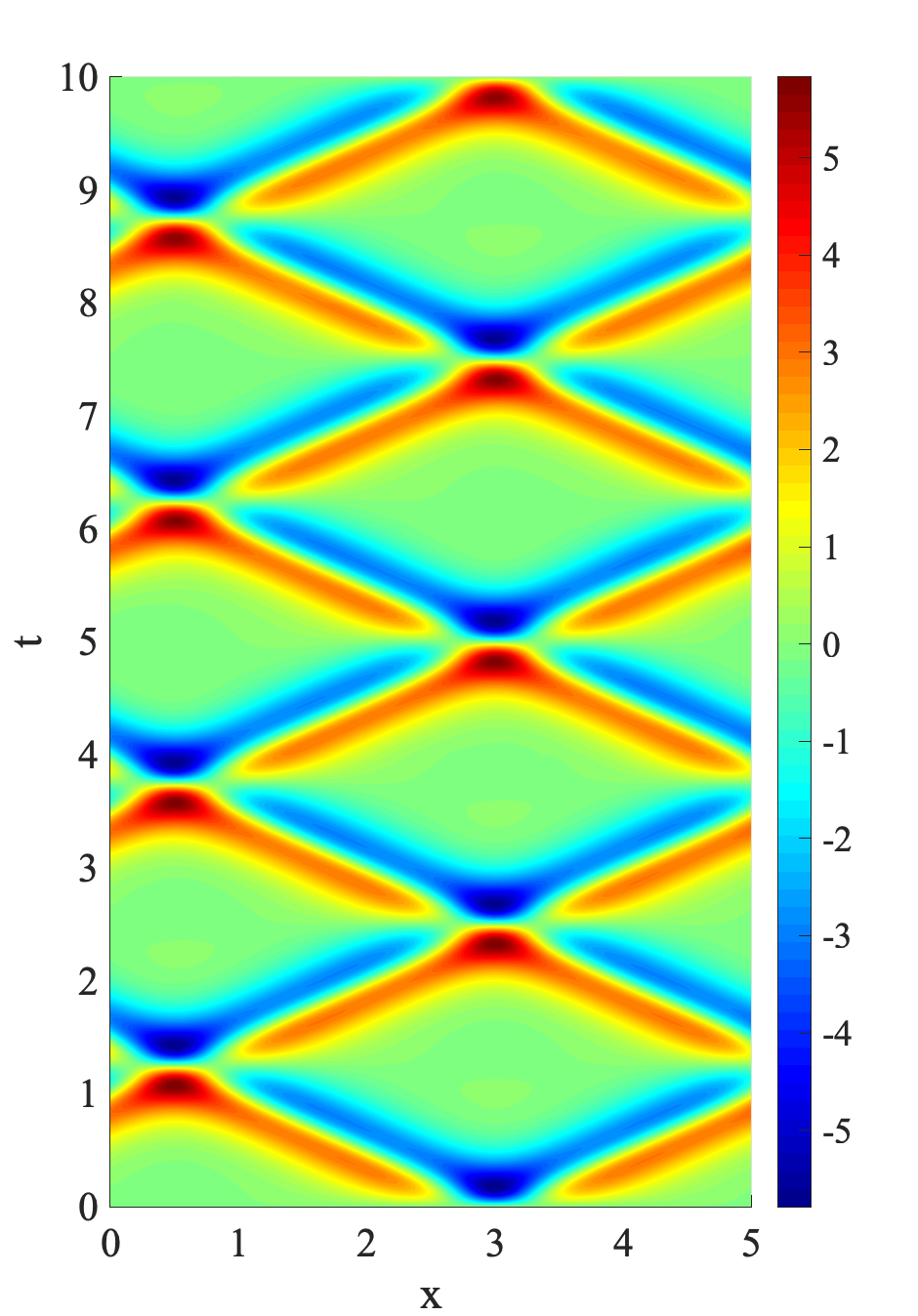}}
    \subfloat[$|v-v_\theta|$]{\includegraphics[width=0.2\linewidth]{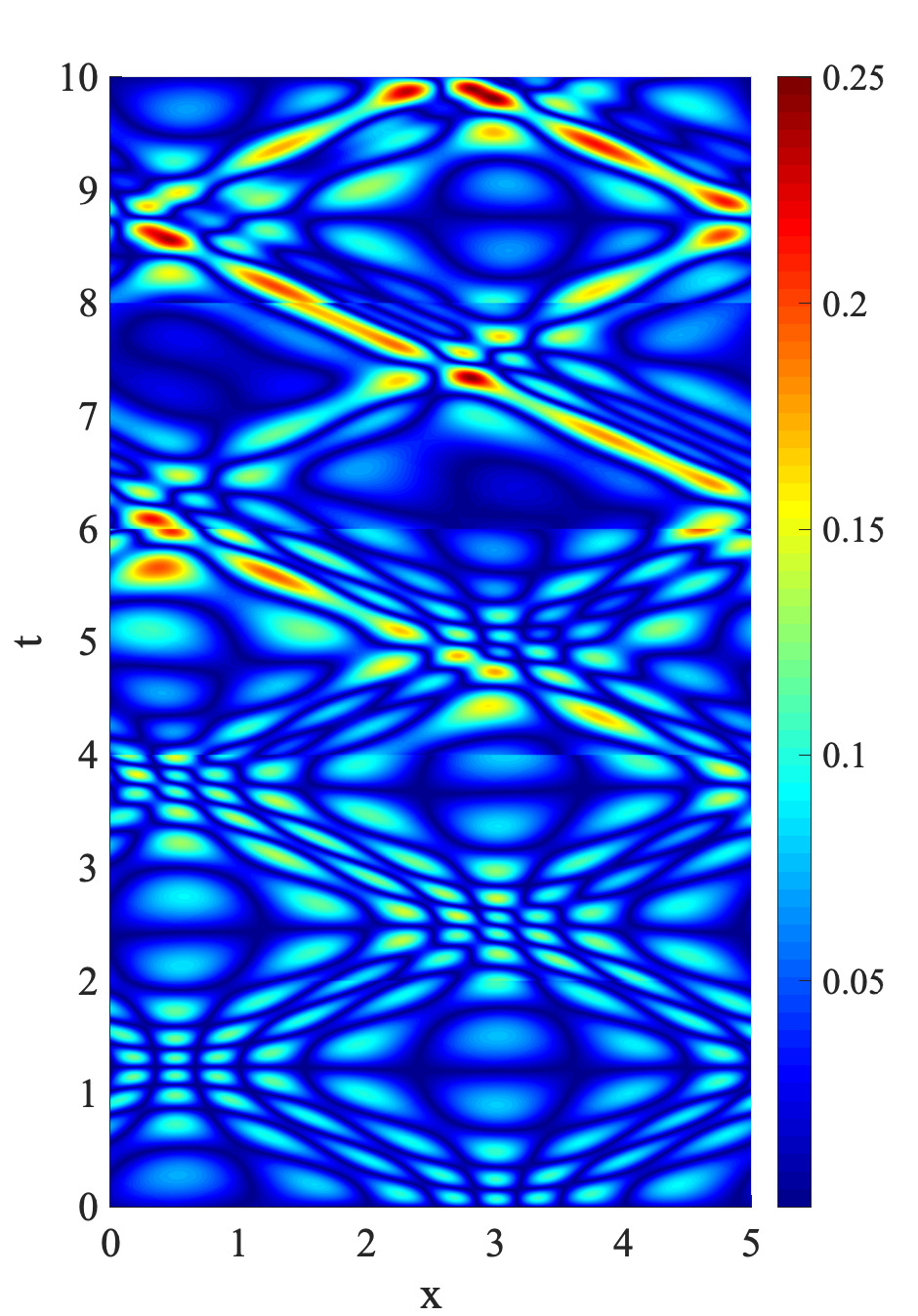}}
    \subfloat[$v^*_\theta$]{\includegraphics[width=0.2\linewidth]{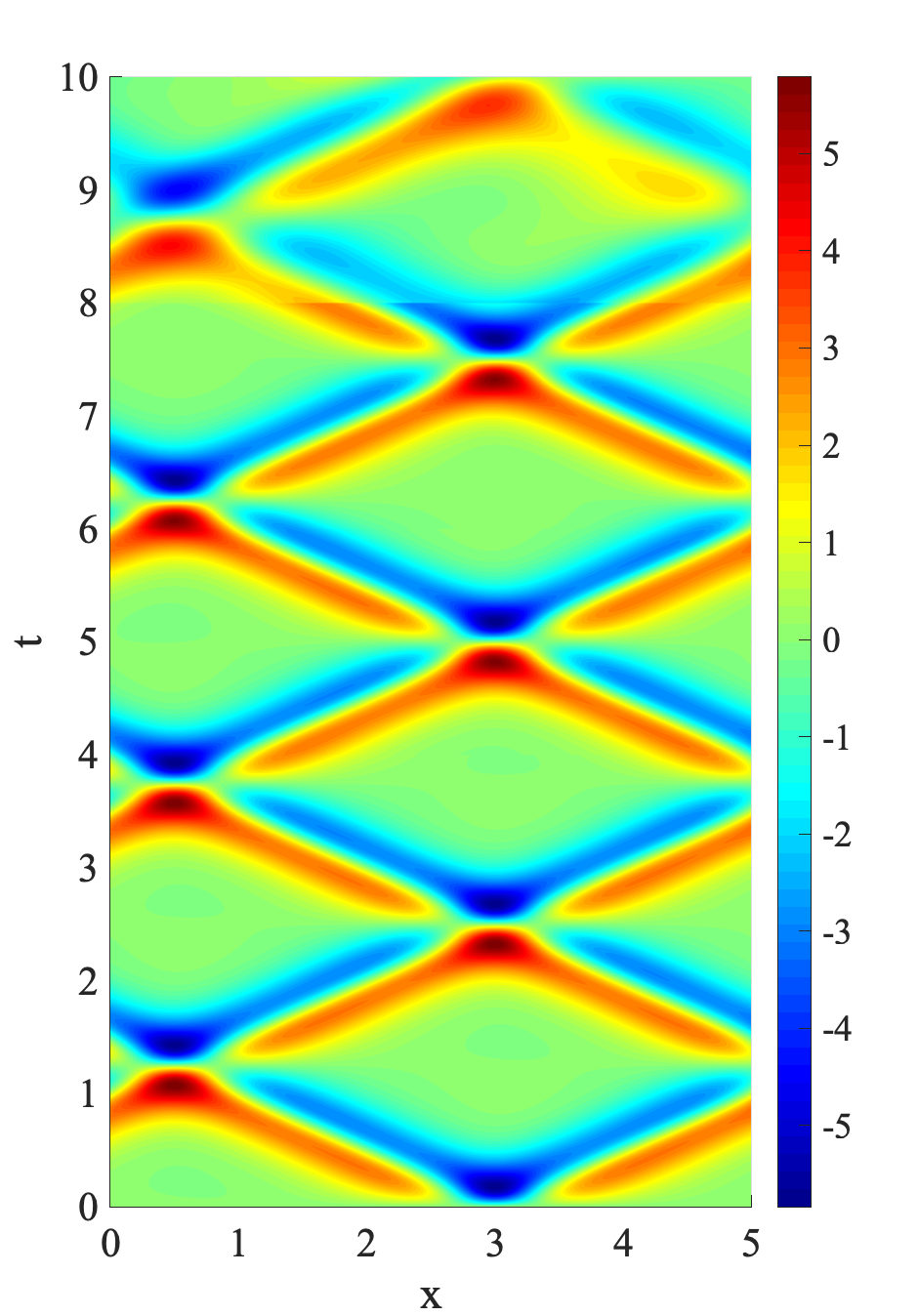}}
    \subfloat[$|v - v^*_\theta|$]{\includegraphics[width=0.2\linewidth]{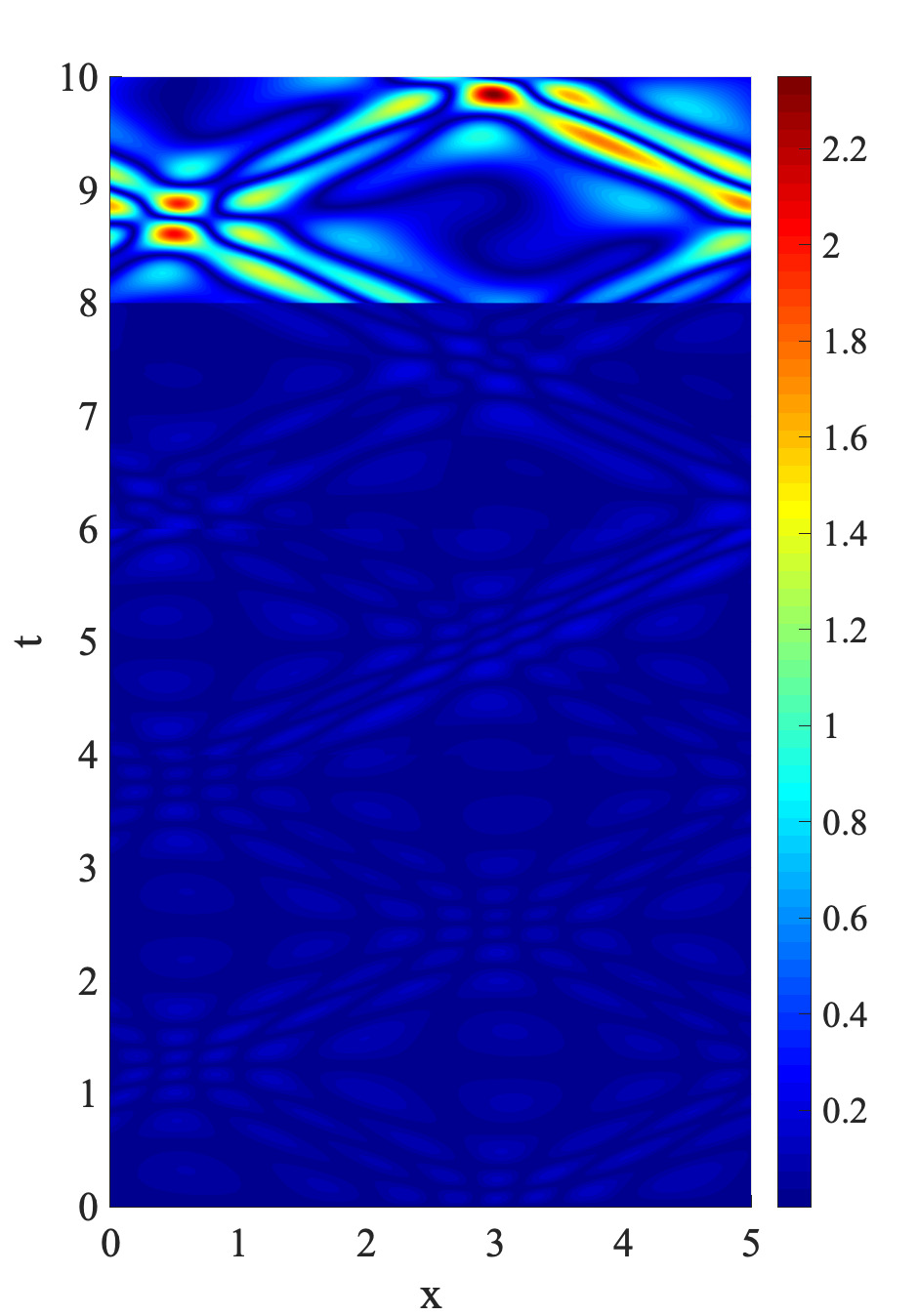}}
	\caption{Wave equation: wave speed distributions ($v=\frac{\partial u}{\partial t}$: true solution; $v_\theta$: HLConcPINN-ExBTM solution; $v_\theta^*$: HLConcPINN-BTM solution).
 Simulation parameters follow those of Figure~\ref{PINN_num_wave_fig1}.
 }
	\label{PINN_num_wave_fig2}
\end{figure}

\begin{figure}[tb]
	\centering
	\subfloat[$ t=2.5 $]{
		\begin{minipage}[b]{0.22\textwidth}
			\includegraphics[scale=0.25]{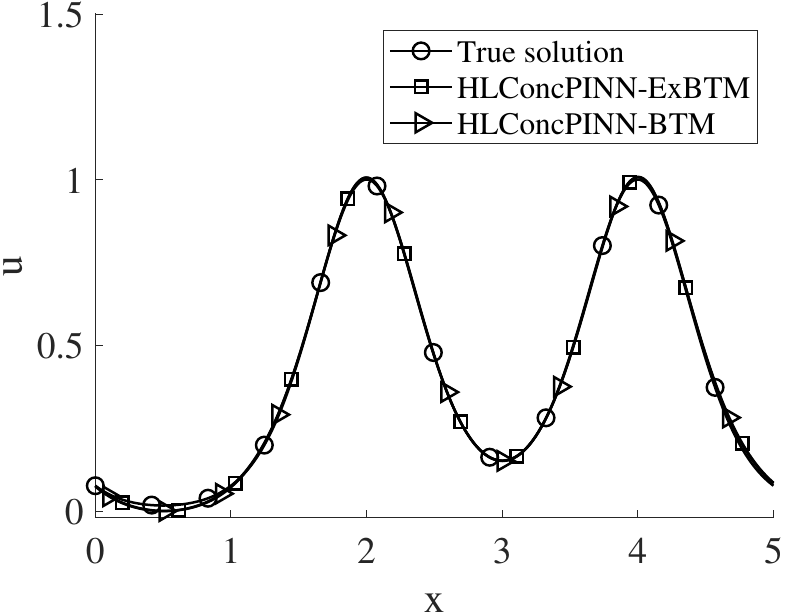}\\
			\includegraphics[scale=0.25]{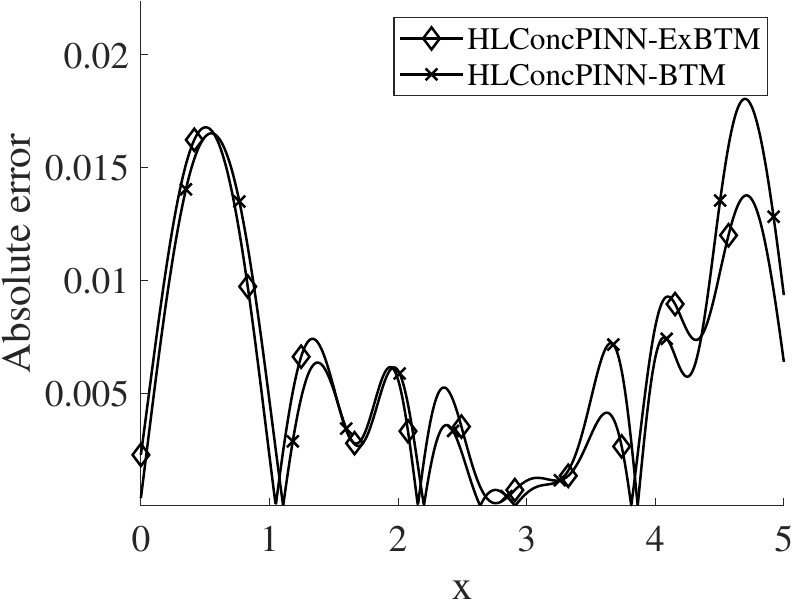}
		\end{minipage}
	}
	\subfloat[$ t=5 $]{
		\begin{minipage}[b]{0.22\textwidth}
			\includegraphics[scale=0.25]{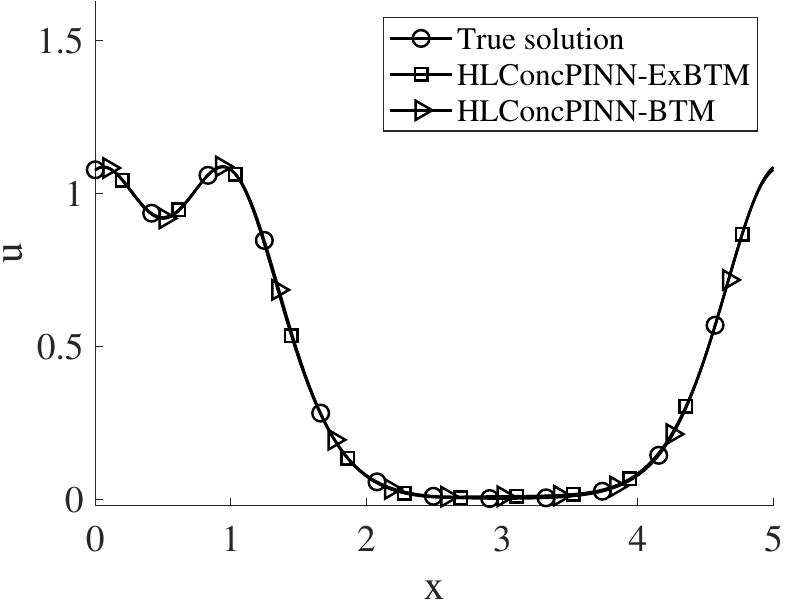}\\
			\includegraphics[scale=0.25]{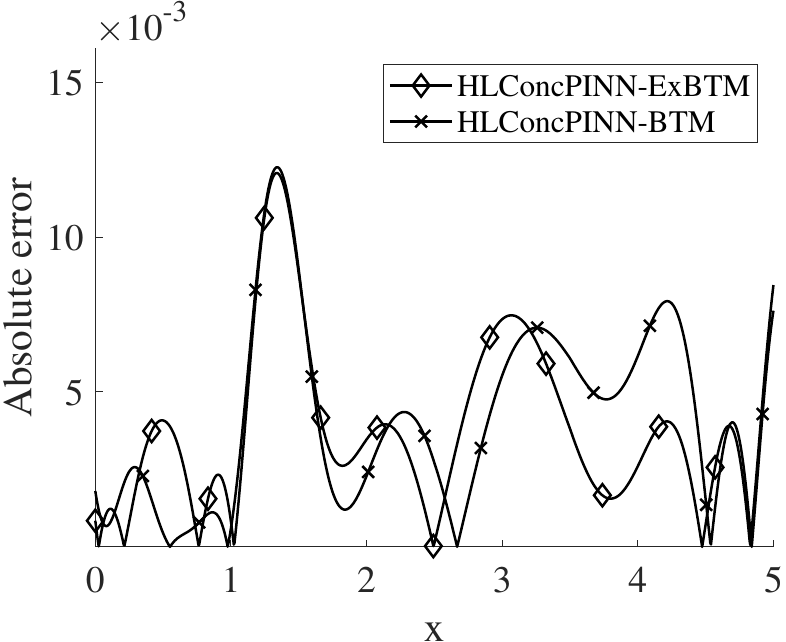}
		\end{minipage}
	}
	\subfloat[$ t=9.5 $]{
		\begin{minipage}[b]{0.22\textwidth}
			\includegraphics[scale=0.25]{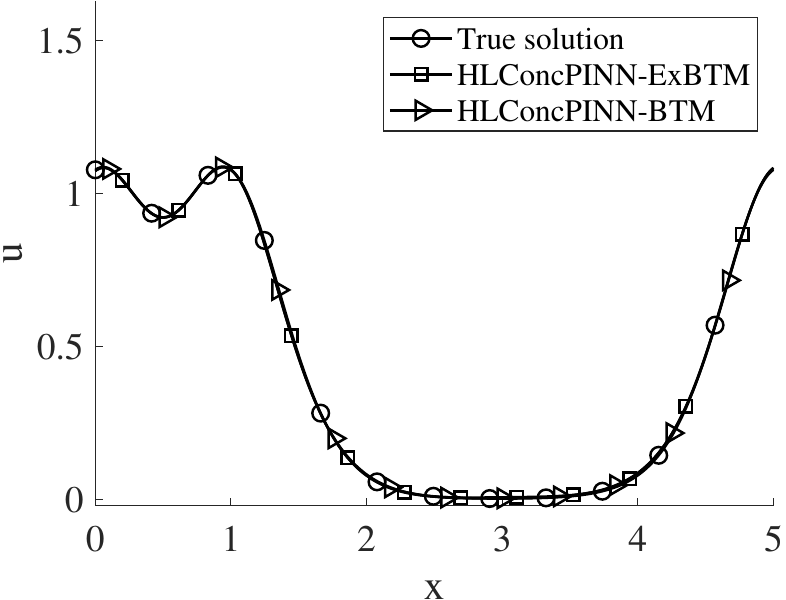}\\
			\includegraphics[scale=0.25]{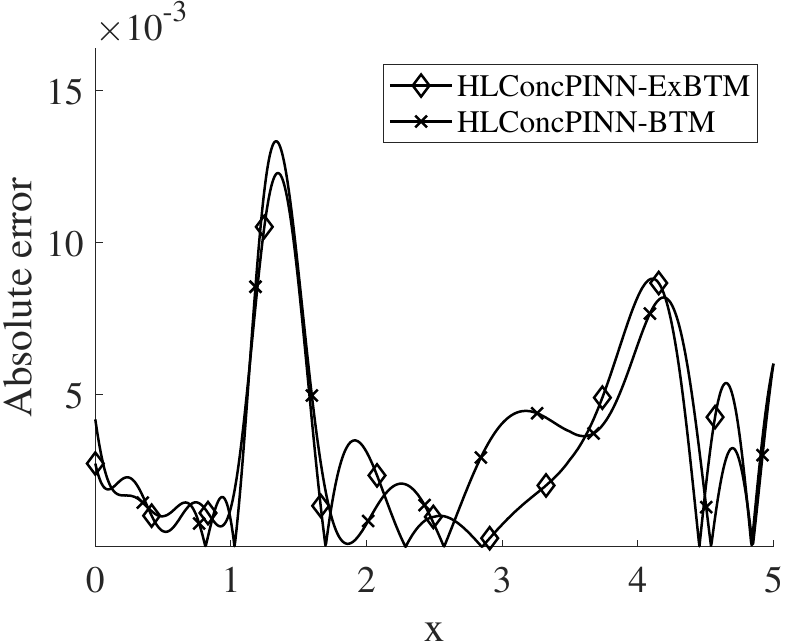}
		\end{minipage}
	}
	\caption{Wave equation: Top row, comparison of wave profiles between the true solution and the HLConcPINN-ExBTM/-BTM solutions at several time instants. Bottom row, absolute-error profiles of HLConcPINN-ExBTM/-BTM. Simulation parameters follow those of Figure~\ref{PINN_num_wave_fig1}.
 }
	\label{PINN_num_wave_fig3_1}
\end{figure}

\begin{figure}[tb]
	\centering
	\subfloat[$ t=2.5 $]{
		\begin{minipage}[b]{0.22\textwidth}
			\includegraphics[scale=0.25]{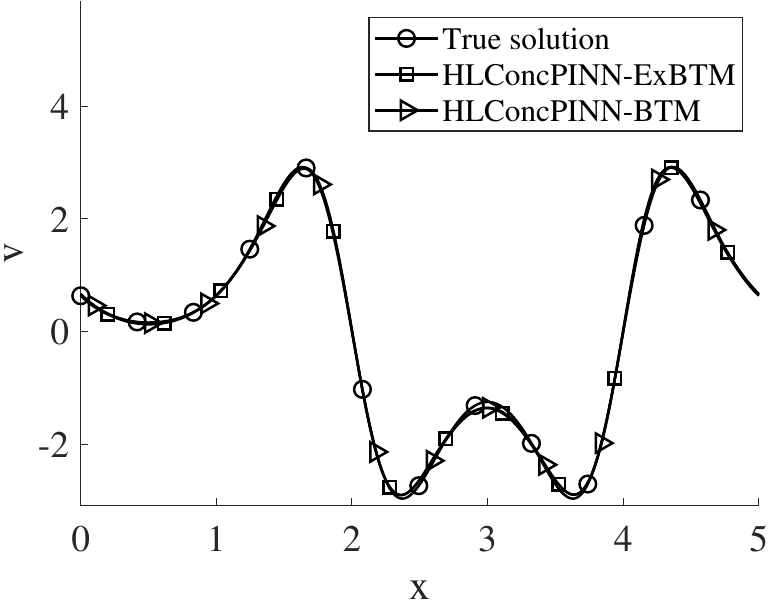}\\
			\includegraphics[scale=0.25]{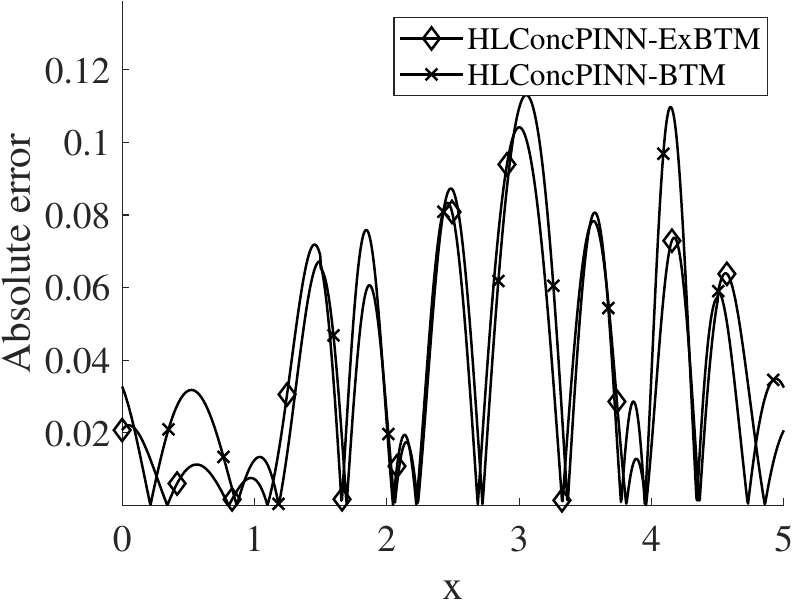}
		\end{minipage}
	}
	\subfloat[$ t=5 $]{
		\begin{minipage}[b]{0.22\textwidth}
			\includegraphics[scale=0.25]{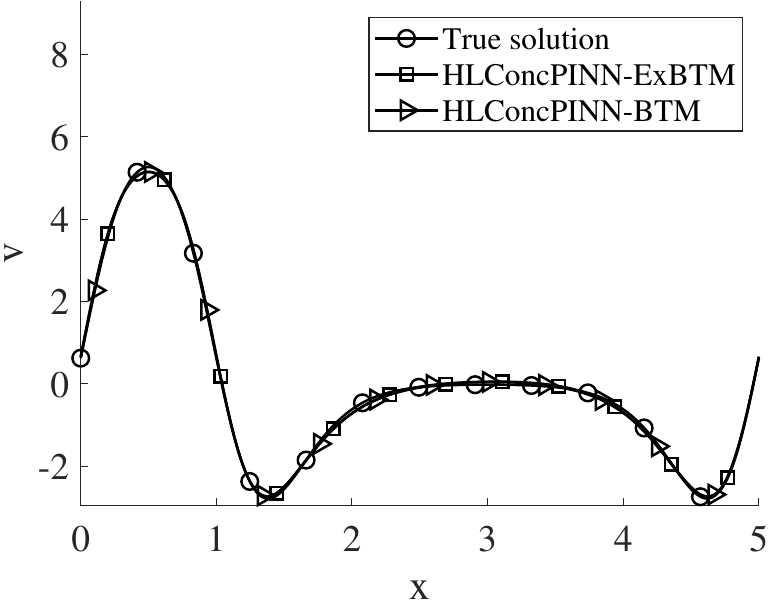}\\
			\includegraphics[scale=0.25]{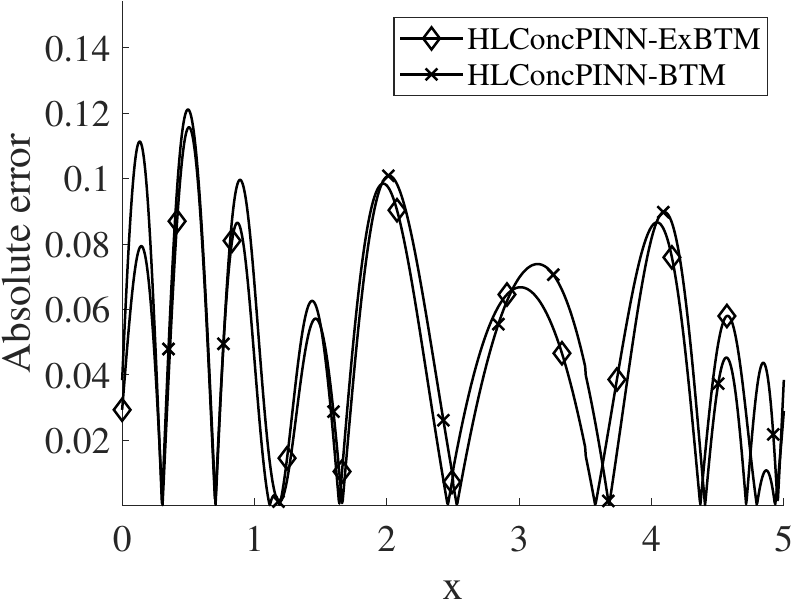}
		\end{minipage}
	}
	\subfloat[$ t=9.5 $]{
		\begin{minipage}[b]{0.22\textwidth}
			\includegraphics[scale=0.25]{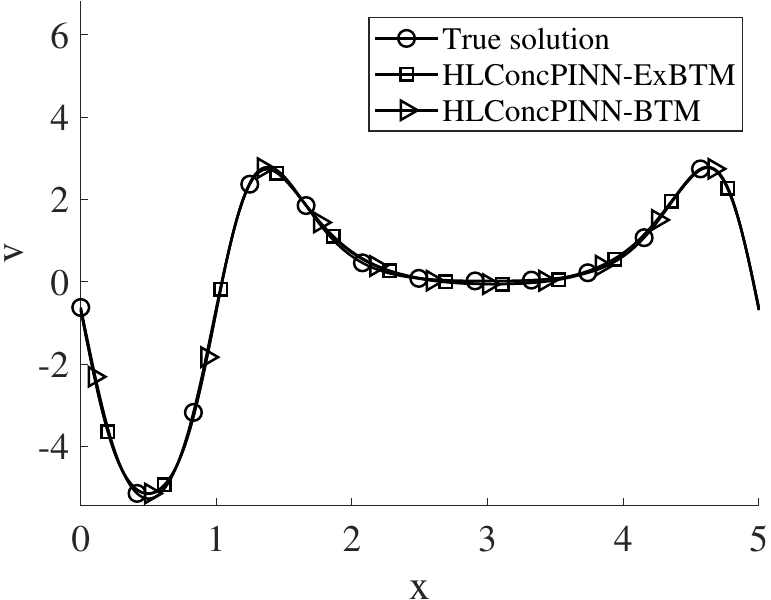}\\
			\includegraphics[scale=0.25]{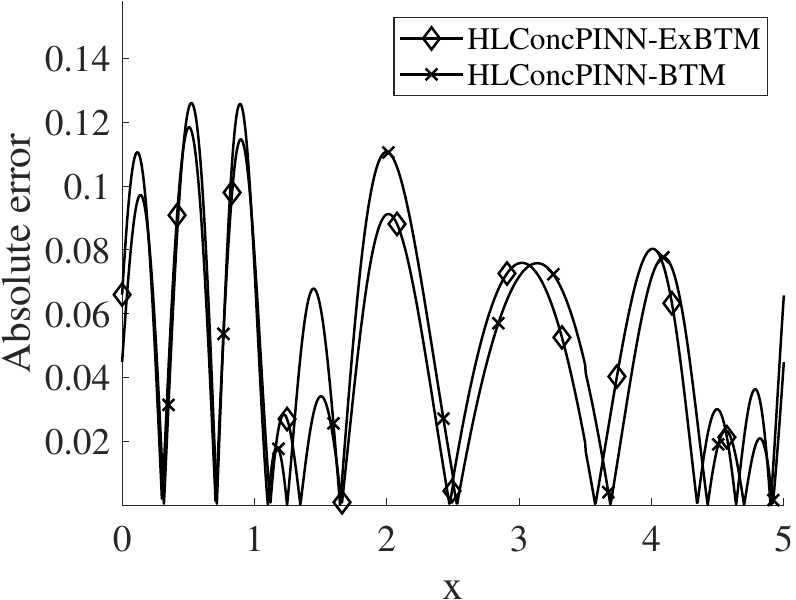}
		\end{minipage}
	}
	\caption{Wave equation: Top row, comparison of wave speed ($v$) profiles between the true solution and the HLConcPINN-ExBTM/BTM solutions at several time instants. Bottom row, profiles of the absolute error of HLConcPINN-ExBTM/-BTM for $v$. Simulation parameters follow those of Figure~\ref{PINN_num_wave_fig1}.
 }
	\label{PINN_num_wave_fig3_2}
\end{figure}

\subsection{Wave Equation}

We next simulate the wave equation in one spatial dimension (plus time) using the current method, following a configuration from~\cite{2021_JCP_Dong_modifiedbatch}. Consider the spatial-temporal domain, $ (x,t)\in D\times[0, T] = [0, 5] \times [0, 10] $, and the following initial-boundary value problem on this domain,
\begin{subequations}\label{num_wave_eq1}
	\begin{align}
		&\frac{\partial^2 u}{\partial t^2} - c^2 \frac{\partial^2 u}{\partial x^2} = 0, \\
		& u(0, t) = u(5, t), \quad
		\frac{\partial u}{\partial x} (0, t) = \frac{\partial u}{\partial x} (5,t), \quad
		u(x,0) = 2\, {\rm sech}^3 \left(\frac{3}{\delta_0}(x-x_0)\right),\quad
		\frac{\partial u}{\partial t}(x,0) = 0,
	\end{align}
\end{subequations}
where $u(x,t)$ is the wave field  to be solved for, $ c $ is the wave speed, $ x_0 $ is the initial peak location of the wave, $ \delta_0 $ is a constant that controls the width of the wave profile, and  periodic boundary conditions are imposed on $ x=0 $ and $ 5 $. We employ $ c=2 $, $ \delta_0=2 $, and $ x_0 = 3 $ for this problem. This  problem has the following solution 
\begin{equation*}\left\{
\begin{split}
	&u(x, t) = {\rm sech}^3\left(\frac{3}{\delta_0}\left(-2.5 +\xi\right)\right) + {\rm sech}^3\left(\frac{3}{\delta_0}\left(-2.5 +\eta\right)\right),\\
	& \xi = {\rm mod}\left(x-x_0 + ct + 2.5, 5\right),\quad
	\eta = {\rm mod}\left(x-x_0 - ct + 2.5, 5\right),
 \end{split}
 \right.
\end{equation*}
where mod refers to the modulo operation. 
In the simulations we introduce the auxiliary field $v(x,t)$ and rewrite \eqref{num_wave_eq1} into
\begin{subequations}\label{num_wave_eq1_1}
\begin{align}
&
\frac{\partial u}{\partial t} - v=0,\qquad
\frac{\partial v}{\partial t} -c^2\frac{\partial^2 u}{\partial x^2}=0, \label{num_wave_eq1_1a}
\\
& u(0,t) = u(5,t), \quad
\frac{\partial u}{\partial x}(0,t) = \frac{\partial u}{\partial x}(5,t), \quad
u(x,0) = 2\, {\rm sech}^3 \left(\frac{3}{\delta_0}(x-x_0)\right),\quad
		v(x, 0) = 0,
\end{align}
\end{subequations}
where $v(x,t)$ is defined by the first equation in \eqref{num_wave_eq1_1a}.

To simulate the system \eqref{num_wave_eq1_1}, the training error in \eqref{wave_T}$-$\eqref{wave_Ti} leads to the following loss function with the HLConcPINN-ExBTM method for the $i$-th time block ($1\leq i\leq l$), 
\begin{align}\label{num_wave_eq2}
    Loss_{i}^I = & \frac{W_1}{N_c} \sum_{n=1}^{N_c}\left[\frac{\partial u_{\theta_i}}{\partial t}(x_{int}^n, t_{int}^n) - v_{\theta_i}(x_{int}^n, t_{int}^n)\right]^2 + \frac{W_2}{N_c} \sum_{n=1}^{N_c} \left[\frac{\partial v_{\theta_i}}{\partial t}(x_{int}^n, t_{int}^n) - 4\frac{\partial^2 u_{\theta_i}}{\partial x^2}(x_{int}^n, t_{int}^n)\right]^2 \nonumber \\
	&+ \frac{W_3}{N_c} \sum_{n=1}^{N_c}\left[\frac{\partial^2u_{\theta_i}}{\partial t\partial x}(x_{int}^n, t_{int}^n) - \frac{\partial v_{\theta_i}}{\partial x}(x_{int}^n, t_{int}^n)\right]^2 + \frac{W_4}{N_c} \sum_{j=1}^{i} \sum_{n=1}^{N_c}\left[ u_{\theta_i}(x_{tb}^n, t_{j-1}) - u_{\theta_{j-1}}(x_{tb}^n, t_{j-1})\right]^2 \nonumber \\
	&+ \frac{W_5}{N_c} \sum_{j=1}^{i} \sum_{n=1}^{N_c} \left[v_{\theta_i}(x_{tb}^n, t_{j-1})-v_{\theta_{j-1}}(x_{tb}^n, t_{j-1}) \right]^2 +  \frac{W_6}{N_c} \sum_{j=1}^{i}\sum_{n=1}^{N_c}\left[\frac{\partial u_{\theta_i}}{\partial x}(x_{tb}^n, t_{j-1}) - \frac{\partial u_{\theta_{j-1}}}{\partial x}(x_{tb}^n, t_{j-1}) \right]^2 \nonumber \\ 
	& + W_7 \Big( \frac{1}{N_c} \sum_{n=1}^{N_c} \left[v_{\theta}(0, t_{sb}^n) - v_{\theta}(5, t_{sb}^n)\right]^2 \Big)^{1/2}
 + W_8 \Big( \frac{1}{N_c} \sum_{n=1}^{N_c} \left[\frac{\partial u_{\theta_i}}{\partial x}(0, t_{sb}^n) - \frac{\partial u_{\theta_i}}{\partial x}(5, t_{sb}^n)\right]^2\Big)^{1/2} \nonumber\\
 &+ Loss_{i-1}^I,
\end{align}
where $Loss_{0}^I=0$, and $W_k>0$ ($1\leq k \leq 8$) are the penalty coefficients added for different loss terms. 
The loss function with the HLConcPINN-BTM method is,
\begin{align}\label{num_wave_eq2_1}
    Loss_{i}^{II} = & \frac{W_1}{N_c} \sum_{n=1}^{N_c}\left[\frac{\partial u_{\theta_i}}{\partial t}(x_{int}^n, t_{int}^n) - v_{\theta_i}(x_{int}^n, t_{int}^n)\right]^2 \notag 
 + \frac{W_2}{N_c} \sum_{n=1}^{N_c} \left[\frac{\partial v_{\theta_i}}{\partial t}(x_{int}^n, t_{int}^n) - 4\frac{\partial^2 u_{\theta_i}}{\partial x^2}(x_{int}^n, t_{int}^n)\right]^2 \nonumber \\
	&+ \frac{W_3}{N_c} \sum_{n=1}^{N_c}\left[\frac{\partial^2u_{\theta_i}}{\partial t\partial x}(x_{int}^n, t_{int}^n) - \frac{\partial v_{\theta_i}}{\partial x}(x_{int}^n, t_{int}^n)\right]^2 \notag 
 + \frac{W_4}{N_c} \sum_{n=1}^{N_c}\left[ u_{\theta_i}(x_{tb}^n, t_{i-1}) - u_{\theta_{i-1}}(x_{tb}^n, t_{i-1})\right]^2 \nonumber \\
	&+ \frac{W_5}{N_c}\sum_{n=1}^{N_c} \left[v_{\theta_i}(x_{tb}^n, t_{i-1})-v_{\theta_{i-1}}(x_{tb}^n, t_{i-1}) \right]^2 +  \frac{W_6}{N_c} \sum_{n=1}^{N_c}\left[\frac{\partial u_{\theta_i}}{\partial x}(x_{tb}^n, t_{i-1}) - \frac{\partial u_{\theta_{i-1}}}{\partial x}(x_{tb}^n, t_{i-1}) \right]^2 \nonumber \\ 
	& + W_7 \Big( \frac{1}{N_c} \sum_{n=1}^{N_c} \left[v_{\theta}(0, t_{sb}^n) - v_{\theta}(5, t_{sb}^n)\right]^2 \Big)^{1/2}
 + W_8 \Big( \frac{1}{N_c} \sum_{n=1}^{N_c} \left[\frac{\partial u_{\theta_i}}{\partial x}(0, t_{sb}^n) - \frac{\partial u_{\theta_i}}{\partial x}(5, t_{sb}^n)\right]^2\Big)^{1/2}.
\end{align}

In the simulations, we employ  neural network architectures with two output nodes, representing the wave field $u$ and the wave speed $v=\frac{\partial u}{\partial t}$, respectively. The penalty coefficients in the loss functions are taken to be $(W_1,...,W_8)=(0.9, 0.9, 0.9, 0.1, 0.1, 0.1, 0.1, 0.1)$. We employ $5$ uniform time blocks in block time marching. The neural network parameters (network depth/width, and activation functions) and the training collocation points are varied in the tests. The adopted neural network structures are listed in Table~\ref{num_table0}.

An overview of the HLConcPINN-ExBTM and HLConcPINN-BTM solutions to the wave equation and their accuracy is provided in Figures~\ref{PINN_num_wave_fig1} to~\ref{PINN_num_wave_fig4_1}.
Figures~\ref{PINN_num_wave_fig1} and \ref{PINN_num_wave_fig2} show  distributions of the wave field $u$ and  the wave speed $v$, corresponding to the true solution,  the HLConcPINN-ExBTM and HLConcPINN-BTM solutions, as well as their point-wise absolute errors, in the spatial-temporal domain. The neural network architecture is specified in the caption of Figure~\ref{PINN_num_wave_fig1}, consisting of three hidden layers, with the $\tanh$ activation function for the first two hidden layers and the sine function for the last hidden layer. $N_c=2500$ has been employed for the training collocation points.
The HLConcPINN-ExBTM method is observed to produce more accurate results than HLConcPINN-BTM, especially toward later time instants. The errors of both methods are observed to grow over time. In particular, the accuracy of HLConcPINN-BTM in the last time block becomes quite poor, with pronounced deviations from the true solution in the wave speed distribution.

\begin{figure}[!ht]
	\centering
	\subfloat[PINN-ExBTM]{\includegraphics[width=0.35\linewidth]{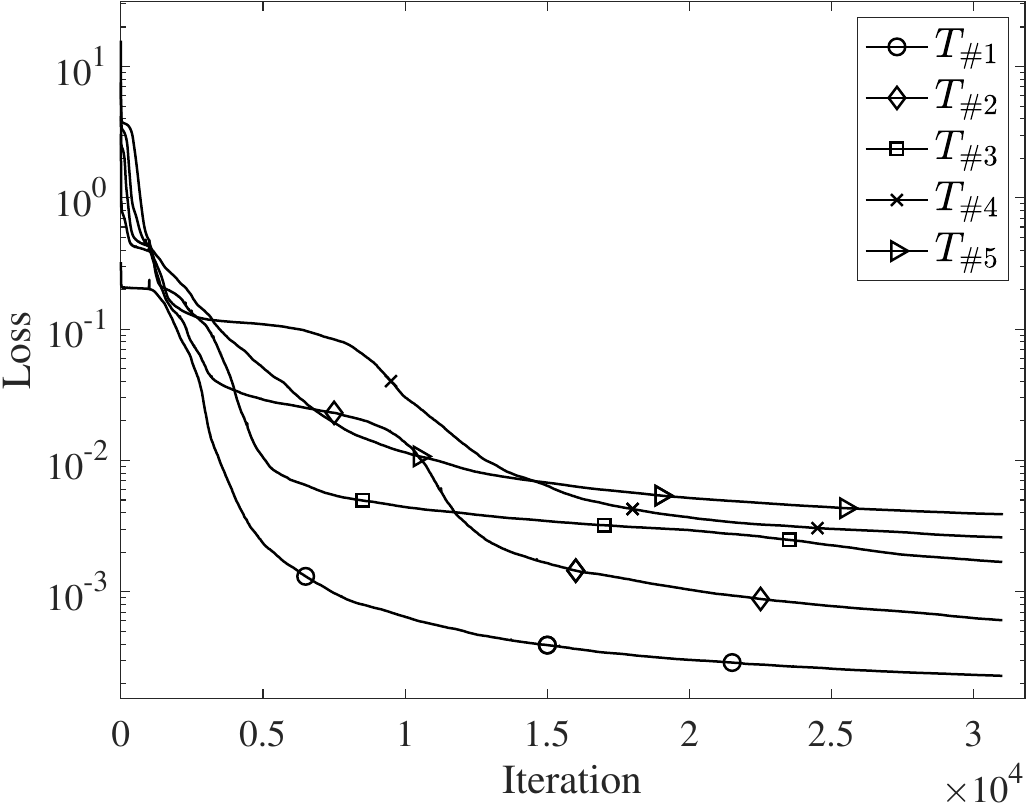}}\qquad
	\subfloat[PINN-BTM]{\includegraphics[width=0.35\linewidth]{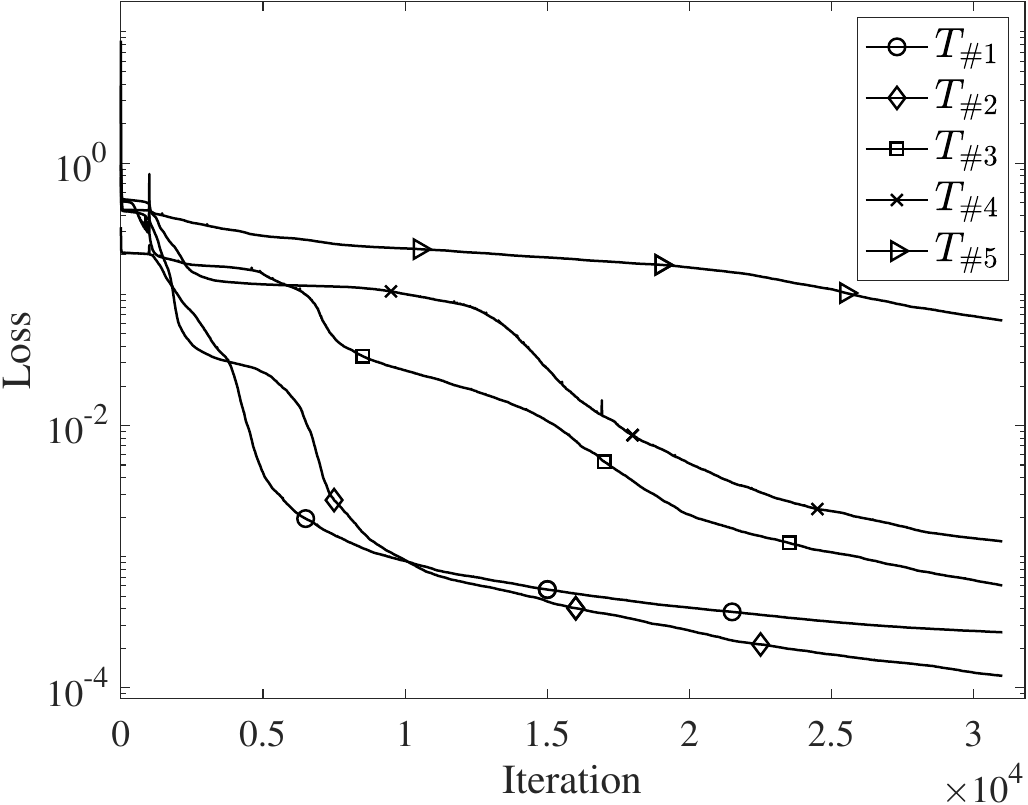}}\hspace{0.1em}
	\caption{Wave equation: Training loss histories in different time blocks of (a) HLConcPINN-ExBTM and (b) HLConcPINN-BTM. Simulation parameters follow those of Figure~\ref{PINN_num_wave_fig1}.
 }
	\label{PINN_num_wave_fig4_1}
\end{figure}

Figures \ref{PINN_num_wave_fig3_1} and \ref{PINN_num_wave_fig3_2} illustrate the solution profiles of the wave field $u$ and the wave speed $v$ obtained using HLConcPINN-ExBTM and HLConcPINN-BTM at three time instants ($t=2.5, 5, 9.5$), accompanied by their corresponding absolute errors. The simulation parameters here follow those of Figure~\ref{PINN_num_wave_fig1}. The  error of HLConcPINN-ExBTM  is generally observed to be  smaller than that of HLConcPINN-BTM.
The training loss histories with this group of tests for HLConcPINN-ExBTM and HLConcPINN-BTM are shown in Figure~\ref{PINN_num_wave_fig4_1}. It can be generally observed that the training process results in higher loss values in later time blocks, implying a growth in the errors over time consistent with what is observed in Figures~\ref{PINN_num_wave_fig1} and~\ref{PINN_num_wave_fig2}.


\begin{table}[tb]\small
	\centering\small
\begin{tabular}{c|@{}c@{}|cc|cc|cc|cc}
\hline
\multirow{2}{*}{Error}    & \multirow{2}{*}{Time block} & \multicolumn{2}{c|}{$N_c=1500$} & \multicolumn{2}{c|}{$N_c=2000$} & \multicolumn{2}{c|}{$N_c=2500$} & \multicolumn{2}{c}{$N_c=3000$} \\ \cline{3-10}
                      &  & ExBTM  & BTM  & ExBTM  & BTM   & ExBTM  & BTM  & ExBTM  & BTM   \\ \hline
\multirow{5}{*}{$l^2$} & $T_{\#1}$  & 9.86e-03 & 1.05e-02 & 9.76e-03 & 1.01e-02 & 1.09e-02 & 1.19e-02 & 9.87e-03 & 9.31e-03 \\ \cline{3-10} 
                      & $T_{\#2}$  & 1.21e-02 & 1.30e-02 & 1.14e-02 & 1.15e-02 & 1.17e-02 & 9.53e-03 & 1.01e-02 & 9.49e-03 \\ \cline{3-10}
                      & $T_{\#3}$  & 1.21e-02 & 3.73e-02 & 1.39e-02 & 1.27e-02 & 1.71e-02 & 1.52e-02 & 1.38e-02 & 1.39e-02 \\ \cline{3-10} 
                      & $T_{\#4}$  & 1.75e-02 & 2.85e-01 & 5.91e-02 & 2.44e-02 & 1.74e-02 & 2.17e-02 & 5.26e-02 & 3.23e-01 \\ \cline{3-10} 
                      & $T_{\#5}$  & 3.34e-02 & 3.85e-01 & 1.25e-01 & 2.15e-01 & 1.88e-02 & 1.76e-01 & 5.70e-02 & 7.36e-01 \\ \hline
\multirow{5}{*}{$l^\infty$} & $T_{\#1}$  & 3.13e-02 & 2.73e-02 & 2.86e-02 & 2.76e-02 & 2.85e-02 & 3.13e-02 & 2.75e-02 & 2.52e-02 \\ \cline{3-10} 
                      & $T_{\#2}$  & 3.13e-02 & 3.57e-02 & 3.22e-02 & 3.43e-02 & 3.46e-02 & 2.64e-02 & 2.85e-02 & 2.72e-02 \\ \cline{3-10}
                      & $T_{\#3}$  & 3.43e-02 & 1.00e-01 & 6.94e-02 & 3.98e-02 & 5.47e-02 & 4.59e-02 & 4.28e-02 & 4.90e-02 \\ \cline{3-10} 
                      & $T_{\#4}$  & 5.31e-02 & 8.33e-01 & 1.40e-01 & 7.40e-02 & 5.85e-02 & 6.56e-02 & 1.44e-01 & 1.13e+00 \\ \cline{3-10} 
                      & $T_{\#5}$  & 8.55e-02 & 1.14e+00 & 2.28e-01 & 6.42e-01 & 6.09e-02 & 5.04e-01 & 1.65e-01 & 2.35e+00 \\ \hline
\end{tabular}
\caption{Wave equation: $l^2$ and $l^\infty$ errors of wave field $u$ in different time blocks obtained with HLConcPINN-ExBTM and HLConcPINN-BTM for a range of training data points $N_c$. NN: [2,90,90,10,2], $\tanh$ activation in first two hidden layers and sine activation in the last hidden layer.
}
	\label{tab_wave_err_1}
\end{table}

\begin{table}[htb]\small
	\centering\small
\begin{tabular}{c|@{}c@{}|cc|cc|cc|cc}
\hline
\multirow{2}{*}{Error}    & \multirow{2}{*}{Time block} & \multicolumn{2}{c|}{tanh} & \multicolumn{2}{c|}{Gaussian} & \multicolumn{2}{c|}{swish} & \multicolumn{2}{c}{softplus} \\ \cline{3-10}
                      &  & ExBTM  & BTM  & ExBTM  & BTM   & ExBTM  & BTM  & ExBTM  & BTM   \\ \hline
\multirow{5}{*}{$l^2$} & $T_{\#1}$  & 1.94e-02 & 2.08e-02 & 1.34e-02 & 8.44e-03 & 2.03e-02 & 1.88e-02 & 1.22e-02 & 1.51e-02 \\ \cline{3-10} 
                      & $T_{\#2}$  & 2.24e-02 & 2.12e-02 & 6.07e-02 & 3.18e-02 & 2.65e-02 & 3.20e-02 & 1.66e-02 & 2.21e-02 \\ \cline{3-10}
                      & $T_{\#3}$  & 1.22e+00 & 7.27e-01 & 9.41e-01 & 4.80e-02 & 5.47e-02 & 1.65e+00 & 3.64e-02 & 3.73e-02 \\ \cline{3-10} 
                      & $T_{\#4}$  & 2.84e+00 & 1.52e+00 & 2.21e+00 & 4.06e-01 & 8.17e-02 & 4.18e+00 & 9.31e-02 & 3.95e-01 \\ \cline{3-10} 
                      & $T_{\#5}$  & 4.91e+00 & 2.34e+00 & 3.74e+00 & 5.42e-01 & 1.42e-01 & 7.25e+00 & 2.82e-01 & 5.83e-01 \\ \hline
\multirow{5}{*}{$l^\infty$} & $T_{\#1}$  & 5.46e-02 & 5.68e-02 & 3.74e-02 & 2.40e-02 & 5.57e-02 & 5.12e-02 & 3.54e-02 & 4.17e-02 \\ \cline{3-10} 
                      & $T_{\#2}$  & 6.13e-02 & 6.77e-02 & 2.38e-01 & 8.23e-02 & 7.70e-02 & 8.74e-02 & 4.92e-02 & 6.82e-02 \\ \cline{3-10}
                      & $T_{\#3}$  & 2.66e+00 & 1.83e+00 & 2.34e+00 & 1.35e-01 & 1.27e-01 & 3.91e+00 & 1.09e-01 & 1.06e-01 \\ \cline{3-10} 
                      & $T_{\#4}$  & 4.65e+00 & 2.63e+00 & 3.23e+00 & 1.41e+00 & 2.35e-01 & 6.33e+00 & 2.44e-01 & 1.43e+00 \\ \cline{3-10} 
                      & $T_{\#5}$  & 7.40e+00 & 3.70e+00 & 5.16e+00 & 1.80e+00 & 4.23e-01 & 1.04e+01 & 7.86e-01 & 1.98e+00 \\ \hline
\end{tabular}
\caption{Wave equation: $l^2$ and $l^\infty$ errors of the wave field $u$ in different time blocks obtained using HLConcPINN-ExBTM and HLConcPINN-BTM with different activation functions in the last hidden layer. NN: [2,90,90,10,2], with $\tanh$ activation in the first two hidden layers. The activation function in the last hidden layer is varied, as listed in the first row of the table. $N_c=2500$ for the training collocation points.
}
	\label{tab_wave_err_3}
\end{table}

Table~\ref{tab_wave_err_1} shows a study of the effect of the training data points on the simulation accuracy of the HLConcPINN-ExBTM and HLConcPINN-BTM methods. Here we list the $l^2$ and $l^{\infty}$ errors of HLConcPINN-ExBTM and HLConcPINN-BTM in different time blocks obtained with training collocation points ranging from $N_c=1500$ to $N_c=3000$ in the simulations. The neural network architecture is given by $[2,90,90,10,2]$, with the $\tanh$ activation function for the first two hidden layers and sine function for the last hidden layer. The data suggest  little sensitivity with respect to number of  training data points in the range tested here.

Table~\ref{tab_wave_err_3} illustrates a test of the effect of different activation functions on the simulation results. The network architecture is characterized by $[2,90,90,10,2]$, with $\tanh$ as the activation function for the first two hidden layers, while the activation function for the last hidden layer is varied among $\tanh$, Gaussian, swish, and softplus functions. The training collocation points are set to $N_c=2500$. The table provides the $l^2$ and $l^{\infty}$ errors of the wave field in different time blocks computed using HLConcPINN-ExBTM and HLConcPINN-BTM corresponding to different activation functions for the last hidden layer. These data can be compared with that of Table~\ref{tab_wave_err_1} corresponding to $N_c=2500$, where the sine activation function has been used. Overall the sine function appears to yield the best results among the activation functions tested here.

\begin{figure}[!ht]
	\centering
	\subfloat[$u$]
 { 
 \includegraphics[width=0.4\linewidth]{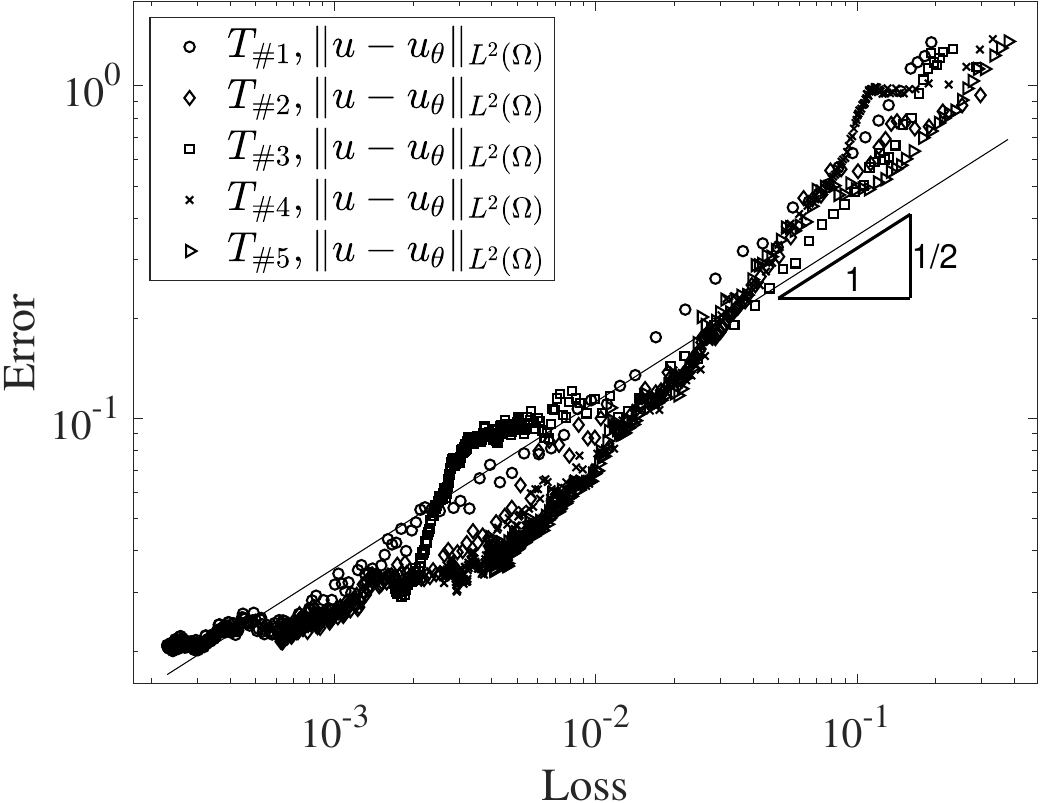}}\qquad
    \subfloat[$v$]
    { 
    \includegraphics[width=0.4\linewidth]{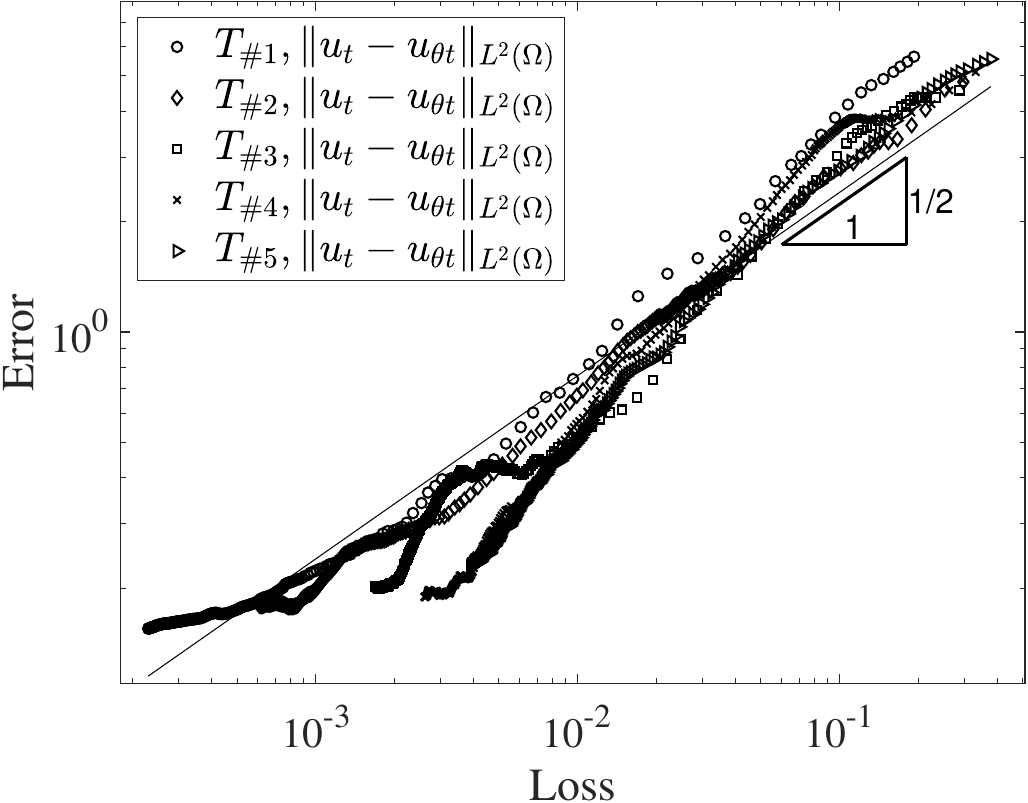}}
    \hspace{0.5em}
	\caption{Wave equation: $l^2$ errors of (a) the wave field $u$, and (b) the wave speed $v=\partial u/\partial t$,  as a function of the training loss for HLConcPINN-ExBTM. NN: [2,90,90,10,2], with $\tanh$ activation for the first two hidden layers and sine activation for the last hidden layer; $N_c=2500$ for the training data points.
 }
	\label{PINN_num_wave_fig5_1}
\end{figure}

Figure \ref{PINN_num_wave_fig5_1} illustrates the relation (in logarithmic scale) between the $l^2$ errors of the wave field $u$ and the wave speed $v=\frac{\partial u}{\partial t}$ as a function of the training loss value for the HLConcPINN-ExBTM method. The neural network architecture, the activation functions, and the training data points are provided in the figure caption. The simulation data approximately exhibit a scaling power of $1/2$, roughly consistent with the conclusion of Theorem~\ref{sec5_Theorem3}.



\comment{
\begin{figure}[tb]
	\centering
	\subfloat[PINN-ExBTM]
 { 
 \includegraphics[width=0.5\linewidth]{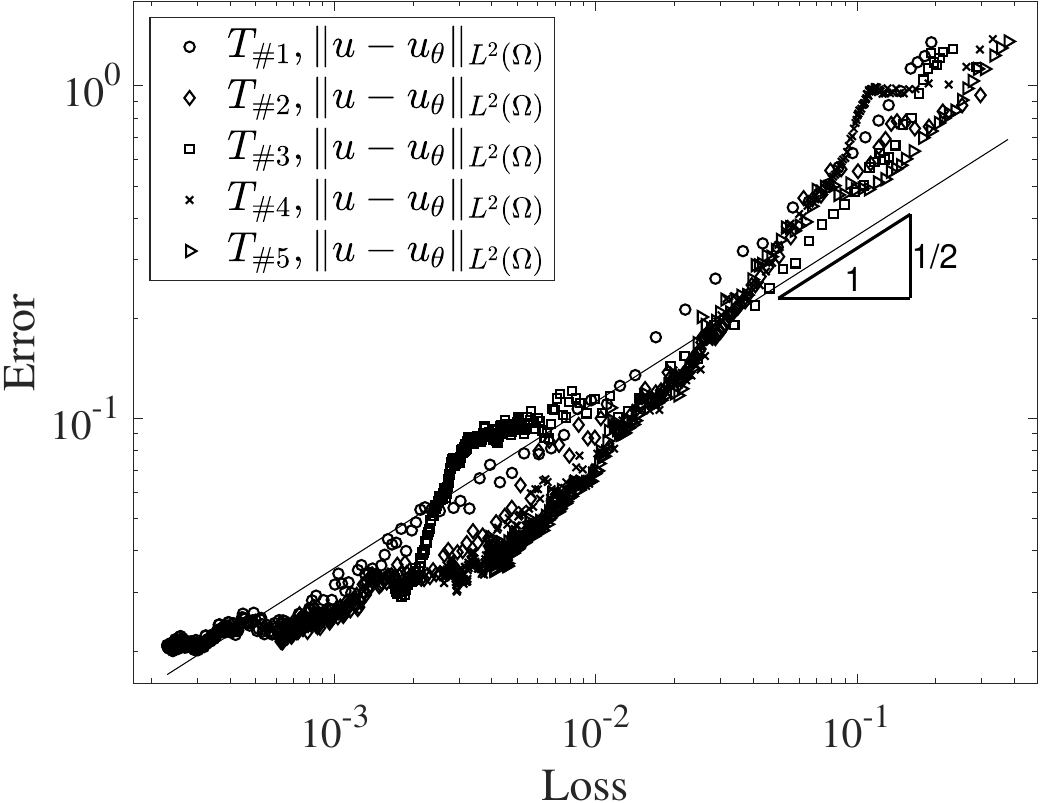}}
	\subfloat[PINN-BTM]
 { 
    \includegraphics[width=0.5\linewidth]{./figures/Wave_BTM_2500_90,90,10_sin_TBlock5_errorRatio_u}}\\
    \subfloat[PINN-ExBTM]
    { 
    \includegraphics[width=0.5\linewidth]{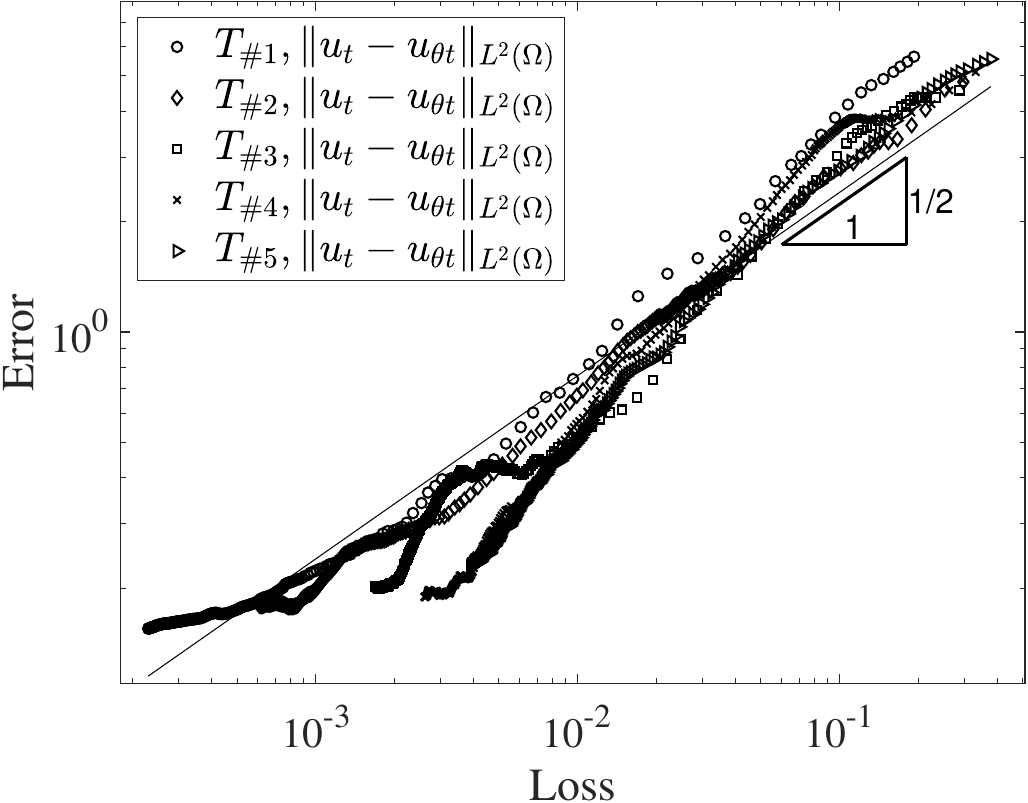}}
	\subfloat[PINN-BTM]
 { 
    \includegraphics[width=0.5\linewidth]{./figures/Wave_BTM_2500_90,90,10_sin_TBlock5_errorRatio_ut}}\hspace{0.5em}
    \hspace{0.5em}
	\caption{Wave equation: The $l^2$ errors of $u$ as a function of the training loss value. [The three hidden layers with $[90, 90, 10]$ neurons. $\tanh-\tanh-\sin$, $N_c=2500$ for training, $N_{ev}=1000$ for prediction.] [(a)$-$(b) ExPINN and PINN for $u$, (c)$-$(d) ExPINN and PINN for $\frac{\partial u}{\partial t}$]}
	\label{PINN_num_wave_fig5_1}
\end{figure}
} 


\comment{
\begin{table}[tb]\small
\caption{Wave equation: For $\frac{\partial u}{\partial t}$, the $l_2$ and $l_\infty$ errors versus the number of training data points $N_c$ and the hidden layers' size. [$\tanh-\tanh-\sin$, layers: [90, 90, 10] for training, $N_{ev}=1000$ for prediction]}
	\label{tab_wave_err_2}
	\centering\small
\begin{tabular}{c|@{}c@{}|cc|cc|cc|cc}
\hline
\multirow{2}{*}{Error}    & \multirow{2}{*}{Time block} & \multicolumn{2}{c|}{$N_c=1500$} & \multicolumn{2}{c|}{$N_c=2000$} & \multicolumn{2}{c|}{$N_c=2500$} & \multicolumn{2}{c}{$N_c=3000$} \\ \cline{3-10}
                      &  & ExBTM  & BTM  & ExBTM  & BTM   & ExBTM  & BTM  & ExBTM  & BTM   \\ \hline
\multirow{5}{*}{$l_2$} & $T_{\#1}$  & 2.89e-02 & 2.33e-02 & 2.64e-02 & 2.34e-02 & 2.70e-02 & 2.85e-02 & 2.52e-02 & 2.39e-02 \\ \cline{3-10} 
                      & $T_{\#2}$  & 2.93e-02 & 3.43e-02 & 3.26e-02 & 2.80e-02 & 3.18e-02 & 2.63e-02 & 2.71e-02 & 2.45e-02 \\ \cline{3-10}
                      & $T_{\#3}$  & 2.85e-02 & 4.92e-02 & 2.69e-02 & 2.47e-02 & 3.43e-02 & 3.41e-02 & 3.43e-02 & 3.08e-02 \\ \cline{3-10} 
                      & $T_{\#4}$  & 3.12e-02 & 5.25e-01 & 4.98e-02 & 3.68e-02 & 3.35e-02 & 3.65e-02 & 1.22e-01 & 5.67e-01 \\ \cline{3-10} 
                      & $T_{\#5}$  & 3.02e-02 & 6.57e-01 & 4.07e-02 & 4.34e-01 & 3.96e-02 & 3.40e-01 & 1.33e-01 & 7.62e-01 \\ \hline
\multirow{5}{*}{$l_\infty$} & $T_{\#1}$  & 8.98e-02 & 6.60e-02 & 7.34e-02 & 6.57e-02 & 6.69e-02 & 7.81e-02 & 6.46e-02 & 6.14e-02 \\ \cline{3-10} 
                      & $T_{\#2}$  & 7.75e-02 & 9.49e-02 & 9.55e-02 & 7.85e-02 & 8.62e-02 & 7.27e-02 & 6.63e-02 & 7.07e-02 \\ \cline{3-10}
                      & $T_{\#3}$  & 8.84e-02 & 1.37e-01 & 2.75e-01 & 9.86e-02 & 1.12e-01 & 1.18e-01 & 1.13e-01 & 1.14e-01 \\ \cline{3-10} 
                      & $T_{\#4}$  & 9.29e-02 & 1.63e+00 & 1.45e-01 & 1.06e-01 & 1.38e-01 & 1.20e-01 & 3.99e-01 & 2.04e+00 \\ \cline{3-10} 
                      & $T_{\#5}$  & 1.04e-01 & 1.84e+00 & 1.36e-01 & 1.78e+00 & 1.38e-01 & 1.29e+00 & 4.83e-01 & 2.48e+00 \\ \hline
\end{tabular}
\end{table}

} 

\comment{
\begin{table}[tb]\small
\caption{Wave equation: For $\frac{\partial u}{\partial t}$, the $l_2$ and $l_\infty$ errors versus the activation functions. [Layers: [90, 90, 10], $N_c=2500$ for training, $N_{ev}=1000$ for prediction]}
	\label{tab_wave_err_4}
	\centering\small
\begin{tabular}{c|@{}c@{}|cc|cc|cc|cc}
\hline
\multirow{2}{*}{Error}    & \multirow{2}{*}{Time block} & \multicolumn{2}{c|}{tanh} & \multicolumn{2}{c|}{Gaussian} & \multicolumn{2}{c|}{swish} & \multicolumn{2}{c}{softplus} \\ \cline{3-10}
                      &  & ExBTM  & BTM  & ExBTM  & BTM   & ExBTM  & BTM  & ExBTM  & BTM   \\ \hline
\multirow{5}{*}{$l_2$} & $T_{\#1}$  & 5.11e-02 & 5.65e-02 & 3.27e-02 & 1.84e-02 & 5.37e-02 & 4.89e-02 & 3.10e-02 & 3.97e-02 \\ \cline{3-10} 
                      & $T_{\#2}$  & 6.58e-02 & 6.01e-02 & 1.80e-01 & 7.55e-02 & 5.91e-02 & 7.17e-02 & 3.85e-02 & 5.17e-02 \\ \cline{3-10}
                      & $T_{\#3}$  & 6.38e-01 & 7.66e-01 & 7.25e-01 & 7.37e-02 & 6.56e-02 & 8.25e-01 & 6.56e-02 & 7.29e-02 \\ \cline{3-10} 
                      & $T_{\#4}$  & 7.91e-01 & 8.46e-01 & 8.05e-01 & 6.44e-01 & 1.79e-01 & 9.73e-01 & 1.98e-01 & 6.48e-01 \\ \cline{3-10} 
                      & $T_{\#5}$  & 8.03e-01 & 8.20e-01 & 8.25e-01 & 7.69e-01 & 2.98e-01 & 9.57e-01 & 5.48e-01 & 7.66e-01 \\ \hline
\multirow{5}{*}{$l_\infty$} & $T_{\#1}$  & 1.45e-01 & 1.58e-01 & 8.45e-02 & 4.94e-02 & 1.43e-01 & 1.26e-01 & 8.49e-02 & 1.09e-01 \\ \cline{3-10} 
                      & $T_{\#2}$  & 2.21e-01 & 2.03e-01 & 8.73e-01 & 1.95e-01 & 1.78e-01 & 2.30e-01 & 1.27e-01 & 1.60e-01 \\ \cline{3-10}
                      & $T_{\#3}$  & 2.32e+00 & 2.68e+00 & 2.62e+00 & 2.24e-01 & 2.32e-01 & 3.13e+00 & 2.73e-01 & 2.62e-01 \\ \cline{3-10} 
                      & $T_{\#4}$  & 2.64e+00 & 2.74e+00 & 2.68e+00 & 2.12e+00 & 6.81e-01 & 3.31e+00 & 6.72e-01 & 2.13e+00 \\ \cline{3-10} 
                      & $T_{\#5}$  & 2.53e+00 & 2.42e+00 & 2.57e+00 & 2.27e+00 & 1.10e+00 & 3.12e+00 & 1.66e+00 & 2.23e+00 \\ \hline
\end{tabular}
\end{table}

} 

%% file: content/numerical_examples_KG.tex
\begin{figure}[!ht]
	\centering
	\subfloat[$u$]{\includegraphics[width=0.2\linewidth]{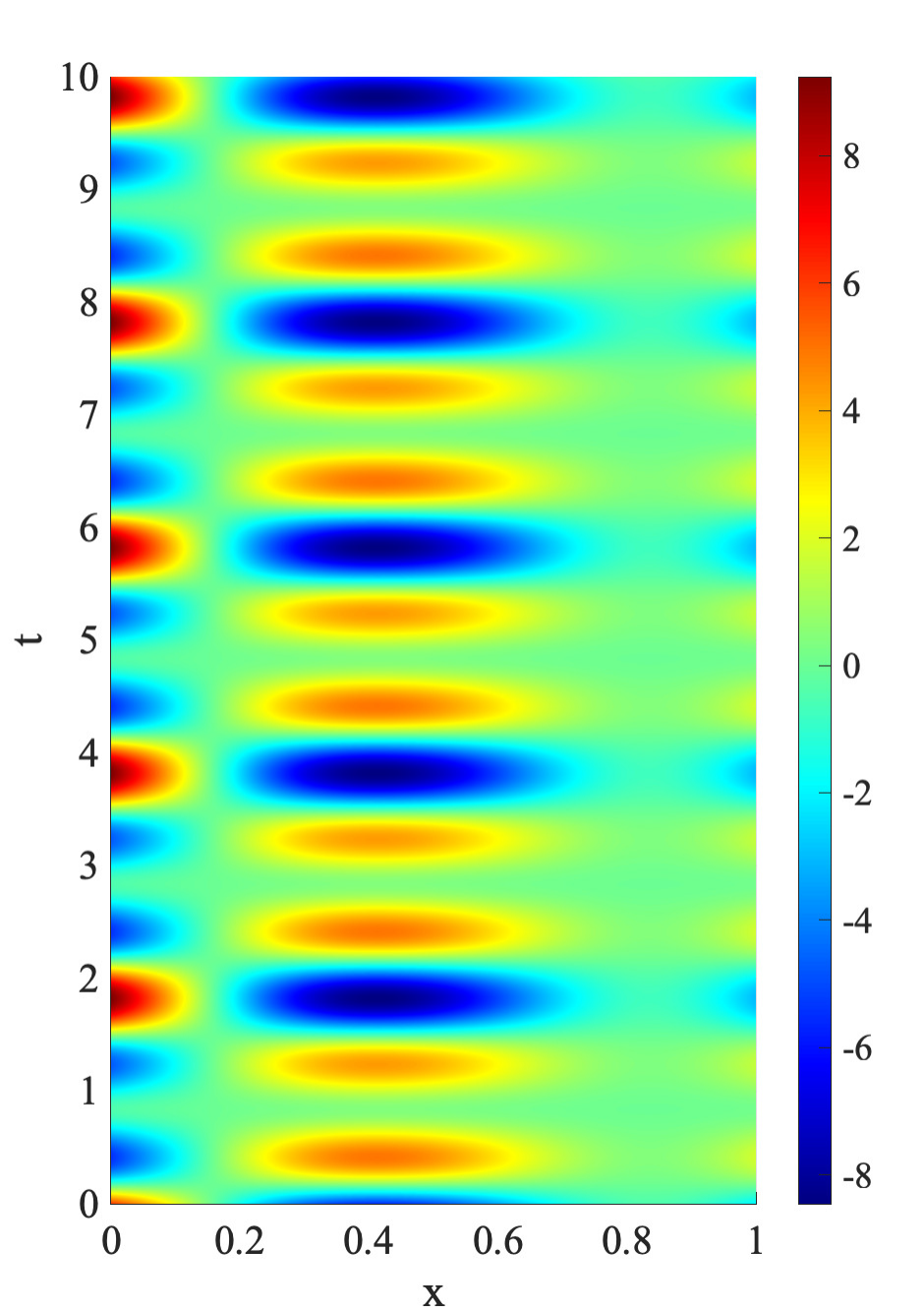}}
    \subfloat[$u_\theta$]{\includegraphics[width=0.2\linewidth]{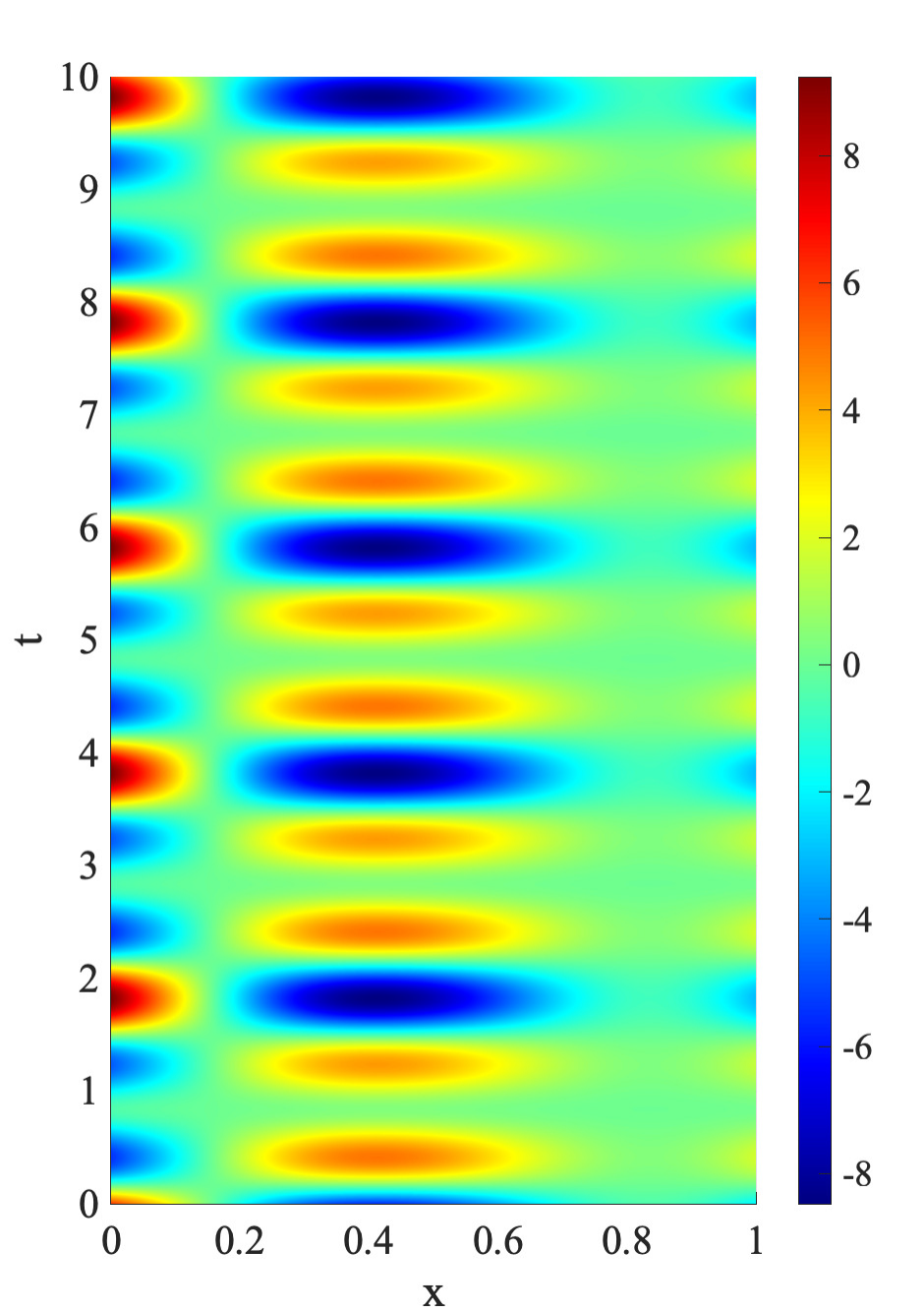}}
    \subfloat[$|u-u_\theta|$]{\includegraphics[width=0.2\linewidth]{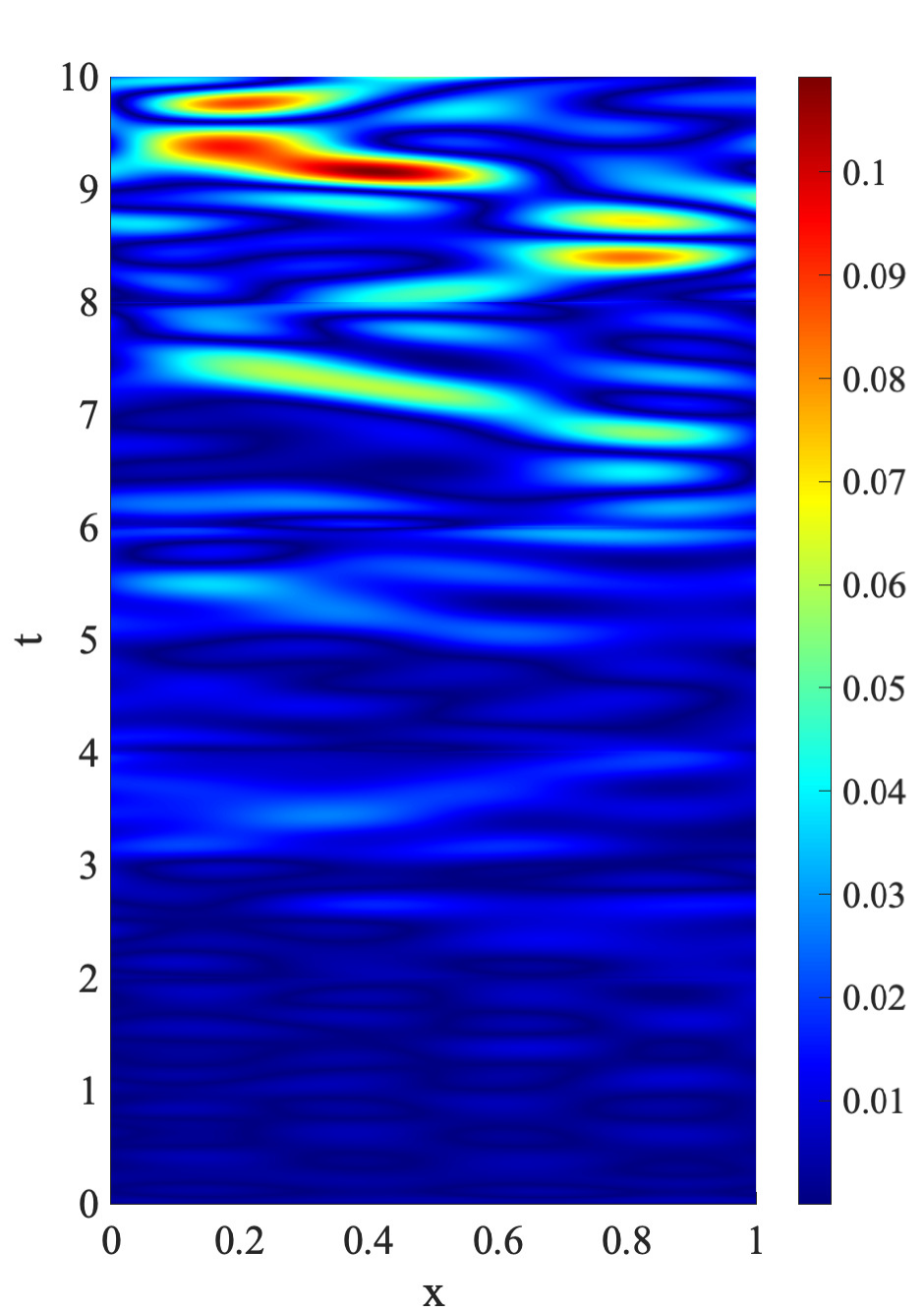}}
    \subfloat[$u^*_\theta$]{\includegraphics[width=0.2\linewidth]{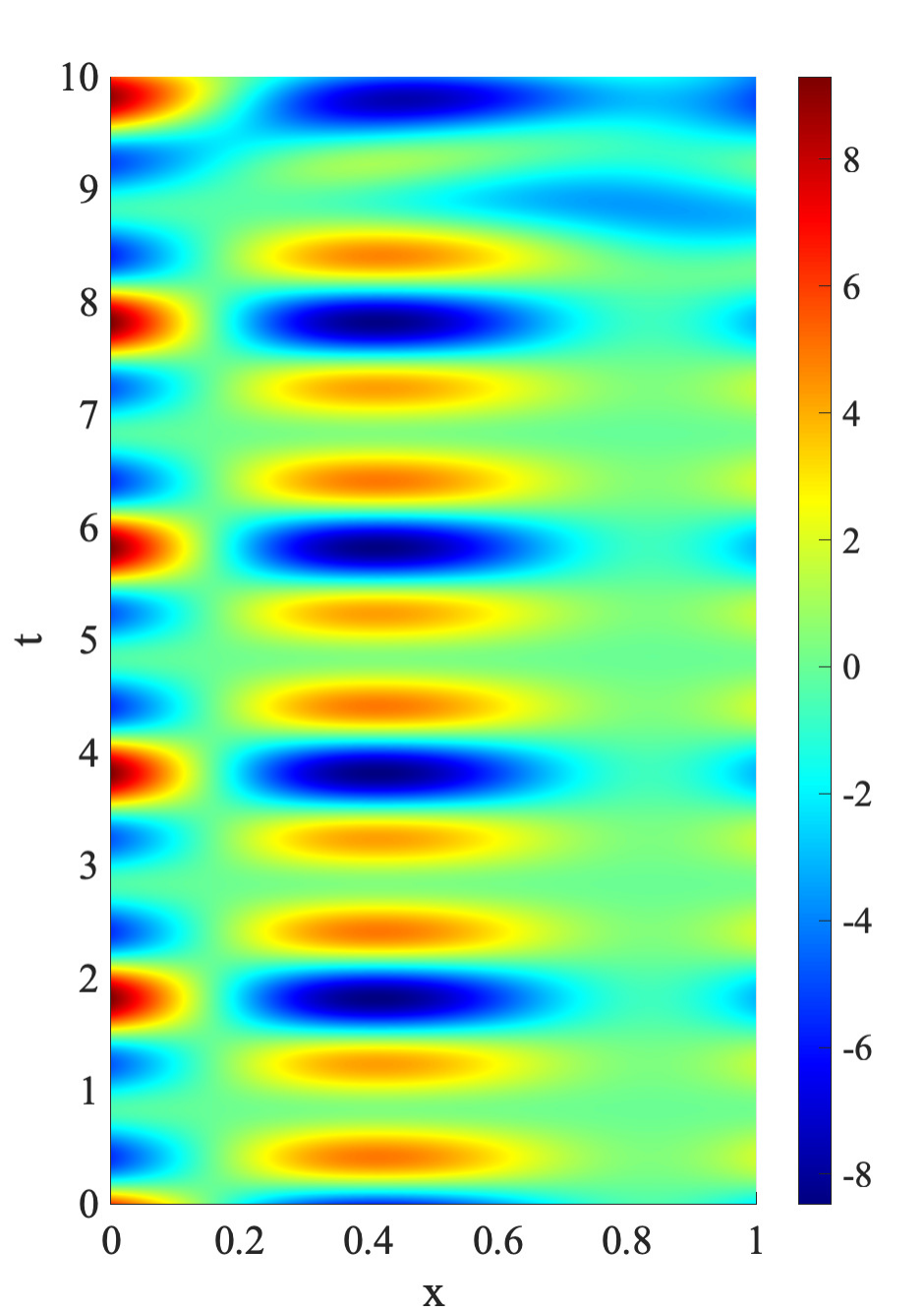}}
    \subfloat[$|u - u^*_\theta|$]{\includegraphics[width=0.2\linewidth]{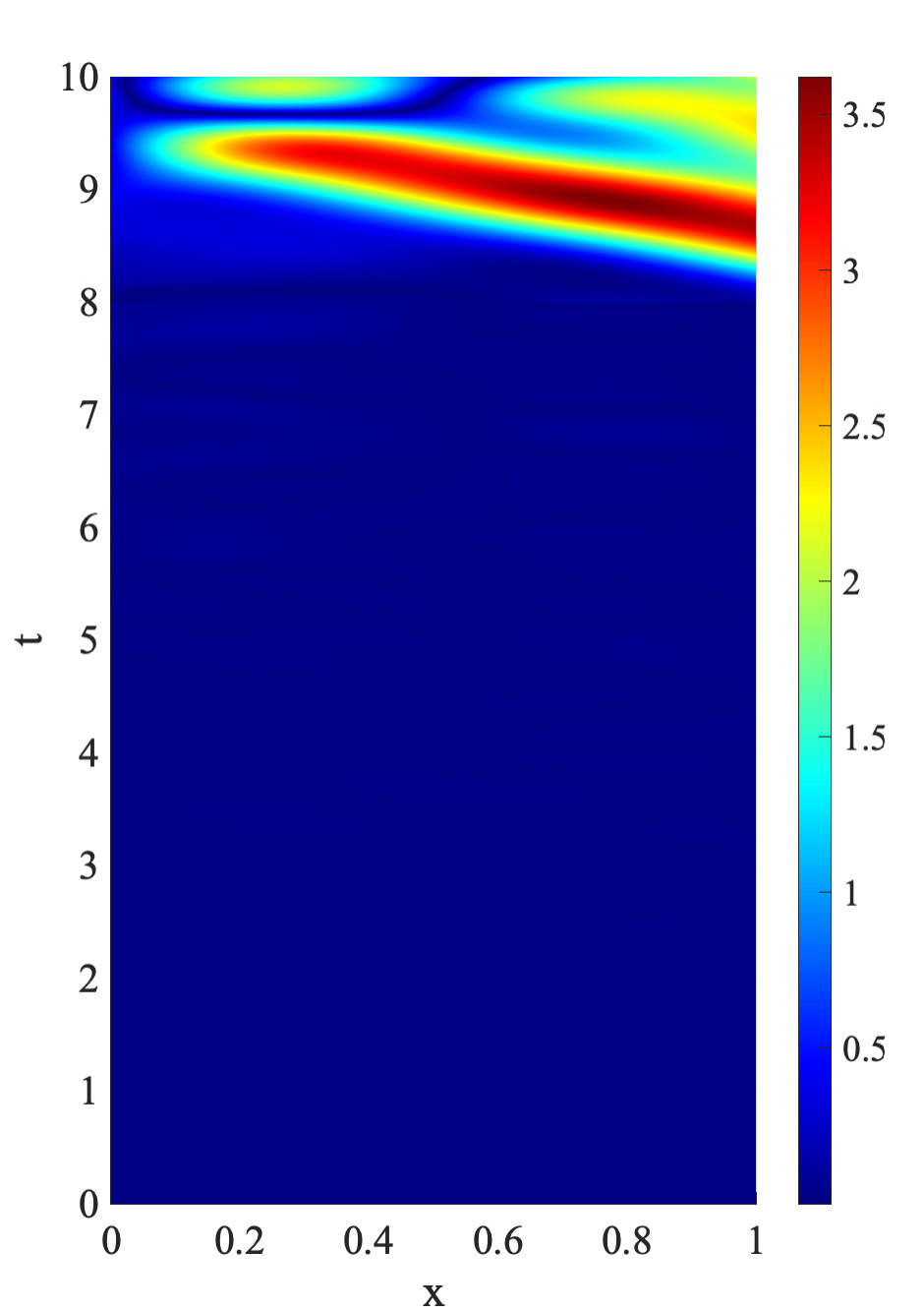}} \\
    \subfloat[$v$]{\includegraphics[width=0.2\linewidth]{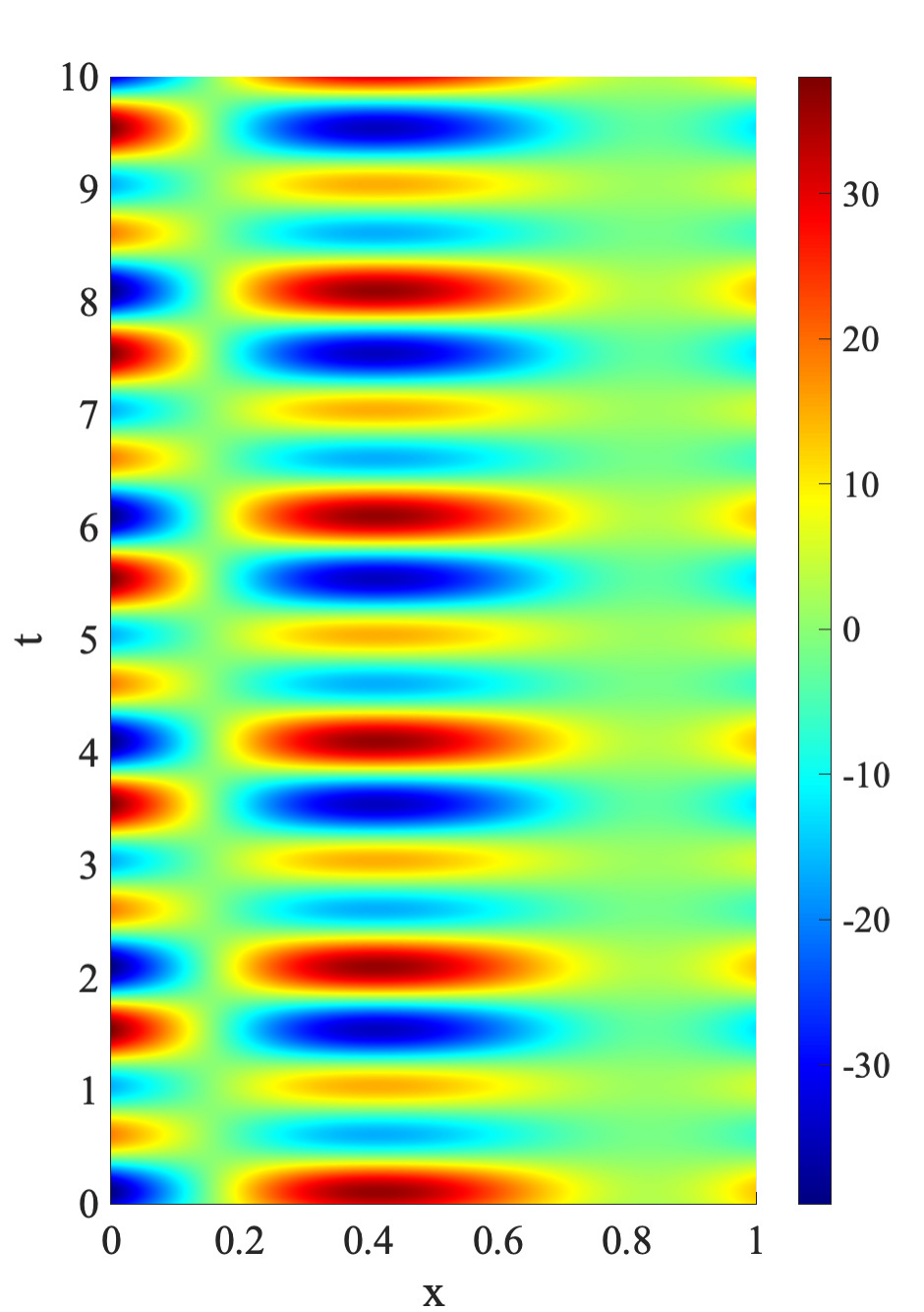}}
    \subfloat[$v_\theta$]{\includegraphics[width=0.2\linewidth]{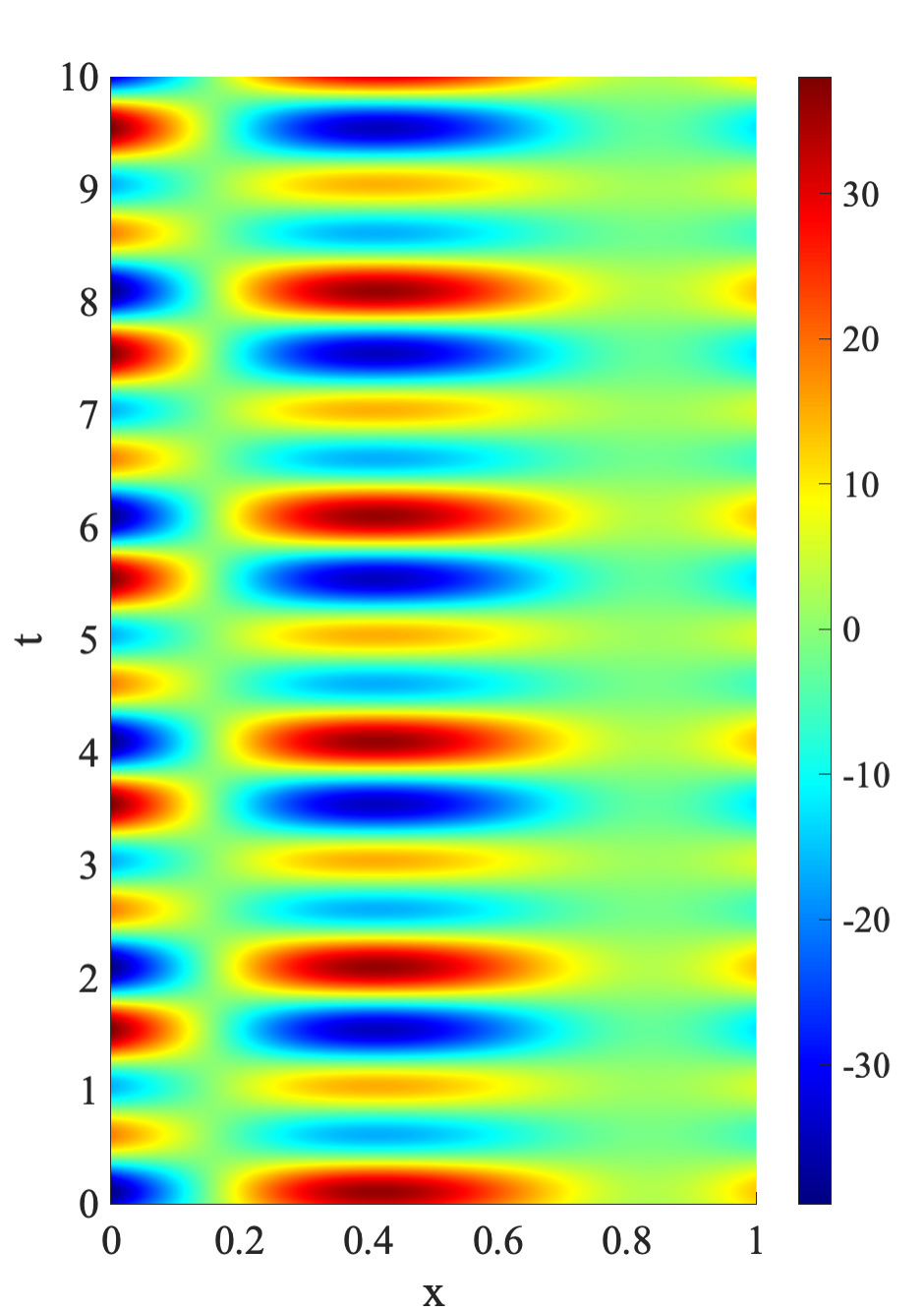}}
    \subfloat[$|v-v_\theta|$]{\includegraphics[width=0.2\linewidth]{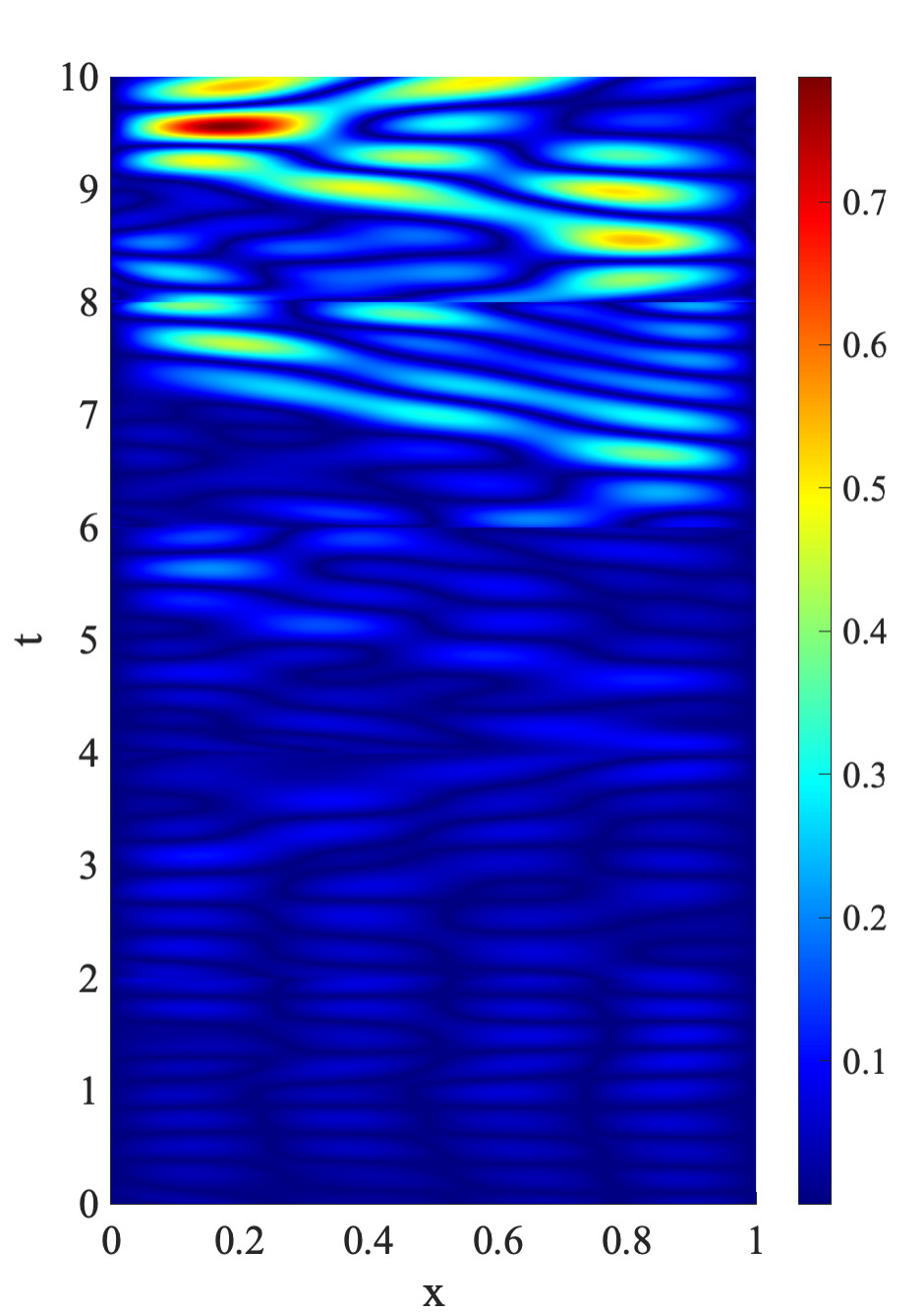}}
    \subfloat[$v^*_\theta$]{\includegraphics[width=0.2\linewidth]{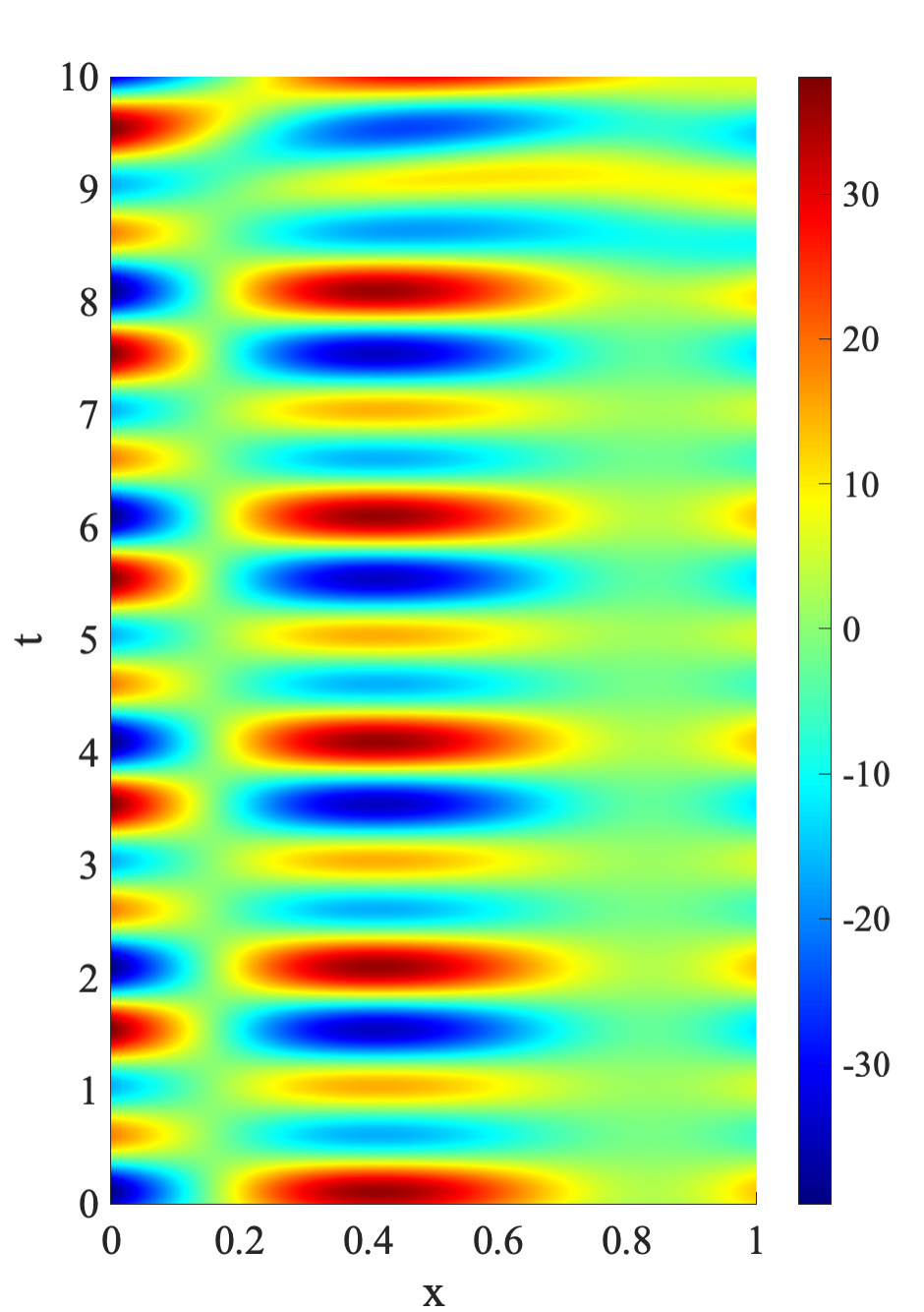}}
    \subfloat[$|v - v^*_\theta|$]{\includegraphics[width=0.2\linewidth]{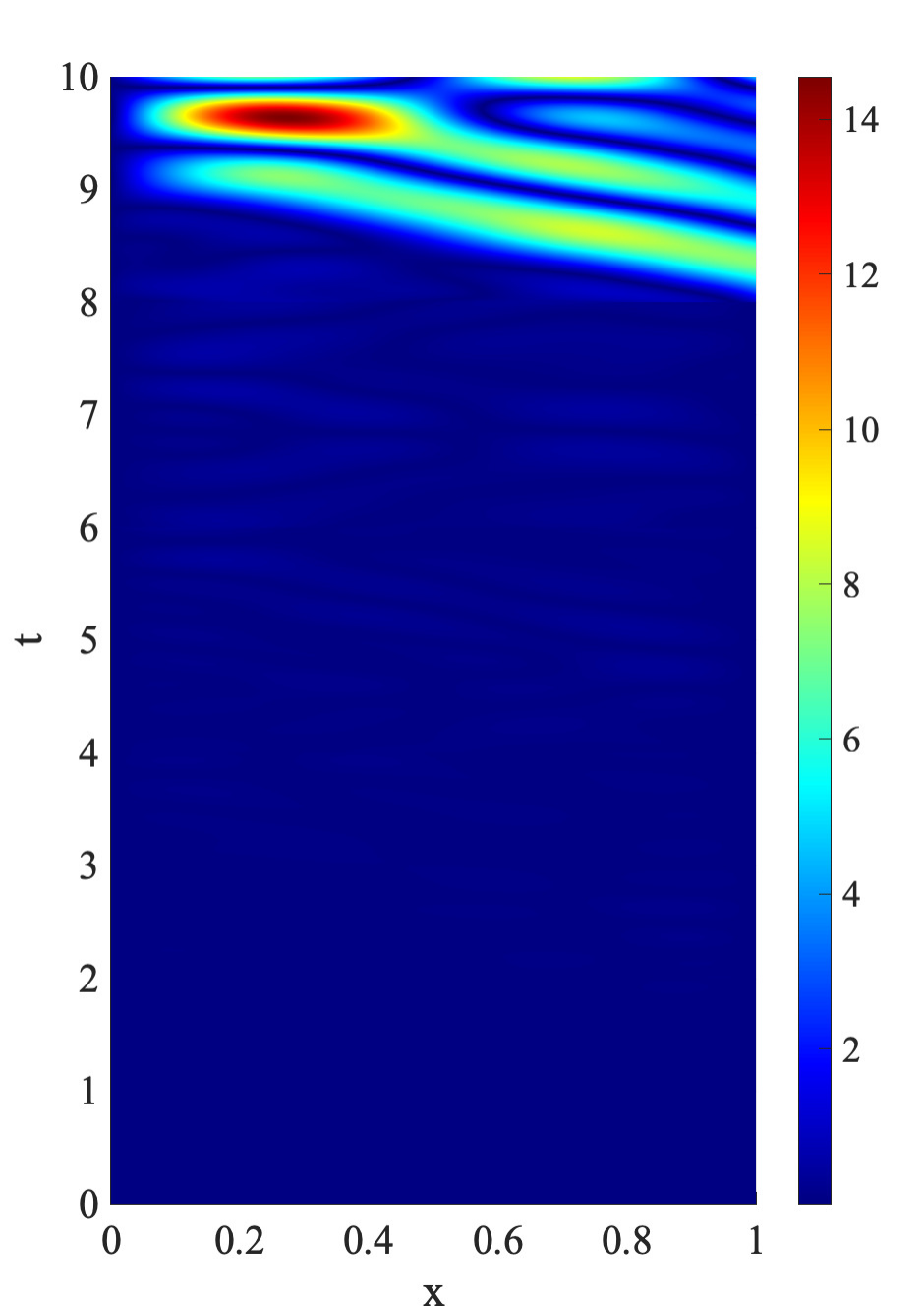}}
	\caption{Nonlinear Klein-Gordon equation: Distributions of the exact solution (a,f), the HLConcPINN-ExBTM solution and its point-wise error (b,g and c,h), and the HLConcPINN-BTM solution and its point-wise error (d,i and e,j), for $u$ (top row) and $v$ (bottom row). 
 NN: [2,90,90,10,2], with $\tanh$ activation function for the first two hidden layers and sine activation function in the last hidden layer; $N_c=2500$ for the training collocation points.
 }
	\label{PINN_num_KG_fig1}
\end{figure}

\begin{figure}[tb]
	\centering
	\subfloat[HLConcPINN-ExBTM]{\includegraphics[width=0.35\linewidth]{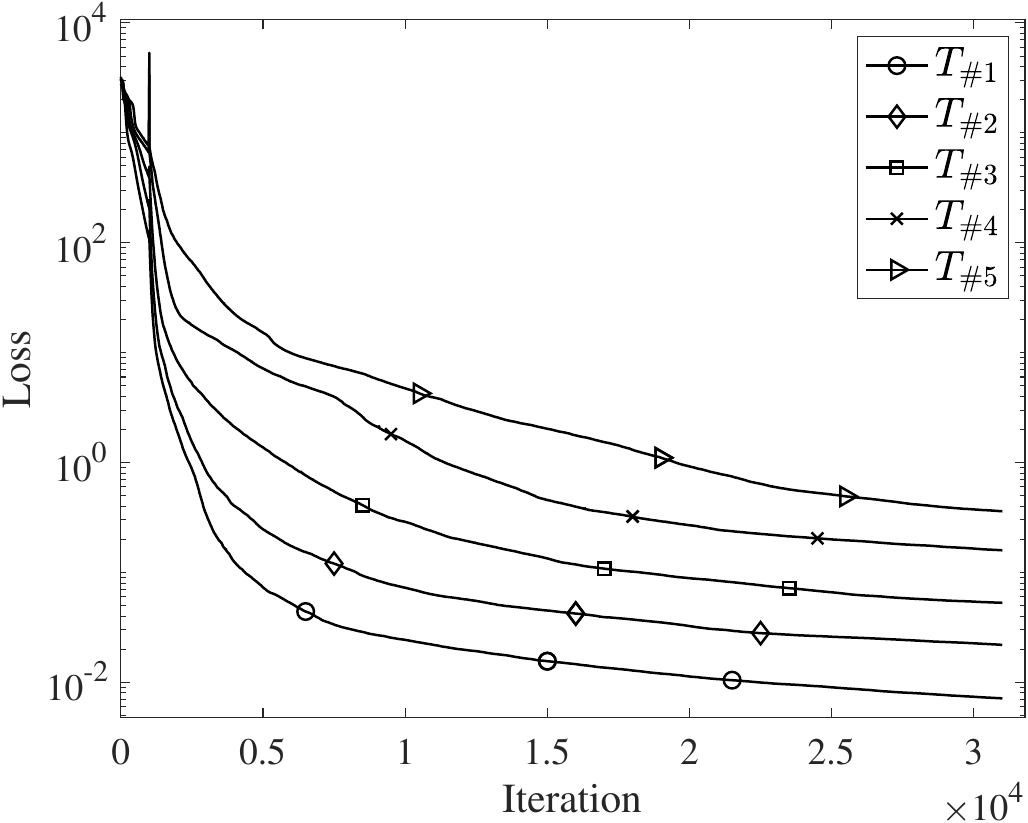}} \qquad\qquad
	\subfloat[HLConcPINN-BTM]{\includegraphics[width=0.35\linewidth]{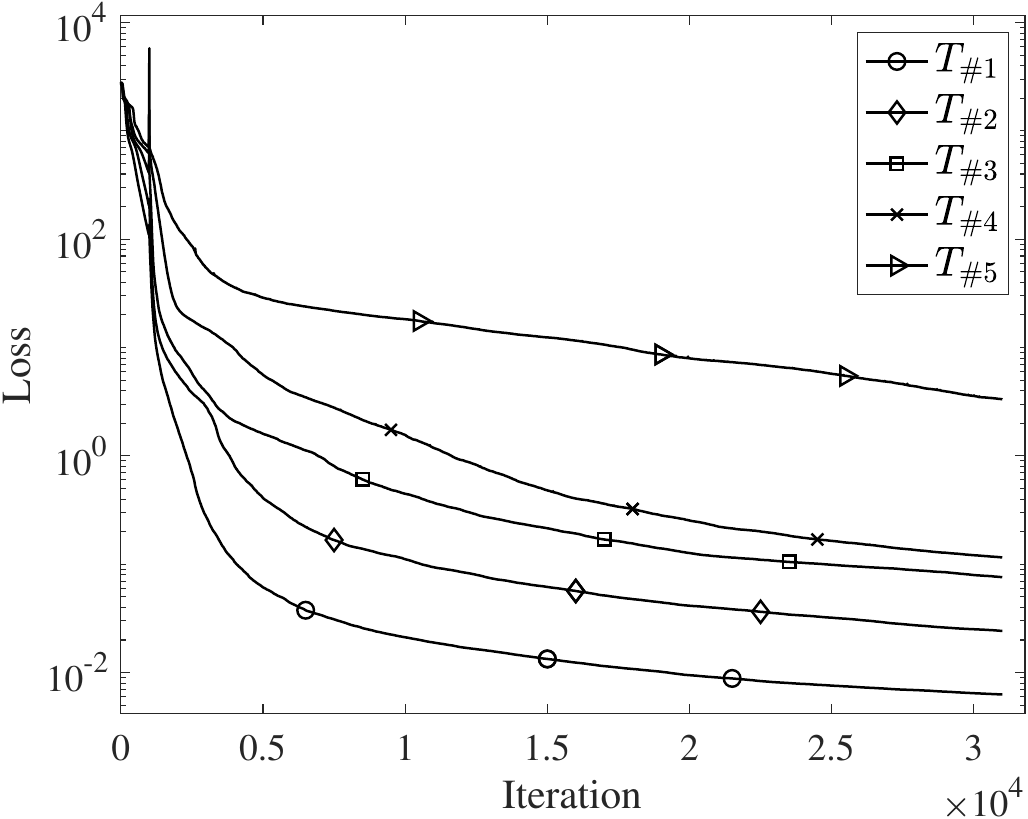}}\hspace{0.1em}
	\caption{Nonlinear Klein-Gordon equation: Histories of training loss for (a) HLConcPINN-ExBTM and (b) HLConcPINN-BTM in different time blocks. NN architecture and simulation parameters follow those of Figure~\ref{PINN_num_KG_fig1}.
 }
	\label{PINN_num_KG_fig4_1}
\end{figure}

\begin{figure}[tb]
	\centering
	\subfloat[$u$]{\includegraphics[width=0.35\linewidth]{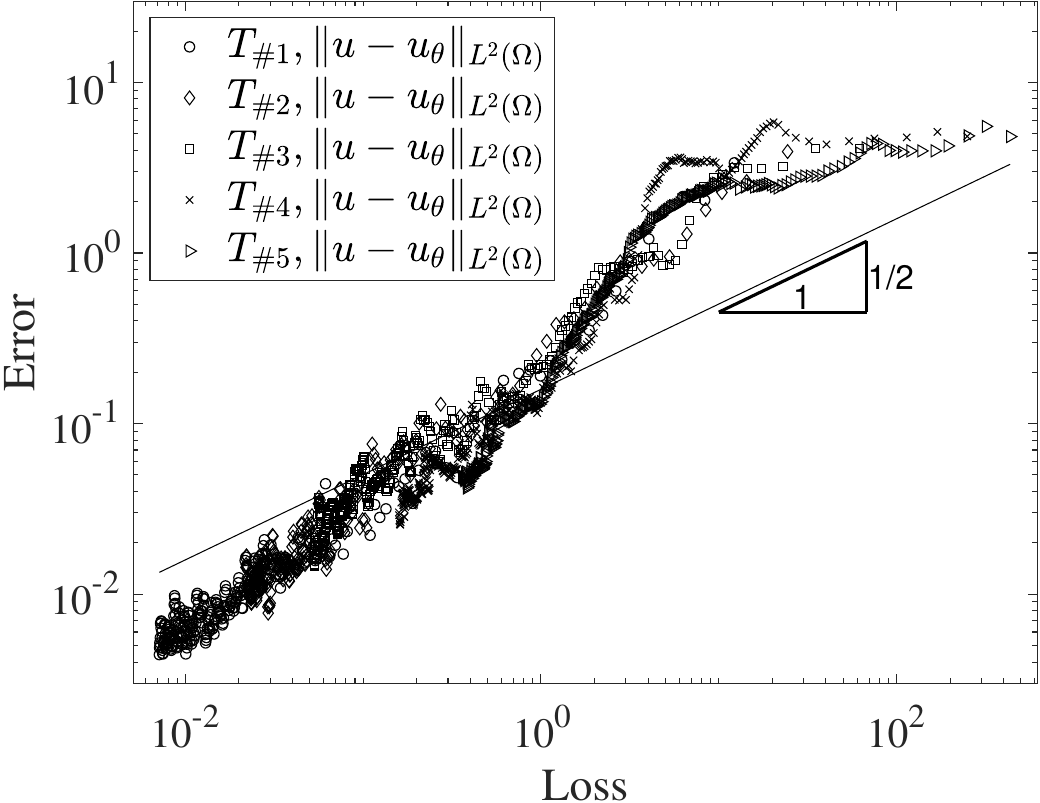}} \qquad\qquad
    \subfloat[$v=\partial u/\partial t$]{\includegraphics[width=0.35\linewidth]{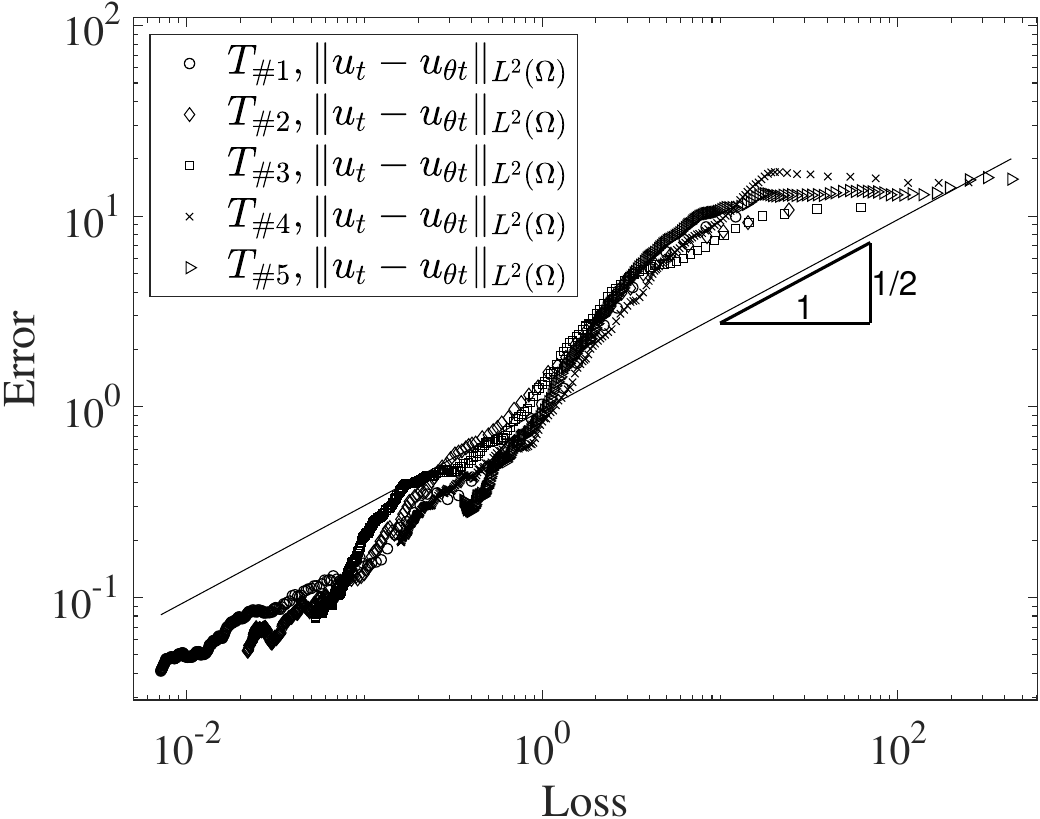}}
    \hspace{0.5em}
	\caption{Nonlinear Klein-Gordon equation: The $l^2$ errors of (a) $u$ and (b) $v$ as a function of the training loss value for the HLConcPINN-ExBTM method. The NN architecture and simulation parameters follow those of Figure~\ref{PINN_num_KG_fig1}.
 }
	\label{PINN_num_KG_fig5_1}
\end{figure}

\subsection{Nonlinear Klein-Gordon Equation}

We consider the spatial-temporal domain
$(x,t)\in\Omega= D\times [0, T] = [0, 1] \times [0, 10]$, and the following initial/boundary value problem  on this domain,
\begin{subequations}\label{num_SG_eq0}
\begin{align}
	& \frac{\partial^2u}{\partial t^2} - \frac{\partial^2 u}{\partial x^2} + u + \sin(u) = f(x,t), \label{eq_70a} \\
 &u({0},t)=\phi_1(t),  \quad u({1},t)=\phi_2(t), \quad
	u({x},0)=\psi_{1}({x}),\quad \frac{\partial u}{\partial t}({x},0)=\psi_{2}({x}).
\end{align}
\end{subequations}
In these equations, $ u(x,t) $ is the field function to be solved for, $ f(x,t) $ is a source term, $ \psi_{1} $ and $ \psi_2 $ are the initial conditions, and $ \phi_1 $ and $ \phi_2 $ are the boundary conditions.
Note that a nonlinear term, $g(u)=\sin u$, has been used, leading to the Sine-Gordon equation in~\eqref{eq_70a}. The source term, initial and boundary conditions are appropriately  chosen such that the problem has the following exact solution,  
\begin{align}\label{num_SG_eq1}
	u(x, t) = \left[2\cos\left(\pi x + \frac{\pi}{5}\right) + \frac{9}{5}\cos\left(2\pi x + \frac{7\pi}{20}\right)\right]\left[ 2\cos\left(\pi t + \frac{\pi}{5}\right) + \frac{9}{5}\cos\left(2\pi t + \frac{7\pi}{20}\right) \right].
\end{align}


To simulate this problem, we reformulate it  as follows,
\begin{subequations}\label{num_SG_eq2}
\begin{align}
    \label{num_SG_eq2a}
	&\frac{\partial u}{\partial t} - v = 0, \quad \frac{\partial v}{\partial t} - \frac{\partial^2 u}{\partial x^2} + u + \sin(u) = f(x,t), \\
    \label{num_SG_eq2b}
    &u({0},t)=\phi_1(t),  \quad u({1},t)=\phi_2(t),\quad
	u({x},0)=\psi_{1}({x}),\quad v({x},0)=\psi_{2}({x}),
\end{align}
\end{subequations}
where $v$ is defined by equation \eqref{num_SG_eq2a}.
In light of \eqref{SG_T}$-$\eqref{SG_Ti}, we set the loss function for HLConcPINN-ExBTM as follows,
\begin{align}\label{num_SG_eq3}
    Loss_i^{I}=&\frac{W_1}{N_c} \sum_{n=1}^{N_c}\left[\frac{\partial u_{\theta_i}}{\partial t}(x_{int}^n, t_{int}^n) - v_{\theta_i}(x_{int}^n, t_{int}^n)\right]^2 \notag\\
    &+ \frac{W_2}{N_c} \sum_{n=1}^{N_c} \left[\frac{\partial v_{\theta_i}}{\partial t}(x_{int}^n, t_{int}^n) - \frac{\partial^2 u_{\theta_i}}{\partial x^2}(x_{int}^n, t_{int}^n) + u_{\theta_i}(x_{int}^n, t_{int}^n) + \sin(u_{\theta_i}(x_{int}^n, t_{int}^n)) - f(x_{int}^n, t_{int}^n) \right]^2 \nonumber \\
    &+ \frac{W_3}{N_c} \sum_{n=1}^{N_c}\left[\frac{\partial^2u_{\theta_i}}{\partial t\partial x}(x_{int}^n, t_{int}^n) - \frac{\partial v_{\theta_i}}{\partial x}(x_{int}^n, t_{int}^n)\right]^2 + \frac{W_4}{N_c} \sum_{j=1}^{i} \sum_{n=1}^{N_c}\left[ u_{\theta_i}(x_{tb}^n, t_{j-1}) - u_{\theta_{j-1}}(x_{tb}^n, t_{j-1})\right]^2 \nonumber \\
    &+ \frac{W_5}{N_c} \sum_{j=1}^{i} \sum_{n=1}^{N_c} \left[v_{\theta_i}(x_{tb}^n, t_{j-1})-v_{\theta_{j-1}}(x_{tb}^n, t_{j-1}) \right]^2 +  \frac{W_6}{N_c} \sum_{j=1}^{i}\sum_{n=1}^{N_c}\left[\frac{\partial u_{\theta_i}}{\partial x}(x_{tb}^n, t_{j-1}) - \frac{\partial u_{\theta_{j-1}}}{\partial x}(x_{tb}^n, t_{j-1}) \right]^2 \nonumber \\ 
    & + W_7 \Big( \frac{1}{N_c} \sum_{n=1}^{N_c} \left[ (v_{\theta}(0, t_{sb}^n) - \frac{\partial \phi_1}{\partial t}(t_{sb}^n) )^2 +  (v_{\theta}(1, t_{sb}^n) - \frac{\partial \phi_2}{\partial t}(t_{sb}^n) )^2   \right] \Big)^{1/2} + Loss_{i-1}^I,
\end{align}
where $W_i$ ($i=1,\dots,7$) are the penalty coefficients for different loss terms.
The loss function for HLConcPINN-BTM is set to, 
\begin{align}\label{num_SG_eq3_1}
    Loss_i^{II}=&\frac{W_1}{N_c} \sum_{n=1}^{N_c}\left[\frac{\partial u_{\theta_i}}{\partial t}(x_{int}^n, t_{int}^n) - v_{\theta_i}(x_{int}^n, t_{int}^n)\right]^2 \notag\\
    &+ \frac{W_2}{N_c} \sum_{n=1}^{N_c} \left[\frac{\partial v_{\theta_i}}{\partial t}(x_{int}^n, t_{int}^n) - \frac{\partial^2 u_{\theta_i}}{\partial x^2}(x_{int}^n, t_{int}^n) + u_{\theta_i}(x_{int}^n, t_{int}^n) + \sin(u_{\theta_i}(x_{int}^n, t_{int}^n)) - f(x_{int}^n, t_{int}^n) \right]^2 \nonumber \\
    &+ \frac{W_3}{N_c} \sum_{n=1}^{N_c}\left[\frac{\partial^2u_{\theta_i}}{\partial t\partial x}(x_{int}^n, t_{int}^n) - \frac{\partial v_{\theta_i}}{\partial x}(x_{int}^n, t_{int}^n)\right]^2 + \frac{W_4}{N_c} \sum_{n=1}^{N_c}\left[ u_{\theta_i}(x_{tb}^n, t_{i-1}) - u_{\theta_{i-1}}(x_{tb}^n, t_{i-1})\right]^2 \nonumber \\
    &+ \frac{W_5}{N_c} \sum_{n=1}^{N_c} \left[v_{\theta_i}(x_{tb}^n, t_{i-1})-v_{\theta_{i-1}}(x_{tb}^n, t_{i-1}) \right]^2 +  \frac{W_6}{N_c}\sum_{n=1}^{N_c}\left[\frac{\partial u_{\theta_i}}{\partial x}(x_{tb}^n, t_{i-1}) - \frac{\partial u_{\theta_{i-1}}}{\partial x}(x_{tb}^n, t_{i-1}) \right]^2 \nonumber \\ 
    & + W_7 \Big( \frac{1}{N_c} \sum_{n=1}^{N_c} \left[ (v_{\theta}(0, t_{sb}^n) - \frac{\partial \phi_1}{\partial t}(t_{sb}^n) )^2 +  (v_{\theta}(1, t_{sb}^n) - \frac{\partial \phi_2}{\partial t}(t_{sb}^n) )^2   \right] \Big)^{1/2}.
\end{align}
We employ the following values for the penalty coefficients, $(W_1,\dots, W_7) = (0.4, 0.4, 0.4, 0.6, 0.6, 0.6, 0.6)$, for this problem. Five uniform time blocks are used in block time marching.


\begin{table}[tb]\small
	\centering\small
\begin{tabular}{c|@{}c@{}|cc|cc|cc|cc}
\hline
\multirow{2}{*}{Error}    & \multirow{2}{*}{Time block} & \multicolumn{2}{c|}{$N_c=1500$} & \multicolumn{2}{c|}{$N_c=2000$} & \multicolumn{2}{c|}{$N_c=2500$} & \multicolumn{2}{c}{$N_c=3000$} \\ \cline{3-10}
                      &  & ExBTM  & BTM  & ExBTM  & BTM   & ExBTM  & BTM  & ExBTM  & BTM   \\ \hline
\multirow{5}{*}{$l^2$} & $T_{\#1}$  & 1.14e-03 & 1.52e-03 & 1.37e-03 & 1.64e-03 & 1.21e-03 & 1.52e-03 & 1.88e-03 & 1.56e-03 \\ \cline{3-10} 
                      & $T_{\#2}$  & 3.08e-03 & 3.99e-03 & 3.28e-03 & 2.97e-03 & 3.94e-03 & 3.53e-03 & 6.14e-03 & 2.55e-03 \\ \cline{3-10}
                      & $T_{\#3}$  & 5.44e-03 & 9.56e-03 & 7.79e-03 & 6.89e-03 & 4.89e-03 & 6.96e-03 & 6.61e-03 & 7.75e-03 \\ \cline{3-10} 
                      & $T_{\#4}$  & 9.69e-03 & 1.67e-02 & 1.90e-02 & 7.19e-03 & 8.31e-03 & 1.44e-02 & 1.01e-02 & 1.36e-02 \\ \cline{3-10} 
                      & $T_{\#5}$  & 7.03e-02 & 9.11e-02 & 3.23e-02 & 1.95e-02 & 1.33e-02 & 6.49e-01 & 2.93e-02 & 7.43e-02 \\ \hline
\multirow{5}{*}{$l^\infty$} & $T_{\#1}$  & 3.33e-03 & 3.24e-03 & 5.01e-03 & 4.75e-03 & 3.62e-03 & 4.21e-03 & 4.40e-03 & 4.54e-03 \\ \cline{3-10} 
                      & $T_{\#2}$  & 8.52e-03 & 8.47e-03 & 9.59e-03 & 1.07e-02 & 1.06e-02 & 9.18e-03 & 1.27e-02 & 7.50e-03 \\ \cline{3-10}
                      & $T_{\#3}$  & 1.93e-02 & 2.74e-02 & 1.77e-02 & 1.78e-02 & 1.44e-02 & 2.25e-02 & 1.89e-02 & 2.10e-02 \\ \cline{3-10} 
                      & $T_{\#4}$  & 2.34e-02 & 4.90e-02 & 5.88e-02 & 1.57e-02 & 2.38e-02 & 4.62e-02 & 2.76e-02 & 3.37e-02 \\ \cline{3-10} 
                      & $T_{\#5}$  & 1.89e-01 & 2.40e-01 & 9.73e-02 & 4.58e-02 & 4.25e-02 & 1.41e+00 & 8.66e-02 & 2.03e-01 \\ \hline
\end{tabular}
\caption{Nonlinear Klein-Gordon equation:  $l^2$ and $l^\infty$ errors of $u$ for HLConcPINN-ExBTM and HLConcPINN-BTM obtained with different training collocation points $N_c$. The NN architecture and activation functions follow those of Figure~\ref{PINN_num_KG_fig1}.
}
	\label{tab_KG_err_1}
\end{table}

\begin{table}[tb]\small
	\centering\small
\begin{tabular}{c|@{}c@{}|cc|cc|cc|cc}
\hline
\multirow{2}{*}{Error}    & \multirow{2}{*}{Time block} & \multicolumn{2}{c|}{tanh} & \multicolumn{2}{c|}{Gaussian} & \multicolumn{2}{c|}{swish} & \multicolumn{2}{c}{softplus} \\ \cline{3-10}
                      &  & ExBTM  & BTM  & ExBTM  & BTM   & ExBTM  & BTM  & ExBTM  & BTM   \\ \hline
\multirow{5}{*}{$l^2$} & $T_{\#1}$ & 2.28e-03 & 2.86e-03 & 2.52e-03 & 3.41e-03 & 9.61e-04 & 1.98e-03 & 1.81e-03 & 1.40e-03 \\ \cline{3-10} 
                      & $T_{\#2}$  & 5.24e-03 & 5.83e-03 & 1.82e-03 & 6.29e-03 & 4.54e-03 & 3.62e-03 & 6.21e-03 & 7.67e-03 \\ \cline{3-10}
                      & $T_{\#3}$  & 8.22e-03 & 1.68e-02 & 1.11e-02 & 1.97e-02 & 7.75e-03 & 5.16e-03 & 8.18e-03 & 1.01e-02 \\ \cline{3-10} 
                      & $T_{\#4}$  & 1.73e-02 & 2.52e-02 & 3.25e-01 & 2.86e-01 & 1.13e-02 & 1.17e-02 & 1.72e-02 & 8.91e-03 \\ \cline{3-10} 
                      & $T_{\#5}$  & 1.33e-01 & 8.68e-02 & 5.56e-01 & 7.53e-01 & 1.11e-01 & 1.82e-02 & 2.98e-02 & 1.43e-02 \\ \hline
\multirow{5}{*}{$l^\infty$} & $T_{\#1}$  & 5.73e-03 & 8.85e-03 & 5.89e-03 & 8.22e-03 & 3.30e-03 & 5.00e-03 & 5.54e-03 & 4.63e-03 \\ \cline{3-10} 
                      & $T_{\#2}$  & 1.64e-02 & 1.74e-02 & 5.34e-03 & 1.94e-02 & 1.15e-02 & 1.61e-02 & 1.99e-02 & 2.67e-02 \\ \cline{3-10}
                      & $T_{\#3}$  & 2.24e-02 & 4.67e-02 & 2.87e-02 & 4.05e-02 & 1.49e-02 & 1.60e-02 & 2.20e-02 & 2.91e-02 \\ \cline{3-10} 
                      & $T_{\#4}$  & 4.27e-02 & 7.18e-02 & 1.05e+00 & 7.84e-01 & 2.53e-02 & 3.32e-02 & 4.77e-02 & 2.33e-02 \\ \cline{3-10} 
                      & $T_{\#5}$  & 3.65e-01 & 2.47e-01 & 1.43e+00 & 1.80e+00 & 3.42e-01 & 6.28e-02 & 8.76e-02 & 4.46e-02 \\ \hline
\end{tabular}
\caption{Nonlinear Klein-Gordon equation:  $l^2$ and $l^\infty$ errors of $u$ for HLConcPINN-ExBTM and HLConcPINN-BTM obtained with different  activation functions for the last hidden layer.
NN: [2,90,90,10,2], with $\tanh$ activation function for the first two hidden layers and the activation function in the last hidden layer varied; $N_c=2500$ for the training collocation points.
}
	\label{tab_KG_err_3}
\end{table}

Figures~\ref{PINN_num_KG_fig1} and~\ref{PINN_num_KG_fig4_1} provide an overview of the simulation results obtained by HLConcPINN-ExBTM and HLConcPINN-BTM for the nonlinear Klein-Gordon equation. Here the distributions of the HLConcPINN-ExBTM and HLConcPINN-BTM solutions for $u$ and $v=\frac{\partial u}{\partial t}$, their point-wise absolute errors, as well as the exact solution field, have been shown. The loss histories for different time blocks obtained using these methods are shown in Figure~\ref{PINN_num_KG_fig4_1}. The network architecture (consisting of three hidden layers), the activation functions, and the training collocation points are given in the caption of Figure~\ref{PINN_num_KG_fig1}. The simulation results obtained with HLConcPINN-ExBTM are markedly more accurate than those of HLConcPINN-BTM for this problem, especially at later time (the last time block). It is also noted that the solution accuracy for $\frac{\partial u}{\partial t}$ is notably lower than that of $u$.




Table~\ref{tab_KG_err_1} summarizes a study of the training collocation points on the PINN solutions. We list the $l^2$ and $l^{\infty}$ errors of both HLConcPINN-ExBTM and HLConcPINN-BTM in different time blocks obtained with a range of training collocation points between $N_c=1500$ and $N_c=3000$. The neural network architecture and activation functions follow those of Figure~\ref{PINN_num_KG_fig1}. The results are in general not sensitive to the number of collocation points, similar to what has been obtained with other test problems in previous subsections.

Table~\ref{tab_KG_err_3} compares the simulation results of HLConcPINN-ExBTM and HLConcPINN-BTM obtained with different activation functions (tanh, Gaussian, swish, softplus) for the last hidden layer. Three hidden layers are employed in the neural network, with $\tanh$ activation for the first two hidden layers and the activation function of the last hidden layer varied. The  network architecture and other simulation parameters are specified in the table caption. These results can be compared with that of Table~\ref{tab_KG_err_1} corresponding to $N_c=2500$, where the sine activation function has been used for the last hidden layer. Among the activation functions tested, the sine function appears to yield the best simulation results.

Finally Figure \ref{PINN_num_KG_fig5_1} illustrates the relation between the errors for $u$ and $v$  and the training loss for the HLConcPINN-ExBTM method from our simulations. The simulation data signify a scaling with a power of approximately $1/2$, which is roughly consistent with the conclusion of Theorem \ref{sec6_Theorem3}. 


\comment{
\begin{figure}[tb]
	\centering
	\subfloat[$v$]{\includegraphics[width=0.22\linewidth]{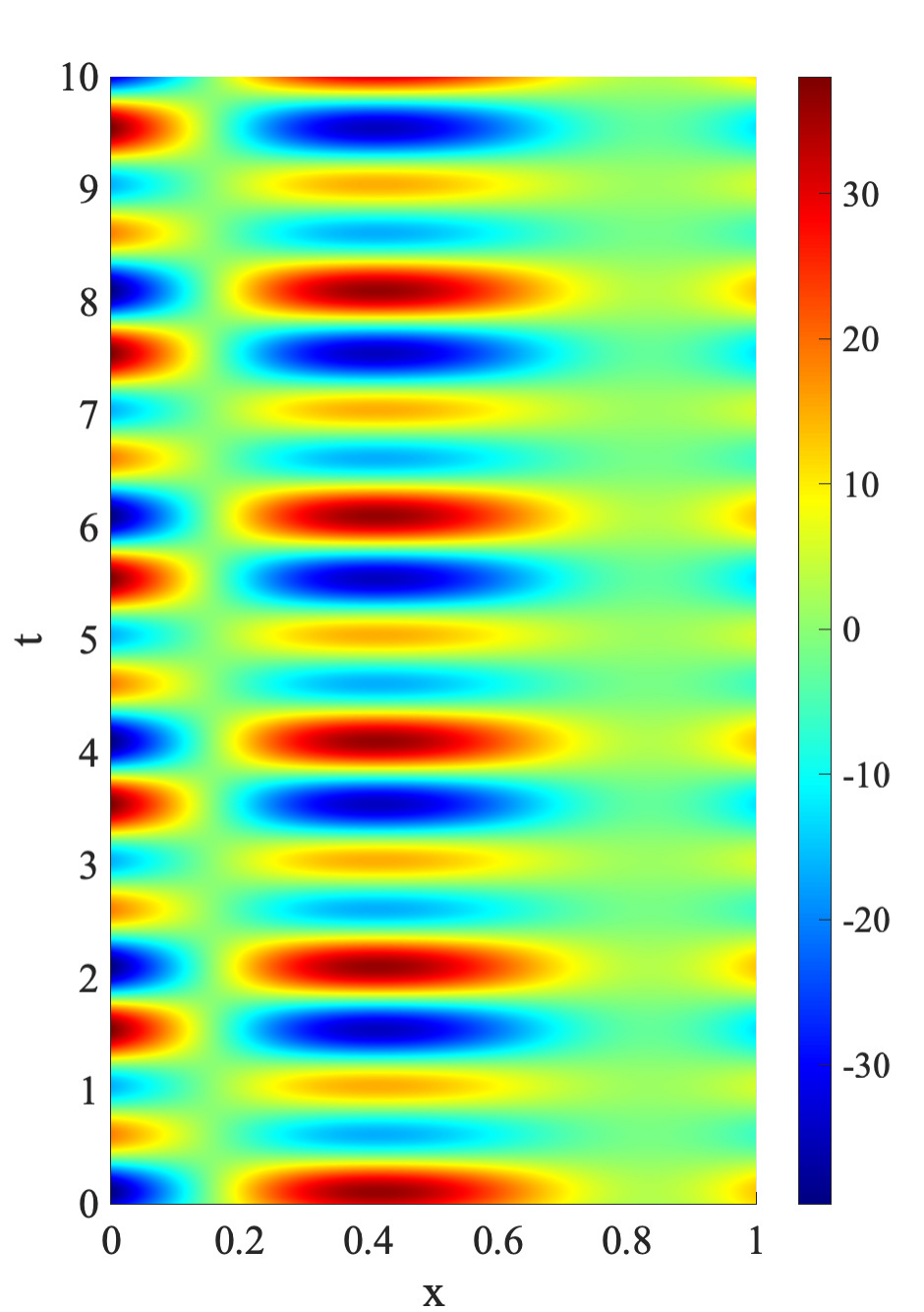}}
    \subfloat[$v_\theta$]{\includegraphics[width=0.22\linewidth]{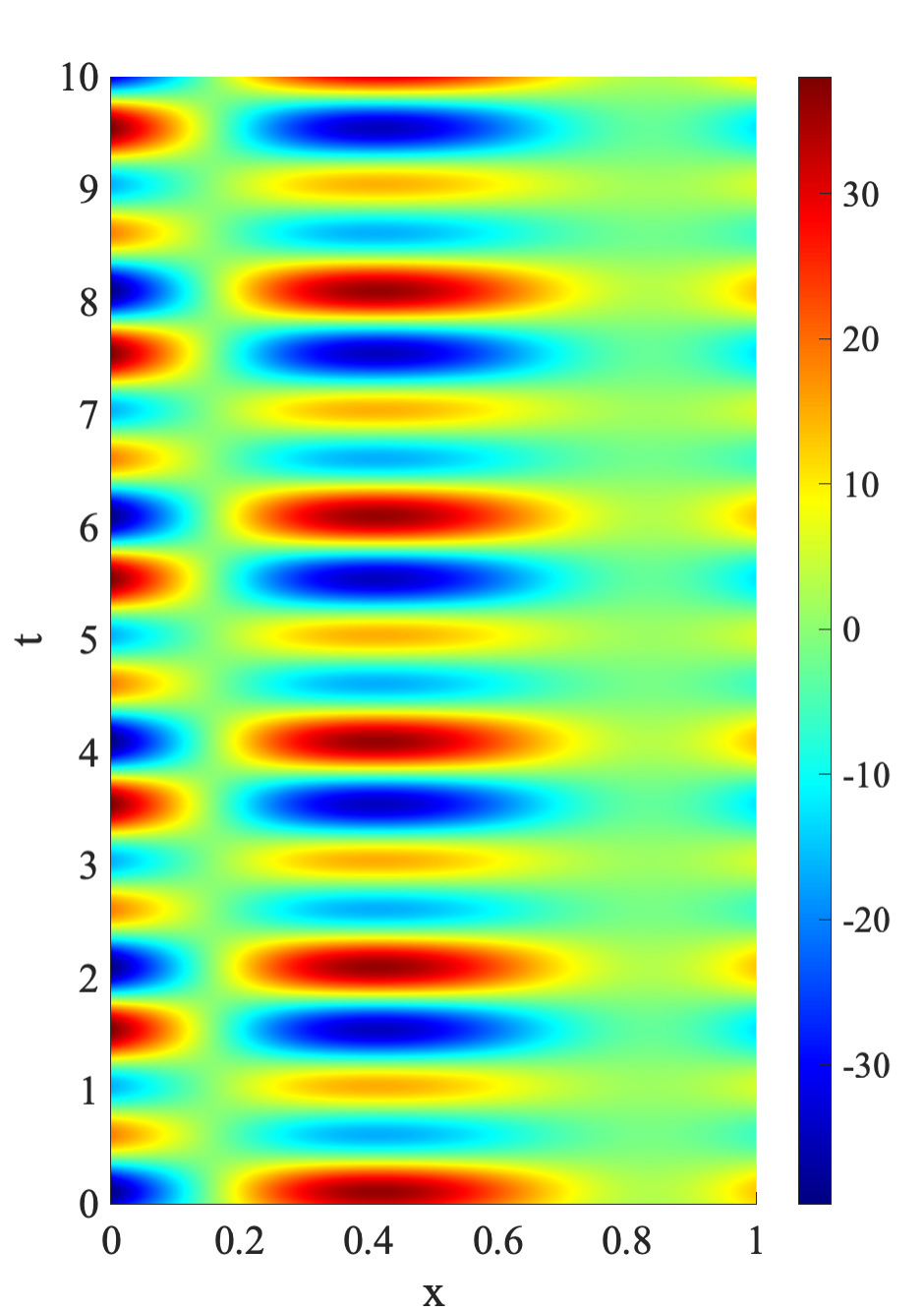}}
    \subfloat[$|v-v_\theta|$]{\includegraphics[width=0.22\linewidth]{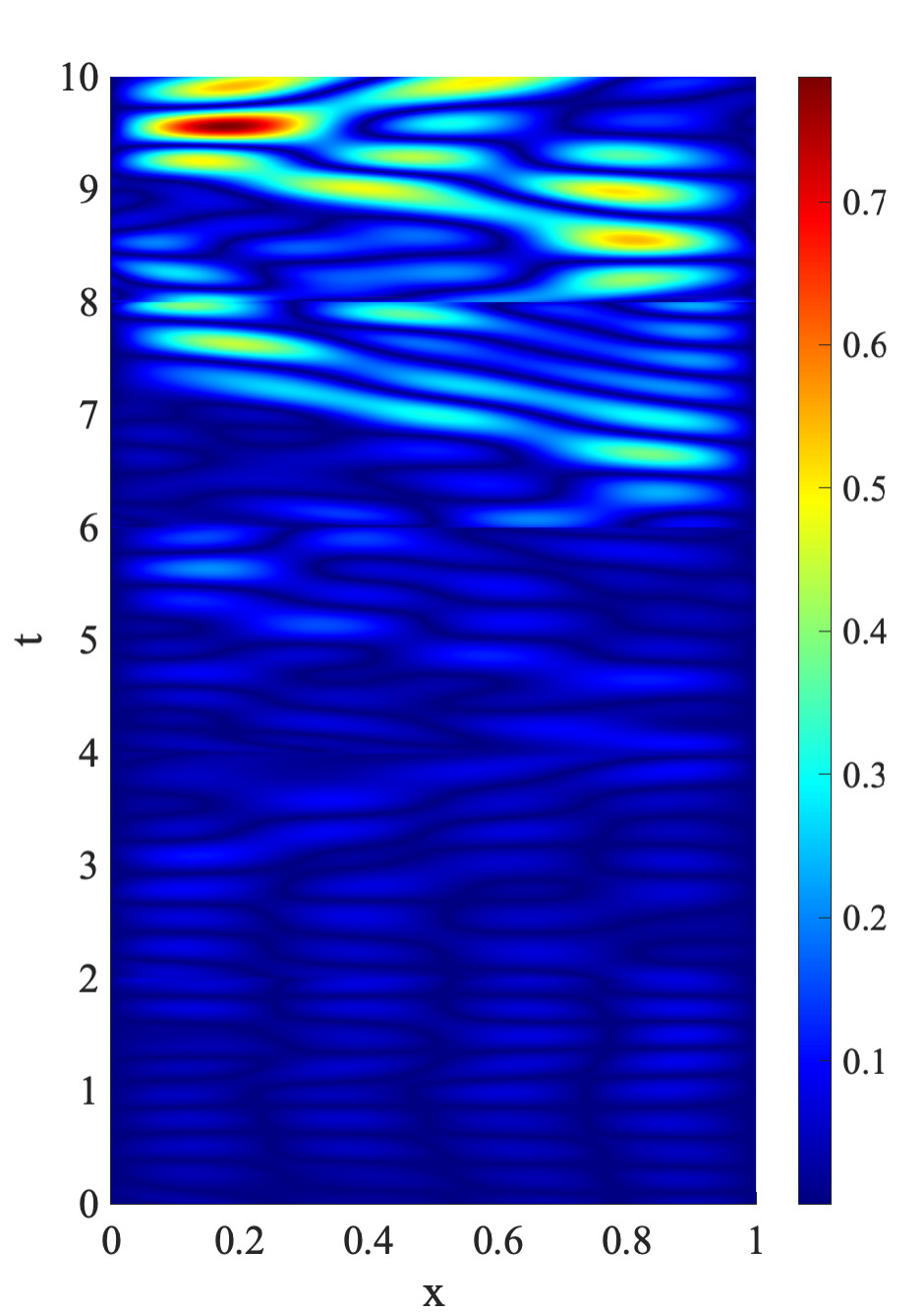}}
    \subfloat[$v^*_\theta$]{\includegraphics[width=0.22\linewidth]{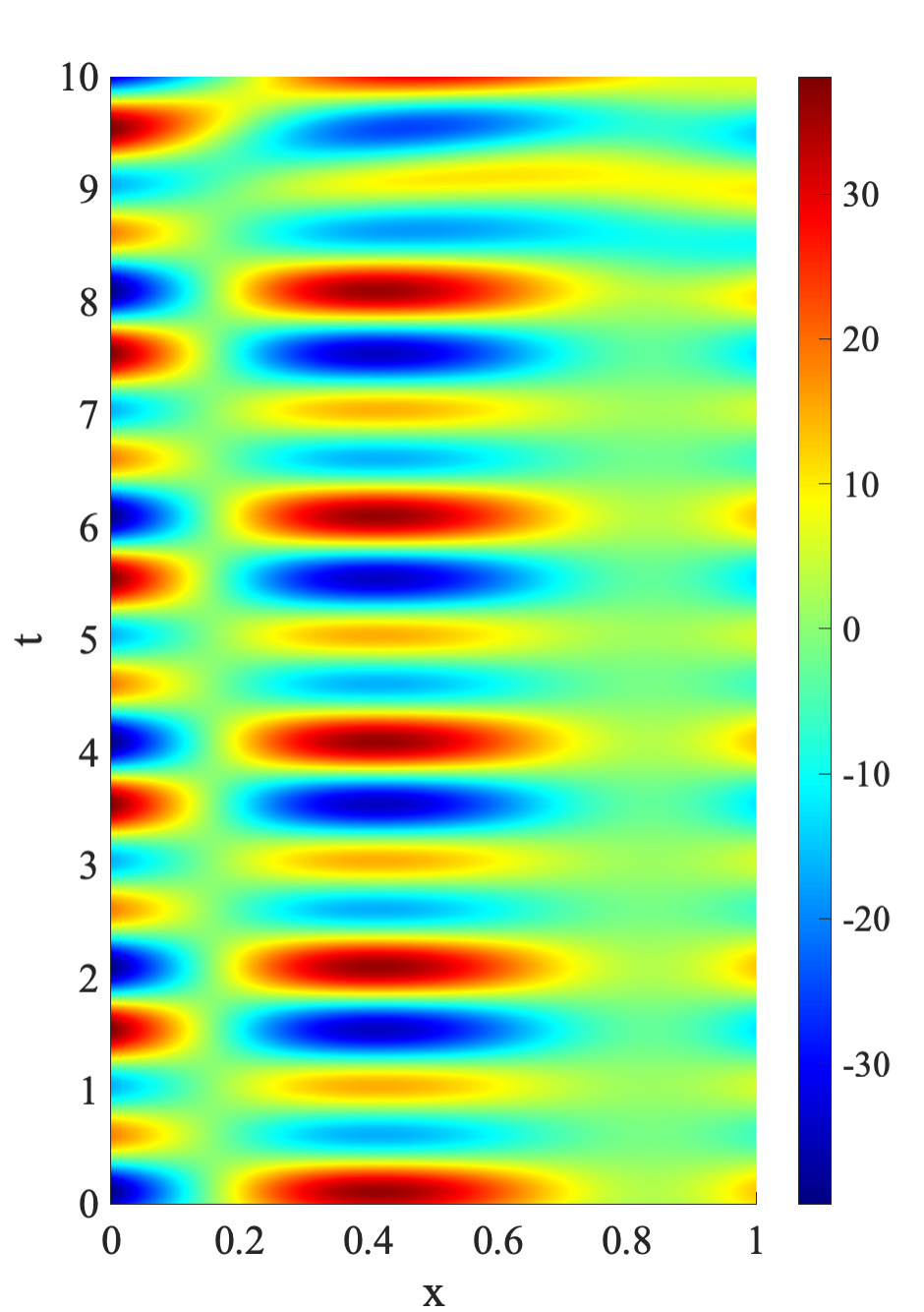}}
    \subfloat[$|v - v^*_\theta|$]{\includegraphics[width=0.22\linewidth]{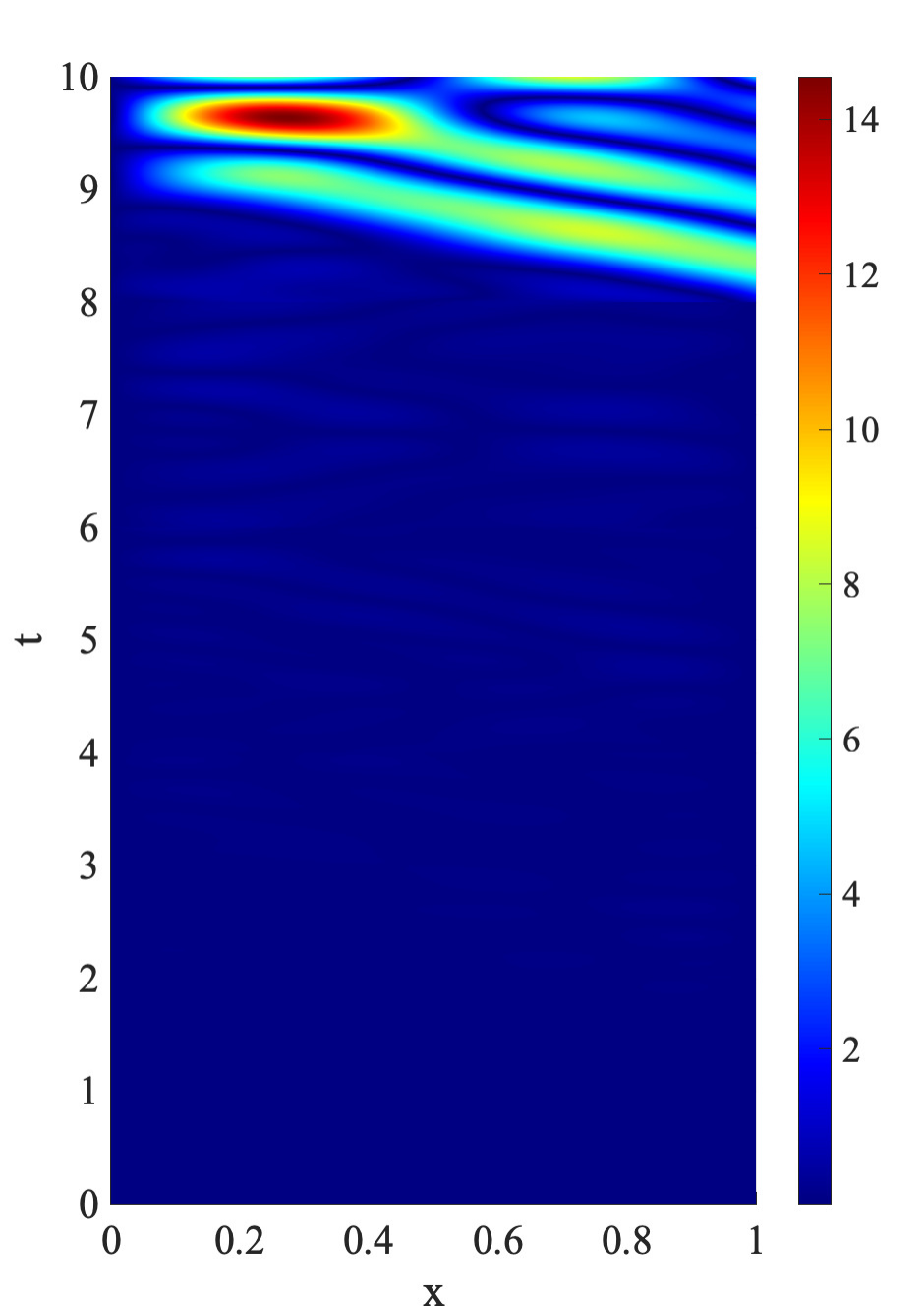}}
	\caption{Nonlinear Klein-Gordon equation: Solutions ($v_\theta$ is the solution of PINN-ExBTM, $v_\theta^*$ is the solution of PINN-BTM). The three hidden layers with $[90, 90, 10]$ neurons. [$\tanh-\tanh-\sin$, $N_c=2500$ for training, $N_{ev}=1000$ for prediction]}
	\label{PINN_num_KG_fig2}
\end{figure}
}  

\comment{
\begin{table}[tb]\small
\caption{KG equation: For $\frac{\partial u}{\partial t}$, the $l_2$ and $l_\infty$ errors versus the number of training data points $N_c$ and the hidden layers' size. [$\tanh-\tanh-\sin$, layers: [90, 90, 10] for training, $N_{ev}=1000$ for prediction]}
	\label{tab_KG_err_2}
	\centering\small
\begin{tabular}{c|@{}c@{}|cc|cc|cc|cc}
\hline
\multirow{2}{*}{Error}    & \multirow{2}{*}{Time block} & \multicolumn{2}{c|}{$N_c=1500$} & \multicolumn{2}{c|}{$N_c=2000$} & \multicolumn{2}{c|}{$N_c=2500$} & \multicolumn{2}{c}{$N_c=3000$} \\ \cline{3-10}
                      &  & ExBTM  & BTM  & ExBTM  & BTM   & ExBTM  & BTM  & ExBTM  & BTM   \\ \hline
\multirow{5}{*}{$l_2$} & $T_{\#1}$  & 1.16e-03 & 6.66e-04 & 1.76e-03 & 1.55e-03 & 2.35e-03 & 1.29e-03 & 2.54e-03 & 3.87e-03 \\ \cline{3-10} 
                      & $T_{\#2}$  & 3.66e-03 & 3.21e-03 & 4.52e-03 & 3.40e-03 & 3.00e-03 & 3.85e-03 & 3.05e-03 & 2.87e-03 \\ \cline{3-10}
                      & $T_{\#3}$  & 6.29e-03 & 1.08e-02 & 5.73e-03 & 6.30e-03 & 4.46e-03 & 8.18e-03 & 7.64e-03 & 8.46e-03 \\ \cline{3-10} 
                      & $T_{\#4}$  & 9.65e-03 & 2.66e-02 & 2.17e-02 & 7.66e-03 & 1.13e-02 & 1.49e-02 & 8.25e-03 & 1.32e-02 \\ \cline{3-10} 
                      & $T_{\#5}$  & 8.90e-02 & 9.41e-02 & 3.89e-02 & 1.89e-02 & 1.85e-02 & 3.73e-01 & 3.25e-02 & 7.39e-02 \\ \hline
\multirow{5}{*}{$l_\infty$} & $T_{\#1}$  & 4.90e-03 & 2.78e-03 & 5.20e-03 & 4.31e-03 & 6.99e-03 & 5.86e-03 & 7.26e-03 & 1.15e-02 \\ \cline{3-10} 
                      & $T_{\#2}$  & 1.22e-02 & 1.11e-02 & 1.35e-02 & 1.16e-02 & 8.98e-03 & 1.29e-02 & 9.89e-03 & 8.99e-03 \\ \cline{3-10}
                      & $T_{\#3}$  & 2.47e-02 & 3.18e-02 & 1.60e-02 & 1.95e-02 & 1.70e-02 & 2.84e-02 & 2.64e-02 & 2.18e-02 \\ \cline{3-10} 
                      & $T_{\#4}$  & 2.94e-02 & 8.01e-02 & 6.45e-02 & 2.28e-02 & 3.61e-02 & 4.23e-02 & 2.39e-02 & 3.89e-02 \\ \cline{3-10} 
                      & $T_{\#5}$  & 3.15e-01 & 2.56e-01 & 1.22e-01 & 5.04e-02 & 6.36e-02 & 1.18e+00 & 1.09e-01 & 2.19e-01 \\ \hline
\end{tabular}
\end{table}
} 

\comment{
\begin{table}[tb]\small
\caption{KG equation: For $\frac{\partial u}{\partial t}$, the $l_2$ and $l_\infty$ errors versus the activation functions. [Layers: [90, 90, 10], $N_c=2500$ for training, $N_{ev}=1000$ for prediction]}
	\label{tab_KG_err_4}
	\centering\small
\begin{tabular}{c|@{}c@{}|cc|cc|cc|cc}
\hline
\multirow{2}{*}{Error}    & \multirow{2}{*}{Time block} & \multicolumn{2}{c|}{tanh} & \multicolumn{2}{c|}{Gaussian} & \multicolumn{2}{c|}{swish} & \multicolumn{2}{c}{softplus} \\ \cline{3-10}
                      &  & ExBTM  & BTM  & ExBTM  & BTM   & ExBTM  & BTM  & ExBTM  & BTM   \\ \hline
\multirow{5}{*}{$l_2$} & $T_{\#1}$  & 1.79e-03 & 2.32e-03 & 2.20e-03 & 2.31e-03 & 1.81e-03 & 1.82e-03 & 2.31e-03 & 2.51e-03 \\ \cline{3-10} 
                      & $T_{\#2}$  & 6.01e-03 & 7.00e-03 & 3.18e-03 & 5.70e-03 & 2.67e-03 & 7.01e-03 & 8.62e-03 & 1.44e-02 \\ \cline{3-10}
                      & $T_{\#3}$  & 1.03e-02 & 1.15e-02 & 9.93e-03 & 1.09e-02 & 4.99e-03 & 6.40e-03 & 1.09e-02 & 1.64e-02 \\ \cline{3-10} 
                      & $T_{\#4}$  & 1.90e-02 & 3.05e-02 & 3.58e-01 & 3.02e-01 & 8.88e-03 & 1.45e-02 & 1.46e-02 & 1.22e-02 \\ \cline{3-10} 
                      & $T_{\#5}$  & 1.57e-01 & 1.16e-01 & 6.11e-01 & 6.99e-01 & 1.38e-01 & 2.52e-02 & 3.75e-02 & 2.23e-02 \\ \hline
\multirow{5}{*}{$l_\infty$} & $T_{\#1}$  & 6.56e-03 & 1.10e-02 & 9.74e-03 & 8.92e-03 & 6.55e-03 & 5.44e-03 & 8.05e-03 & 8.25e-03 \\ \cline{3-10} 
                      & $T_{\#2}$  & 2.49e-02 & 2.14e-02 & 1.13e-02 & 2.19e-02 & 1.05e-02 & 2.84e-02 & 2.82e-02 & 4.55e-02 \\ \cline{3-10}
                      & $T_{\#3}$  & 3.01e-02 & 3.56e-02 & 3.25e-02 & 4.09e-02 & 1.64e-02 & 2.65e-02 & 3.21e-02 & 4.43e-02 \\ \cline{3-10} 
                      & $T_{\#4}$  & 7.20e-02 & 7.89e-02 & 1.15e+00 & 1.00e+00 & 2.91e-02 & 4.61e-02 & 4.40e-02 & 4.62e-02 \\ \cline{3-10} 
                      & $T_{\#5}$  & 4.86e-01 & 4.26e-01 & 1.45e+00 & 1.81e+00 & 4.69e-01 & 9.65e-02 & 1.16e-01 & 7.31e-02 \\ \hline
\end{tabular}
\end{table}
} 

\comment{
\begin{figure}[tb]
	\centering
	\subfloat[$ t=2.5 $]{
		\begin{minipage}[b]{0.22\textwidth}
			\includegraphics[scale=0.25]{./figures/KG_BTM_2500_90,90,10_sin_TBlock2_T2.5_u}\\
			\includegraphics[scale=0.25]{./figures/KG_BTM_2500_90,90,10_sin_TBlock2_T2.5_uerr}
		\end{minipage}
	}\qquad
	\subfloat[$ t=5 $]{
		\begin{minipage}[b]{0.22\textwidth}
			\includegraphics[scale=0.25]{./figures/KG_BTM_2500_90,90,10_sin_TBlock3_T5_u}\\
			\includegraphics[scale=0.25]{./figures/KG_BTM_2500_90,90,10_sin_TBlock3_T5_uerr}
		\end{minipage}
	}\qquad
	\subfloat[$ t=9.5 $]{
		\begin{minipage}[b]{0.22\textwidth}
			\includegraphics[scale=0.25]{./figures/KG_BTM_2500_90,90,10_sin_TBlock5_T9.5_u}\\
			\includegraphics[scale=0.25]{./figures/KG_BTM_2500_90,90,10_sin_TBlock5_T9.5_uerr}
		\end{minipage}
	}
	\caption{KG equation: Top row, comparison of profiles between the exact solution and PINN-ExBTM/PINN-BTM solutions for $u$ at several time instants. Bottom row, profiles of the absolute error of the PINN-ExBTM and PINN-BTM solutions for $u$. $N_c=2000$ training collocation points. The three hidden layers with $[90, 90, 10]$ neurons. [$\tanh-\tanh-\sin$, $N_c=2500$ for training, $N_{ev}=1000$ for prediction]}
	\label{PINN_num_KG_fig3_1}
\end{figure}

\begin{figure}[tb]
	\centering
	\subfloat[$ t=2.5 $]{
		\begin{minipage}[b]{0.22\textwidth}
			\includegraphics[scale=0.25]{./figures/KG_BTM_2500_90,90,10_sin_TBlock2_T2.5_v}\\
			\includegraphics[scale=0.25]{./figures/KG_BTM_2500_90,90,10_sin_TBlock2_T2.5_verr}
		\end{minipage}
	}
	\subfloat[$ t=5 $]{
		\begin{minipage}[b]{0.22\textwidth}
			\includegraphics[scale=0.25]{./figures/KG_BTM_2500_90,90,10_sin_TBlock3_T5_v}\\
			\includegraphics[scale=0.25]{./figures/KG_BTM_2500_90,90,10_sin_TBlock3_T5_verr}
		\end{minipage}
	}
	\subfloat[$ t=9.5 $]{
		\begin{minipage}[b]{0.22\textwidth}
			\includegraphics[scale=0.25]{./figures/KG_BTM_2500_90,90,10_sin_TBlock5_T9.5_v}\\
			\includegraphics[scale=0.25]{./figures/KG_BTM_2500_90,90,10_sin_TBlock5_T9.5_verr}
		\end{minipage}
	}
	\caption{KG equation: Top row, comparison of profiles between the exact solution and PINN-ExBTM/PINN-BTM solutions for $\frac{\partial u}{\partial t}$ at several time instants. Bottom row, profiles of the absolute error of the PINN-ExBTM and PINN-BTM solutions for $u$. $N=2000$ training collocation points. The three hidden layers with $[90, 90, 10]$ neurons. [$\tanh-\tanh-\sin$, $N_c=2500$ for training, $N_{ev}=1000$ for prediction]}
	\label{PINN_num_KG_fig3_2}
\end{figure}
}  

\comment{
\begin{figure}[tb]
	\centering
	\subfloat[PINN-ExBTM]{\label{PINN_num_KG_fig5_1_a}\includegraphics[width=0.5\linewidth]{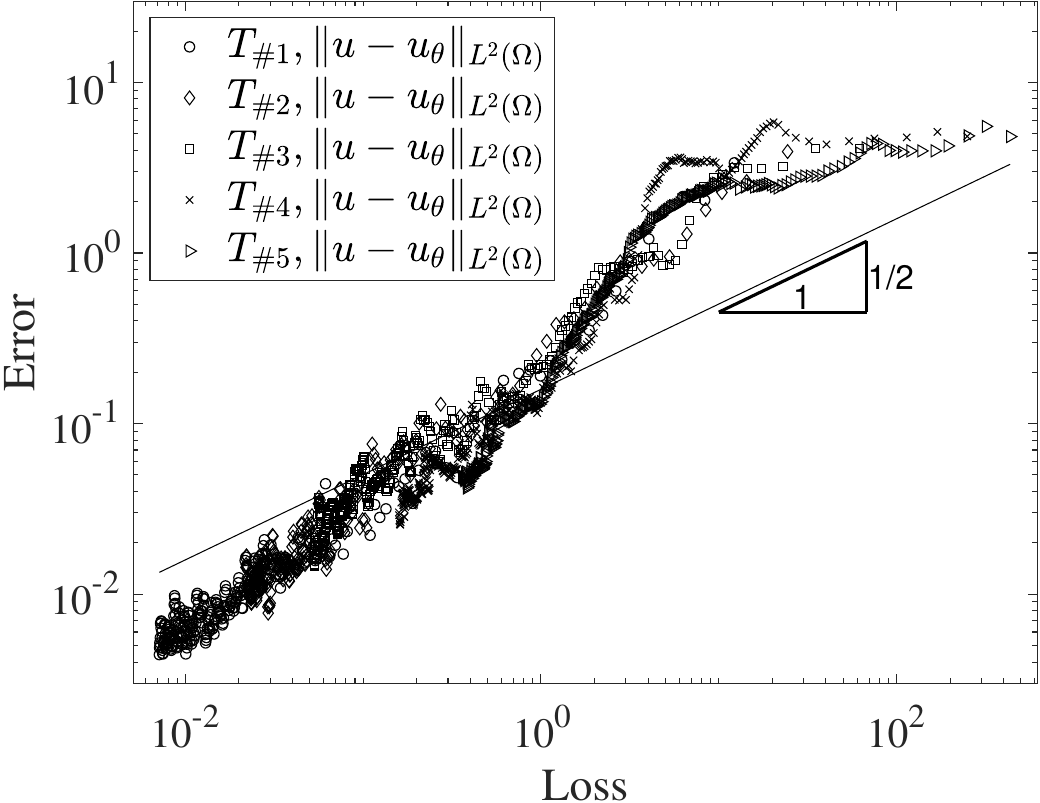}}
	\subfloat[PINN-BTM]{\label{PINN_num_KG_fig5_1_b}\includegraphics[width=0.5\linewidth]{./figures/KG_BTM_2500_90,90,10_sin_TBlock5_errorRatio_u}}\\
    \subfloat[PINN-ExBTM]{\label{PINN_num_KG_fig5_1_c}\includegraphics[width=0.5\linewidth]{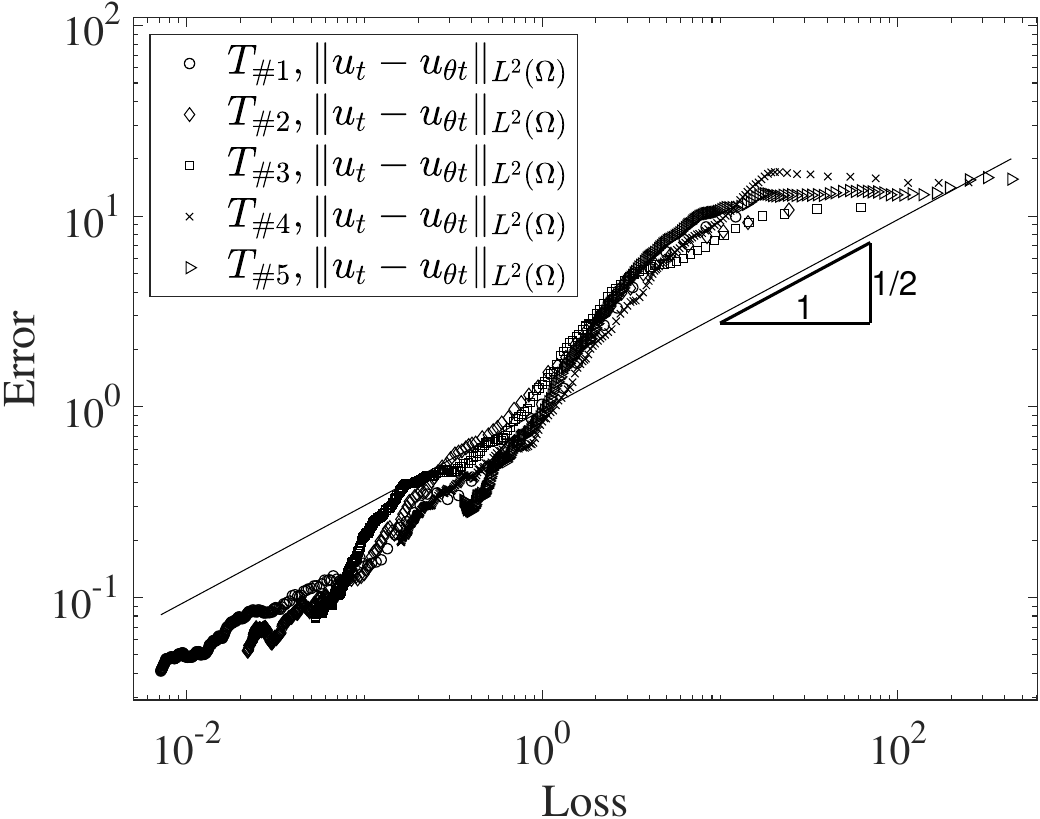}}
	\subfloat[PINN-BTM]{\label{PINN_num_KG_fig5_1_d}\includegraphics[width=0.5\linewidth]{./figures/KG_BTM_2500_90,90,10_sin_TBlock5_errorRatio_ut}}\hspace{0.5em}
    \hspace{0.5em}
	\caption{KG equation: The $l^2$ errors of $u$ as a function of the training loss value. [The three hidden layers with $[90, 90, 10]$ neurons. $\tanh-\tanh-\sin$, $N_c=2500$ for training, $N_{ev}=1000$ for prediction.] [(a)$-$(b) ExPINN and PINN for $u$, (c)$-$(d) ExPINN and PINN for $\frac{\partial u}{\partial t}$]}
	\label{PINN_num_KG_fig5_1}
\end{figure}
} 

%% file: content/summary.tex
\section{Concluding Remarks}
\label{sec_summary}

We have presented a hidden-layer concatenated physics informed neural network (HLConcPINN) method for approximating PDEs, by combining hidden-layer concatenated feed-forward neural networks (HLConcFNN), an extended block time marching strategy, and the physics informed approach. We analyze the convergence properties and the errors of this method for parabolic and hyperbolic type PDEs. Our analyses show that with this method the approximation error of the solution field can be effectively controlled by the training loss for dynamic simulations with long time horizons.
HLConcPINN allows network architectures with an arbitrary number of hidden layers of two or larger, and any of the commonly-used smooth activation functions for all hidden layers beyond the first two. Our method generalizes several existing PINN techniques, which have theoretical guarantees but are confined to network architectures with two hidden layers and the $\tanh$ activation function.
We implement the HLConcPINN algorithm, and have presented a number of computational examples based on this method. The numerical results demonstrate the effectiveness of our method and corroborate the relationship between the approximation error and the training loss function from theoretical analyses.

Finally we would like to comment that
in our analyses we have focused on  parabolic and hyperbolic type PDEs. However, the analysis of the HLConcPINN technique, excluding the block time marching component, can be extended to elliptic type equations. Discussion on this type of equations is not included here for conciseness of the paper. 


%% file: content/auxiliary.tex
\section{Appendix: Auxiliary Results and Proofs of Main Theorems}
\label{Appendix}

\subsection{Some Auxiliary Results}\label{Auxiliary lemmas}

Let a $d$-tuple of non-negative integers $\alpha\in \mathbb{N}_0^d$ be multi-index with $d \in\mathbb{N}$. 
For given two multi-indices $\alpha,\beta \in \mathbb{N}_0^d$, we say that $\alpha\leq \beta$,
if and only if, $\alpha_i\leq \beta_i$ for all $i=1,\cdots,d$. Denote 
$ 
|\alpha| = \sum_{i=1}^d\alpha_i,\ \alpha!=\prod_{i=1}^d\alpha_i!, 
\ \begin{pmatrix}
	\alpha\\	
	\beta
\end{pmatrix}=\frac{\alpha!}{\beta!(\alpha-\beta)!}.
$ 
Let $P_{m,n}=\{\alpha\in \mathbb{N}_0^n,  |\alpha|=m\}$, for which it holds
$ 
|P_{m,n}| =\begin{pmatrix}
	m+n-1\\	
	m
\end{pmatrix}.
$ 


\begin{Lemma}\label{Ar_01} Let $d\in \mathbb{N}, k\in \mathbb{N}_0$, $f\in H^{k}(\Omega)$ and $g\in W^{k,\infty}(\Omega)$ with $\Omega \subset \mathbb{R}^d$, then	
	$ 
	\|fg\|_{H^k(\Omega)}\leq 2^k\|f\|_{H^k(\Omega)}\|g\|_{W^{k,\infty}(\Omega)}.
	$ 
\end{Lemma}

\begin{Lemma}[Multiplicative trace inequality, e.g. \cite{2023_IMA_Mishra_NS}]\label{Ar_2} 
Let $d\geq 2$, $\Omega \subset \mathbb{R}^d$ be a Lipschitz domain and let $\gamma_0: H^1(\Omega)\rightarrow L^2(\partial\Omega): u \mapsto u|_{\partial\Omega}$ be the trace operator. Denote by $h_{\Omega}$ the diameter of $\Omega$ and by $\rho_{\Omega}$ the radius of the largest $d$-dimensional ball that can be inscribed into $\Omega$. Then it holds that
	\begin{equation}\label{Ar_eq1}
	  \|\gamma_0 u\|_{L^2(\partial \Omega)}\leq C_{h_{\Omega},d,\rho_{\Omega}}\|u\|_{H^1(\Omega)},
          \quad\text{where}\
          C_{h_{\Omega},d,\rho_{\Omega}}=\sqrt{2\max\{2h_{\Omega},d\}/\rho_{\Omega}}.
	\end{equation}
\end{Lemma}

\begin{Lemma}[\cite{2023_IMA_Mishra_NS}]\label{Ar_3}
Let $d, n, L, W\in \mathbb{N}$ and let $u_{\vartheta}:\mathbb{R}^{d}\rightarrow \mathbb{R}^{l_L}$ be a neural network with $\vartheta\in \Theta$ for $d, L\geq 2, R, W\geq 1$, c.f. Definition \ref{pre_lem1}. Assume that $\|\sigma\|_{C^n}\geq 1$. Then it holds for $1\leq j\leq l_L$ that
	\begin{equation}\label{Ar_eq2}
		\|(u_{\vartheta})_j\|_{C^n(\Omega)}\leq 16^Ld^{2n}(e^2n^4W^3R^n\|\sigma\|_{C^n(\Omega)})^{nL}.
	\end{equation}
\end{Lemma}

\begin{Remark}\label{Ar_Remark} 
Let $u_{\theta}: \mathbb{R}^d \rightarrow \mathbb{R}^{l_L}$ denote a neural network with smooth activation functions, in accordance with Definition \ref{pre_lem1}. Suppose the first two hidden layers of the network are endowed with the tanh activation function, whereas (if $L>3$) the subsequent hidden layers utilize a variety of smooth activation functions, including (but not restricted to) e.g.~the tanh, sine, sigmoid, Gaussian, and softplus functions. 
Let $\hat{\sigma}$ denote a collection of these smooth activation functions. Under the conditions specified in Lemma \ref{Ar_3}, by defining $\|\sigma\|_{C^k} = \max_{\tilde{\sigma}\in\hat{\sigma}}\{\|\tilde{\sigma}\|_{C^k}\}$, it can be shown that Lemma \ref{Ar_3} remains valid.
Furthermore, thanks to the inherent properties of hidden-layer concatenated feedforward neural networks, the output fields of the $i$-th ($i = 1, \ldots, L - 1$) hidden layer and the output layer exhibit analogous behavior based on Lemma \ref{Ar_3}. For the sake of conciseness, we omit the proof here and refer to the results presented in Lemma \ref{Ar_3}.
\end{Remark}


\begin{Lemma}[\cite{2023_IMA_Mishra_NS}]\label{Ar_4}  
Let $d\geq2, n\geq2,  m\geq 3, \sigma>0, a_i, b_i \in \mathbb{Z}$ with $a_i <b_i$ for $1\leq i\leq d$, $\Omega=\prod_{i=1}^d[a_i,b_i]$ and $f\in H^m(\Omega)$. Then for every $N\in \mathbb{N}$ with $N>5$, there exists a tanh neural network $\widetilde{f}^N$ with two hidden layers, one of width at most $3\lceil\frac{m+n-2}{2}\rceil|P_{m-1,d+1}|+\sum_{i=1}^d(b_i-a_i)(N-1)$ and another of width at most $3\lceil\frac{d+n}{2}\rceil|P_{d+1,d+1}|N^d\prod_{i=1}^d(b_i-a_i)$, such that for $k=0,1,2$ it holds that
	\begin{equation}\label{Ar_eq3}
		\|f-\widetilde{f}^N\|_{H^k(\Omega)}\leq 2^k3^dC_{k,m,d,f}(1+\delta){\rm ln}^k\left(\beta_{k,\delta,d,f}N^{d+m+2}\right)N^{-m+k},
	\end{equation}
	and where
	\begin{align}
		&C_{k,m,d,f}=\max_{0\leq l \leq k}\left(
		\begin{array}{c}
			d+l-1 \\
			l \\
		\end{array}
		\right)^{1/2}\frac{((m-l)!)^{1/2}}{(\lceil\frac{m-l}{d}\rceil!)^{d/2}}\left(\frac{3\sqrt{d}}{\pi}\right)^{m-l}|f|_{H^m(\Omega)},\\
		&\beta_{k,\delta,d,f}=\frac{5\cdot2^{kd}\max\{\prod_{i=1}^d(b_i-a_i),d\}\max\{\|f\|_{W^{k,\infty}(\Omega)},1\}}{3^d\delta\min\{1,C_{k,m,d,f}\}}.	
	\end{align}
Moreover, the weights of $\widetilde{f}^N$ scale as $O(N^{\gamma}+N{\rm ln}N)$ with $\gamma=\max\{m^2/n,d(2+m+d)/n\}$.
\end{Lemma}

\begin{Lemma}[\cite{NiDong_hLC_2023}]\label{Ar_5} Given an architectural vector $\bm{l}_1= (l_0, l_1, \cdots,l_{L-1}, l_L)$ with $l_L=1$, define a new vector $\bm{l}_2= (l_0, l_1, \cdots,l_{L-1},n,l_L)$, where $n\geq 1$ is an integer. For a given domain $D\subset \mathbb{R}^{l_0}$ and an activation function $\sigma$, the following relation holds
\begin{equation}\label{Ar_eq4}
   U(D,\bm{l}_1,\sigma) \subseteq U(D,\bm{l}_2,\sigma),
\end{equation} 
where $U$ is defined by \eqref{pre_eq4}.
\end{Lemma}

\begin{Lemma}[\cite{NiDong_hLC_2023}]\label{Ar_6} Given an architectural vector $\bm{l}_1= (l_0, l_1, \cdots,l_{L-1}, l_L)$ with $l_L=1$, define a new vector $\bm{l}_2= (l_0, l_1, \cdots,l_{s-1},l_s+1,l_{s+1},\cdots, l_L)$ for some $s$ $(1\leq s\leq L-1)$. For a given domain $D\subset \mathbb{R}^{l_0}$ and an activation function $\sigma$, the following relation holds
\begin{equation}\label{Ar_eq5}
  U(D,\bm{l}_1,\sigma) \subseteq U(D,\bm{l}_2,\sigma),
\end{equation} 
where $U$ is defined by \eqref{pre_eq4}.
\end{Lemma}

\begin{Lemma}\label{Ar_7} Under the conditions specified in Lemma \ref{Ar_4}, for every $N \in \mathbb{N}$ where $N > 5$, there exists a hidden-layer concatenated feedforward neural network denoted as $\hat{f}^N$, defined by
\begin{equation}\label{Ar_eq6}
   \hat{f}^N=\sum_{i=1}^{L-1}M_i\hat{f}_i^N+b_{L}\qquad L\geq 3,
\end{equation} 
where $\hat{f}_i^N$, $M_i\in \mathbb{R}^{1\times l_i}$ $(1\leq i \leq L-1)$ and $b_L\in\mathbb{R}^{1}$ represent the output of the $i$-th hidden layer, the connection coefficients between the output layer and the $i$-th hidden layer, and the bias of the output layer, respectively. Note that the first two hidden layers of the network employ the tanh activation function, while the other hidden layers can use any other smooth activation function. 
Therefore, for $k=0, 1, 2$, this neural network satisfies\begin{equation}\label{Ar_eq7}
	\|f-\hat{f}^N\|_{H^k(\Omega)}\lesssim {\rm ln}^k\left(N^{d+m+2}\right)N^{-m+k}.
\end{equation}
\end{Lemma}
\begin{proof} 
Lemma \ref{Ar_5} and \ref{Ar_6} imply that $\hat{f}^N$ possesses a greater representational capacity compared to $\widetilde{f}^N$. 
		
It should be noted that smooth functions are both continuous and bounded on a closed interval. For neural networks, activation function $\sigma$ such as the sigmoid and hyperbolic tangent (tanh) are examples of smooth functions. These functions can be bounded by a constant $C$ in the $H^k(\Omega)$ norm, i.e., $\|\sigma\|_{H^k(\Omega)}\leq C$.

We set $M_1=0^{1\times l_1} \in \mathbb{R}^{1\times l_1}$, $M_2=W_3$, $b_L=b_3$, and $M_i=\frac{\varepsilon}{Cl_i(L-3)} 1^{1\times l_i} \in \mathbb{R}^{1\times l_i}$ $(i=3,\cdots,L-1)$, while assigning $0^{l_{i}\times l_{i-1}} \in \mathbb{R}^{l_{i}\times l_{i-1}}$ to the weight coefficients $W_i$ of the $i$-th hidden layer for all $i=3,\cdots,L-1$. 
By retaining the initial two hidden layers in Lemma \ref{Ar_4} and setting $\varepsilon={\rm ln}^k\left(N^{d+m+2}\right)N^{-m+k}$, we obtain $W_3\hat{f}_2^N+b_3=\widetilde{f}^N$ (defined 
in Lemma \ref{Ar_4}) and $\hat{f}^N=\widetilde{f}^N+\sum_{i=3}^{L-1}M_i\sigma(b_i)$. Consequently, the approximation can be bounded as follows
\[
\|f-\hat{f}^N\|_{H^k(\Omega)}\leq \|f-\widetilde{f}^N\|_{H^k(\Omega)} +\sum_{i=3}^{L-1}\|M_i\sigma(b_i)\|_{H^k(\Omega)}\lesssim{\rm ln}^k\left(N^{d+m+2}\right)N^{-m+k},
\]
where $l_i$ denotes the number of nodes in the $i$-th hidden layer.
\end{proof}

\subsection{Proof of Main Theorems from Section~\ref{Heat}: Heat Equation}\label{Proof}

\vspace{0.1in}
\noindent\underline{\bf Proof of Theorem \ref{sec4_Theorem1} :}
\begin{proof} Based on Lemma \ref{Ar_7}, there exists a HLConcPINN $u_{\theta_i}$ such that for every $0 \leq m \leq 2$,
	\begin{equation}\label{sec4_eq0}
		\|u_{\theta_i}-u\|_{H^m(\widetilde\Omega_i)}\lesssim {\rm ln}^m(N)N^{-k+m}.
	\end{equation}
	According to Lemma \ref{Ar_2}, we can bound the HLConcPINN residual terms,
    \begin{align*}
		&\|\frac{\partial \hat{u}_{i}}{\partial t}\|_{L^2(\widetilde\Omega_{i})}\leq\|\hat{u}_{i}\|_{H^1(\widetilde\Omega_{i})},\qquad \|\Delta\hat{u}_{i}\|_{L^2(\widetilde\Omega_{i})}\leq\|\hat{u}_{i}\|_{H^2(\widetilde\Omega_{i})},\\
		&\|\hat{u}_{i}\|_{L^2(\widetilde\Omega_{*i})}\leq \|\hat{u}_{i}\|_{L^2(\partial\widetilde\Omega_{i})}\leq C_{h_{\widetilde\Omega_{i}},d+1,\rho_{\widetilde\Omega_{i}}}\|\hat{u}_{i}\|_{H^1(\widetilde\Omega_{i})}.
	\end{align*}
     For $j=1$, $R_{tb_{i}}|_{t=t_{0}}=\hat{u}_{i}|_{t=t_{0}}$, we obtain
		\[
		\|R_{tb_{i}}(\bm{x},0)\|_{L^2(D)}\leq\|\hat{u}_{i}\|_{L^2(\partial\widetilde\Omega_{i})}\leq C_{h_{\widetilde\Omega_{i}},d+1,\rho_{\widetilde\Omega_{i}}}\|\hat{u}_{i}\|_{H^1(\widetilde\Omega_{i})}.
       \]
	   For $j>1$, it holds 
	   \begin{align*}
	   &\|R_{tb_{i}}(\bm{x},t_{j-1})\|_{L^2(D)}
	   	\leq \|\hat{u}_{i}|_{t=t_{j-1}}\|_{L^2(D)}+\|\hat{u}_{j-1}|_{t=t_{j-1}}\|_{L^2(D)}
	   \leq\|\hat{u}_{i}\|_{L^2(\partial\widetilde\Omega_{j-1})}+\|\hat{u}_{j-1}\|_{L^2(\partial\widetilde\Omega_{j-1})}\\
	   &\qquad\leq C_{h_{\widetilde\Omega_{j-1}},d+1,\rho_{\widetilde\Omega_{j-1}}}(\|\hat{u}_{i}\|_{H^1(\widetilde\Omega_{j-1})}+\|\hat{u}_{j-1}\|_{H^1(\widetilde\Omega_{j-1})})\leq  C_{h_{\widetilde\Omega_{j-1}},d+1,\rho_{\widetilde\Omega_{j-1}}}(\|\hat{u}_{i}\|_{H^1(\widetilde\Omega_{i})}+\|\hat{u}_{j-1}\|_{H^1(\widetilde\Omega_{j-1})}).
      \end{align*} 
	By combining these relations with \eqref{sec4_eq0}, we can obtain
	\begin{align*}
		&\|R_{int_{i}}\|_{L^2(\widetilde\Omega_{i})}=\|\frac{\partial \hat{u}_{i}}{\partial t}-\Delta\hat{u}_{i}\|_{L^2(\widetilde\Omega_{i})}\leq
		\|\hat{u}_i\|_{H^1(\widetilde\Omega_{i})}+\|\hat{u}_{i}\|_{H^2(\widetilde\Omega_{i})}\lesssim N^{-k+2}{\rm ln}^2N,\\
		&\|R_{tb_{i}}(\bm{x},t_{j-1})\|_{L^2(D)}, \|R_{sb_{i}}\|_{L^2(\widetilde\Omega_{*i})}\lesssim\|\hat{u}_{i}\|_{H^1(\widetilde\Omega_{i})}+\|\hat{u}_{j-1}\|_{H^1(\widetilde\Omega_{j-1})}\lesssim N^{-k+1}{\rm ln}N\quad 1\leq j\leq i.
	\end{align*} 
Then, we finish our proof.
\end{proof}

\vspace{0.1in}
\noindent\underline{\bf Proof of Theorem \ref{sec4_Theorem2} :}
\begin{proof} Take the inner product of \eqref{heat_error_eq1} and $\hat{u}_{i}$ over $D$ to obtain,
	\begin{align}\label{sec4_eq1}
		\frac{d}{2dt}\int_{D} |\hat{u}_{i}|^2\dx &=- \int_{D}|\nabla\hat{u}_{i}|^2\dx+\int_{\partial D} R_{sb_{i}}\nabla\hat{u}_{i}\cdot\bm{n}\ds+\int_{D} R_{int_{i}}\hat{u}_{i}\dx
		\nonumber\\
		&\leq - \int_{D}|\nabla\hat{u}_{i}|^2\dx+\frac{1}{2}\int_{D} |\hat{u}_{i}|^2\dx+\frac{1}{2}\int_{D} |R_{int_{i}}|^2\dx+C_{\partial D_i}\big(\int_{\partial D}|R_{sb_{i}}|^2\ds\big)^{\frac{1}{2}},
	\end{align}
	where $C_{\partial D_i}=|\partial D|^{\frac{1}{2}}(\|u\|_{C^1(\partial D\times[0,t_i])}+\|u_{\theta_{i}}\|_{C^1(\partial D\times[0,t_i])})$. \\
	Integrating \eqref{sec4_eq1} over $[t_{i-1},
	\tau]$ for any $t_{i-1}\leq \tau \leq t_i$, using the initial condition \eqref{heat_error_eq2}, and applying Cauchy–Schwarz inequality, we obtain
	\begin{align*}
		&\int_{D} |\hat{u}_{i}(\bm{x},\tau)|^2\dx +2\int_{t_{i-1}}^{t_{i}}\int_{D}|\nabla\hat{u}_{i}|^2\dx\dt\\
		&\qquad\leq\int_{D}|\hat{u}_{i}(\bm{x},t_{i-1})|^2\dx +\int_{t_{i-1}}^{t_{i}}\int_{D} |\hat{u}_{i}|^2\dx\dt+ \int_{t_{i-1}}^{t_{i}}\int_{D}|R_{int_{i}}|^2\dx\dt
		+2C_{\partial D_i}|\Delta t|^{\frac{1}{2}}\big(\int_{t_{i-1}}^{t_{i}}\int_{\partial D}|R_{sb_{i}}|^2\ds\dt\big)^{\frac{1}{2}}\\
		&\qquad\leq\int_{D}|\hat{u}_{i}(\bm{x},t_{i-1})|^2\dx+\sum_{j=1}^{i-1}\int_{D}|R_{tb_{i}}(\bm{x},t_{j-1})|^2\dx +\int_{t_{i-1}}^{t_{i}}\int_{D} |\hat{u}_{i}|^2\dx\dt+ \int_{t_{i-1}}^{t_{i}}\int_{D}|R_{int_{i}}|^2\dx\dt\\
		&\qquad\qquad+2C_{\partial D_i}|\Delta t|^{\frac{1}{2}}\big(\int_{t_{i-1}}^{t_{i}}\int_{\partial D}|R_{sb_{i}}|^2\ds\dt\big)^{\frac{1}{2}}\\
		&\qquad\leq 2\int_{D}|\hat{u}_{i-1}(\bm{x},t_{i-1})|^2\dx+2\int_{D}|R_{tb_{i}}(\bm{x},t_{i-1})|^2\dx+\sum_{j=1}^{i-1}\int_{D}|R_{tb_{i}}(\bm{x},t_{j-1})|^2\dx +\int_{t_{i-1}}^{t_{i}}\int_{D} |\hat{u}_{i}|^2\dx\dt\\
		&\qquad\qquad +\int_{t_{i-1}}^{t_{i}}\int_{D}|R_{int_{i}}|^2\dx\dt+2C_{\partial D_i}|\Delta t|^{\frac{1}{2}}\big(\int_{t_{i-1}}^{t_{i}}\int_{\partial D}|R_{sb_{i}}|^2\ds\dt\big)^{\frac{1}{2}}.
	\end{align*}
    Then, applying the integral form of the Gr${\rm\ddot{o}}$nwall inequality to the above inequality, it holds
    \begin{align}\label{sec4_eq2}
		&\int_{D} |\hat{u}_{i}(\bm{x},\tau)|^2\dx+2\int_{t_{i-1}}^{t_{i}}\int_{D}|\nabla\hat{u}_{i}|^2\dx\dt
		\nonumber\\
		&\qquad\leq \big(2\int_{D}|\hat{u}_{i-1}(\bm{x},t_{i-1})|^2\dx+2\sum_{j=1}^{i}\int_{D}|R_{tb_{i}}(\bm{x},t_{j-1})|^2\dx+ \int_{t_{i-1}}^{t_{i}}\int_{D}|R_{int_{i}}|^2\dx\dt\big)\exp(\Delta t)
		\nonumber\\
	    &\qquad\qquad +
	    2C_{\partial D_i}|\Delta t|^{\frac{1}{2}}\big(\int_{t_{i-1}}^{t_{i}}\int_{\partial D}|R_{sb_{i}}|^2\ds\dt\big)^{\frac{1}{2}}\exp(\Delta t).
	\end{align}

    Firstly, 
	by \eqref{sec4_eq2} and integrating \eqref{sec4_eq2} over $[t_{0},t_{1}]$, 
	Theorem \ref{sec4_Theorem2} holds for $i=1$ according to the fact that $\mathcal{E}_{G_{i-1}}(\theta) =0$ and $\hat{u}_{i-1}|_{t=t_{i-1}}=0$ for $i=1$.
	
	Secondly, we assume that Theorem \ref{sec4_Theorem2} holds for all $i\leq l-1$, i.e.,
	\[
		\int_{D}|\hat{u}_{i}(\bm{x},\tau)|^2\dx\leq C_{G_{i}}\exp(\Delta t), \qquad 
		\int_{t_{i-1}}^{t_{i}}\int_{D}|\hat{u}_{i}(\bm{x},\tau)|^2\dx\dt\leq C_{G_{i}}\Delta t\exp(\Delta t).
  \]
    For $i=l-1$, as it should be,
	\begin{equation}\label{sec4_eq4}
		\int_{D}|\hat{u}_{l-1}(\bm{x},\tau)|^2\dx \leq C_{G_{l-1}}\exp(\Delta t) \qquad t_{l-2}\leq \tau\leq t_{l-1}.
	\end{equation}

	Finally, we begin to verify that Theorem \ref{sec4_Theorem2} is true at $i = l$. 
	
	Let $i=l$ in \eqref{sec4_eq2}, under the established conditions \eqref{sec4_eq4} and the Gr${\rm\ddot{o}}$nwall inequality, we derive
	\begin{align*}
		&\int_{D} |\hat{u}_l(\bm{x},\tau)|^2\dx+2\int_{t_{l-1}}^{t_{l}}\int_{D}|\nabla\hat{u}_l|^2\dx\dt
		\nonumber\\
		&\qquad\leq \big(2\int_{D}|\hat{u}_{l-1}(\bm{x},t_{l-1})|^2\dx+2\sum_{j=1}^{l}\int_{D}|R_{tb_{l}}(\bm{x},t_{j-1})|^2\dx+ \int_{t_{l-1}}^{t_{l}}\int_{D}|R_{int_{l}}|^2\dx\dt\big)\exp(\Delta t)
		\nonumber\\
		&\qquad\qquad +
		2C_{\partial D_{l}}|\Delta t|^{\frac{1}{2}}\big(\int_{t_{l-1}}^{t_{l}}\int_{\partial D}|R_{sb_{l}}|^2\ds\dt\big)^{\frac{1}{2}}\exp(\Delta t)
		\nonumber\\
		&\qquad\leq \big(2C_{G_{l-1}}\exp(\Delta t) +2\sum_{j=1}^{l}\int_{D}|R_{tb_{l}}(\bm{x},t_{j-1})|^2\dx+ \int_{t_{l-1}}^{t_{l}}\int_{D}|R_{int_{l}}|^2\dx\dt\big)\exp(\Delta t)
		\nonumber\\
		&\qquad\qquad +
		2C_{\partial D_{l}}|\Delta t|^{\frac{1}{2}}\big(\int_{t_{l-1}}^{t_{l}}\int_{\partial D}|R_{sb_{l}}|^2\ds\dt\big)^{\frac{1}{2}}\exp(\Delta t)
		\nonumber\\
		&\qquad\leq (\widetilde{C}_{G_{l}}+2C_{G_{l-1}}\exp(\Delta t))\exp(\Delta t),
	\end{align*}
	where 
	\[
	\widetilde{C}_{G_{l}}=2\sum_{j=1}^{l}\int_{D}|R_{tb_{l}}(\bm{x},t_{j-1})|^2\dx+\int_{t_{l-1}}^{t_{l}}\int_{D}|R_{int_{l}}|^2\dx\dt+2C_{\partial D_{l}}|\Delta t|^{\frac{1}{2}}\big(\int_{t_{l-1}}^{t_{l}}\int_{\partial D}|R_{sb_{l}}|^2\ds\dt\big)^{\frac{1}{2}}.
	\]
	By using the mathematical induction and deduction methods, we finish the proof.
\end{proof} 

\vspace{0.1in}
\noindent\underline{\bf Proof of Theorem \ref{sec4_Theorem3} :}
\begin{proof} By combining Theorem \ref{sec4_Theorem2} with
	the quadrature error formula \eqref{int1}, we have
	\begin{align*}
		&\int_{D}|R_{tb_{i}}|^2\dx=\int_{D}|R_{tb_{i}}|^2\dx-\mathcal{Q}_{M_{tb_{i}}}^{D}(R_{tb_{i}}^2)+\mathcal{Q}_{M_{tb_{i}}}^{D}(R_{tb_{i}}^2)
		\leq C_{({R_{tb_{i}}^2})}M_{tb_{i}}^{-\frac{2}{d}}+\mathcal{Q}_{M_{tb_{i}}}^{D}(R_{tb_{i}}^2),\\
		&\int_{\Omega_{i}}|R_{int_{i}}|^2\dx\dt=\int_{\Omega_{i}}|R_{int_{i}}|^2\dx\dt-\mathcal{Q}_{M_{int_{i}}}^{\Omega_i}(R_{int_{i}}^2)+\mathcal{Q}_{M_{int_{i}}}^{\Omega_i}(R_{int_{i}}^2)
		\leq C_{({R_{int_{i}}^2})}M_{int_{i}}^{-\frac{2}{d+1}}+\mathcal{Q}_{M_{int_{i}}}^{\Omega_i}(R_{int_{i}}^2),\\
		&\int_{\Omega_{*i}}|R_{sb_{i}}|^2\ds\dt=\int_{\Omega_{*i}}|R_{sb_{i}}|^2\ds\dt-\mathcal{Q}_{M_{sb_{i}}}^{\Omega_{*i}}(R_{sb_{i}}^2)+\mathcal{Q}_{M_{sb_{i}}}^{\Omega_{*i}}(R_{sb_{i}}^2)
		\leq C_{({R_{sb_{i}}^2})}M_{sb_{i}}^{-\frac{2}{d}}+\mathcal{Q}_{M_{sb_{i}}}^{\Omega_{*i}}(R_{sb_{i}}^2).
	\end{align*}
	Combining the fact that $C_{(R_{tb_{i}}^2)}\lesssim\|R_{tb_{i}}^2\|_{C^2}$ and $\|R_{tb_{i}}^2\|_{C^n}\leq 2^n\|R_{tb_{i}}\|_{C^n}^2$ with Lemma \ref{Ar_3}, it holds
	\begin{align}\label{sec4_eq5}
		&C_{(R_{tb_{i}}^2(\bm{x},t_{j-1}))}\lesssim\|R_{tb_{i}}(\bm{x},t_{j-1})\|_{C^2}^2\leq 2(\|\hat{u}_{i}|_{t=t_{j-1}} \|_{C^2}^2+\|\hat{u}_{j-1}|_{t=t_{j-1}} \|_{C^2}^2)
		\nonumber\\
		&\qquad\lesssim \|u\|_{C^2}^2+(e^22^4W^3R^2\|\sigma\|_{C^2})^{4L}.
	\end{align}
	In a similar way, we can estimate the terms $\int_{\Omega_{i}}|R_{int_{i}}|^2\dx\dt$ and $\int_{\Omega_{*i}}|R_{sb_{i}}|^2\ds\dt$. 

    Then, combining the above inequalities with \eqref{lem4.3}, it holds that
	\begin{equation*}
	\int_{t_{i-1}}^{t_{i}}\int_{D}|\hat{u}_{i}(\bm{x},t)|^2\dx\dt\leq  C_{T_i}\Delta t\exp(\Delta t),
	\end{equation*}
	where the constant $C_{T_i}$ is defined in \eqref{lem4.5}.
\end{proof}

%% file: main.bbl
\begin{thebibliography}{49}
\providecommand{\natexlab}[1]{#1}
\providecommand{\url}[1]{\texttt{#1}}
\providecommand{\urlprefix}{URL }
\expandafter\ifx\csname urlstyle\endcsname\relax
  \providecommand{\doi}[1]{doi:\discretionary{}{}{}#1}\else
  \providecommand{\doi}[1]{doi:\discretionary{}{}{}\begingroup
  \urlstyle{rm}\url{#1}\endgroup}\fi
\providecommand{\bibinfo}[2]{#2}

\bibitem[{Bai et~al.(2021)Bai, Koley, Mishra, and Molinaro}]{Bai2021PINN}
\bibinfo{author}{G.~Bai}, \bibinfo{author}{U.~Koley},
  \bibinfo{author}{S.~Mishra}, \bibinfo{author}{R.~Molinaro},
  \bibinfo{title}{Physics informed neural networks ({PINN}s) for approximating
  nonlinear dispersive {PDE}s}, \bibinfo{journal}{J. Comput. Math.}
  \bibinfo{volume}{39}~(\bibinfo{number}{6}) (\bibinfo{year}{2021})
  \bibinfo{pages}{816--847}.

\bibitem[{Beck et~al.(2019)Beck, E, and Jentzen}]{Beck2019Machine}
\bibinfo{author}{C.~Beck}, \bibinfo{author}{W.~E},
  \bibinfo{author}{A.~Jentzen}, \bibinfo{title}{Machine learning approximation
  algorithms for high-dimensional fully nonlinear partial differential
  equations and second-order backward stochastic differential equations},
  \bibinfo{journal}{J. Nonlinear Sci.}
  \bibinfo{volume}{29}~(\bibinfo{number}{4}) (\bibinfo{year}{2019})
  \bibinfo{pages}{1563--1619}.

\bibitem[{Berner et~al.(2020)Berner, Grohs, and Jentzen}]{Berner2020Analysis}
\bibinfo{author}{J.~Berner}, \bibinfo{author}{P.~Grohs},
  \bibinfo{author}{A.~Jentzen}, \bibinfo{title}{Analysis of the generalization
  error: empirical risk minimization over deep artificial neural networks
  overcomes the curse of dimensionality in the numerical approximation of
  {B}lack-{S}choles partial differential equations}, \bibinfo{journal}{SIAM J.
  Math. Data Sci.} \bibinfo{volume}{2}~(\bibinfo{number}{3})
  (\bibinfo{year}{2020}) \bibinfo{pages}{631--657}.

\bibitem[{Berrone et~al.(2022)Berrone, Canuto, and
  Pintore}]{Berrone2022posteriori}
\bibinfo{author}{S.~Berrone}, \bibinfo{author}{C.~Canuto},
  \bibinfo{author}{M.~Pintore}, \bibinfo{title}{Solving {PDE}s by variational
  physics-informed neural networks: an a posteriori error analysis},
  \bibinfo{journal}{Ann. Univ. Ferrara Sez. VII Sci. Mat.}
  \bibinfo{volume}{68}~(\bibinfo{number}{2}) (\bibinfo{year}{2022})
  \bibinfo{pages}{575--595}.

\bibitem[{Biswas et~al.(2022)Biswas, Tian, and Ulusoy}]{Biswas2022Error}
\bibinfo{author}{A.~Biswas}, \bibinfo{author}{J.~Tian},
  \bibinfo{author}{S.~Ulusoy}, \bibinfo{title}{Error estimates for deep
  learning methods in fluid dynamics}, \bibinfo{journal}{Numer. Math.}
  \bibinfo{volume}{151}~(\bibinfo{number}{3}) (\bibinfo{year}{2022})
  \bibinfo{pages}{753--777}.

\bibitem[{Cai et~al.(2020)Cai, Chen, Liu, and Liu}]{CaiCLL2020}
\bibinfo{author}{Z.~Cai}, \bibinfo{author}{J.~Chen}, \bibinfo{author}{M.~Liu},
  \bibinfo{author}{X.~Liu}, \bibinfo{title}{Deep least-squares methods: an
  unsupervised learning-based numerical method for solving elliptic {PDE}s},
  \bibinfo{journal}{J. Comput. Phys.} \bibinfo{volume}{420}
  (\bibinfo{year}{2020}) \bibinfo{pages}{109707}.

\bibitem[{Calabro et~al.(2021)Calabro, Fabiani, and Siettos}]{CalabroFS2021}
\bibinfo{author}{F.~Calabro}, \bibinfo{author}{G.~Fabiani},
  \bibinfo{author}{C.~Siettos}, \bibinfo{title}{Extreme learning machine
  collocation for the numerical solution of elliptic {PDEs} with sharp
  gradients}, \bibinfo{journal}{Comput. Methods Appl. Mech. Engrg.}
  \bibinfo{volume}{387} (\bibinfo{year}{2021}) \bibinfo{pages}{114188}.

\bibitem[{Cuomo et~al.(2022)Cuomo, Schiano Di~Cola, Giampaolo, Rozza, Raissi,
  and Piccialli}]{Cuomo2022Scientific}
\bibinfo{author}{S.~Cuomo}, \bibinfo{author}{V.~Schiano Di~Cola},
  \bibinfo{author}{F.~Giampaolo}, \bibinfo{author}{G.~Rozza},
  \bibinfo{author}{M.~Raissi}, \bibinfo{author}{F.~Piccialli},
  \bibinfo{title}{Scientific machine learning through physics-informed neural
  networks: where we are and what's next}, \bibinfo{journal}{J. Sci. Comput.}
  \bibinfo{volume}{92}~(\bibinfo{number}{3}) (\bibinfo{year}{2022})
  \bibinfo{pages}{88}.

\bibitem[{Cyr et~al.(2020)Cyr, Gulian, Patel, Perego, and Trask}]{CyrGPPT2020}
\bibinfo{author}{E.~Cyr}, \bibinfo{author}{M.~Gulian},
  \bibinfo{author}{R.~Patel}, \bibinfo{author}{M.~Perego},
  \bibinfo{author}{N.~Trask}, \bibinfo{title}{Robust training and
  initialization of deep neural networks: {A}n adaptive basis viewpoint},
  \bibinfo{journal}{Proceedings of Machine Learning Research}
  \bibinfo{volume}{107} (\bibinfo{year}{2020}) \bibinfo{pages}{512--536}.

\bibitem[{Davis and Rabinowitz(2007)}]{DavisR2007}
\bibinfo{author}{P.~Davis}, \bibinfo{author}{P.~Rabinowitz},
  \bibinfo{title}{Methods of numerical integration}, \bibinfo{publisher}{Dover
  Publications, Inc}, \bibinfo{year}{2007}.

\bibitem[{De~Ryck et~al.(2023)De~Ryck, Jagtap, and Mishra}]{2023_IMA_Mishra_NS}
\bibinfo{author}{T.~De~Ryck}, \bibinfo{author}{A.~D. Jagtap},
  \bibinfo{author}{S.~Mishra}, \bibinfo{title}{Error estimates for
  physics-informed neural networks approximating the {Navier–Stokes}
  equations}, \bibinfo{journal}{IMA J. Numer. Anal.}
  \bibinfo{volume}{44}~(\bibinfo{number}{1}) (\bibinfo{year}{2023})
  \bibinfo{pages}{83--119}.

\bibitem[{De~Ryck et~al.(2021)De~Ryck, Lanthaler, and Mishra}]{DeRyck2021On}
\bibinfo{author}{T.~De~Ryck}, \bibinfo{author}{S.~Lanthaler},
  \bibinfo{author}{S.~Mishra}, \bibinfo{title}{On the approximation of
  functions by tanh neural networks}, \bibinfo{journal}{Neural Networks}
  \bibinfo{volume}{143} (\bibinfo{year}{2021}) \bibinfo{pages}{732--750}.

\bibitem[{Dong and Li(2021{\natexlab{a}})}]{2021_CMAME_LEMDD}
\bibinfo{author}{S.~Dong}, \bibinfo{author}{Z.~Li}, \bibinfo{title}{Local
  extreme learning machines and domain decomposition for solving linear and
  nonlinear partial differential equations}, \bibinfo{journal}{Comput. Methods
  Appl. Mech. Engrg.} \bibinfo{volume}{387}
  (\bibinfo{year}{2021}{\natexlab{a}}) \bibinfo{pages}{114129}.

\bibitem[{Dong and Li(2021{\natexlab{b}})}]{2021_JCP_Dong_modifiedbatch}
\bibinfo{author}{S.~Dong}, \bibinfo{author}{Z.~Li}, \bibinfo{title}{A modified
  batch intrinsic plasticity method for pre-training the random coefficients of
  extreme learning machines}, \bibinfo{journal}{J. Comput. Phys.}
  \bibinfo{volume}{445} (\bibinfo{year}{2021}{\natexlab{b}})
  \bibinfo{pages}{110585}.

\bibitem[{Dong and Ni(2021)}]{DongN2021}
\bibinfo{author}{S.~Dong}, \bibinfo{author}{N.~Ni}, \bibinfo{title}{A method
  for representing periodic functions and enforcing exactly periodic boundary
  conditions with deep neural networks}, \bibinfo{journal}{J. Comput. Phys.}
  \bibinfo{volume}{435} (\bibinfo{year}{2021}) \bibinfo{pages}{110242}.

\bibitem[{Dong and Wang(2023)}]{DongW2022}
\bibinfo{author}{S.~Dong}, \bibinfo{author}{Y.~Wang}, \bibinfo{title}{A method
  for computing inverse parametric {PDE} problems with random-weight neural
  networks}, \bibinfo{journal}{Journal of Computational Physics}
  \bibinfo{volume}{489} (\bibinfo{year}{2023}) \bibinfo{pages}{112263},
  \bibinfo{note}{(also arXiv:2210.04338)}.

\bibitem[{Dong and Yang(2022{\natexlab{a}})}]{DongY2022}
\bibinfo{author}{S.~Dong}, \bibinfo{author}{J.~Yang}, \bibinfo{title}{Numerical
  approximation of partial differential equations by a variable projection
  method with artificial neural networks}, \bibinfo{journal}{Comput. Methods
  Appl. Mech. Engrg.} \bibinfo{volume}{398}
  (\bibinfo{year}{2022}{\natexlab{a}}) \bibinfo{pages}{115284},
  \bibinfo{note}{(also arXiv:2201.09989)}.

\bibitem[{Dong and Yang(2022{\natexlab{b}})}]{DongY2022rm}
\bibinfo{author}{S.~Dong}, \bibinfo{author}{J.~Yang}, \bibinfo{title}{On
  computing the hyperparameter of extreme learning machines: algorithms and
  applications to computational {PDE}s, and comparison with classical and
  high-order finite elements}, \bibinfo{journal}{J. Comput. Phys.}
  \bibinfo{volume}{463} (\bibinfo{year}{2022}{\natexlab{b}})
  \bibinfo{pages}{111290}, \bibinfo{note}{(also arXiv:2110.14121)}.

\bibitem[{E and Yu(2018)}]{EY2018}
\bibinfo{author}{W.~E}, \bibinfo{author}{B.~Yu}, \bibinfo{title}{The deep
  {R}itz method: a deep learning-based numerical algorithm for solving
  variational problems}, \bibinfo{journal}{Commun. Math. Stat.}
  \bibinfo{volume}{6} (\bibinfo{year}{2018}) \bibinfo{pages}{1--12}.

\bibitem[{Fabiani et~al.(2021)Fabiani, Calabro, Russo, and
  Siettos}]{FabianiCRS2021}
\bibinfo{author}{G.~Fabiani}, \bibinfo{author}{F.~Calabro},
  \bibinfo{author}{L.~Russo}, \bibinfo{author}{C.~Siettos},
  \bibinfo{title}{Numerical solution and bifurcation analysis of nonlinear
  partial differential equations with extreme learning machines},
  \bibinfo{journal}{J. Sci. Comput.} \bibinfo{volume}{89}
  (\bibinfo{year}{2021}) \bibinfo{pages}{44}.

\bibitem[{Gao and Zakharian(????)}]{Gao2305.11915}
\bibinfo{author}{J.~Gao}, \bibinfo{author}{Y.~Zakharian}, \bibinfo{title}{PINNs
  error estimates for nonlinear equations in R-smooth Banach spaces},
  \bibinfo{journal}{arXiv:2305.11915} .

\bibitem[{He and Xu(2019)}]{HeX2019}
\bibinfo{author}{J.~He}, \bibinfo{author}{J.~Xu}, \bibinfo{title}{{MgNet}: A
  unified framework for multigrid and convolutional neural network},
  \bibinfo{journal}{Sci. China Math.} \bibinfo{volume}{62}
  (\bibinfo{year}{2019}) \bibinfo{pages}{1331--1354}.

\bibitem[{Hu et~al.(2023)Hu, Lin, Raydan, and Tang}]{Ruimeng2209.11929}
\bibinfo{author}{R.~Hu}, \bibinfo{author}{Q.~Lin}, \bibinfo{author}{A.~Raydan},
  \bibinfo{author}{S.~Tang}, \bibinfo{title}{Higher-order error estimates for
  physics-informed neural networks approximating the primitive equations},
  \bibinfo{journal}{Partial Differ. Equ. Appl.}
  \bibinfo{volume}{4}~(\bibinfo{number}{4}) (\bibinfo{year}{2023})
  \bibinfo{pages}{No. 34, 35}.

\bibitem[{Hu et~al.(2022{\natexlab{a}})Hu, Jagtap, Karniadakis, and
  Kawaguchi}]{Hu2022XPINN}
\bibinfo{author}{Z.~Hu}, \bibinfo{author}{A.~D. Jagtap}, \bibinfo{author}{G.~E.
  Karniadakis}, \bibinfo{author}{K.~Kawaguchi}, \bibinfo{title}{When do
  extended physics-informed neural networks ({XPINN}s) improve
  generalization?}, \bibinfo{journal}{SIAM J. Sci. Comput.}
  \bibinfo{volume}{44}~(\bibinfo{number}{5})
  (\bibinfo{year}{2022}{\natexlab{a}}) \bibinfo{pages}{A3158--A3182}.

\bibitem[{Hu et~al.(2022{\natexlab{b}})Hu, Liu, Wang, and Xu}]{HuLWX2022}
\bibinfo{author}{Z.~Hu}, \bibinfo{author}{C.~Liu}, \bibinfo{author}{Y.~Wang},
  \bibinfo{author}{Z.~Xu}, \bibinfo{title}{Energetic variational neural network
  discretizations to gradient flows}, \bibinfo{journal}{arXiv:2206.07303} .

\bibitem[{Jagtap and Karniadakis(2020)}]{JagtapK2020}
\bibinfo{author}{A.~Jagtap}, \bibinfo{author}{G.~Karniadakis},
  \bibinfo{title}{Extended physics-informed neural network ({XPINNs}): A
  generalized space-time domain decomposition based deep learning framework for
  nonlinear partial differential equations}, \bibinfo{journal}{Commun. Comput.
  Phys.} \bibinfo{volume}{28} (\bibinfo{year}{2020})
  \bibinfo{pages}{2002--2041}.

\bibitem[{Jagtap et~al.(2020)Jagtap, Kharazmi, and Karniadakis}]{JagtapKK2020}
\bibinfo{author}{A.~Jagtap}, \bibinfo{author}{E.~Kharazmi},
  \bibinfo{author}{G.~Karniadakis}, \bibinfo{title}{Conservative
  physics-informed neural networks on discrete domains for conservation laws:
  applications to forward and inverse problems}, \bibinfo{journal}{Comput.
  Methods Appl. Mech. Engrg.} \bibinfo{volume}{365} (\bibinfo{year}{2020})
  \bibinfo{pages}{113028}.

\bibitem[{Karniadakis et~al.(2021)Karniadakis, Kevrekidis, Lu, Perdikaris,
  Wang, and Yang}]{Karniadakisetal2021}
\bibinfo{author}{G.~Karniadakis}, \bibinfo{author}{G.~Kevrekidis},
  \bibinfo{author}{L.~Lu}, \bibinfo{author}{P.~Perdikaris},
  \bibinfo{author}{S.~Wang}, \bibinfo{author}{L.~Yang},
  \bibinfo{title}{Physics-Informed Machine Learning}, \bibinfo{journal}{Nat.
  Rev. Phys.} \bibinfo{volume}{3} (\bibinfo{year}{2021})
  \bibinfo{pages}{422--440}.

\bibitem[{Krishnapriyan et~al.(2021)Krishnapriyan, Gholami, Zhe, Kirby, and
  Mahoney}]{KrishnapriyanGZKM2021}
\bibinfo{author}{A.~Krishnapriyan}, \bibinfo{author}{A.~Gholami},
  \bibinfo{author}{S.~Zhe}, \bibinfo{author}{R.~Kirby},
  \bibinfo{author}{M.~Mahoney}, \bibinfo{title}{Characterizing possible failure
  modes in physics-informed neural networks}, \bibinfo{journal}{Advances in
  Neural Information Processing Systems} \bibinfo{volume}{34}
  (\bibinfo{year}{2021}) \bibinfo{pages}{26548--26560}.

\bibitem[{LeCun et~al.(2015)LeCun, Bengio, and Hinton}]{LeCun2015DP}
\bibinfo{author}{Y.~LeCun}, \bibinfo{author}{Y.~Bengio},
  \bibinfo{author}{G.~Hinton}, \bibinfo{title}{Deep learning},
  \bibinfo{journal}{Nature} \bibinfo{volume}{521} (\bibinfo{year}{2015})
  \bibinfo{pages}{436--444}.

\bibitem[{Lu et~al.(2021)Lu, Meng, Mao, and Karniadakis}]{Lu2021DeepXDE}
\bibinfo{author}{L.~Lu}, \bibinfo{author}{X.~Meng}, \bibinfo{author}{Z.~Mao},
  \bibinfo{author}{G.~E. Karniadakis}, \bibinfo{title}{Deep{XDE}: a deep
  learning library for solving differential equations}, \bibinfo{journal}{SIAM
  Rev.} \bibinfo{volume}{63}~(\bibinfo{number}{1}) (\bibinfo{year}{2021})
  \bibinfo{pages}{208--228}.

\bibitem[{Mishra and Molinaro(2021)}]{Mishra2021pinn}
\bibinfo{author}{S.~Mishra}, \bibinfo{author}{R.~Molinaro},
  \bibinfo{title}{Physics informed neural networks for simulating radiative
  transfer}, \bibinfo{journal}{J. Quant. Spectrosc. Radiat. Transfer}
  \bibinfo{volume}{270} (\bibinfo{year}{2021}) \bibinfo{pages}{107705}.

\bibitem[{Mishra and Molinaro(2022)}]{Mishra2022inverse}
\bibinfo{author}{S.~Mishra}, \bibinfo{author}{R.~Molinaro},
  \bibinfo{title}{Estimates on the generalization error of physics-informed
  neural networks for approximating a class of inverse problems for {PDE}s},
  \bibinfo{journal}{IMA J. Numer. Anal.}
  \bibinfo{volume}{42}~(\bibinfo{number}{2}) (\bibinfo{year}{2022})
  \bibinfo{pages}{981--1022}.

\bibitem[{Mishra and Molinaro(2023)}]{Mishra2022Estimates}
\bibinfo{author}{S.~Mishra}, \bibinfo{author}{R.~Molinaro},
  \bibinfo{title}{Estimates on the generalization error of physics-informed
  neural networks for approximating {PDE}s}, \bibinfo{journal}{IMA J. Numer.
  Anal.} \bibinfo{volume}{43}~(\bibinfo{number}{1}) (\bibinfo{year}{2023})
  \bibinfo{pages}{1--43}.

\bibitem[{Ni and Dong(2023)}]{NiDong_hLC_2023}
\bibinfo{author}{N.~Ni}, \bibinfo{author}{S.~Dong}, \bibinfo{title}{Numerical
  computation of partial differential equations by hidden-layer concatenated
  extreme learning machine}, \bibinfo{journal}{J. Sci. Comput.}
  \bibinfo{volume}{95}~(\bibinfo{number}{2}) (\bibinfo{year}{2023})
  \bibinfo{pages}{35}.

\bibitem[{Pantidis et~al.(2023)Pantidis, Eldababy, Tagle, and
  Mobasher}]{Panos2023IFENN}
\bibinfo{author}{P.~Pantidis}, \bibinfo{author}{H.~Eldababy},
  \bibinfo{author}{C.~M. Tagle}, \bibinfo{author}{M.~E. Mobasher},
  \bibinfo{title}{Error convergence and engineering-guided hyperparameter
  search of {PINN}s: {T}owards optimized {I}-{FENN} performance},
  \bibinfo{journal}{Comput. Methods Appl. Mech. Engrg.} \bibinfo{volume}{414}
  (\bibinfo{year}{2023}) \bibinfo{pages}{116160}.

\bibitem[{Penwarden et~al.(2023)Penwarden, Jagtap, Zhe, Karniadakis, and
  Kirby}]{Penwardenetal2023}
\bibinfo{author}{M.~Penwarden}, \bibinfo{author}{A.~D. Jagtap},
  \bibinfo{author}{S.~Zhe}, \bibinfo{author}{G.~E. Karniadakis},
  \bibinfo{author}{R.~M. Kirby}, \bibinfo{title}{A unified scalable framework
  for causal sweeping strategies for physics-informed neural networks ({PINN}s)
  and their temporal decompositions}, \bibinfo{journal}{J. Comput. Phys.}
  \bibinfo{volume}{493} (\bibinfo{year}{2023}) \bibinfo{pages}{112464}.

\bibitem[{Qian et~al.(2023)Qian, Zhang, Huang, and Dong}]{Qian2303.12245}
\bibinfo{author}{Y.~Qian}, \bibinfo{author}{Y.~Zhang},
  \bibinfo{author}{Y.~Huang}, \bibinfo{author}{S.~Dong},
  \bibinfo{title}{Physics-informed neural networks for approximating dynamic
  (Hyperbolic) PDEs of second order in time: Error analysis and numerical
  algorithms}, \bibinfo{journal}{J. Comput. Phys.} \bibinfo{volume}{495}
  (\bibinfo{year}{2023}) \bibinfo{pages}{112527}.

\bibitem[{Raissi et~al.(2019)Raissi, Perdikaris, and
  Karniadakis}]{Raissi2019pinn}
\bibinfo{author}{M.~Raissi}, \bibinfo{author}{P.~Perdikaris},
  \bibinfo{author}{G.~E. Karniadakis}, \bibinfo{title}{Physics-informed neural
  networks: a deep learning framework for solving forward and inverse problems
  involving nonlinear partial differential equations}, \bibinfo{journal}{J.
  Comput. Phys.} \bibinfo{volume}{378} (\bibinfo{year}{2019})
  \bibinfo{pages}{686--707}.

\bibitem[{Shin et~al.(2020)Shin, Darbon, and Karniadakis}]{Shin2020On}
\bibinfo{author}{Y.~Shin}, \bibinfo{author}{J.~Darbon}, \bibinfo{author}{G.~E.
  Karniadakis}, \bibinfo{title}{On the convergence of physics informed neural
  networks for linear second-order elliptic and parabolic type {PDE}s},
  \bibinfo{journal}{Commun. Comput. Phys.}
  \bibinfo{volume}{28}~(\bibinfo{number}{5}) (\bibinfo{year}{2020})
  \bibinfo{pages}{2042--2074}.

\bibitem[{Shin et~al.(2023)Shin, Zhang, and Karniadakis}]{Shin2010.08019}
\bibinfo{author}{Y.~Shin}, \bibinfo{author}{Z.~Zhang}, \bibinfo{author}{G.~E.
  Karniadakis}, \bibinfo{title}{Error estimates of residual minimization using
  neural networks for linear {PDE}s}, \bibinfo{journal}{J. Mech. Learn. Model.
  Comput.} \bibinfo{volume}{4}~(\bibinfo{number}{4}) (\bibinfo{year}{2023})
  \bibinfo{pages}{73--101}.

\bibitem[{Siegel et~al.(2023)Siegel, Hong, Jin, Hao, and Xu}]{Siegeletal2022}
\bibinfo{author}{J.~W. Siegel}, \bibinfo{author}{Q.~Hong},
  \bibinfo{author}{X.~Jin}, \bibinfo{author}{W.~Hao}, \bibinfo{author}{J.~Xu},
  \bibinfo{title}{Greedy training algorithms for neural networks and
  applications to {PDE}s}, \bibinfo{journal}{J. Comput. Phys.}
  \bibinfo{volume}{484} (\bibinfo{year}{2023}) \bibinfo{pages}{No. 112084, 27}.

\bibitem[{Sirignano and Spoliopoulos(2018)}]{SirignanoS2018}
\bibinfo{author}{J.~Sirignano}, \bibinfo{author}{K.~Spoliopoulos},
  \bibinfo{title}{{DGM}: A deep learning algorithm for solving partial
  differential equations}, \bibinfo{journal}{J. Comput. Phys.}
  \bibinfo{volume}{375} (\bibinfo{year}{2018}) \bibinfo{pages}{1339--1364}.

\bibitem[{Tartakovsky et~al.(2020)Tartakovsky, Marrero, Perdikaris,
  Tartakovsky, and Barajas-Solano}]{Tartakovskyetal2020}
\bibinfo{author}{A.~Tartakovsky}, \bibinfo{author}{C.~Marrero},
  \bibinfo{author}{P.~Perdikaris}, \bibinfo{author}{G.~Tartakovsky},
  \bibinfo{author}{D.~Barajas-Solano}, \bibinfo{title}{Physics-informed deep
  neural networks for learning parameters and constitutive relationships in
  subsurface flow problems}, \bibinfo{journal}{Water Resour. Res.}
  \bibinfo{volume}{56} (\bibinfo{year}{2020}) \bibinfo{pages}{e2019WR026731}.

\bibitem[{Wan and Wei(2022)}]{WanW2022}
\bibinfo{author}{X.~Wan}, \bibinfo{author}{S.~Wei}, \bibinfo{title}{{VAE-KRnet}
  and its applications to variational {B}ayes}, \bibinfo{journal}{Commun.
  Comput. Phys.} \bibinfo{volume}{31} (\bibinfo{year}{2022})
  \bibinfo{pages}{1049--1082}.

\bibitem[{Wang et~al.(2022)Wang, Yu, and Perdikaris}]{WangYP2022}
\bibinfo{author}{S.~Wang}, \bibinfo{author}{X.~Yu},
  \bibinfo{author}{P.~Perdikaris}, \bibinfo{title}{When and why {PINN}s fail to
  train: a neural tangent kernel perspective}, \bibinfo{journal}{J. Comput.
  Phys.} \bibinfo{volume}{449} (\bibinfo{year}{2022}) \bibinfo{pages}{110768}.

\bibitem[{Wang and Dong(2024)}]{WangD2024}
\bibinfo{author}{Y.~Wang}, \bibinfo{author}{S.~Dong}, \bibinfo{title}{An
  extreme learning machine-based method for computational PDEs in higher
  dimensions}, \bibinfo{journal}{Comput. Methods Appl. Mech. Engrg.}
  \bibinfo{volume}{418} (\bibinfo{year}{2024}) \bibinfo{pages}{116578}.

\bibitem[{Wang and Lin(2020)}]{WangL2020}
\bibinfo{author}{Y.~Wang}, \bibinfo{author}{G.~Lin}, \bibinfo{title}{Efficient
  deep learning techniques for multiphase flow simulation in heterogeneous
  porous media}, \bibinfo{journal}{J. Comput. Phys.} \bibinfo{volume}{401}
  (\bibinfo{year}{2020}) \bibinfo{pages}{108968}.

\bibitem[{Zerbinati(2022)}]{Zerbinati2022pinns}
\bibinfo{author}{U.~Zerbinati}, \bibinfo{title}{{PINN}s and {GaLS}: A Priori
  error estimates for shallow physics informed neural networks applied to
  elliptic problems}, \bibinfo{journal}{IFAC-PapersOnLine}
  \bibinfo{volume}{55}~(\bibinfo{number}{20}) (\bibinfo{year}{2022})
  \bibinfo{pages}{61--66}.

\end{thebibliography}
